\documentclass{amsart}[11pt]

\usepackage[margin=1in]{geometry}

\usepackage{amsmath,amsthm,amssymb, amsfonts,mathtools}
\usepackage{latexsym}
\usepackage{url}
\usepackage{xcolor}
\usepackage{graphicx}
\usepackage{float}
\usepackage{caption}
\usepackage{subcaption}
\usepackage{verbatim}
\usepackage{enumitem}
\usepackage{amsrefs}
\usepackage{array}
\usepackage{booktabs}
\usepackage{color}
\usepackage[utf8]{inputenc}
\usepackage{tcolorbox}
\usepackage{bbm}
\usepackage{commath}
\usepackage{enumitem}
\usepackage{hyperref}

\newcommand{\R}{\mathbb{R}}
\newcommand{\E}{\mathbb{E}}
\newcommand{\F}{\mathcal{F}}
\newcommand{\Prob}{\mathbb{P}}
\newcommand{\1}{\mathbbm{1}}
\newcommand{\bb}[1]{\mathbb{#1}}
\newcommand{\br}[1]{\lbrace #1 \rbrace}

\newcommand{\toinf}{\rightarrow\infty}
\newcommand{\tto}{\rightarrow}

\newcommand{\W}{\Omega}
\newcommand{\mc}[1]{\mathcal{ #1}}

\newcommand{\red}[1]{\textcolor{red}{#1}}

\DefineSimpleKey{bib}{arxiveprint}
\DefineSimpleKey{bib}{arxivid}
\DefineSimpleKey{bib}{arxivclass}
\makeatletter
\catcode`\'=11 % letter; it was originally 12
  \def\arxiveprint{%
    \resolve@inner{\bib@arxiveprint}
  }
  \def\bib@arxiveprint#1{%
    \begingroup
        #1\relax
        \bib@resolve@xrefs
        \bib@field@patches
        \bib'setup
        \let\PrintPrimary\@empty
        {%
          \IfEmptyBibField{arxivid}{\url{https://arxiv.org/}}
          {%
            \href{https://arxiv.org/abs/\bib'arxivid}{\nolinkurl{arXiv:\bib'arxivid}}%
            \IfEmptyBibField{arxivclass}{}{~\nolinkurl{[\bib'arxivclass]}}% \nopunct
          }
        }\bib'transition
        \setbib@@
    \endgroup
  }
\catcode`\'=12 % other
\makeatother
\BibSpec{arxiv}{%
    +{}  {\PrintAuthors}                {author}
    +{,} { \textit}                     {title}
    +{.} { }                            {part}
    +{:} { \textit}                     {subtitle}
    +{,} { \PrintContributions}         {contribution}
    +{.} { \PrintPartials}              {partial}
    +{,} { }                            {journal}
    +{}  { \textbf}                     {volume}
    +{}  { \PrintDatePV}                {date}
    +{,} { \issuetext}                  {number}
    +{,} { \eprintpages}                {pages}
    +{,} { }                            {status}
    +{,} { \PrintDOI}                   {doi}
    +{,} { available at \eprint}        {eprint}
    +{,} { available at \arxiveprint}   {arxiveprint}
    +{}  { \parenthesize}               {language}
    +{}  { \PrintTranslation}           {translation}
    +{;} { \PrintReprint}               {reprint}
    +{.} { }                            {note}
    +{.} {}                             {transition}
    +{}  {\SentenceSpace \PrintReviews} {review}
}

\newtheorem{thm}{Theorem}[section]
\newtheorem{theo}[thm]{Theorem}
\newtheorem{corollary}[thm]{Corollary}
\newtheorem{lem}[thm]{Lemma}	
\newtheorem{proposition}[thm]{Proposition}

\newtheorem{defi}[thm]{Definition}

\theoremstyle{remark}
\newtheorem{remark}[thm]{Remark}

\newtheorem*{lemma*}{Lemma}

\author{Z.W. Bezemek, and K. Spiliopoulos}
\address{Boston University, Department of Mathematics and Statistics\\ 111 Cummington Mall, Boston, MA 02215, USA}
\email[Zachary William Bezemek]{bezemek@bu.edu}
\email[Konstantinos Spiliopoulos]{kspiliop@bu.edu}
\thanks{This work has been partially supported by the National Science Foundation (DMS 1550918, DMS 2107856) and Simons Foundation Award  672441. We would like to specially thank the referee of this paper for a very careful and thorough review of the paper.
}
\title{Large deviations for interacting multiscale particle systems}
\date{\today}
\allowdisplaybreaks

\begin{document}
\begin{abstract}
We consider a collection of weakly interacting diffusion processes moving in a two-scale locally periodic environment. We study the large deviations principle of the empirical distribution of the particles' positions in the combined limit as the number of particles grow to infinity  and the time-scale separation parameter goes to zero. We make use of weak convergence methods providing a convenient representation for the large deviations rate function, which allow us to characterize the effective controlled mean field dynamics. In addition, we rigorously obtain equivalent non-variational representations for the large deviations rate function as introduced by Dawson-G\"artner.
\end{abstract}

\maketitle
\section{Introduction}
The goal of this article is to obtain the large deviations principle (LDP) for interacting particle systems of  diffusion type in multiscale environments.  We use methods from weak convergence and stochastic control,
\cite{DE}, ultimately making connections with mean field stochastic control problems
\cite{CD}.

In particular, we consider on some filtered probability space satisfying the usual conditions $(\W,\bb{F},\Prob),\br{\F_t}_{t\in[0,1]}$ the interacting particle system
\begin{align}
\label{eq:multimeanfield}
dX^{i,N}_t &= \left[\frac{1}{\epsilon}f(X_t^{i,N},X_t^{i,N}/\epsilon,\mu^N_t)+b(X^{i,N}_t,X^{i,N}_t/\epsilon,\mu^N_t)\right]dt+\sigma(X^{i,N}_t,X^{i,N}_t/\epsilon,\mu^N_t)dW^i_t,\\
X^{i,N}_0&=x^{i,N}\nonumber
\end{align}
where $t\in[0,1]$, $W^i_t,i=1,...,N$ are $m$-dimensional independent $\F_t$-Brownian motions, $$\mu^{N}_t(\omega)\coloneqq \frac{1}{N}\sum_{i=1}^N \delta_{X^{i,N}_t(\omega)},$$ $X_t^{i,N}(\omega),f(x,y,\mu),b(x,y,\mu)\in\R^d$, $\sigma(x,y,\mu)\in \R^{d\times m}$ and all coefficients are $1-$periodic in the second coordinate. Suppose also that $\epsilon>0,N\in\bb{N}$ and $\epsilon(N)\tto 0$ as $N\toinf$.

Our goal is to obtain the large deviations principle for the measure-valued process $\left\{\mu^{N}_t,t\in [0,1]\right\}_{N\in\mathbb{N}}$ in the combined limit  $N\rightarrow\infty$ and $\epsilon\downarrow 0$. Here, $\epsilon$ is the time scale separation parameter. One can regard $X^{i,N}$ as the slow $i^{th}$ component and $Y^{i,N}=X^{i,N}/\epsilon$ as the fast $i^{th}$ component.

Systems of interacting diffusions arise in many areas of science, finance and engineering, see for example \cites{BinneyTremaine,Garnier1, Garnier2, IssacsonMS,Lucon2016,MotschTadmor2014} to name just a few. On the other hand, diffusions in multiscale environments are also common in many applications ranging from chemical physics to finance and climate modeling, see for example \cites{Ansari,BryngelsonOnuchicWolynes, feng2010short,feng2012small,jean2000derivatives, HyeonThirumalai, majda2008applied,
 Zwanzig} for a representative, but by no means complete, list. Our goal in this paper is to study the combined effect of weak mean field interactions in a fast oscillating multiscale environment from the point of view of large deviations for the empirical measure of the particles.

In the case $\epsilon=1$, i.e. in the absence of multiple scales, the limiting problem of $N\rightarrow\infty$ has been very well studied in the literature. Typical behavior, fluctuations, as well as large deviations have been obtained, see for example \cites{Dawson,DG,BDF} for related classical works. Analogously, if $N=1$, i.e., in the single particle case, the limiting behavior as $\epsilon\downarrow 0$ has also been extensively studied in the literature under various modeling assumptions, see for example \cites{Baldi,BorkarGaitsgory,DS,FS, GaitsgoryNguyen,  Lipster,PV1,PV2, Spiliopoulos2013a, Spiliopoulos2014Fluctuations, Spiliopoulos2013QuenchedLDP, Veretennikov, VeretennikovSPA2000,HLLS} and the references therein. In this paper, we study the combined limit as $N\rightarrow\infty$ and $\epsilon(N)\downarrow 0$.  The main result of the paper is Theorem \ref{thm:LaplacePrinciple} (see also Theorem \ref{theorem:DGform}) that gives the large deviations principle of the empirical distribution of the particles in the combined limit $N\rightarrow\infty$ and $\epsilon\downarrow 0$. As a byproduct we also obtain in Theorem \ref{thm:propofchaos} the typical behavior, i.e. the law of large numbers.  We use weak convergence methods of
\cite{DE} which leads to the study of related mean field stochastic control problems \cites{CD, Lacker,Fischer}.
In addition, in Subsection \ref{subsec:DGequivalentform}, we connect the variational form of the action functional that we obtain in Theorem \ref{thm:LaplacePrinciple} with the expected ``dual'' form based on the classical work of Dawson and G\"{a}rtner \cite{DG} in the $\epsilon=1$ case and \cite{DS} in the $N=1$ case. As far as we know, this is the first time that the connection of the variational form of the large deviations rate function for the empirical measures as proved in \cite{BDF} in the setting without multiscale structure and that found in \cite{DG} has been rigorously established. This formulation of the rate function has many parallels to the corresponding form of the rate function for small noise diffusions in both the case with and without multiscale structure, as discussed in Remark \ref{remark:comparingwiththesmallnoisecase}. We expect that employing this exciting duality between the rate functions will open the doors to studying the dynamical effects of multiscale structure on phase transitions and exit times from basins of attraction for the empirical measures of weakly interacting diffusions, see \cites{Dawson,DGbook}, and allow for the design provably optimal importance sampling schemes for functionals of the empirical measure in the multiscale and non-multiscale settings \cites{Spiliopoulos2013a,DSX,AS}.

As an example, we consider in Section \ref{S:DawsonExample} a class of examples in which particles diffuse in a rough confining potential and interact through a general interaction potential. These examples are motivated by the seminal work of Dawson in \cite{Dawson}, and in Remark \ref{remark:unboundedconfiningandinteraction} we discuss how to verify that even a rough version of the system of \cite{Dawson} where the confining potential is bi-stable and unbounded and the interaction potential is of Curie-Weiss form can be seen to satisfy the large deviations principle proved here.

To our knowledge, this is the first large deviations result for the combined $\epsilon\downarrow 0$ and $N\toinf$ limit. Some similar results include the proof of an averaging principle for slow-fast McKean-Vlasov SDEs found in \cite{Rockner}, i.e. the $\epsilon\downarrow 0$ limit for a system of the type we get after $N\toinf$. There is also the result of \cite{gammaconv}, wherein the object of study is $J^\epsilon$ which corresponds to the large deviations rate functional with rate $N$ for the empirical density of a multiscale interacting particle system similar to Equation \eqref{eq:multimeanfield} with $\sigma$ independent of $\mu$ and with $\epsilon>0$ fixed, which is known as per the results of \cite{DG}. In our setting with $\sigma=I$, $J^\epsilon:C([0,T];\mc{P}(\R^d))\tto [0,+\infty]$ would be given by
\begin{align*}
J^\epsilon(\theta) &= \frac{1}{2}\int_0^1 \sup_{g\in C^\infty_c(\R^d):\langle |\nabla g|^2,\theta(t)\rangle \neq 0} \frac{|\langle g,\dot{\theta}(t)-[\mc{L}^\epsilon(\theta(t))]^*\theta\rangle|^2}{\langle |\nabla g|^2,\theta(t)\rangle}dt\\
\mc{L}^\epsilon(\theta(t))[g](x)& = \frac{1}{\epsilon}f(x,x/\epsilon, \theta(t))\cdot \nabla g(x) + b(x,x/\epsilon, \theta(t))\cdot \nabla g(x)+\frac{1}{2}\Delta g(x),
\end{align*}
(see the notation in our Theorem \ref{theorem:DGform}), though their system and setup are different. They are able to prove $\Gamma$-convergence of the sequence $\br{J^\epsilon}_{\epsilon>0}$ to a functional $J$ as $\epsilon\downarrow 0$, in some sense establishing an averaging principle for the empirical density a system of mean field multiscale diffusions at the level of large deviations. Lastly, in \cite{delgadino2020}, a result similar to Theorem \ref{thm:propofchaos} appears (only typical behavior, not LDP). A key difference between the regime of \cite{delgadino2020} and the regime of our paper is that rather than depending on the slow process $X^{i,N}_t$, the fast process $X^{i,N}_t/\epsilon$, and the empirical measure $\mu^N_t$, their coefficients depend on the fast process $X^{i,N}_t/\epsilon$ and the ``fast empirical measure'' $\mu^{N,\epsilon}_t \coloneqq \frac{1}{N}\sum_{i=1}^N \delta_{X^{i,N}_t/\epsilon}$. As a result, the invariant measure $\pi$ (see Equation \eqref{eq:pi}) depends not on the parameter $\mu =\mc{L}(\bar{X}_{\cdot})$ in the limit, as in our regime, but implicitly on itself as $\mu =\pi$. Consequently, in \cite{delgadino2020}, multiple steady states can exist, potentially affecting the way in which the limits $\epsilon\rightarrow 0$ and $N\rightarrow\infty$ interact. We discuss this further in our conclusion Section \ref{S:Conclusions}.

The rest of the  paper is organized as follows. In Section \ref{S:Assumptions}, we lay out notation and main assumptions in regard to the model (\ref{eq:multimeanfield}). In addition, we introduce the corresponding controlled particle system and controlled McKean-Vlasov process which will be crucial components of the large deviations analysis. In Section \ref{S:MainResult} we present our main result on large deviations for the measure-valued process $\left\{\mu^{N}_t,t\in [0,1]\right\}_{N\in\mathbb{N}}$ in the combined limit  $N\rightarrow\infty$ and $\epsilon\downarrow 0$. Section \ref{S:DawsonExample} discusses a class of physically motivated examples which take the form of aggregation-diffusion equations. Section \ref{section:connectiontootheratefunctions} connects the obtained Laplace principle with other classical works in the literature, i.e. the LDP in the $\epsilon=1$ case of \cite{DG} and the LDP in the $N=1$ case of \cite{DS}, establishes an alternative variational form of the rate function provided in Theorem \ref{thm:LaplacePrinciple}, and establishes a non-variational, ``negative-Sobolev'' form of the rate function in Theorem \ref{theorem:DGform}.  In Section \ref{sec:limitingbehavior} we discuss the limiting behavior of the controlled particle system, proving tightness and identifying the limit. In Sections \ref{sec:lowerbound} and \ref{sec:upperbound} we prove the Laplace principle (which is equivalent to the large deviations principle) lower and upper bounds respectively. Compactness of the level sets of the rate function is proven in Section \ref{sec:compactlevelsets}. In Appendices \ref{Appendix:Prelimit} and \ref{Appendix:L1}, we discuss  technical preliminary results that are used in various places of the paper. For purposes of self containment and for the reader's convenience, Appendix \ref{Appendix:LionsDifferentiation} reviews the necessary material from Lions differentiation.  Section \ref{S:Conclusions} has our conclusions and directions for future work.

\section{Notation, Assumptions, and the Controlled McKean-Vlasov process}\label{S:Assumptions}
For $S$ a Polish space, we will use $C([0,1];\mc{S})$ to denote the space of continuous functions from $[0,1]$ to $\mc{S}$, equipped with the topology of uniform convergence. A useful fact is that $C([0,1];\R^d)$ with the previously described topology is a Polish space (see \cite{DE} Theorem A.6.5). $\mc{M}(\mc{S}_1;\mc{S}_2)$ will denote the space of Borel-measurable functions $g:\mc{S}_1\tto \mc{S}_2$ for Polish spaces $\mc{S}_1,\mc{S}_2$. We will use $C_b(\mc{S})$ to denote the space of continuous, bounded functions $B:\mc{S}\tto \R$, and let $\norm{B}_\infty\coloneqq \sup_{x\in \mc{S}}|B(x)|$. In addition, we use $C_{b,L}(\mc{S})$ to denote the space of bounded, Lipschitz functions $B:\mc{S}\tto \R$. We use $C_b^k(\R^d)$ to denote the space of continuous, bounded functions $B:\R^d\tto \R$ with $k$ continuous, bounded derivatives. We use $C_c^\infty(\R^d)$ to denote the space of continuous, infinitely differentiable functions $B:\R^d\tto \R$ with compact support.  These spaces are defined in the same way when $\R^d$ is replaced by the $d$-dimensional unit torus $\bb{T}^d$. $L^2(\mc{S},\mu;\R^k)$, where $\mu$ is a measure on $\mc{S}$ will denote the class of functions $B:\mc{S}\tto \R^k$ such that $\norm{B}_{L^2(\mc{S},\mu)}\coloneqq \biggl(\int_\mc{S} |B(x)|^2\mu(dx) \biggr)^{1/2}<\infty$. We may omit the codomain in this notation when convenient. We will also at times denote $L^2(\R^d\times\R^d,\mu\otimes\mu)$ by $L^2(\R^d,\mu)\otimes L^2(\R^d,\mu)$. $\mc{P}(\mc{S})$ will denote the space of probability measures on the Borel $\sigma$-field $\mc{B}(\mc{S})$, where open sets are induced by the metric on $\mc{S}$. $\mc{P}(\mc{S})$ is given the topology of weak convergence and Prokhorov's metric, and is itself a Polish space (\cite{EK} Theorem 3.1.7). $\mc{P}_2(\mc{S})\subset \mc{P}(\mc{S})$ will denote the set of square integrable measures on $\mc{S}$. It is given the $L^2$-Wasserstein distance (see Definition \ref{def:lionderivative}) as its metric and is also a Polish space (\cite{CD} p.360). Given a random variable $\eta$, $\mc{L}(\eta)$ will denote the distribution of $\eta$. For a function $\phi:\mc{S}\tto \R^d$ which is integrable with respect to $\mu\in \mc{P}(\mc{S})$, we will denote $\langle \mu, \phi\rangle \coloneqq \int_{\mc{S}} \phi(x) \mu(dx)$.

Assume the following:
\begin{enumerate}[label=(A\arabic*)]
\item \label{assumption:initialconditions}For some $\nu_0\in\mc{P}_2(\R^d),\frac{1}{N}\sum_{i=1}^N\delta_{x^{i,N}}\tto \nu_0$ in $\mc{P}_2(\R^d)$ as $N\toinf$, where $x^{i,N}$ are the initial conditions from Equation \eqref{eq:multimeanfield}.
\item \label{assumption:LipschitzandBounded} There exists $L\in(0,\infty)$ such that for $x_1,x_2\in\R^d,y_1,y_2\in\bb{T}^d,\mu_1,\mu_2\in\mc{P}_2(\R^d)$, $$|g(x_1,y_1,\mu_1)-g(x_2,y_2,\mu_2)|\leq L\biggl(|x_1-x_2|+|y_1-y_2|+\bb{W}_2(\mu_1,\mu_2)\biggr),$$ where $g=f,b,$ or $\sigma$ and $\bb{W}_2$ is the $L^2$-Wasserstein distance (see Definition \ref{def:lionderivative}). In addition, $f,b,$ and $\sigma$ are bounded and jointly continuous on $\R^d\times \bb{T}^d\times \mc{P}(\R^d)$.
\item \label{assumption:uniformellipticity} For $A=\sigma\sigma^{\top}$ there exists $\lambda_1>0$ such that uniformly in $x\in\R^d,y\in\bb{T}^d,\mu\in\mc{P}_2(\R^d)$,
\begin{align*}
\xi^\top A(x,y,\mu)\xi\geq \lambda_1 |\xi|^2,\forall \xi\in\R^d.
\end{align*}
\item \label{assumption:fsigmaregularity} For $g=f$ or $\sigma$, $g$,$\nabla_x g$, and $\nabla_x\nabla_x g$ exist and are uniformly bounded. Moreover, for each $x\in\R^d,y\in\bb{T}^d$, $g(x,y,\cdot)$ is Fully $C^2$ and $\nabla_x g$ is $C^1$ in the sense of Lions differentiation (see Definition \ref{def:fullyC2}), and $\partial_\mu g(x,y,\mu)(\cdot),$ $\nabla_x\partial_\mu g(x,y,\mu)(\cdot),$ and $\partial_v\partial_\mu g(x,y,\mu)(\cdot)$ are bounded in $L^2(\R^d,\mu)$ uniformly in $x,y,$ and $\mu$ and $\partial^2_\mu g(x,y,\mu)(\cdot,\cdot)$ is bounded in $L^2(\R^d,\mu)\otimes L^2(\R^d,\mu)$ uniformly in $x,y,$ and $\mu$. All the first and second derivatives of $f$ and $\sigma$ listed in here are H\"oldarian in $y$ uniformly in $x$ and $\mu$, and jointly continuous on $\R^d\times \bb{T}^d\times \mc{P}(\R^d)$.
\end{enumerate}

Note that many of these assumptions can be relaxed, including the boundedness of the coefficients (see Remark \ref{remark:unboundedconfiningandinteraction} for an example). In fact, even the assumption that the fast component of the coefficients in Equation \ref{eq:multimeanfield} are periodic can be relaxed - see Remark \ref{eq:onhavingacoupledslowfastsystem}. However, we choose to present the proofs under the simple yet restrictive assumptions posed in this Section for readability purposes.

Assumption \ref{assumption:initialconditions} is used to determine the initial distribution of the limiting McKean-Vlasov Equation \ref{eq:McKeanLimit}, and ensure that it has sufficient moments for the analysis to go through. Assumption \ref{assumption:LipschitzandBounded} is used to ensure unique strong solutions to the system of prelimit Equation \ref{eq:multimeanfield} for each $N$ (see Proposition \ref{prop:uniquestrongsol}). The uniform ellipticity assumption \ref{assumption:uniformellipticity} and the regularity of certain derivatives of $f$ and $\sigma$ imposed in assumption \ref{assumption:fsigmaregularity} are used along with the centering condition \ref{assumption:centeringcondition} below to ensure the analogous regularity of the cell problem \eqref{eq:cellproblem}, which we will now introduce.

An important object of study will be the operator $\mc{L}^1_{x,\mu}$, parameterized by $x\in\R^d$ and $\mu\in\mc{P}_2(\R^d)$ which acts on $g\in C^2(\bb{T}^d)$ by
\begin{align}\label{eq:L1}
\mc{L}^1_{x,\mu}g(y) \coloneqq f(x,y,\mu)\cdot \nabla g(y)+\frac{1}{2} A(x,y,\mu):\nabla\nabla g(y).
\end{align}
Related to this operator we consider the measure $\pi(\cdot|x,\mu)\in\mc{P}(\bb{T}^d)$, parameterized by $\mu\in\mc{P}_2(\R^d)$, whose density $\tilde{\pi}(\cdot|x,\mu)$ satisfies the adjoint equation
\begin{align}\label{eq:pi}
\left(\mc{L}^1_{x,\mu}\right)^*\tilde{\pi}(y|x,\mu)&=0\\
\int_{\bb{T}^d} \tilde{\pi}(y|x,\mu)dy&=1,\forall x\in\R^d,\mu\in\mc{P}_2(\R^d),\nonumber
{}\end{align}
and the function $\Phi:\R^d\times \bb{T}^d\times \mc{P}_2(\R^d)\tto \R^d$, $\Phi = (\Phi_1,...,\Phi_d)$ solving
\begin{align}
\label{eq:cellproblem}
\mc{L}^1_{x,\mu} \Phi_l(x,y,\mu) &= -f_l(x,y,\mu)\\
 \int_{\bb{T}^d} \Phi_l(x,y,\mu)  \pi(dy|x,\mu) &=0, \nonumber
\end{align}
where both of these equations are given periodic boundary conditions.

In order to ensure the existence of solutions to Equation \eqref{eq:cellproblem}, we impose the centering condition, which is standard in the theory of averaging:
\begin{enumerate}[label=(A\arabic*)]\setcounter{enumi}{4}
\item \label{assumption:centeringcondition} The centering condition:
\begin{align*}
\int_{\bb{T}^d} f(x,y,\mu)\pi(dy|x,\mu)=0,\forall x\in\R^d,\mu\in\mc{P}_2(\R^d),
\end{align*}
where $\pi$ is the unique invariant measure for the frozen fast generator associated to Equation \eqref{eq:multimeanfield} as defined in Equation \eqref{eq:pi}, holds.
\end{enumerate}

As we will see in Propositions \ref{prop:invtmeasure} and \ref{prop:Phiexistenceregularity}, $\pi$ and $\Phi$ are uniquely defined and $\pi$ indeed admits a density $\tilde{\pi}$ under assumptions \ref{assumption:uniformellipticity}, \ref{assumption:fsigmaregularity}, and \ref{assumption:centeringcondition}.

$\mc{L}^1_{x,\mu}$ defined in Equation \eqref{eq:L1} is the generator of the diffusion process $Y^{x,y,\mu}_t$ from Equation \eqref{eq:frozeneqn}, which can be obtained from Equation \eqref{eq:multimeanfield} by writing down the generator associated to $Y^{i,N}_t =X^{i,N}_t/\epsilon$ and only keeping the $O(1/\epsilon^2)$ terms and freezing the $x$ and $\mu$ components in time. Intuitively, $X^{i,N}_t$ and $\mu^N_t$ will evolve much slower relative to $Y^{i,N}_t$ in Equation \eqref{eq:multimeanfield}, and as $\epsilon\downarrow 0$, $Y^{i,N}_t$'s dynamics will be immediately stabilized at its invariant measure. Thus we will have as $N\toinf$ (and $\epsilon\downarrow 0$) the arguments $X^{i,N}_t/\epsilon$ will be replaced by integrating against $\pi$ from Equation \eqref{eq:pi}. In addition, derivatives of the solution $\Phi$ to the cell problem as defined in Equation \eqref{eq:cellproblem} will enter the limiting equation for $\mu^N_t$ in order to correct for the fact that the drift term containing $f$ blows up as $\epsilon\downarrow 0$ (see Equation \eqref{eq:McKeanLimit} below). This effect of averaging for single particle systems is well understood. See, e.g. \cite{Bensoussan} and \cite{PS} for more intuition on the role of the cell-problem and invariant measure in a wide array of averaging problems.

We wish to observe the behavior of the sequence of $\mc{P}(C([0,1];\R^d))$-valued random variables
\begin{align}
\label{eq:empiricalmesaure}
\mu^{N}(\omega):=\frac{1}{N}\sum_{i=1}^N\delta_{X^{i,N}(\omega)}
\end{align}
as $N\toinf$. Specifically, letting $ev(t):C([0,1];\R^d)\tto \R^d$ be the evaluation map at time $t$, in Theorem \ref{thm:propofchaos} we see that, under assumptions \ref{assumption:initialconditions}-\ref{assumption:limitinguniformellipticity}, $\mc{L}(\mu^{N})\tto \delta_{\mu^*}$ in $\mc{P}(\mc{P}(C([0,1];\R^d)))$, where deterministic $\mu^*\in \mc{P}(C([0,1];\R^d))$ satisfies $\mu^*\circ [ev(t)]^{-1}= \mc{L}(X_t),t\in[0,1]$ for $X$ solving the McKean-Vlasov SDE:
\begin{align}
\label{eq:McKeanLimit}
dX_t &= \bar{\beta}(X_t,\mc{L}(X_t))dt +\bar{B}(X_t,\mc{L}(X_t))dW_t \\
\bar{B}(x,\mu)\bar{B}(x,\mu)^\top &= \bar{D}(x,\mu) \nonumber\\
X_0&\sim \nu_0 \nonumber
\end{align}
on some (possibly different than the original) filtered probability space $(\tilde\W,\tilde{\bb{F}},\tilde\Prob),\br{\tilde{\F}_t}_{t\in[0,1]}$, where $W_t$ is a standard $d$-dimensional $\tilde{\F}_t-$Brownian motion.  Here we define
\begin{align}\label{eq:limitingcoefficients}
\beta(x,y,\mu)&\coloneqq [\nabla_y\Phi(x,y,\mu)+I]b(x,y,\mu)+\nabla_x\Phi(x,y,\mu)f(x,y,\mu)+A:\nabla_x\nabla_y\Phi(x,y,\mu)\\
D(x,y,\mu)&\coloneqq \nabla_y\Phi(x,y,\mu)A(x,y,\mu)+A(x,y,\mu)[\nabla_y\Phi]^\top(x,y,\mu)+f\otimes\Phi(x,y,\mu)+\Phi\otimes f(x,y,\mu)\nonumber\\ 
&\hspace{12cm} +A(x,y,\mu)\nonumber\\
\bar{\beta}(x,\mu)&\coloneqq \int_{\bb{T}^d}\gamma(x,y,\mu)\pi(dy;x,\mu)\nonumber\\
\bar{D}(x,\mu)&\coloneqq \int_{\bb{T}^d}D(x,y,\mu)\pi(dy;x,\mu)\nonumber,
\end{align}
and
\begin{align*}
A:\nabla_x\nabla_y\Phi(x,y,\mu)\coloneqq (A(x,y,\mu):\nabla_x\nabla_y\Phi_1(x,y,\mu),...,A(x,y,\mu):\nabla_x\nabla_y\Phi_d(x,y,\mu))^\top.
\end{align*}

The following useful remark provides an alternative form for the limiting diffusion $\bar{D}$:
\begin{remark}\label{remark:altdiffusionrep}
It is worth noting that via an integration-by-parts argument, letting $D(x,y,\mu)$ be as in Equation \eqref{eq:limitingcoefficients} and introducing
\begin{align}\label{eq:tildeD}
\tilde{D}(x,y,\mu)=[I+\nabla_y\Phi(x,y,\mu)]A(x,y,\mu)[I+\nabla_y\Phi(x,y,\mu)]^\top,
\end{align}
that
\begin{align*}
\int_{\bb{T}^d}\tilde{D}(x,y,\mu)\pi(dy;x,\mu) = \int_{\bb{T}^d}D(x,y,\mu)\pi(dy;x,\mu)=\bar{D}(x,\mu).
\end{align*}
Thus the diffusion coefficient $\bar{B}(x,\mu)$ in Equation \eqref{eq:McKeanLimit} (and hence in Equations \eqref{eq:paramMcKeanLimit} and \eqref{eq:controlledMcKeanLimit}) can also be written as
\begin{align}\label{eq:equivalentaveragediffusion}
\bar{B}(x,\mu)\bar{B}(x,\mu)^\top = \int_{\bb{T}^d} \tilde{D}(x,y,\mu) \pi(dy|x,\mu).
\end{align}
See e.g. \cite{PS} Remark 11.4. In particular, this implies that $\bar{D}(x,\mu)$ is positive semi-definite for all $x\in\R^d$ and $\mu\in\mc{P}_2(\R^d)$.
\end{remark}

In the course of the proofs, we will need boundedness and continuity assumptions not only on the prelimit diffusion, but also on the limiting diffusion $\bar{B}=\sqrt{\bar{D}}$ and $\bar{B}^{-1}$. In order to ensure this, we finally impose the following assumption, which requires that $\bar{D}$ is uniformly positive definite:
\begin{enumerate}[label=(A\arabic*)]\setcounter{enumi}{5}
\item \label{assumption:limitinguniformellipticity} For some $\bar{\lambda}_1>0$, the second order term in the generator of the limiting McKean-Vlasov Equation $\bar{D}(x,\mu)=\bar{B}(x,\mu)\bar{B}^{\top}(x,\mu)$ from Equation \eqref{eq:limitingcoefficients} satisfies $\bar{\lambda}_1|\xi|^2\leq \xi^\top \bar{D}(x,\mu)\xi$ for each $\xi\in\R^d$ uniformly in $x\in\R^d,\mu\in\mc{P}_2(\R^d)$ .
\end{enumerate}

We seek now to quantify the rate at which the convergence of the random measures given by Equation \eqref{eq:empiricalmesaure} to the law of the solution of Equation \eqref{eq:McKeanLimit} occurs via deriving a large deviations principle for $\br{\mu^N}_{N\in\bb{N}}$ under these assumptions.

\subsection{The Controlled Process}
We start by constructing a controlled version of the system of mean-field SDEs \eqref{eq:multimeanfield} which will then allow us to use the weak  convergence approach to large deviations of \cite{DE}.

For $N\in\bb{N}$ let $\mc{U}_N$ denote the space of $\mc{F}_t$-progressively measurable functions $u:[0,1]\times\W\tto \R^{N\times m}$ such that $\E[\int_0^1|u(t)|^2dt]<\infty$, where $\E$ denotes the expectation with respect to $\Prob$ and $|\cdot|$ the Euclidean norm. For $u\in\mc{U}_N$, we write $u=(u_1,...,u_N)$ where $u_i\in \R^m$, $i=1,...,N$.

Given $u^{N}\in\mc{U}_{N}$, we consider the controlled system of SDEs
\begin{align}
\label{eq:controlledprelimit}
d\bar{X}^{i,N}_t &= \left[\frac{1}{\epsilon}f(\bar{X}_t^{i,N},\bar{X}_t^{i,N}/\epsilon,\bar\mu^N_t)+b(\bar{X}_t^{i,N},\bar{X}_t^{i,N}/\epsilon,\bar \mu^N_t)+\sigma(\bar{X}_t^{i,N},\bar{X}_t^{i,N}/\epsilon,\bar \mu^N_t)u_i^{N}(t)\right]dt\\
&+\sigma(\bar{X}_t^{i,N},\bar{X}_t^{i,N}/\epsilon,\bar\mu^N_t)dW^i_t\nonumber\\
\bar{X}^{i,N}_0&=x^{i,N}\nonumber
\end{align}
where $\bar{\mu}^N(t)$ and $\bar{\mu}^N$ are the empirical measures of $\bar{X}^{i,N}(t)$ and $\bar{X}^{i,N}$ respectively,
\begin{align}
\label{eq:barmu}
\bar{\mu}^{N}_t(\omega):=\frac{1}{N}\sum_{i=1}^N\delta_{\bar{X}^{i,N}_t(\omega)},\hspace{2cm}\bar{\mu}^{N}(\omega):=\frac{1}{N}\sum_{i=1}^N\delta_{\bar{X}^{i,N}(\omega)}.
\end{align}
Note existence and uniqueness of strong solutions to the controlled system of SDEs \eqref{eq:controlledprelimit} follows from Proposition \ref{prop:uniquestrongsol} and the discussion on p.81 of \cite{BDF}.

For notational convenience, we now introduce some spaces of interest. Let $\mc{X}\coloneqq C([0,1];\R^d)$, $\mc{Y}\coloneqq \mc{R}^1_1(\bb{T}^d\times\R^m)$, $\mc{W}\coloneqq C([0,1];\R^d)$ and $\mc{C}=\mc{X}\times\mc{Y}\times\mc{W}$. Here
\begin{align*}
\mc{R}_1^\alpha(\bb{T}^d\times\R^m) \coloneqq \br{&r:r\text{ is a positive Borel measure on }\bb{T}^d\times\R^m\times [0,\alpha],r(\bb{T}^d\times\R^m\times[0,t])=t,\forall t\in [0,\alpha],\\
&\text{ and }\int_{\bb{T}^d\times\R^m\times [0,\alpha]}|z|r(dy dz dt)<\infty}.
\end{align*}
Note that we construct $\mc{Y}$ this way to allow for extension of the results of this paper to bounded time intervals other than $[0,1]$. Also note that by Section 6.3 in \cite{Rachev}, $\mc{Y}$ is a Polish space.

Note that if $u\in\mc{U}^N$ for any $N\in\bb{N}$, then $u$ induces a  $\mc{Y}$-valued random variable $r$ via
\begin{align}
\label{eq:inducedrelaxedcontrol}
r_\omega(D\times E\times  I) \coloneqq \int_I \delta_{(\bar{X}^{i,N}_t/\epsilon) \text{mod}1}(D)\delta_{u(t,\omega)}(E)dt, \hspace{1cm} D\in \mc{B}(\bb{T}^d) ,E\in\mc{B}(\R^m),I\in \mc{B}([0,1]),\omega\in\W,
\end{align}
where $\bar{X}^{i,N}_t$ is as in Equation \eqref{eq:controlledprelimit} with this choice of control $u\in\mc{U}^N$.

Since for $r\in \mc{Y}$, $t\mapsto r(B\times [0,t])$ for $B\in \mc{B}(\bb{T}^d\times\R^d)$ is absolutely continuous, there exists $r_t:[0,1]\tto \mc{P}(\bb{T}^d\times \R^d)$ such that $r(dydzdt)=r_t(dydz)dt$.

Consider the McKean-Vlasov SDE parameterized by $\nu\in C([0,1];\mc{P}(\R^d))$ given by:
\begin{align}
\label{eq:paramMcKeanLimit}
  d\tilde{X}_t^{\nu}&= \biggl[\bar{\beta}(\tilde{X}_t^\nu,\nu(t))+\int_{\bb{T}^d\times\R^d}[\nabla_y \Phi(\tilde{X}_t^{\nu},y,\nu(t))+I]\sigma(\tilde{X}_t^{\nu},y,\nu(t))  z \rho_t(dydz)\biggr]dt+\bar{B}(X_t,\nu(t))dW_t
{}\end{align}
for $\tilde{X}^{\nu}\in\mc{X}, \rho\in\mc{Y}$, and $W\in\mc{W}$ a standard $d$-dimensional Wiener process. Here $\bar{\beta}$ and $\bar{B}$ are as in Equations \eqref{eq:McKeanLimit} and \eqref{eq:limitingcoefficients}. For fixed $\nu\in C([0,1];\mc{P}(\R^d))$, $Q\in\mc{P}(\mc{C})$ corresponds to a weak solution of (\ref{eq:paramMcKeanLimit}) if there exists a filtered probability space $(\tilde{\W},\tilde{\F},\tilde{\Prob}),\br{\tilde{\mc{F}}_t}$ supporting a $\tilde{\F}$-adapted $\R^d$-valued process $\tilde{X}^\nu_t$, a $\mc{P}(\bb{T}^d\times\R^m)$-valued $\tilde{\F}$-predictable process $\rho_t$, and a standard $d$-dimensional $\tilde{\F}_t$-Brownian Motion $W$ such that $(\tilde{X}^\nu,\rho_t(dydz)dt,W)$ is a $\mc{C}$- valued random variable satisfying Equation \eqref{eq:paramMcKeanLimit} that has distribution $Q$ under $\tilde{\Prob}$. Note that $\tilde{X}^{\nu},\rho,$ and $W$ are each random processes, unlike similar constructions in the case of small-noise large deviations, where the limiting process and controls can be taken to be deterministic (see e.g. \cites{DS}, \cite{BD} Section 4). Also note the inclusion of $\mc{W}$ in the construction of the canonical space $\mc{C}$, which allows us to identify the joint distribution of the control and driving Wiener process. This is important as per the discussion on \cite{BDF} p.79 , and in particular, since the driving Brownian motion of the averaged system cannot be realized as a copy of the Brownian motions from the prelimit system \eqref{eq:controlledprelimit}, this highly informs our construction of the occupation measures in Equations \eqref{eq:occmeas} and \eqref{eq:barW} and our proof of the Laplace Principle Upper Bound in Section \ref{sec:upperbound} (see also Remark \ref{remark:ontheoccmeasures}).

We are interested in particular in $Q\in\mc{P}(\mc{C})$ corresponding to weak solutions of $\tilde{X}^{\nu_Q}$, where $\nu_Q(t):[0,1]\tto \mc{P}(\R^d)$ is the Borel measurable mapping defined by
\begin{align}\label{eq:nuQ}
    \nu_Q(t)\coloneqq Q(\br{(\phi,r,w)\in\mc{C}:\phi(t)\in B}),\hspace{1cm}B\in\mc{B}(\R^d),t\in[0,1].
\end{align}
(For a description of $\mc{B}(\mc{P}(\R^d))$ see \cite{DE} Lemma A.5.1). This map is in fact seen to be continuous in Proposition \ref{prop:nucontmeasure}.

Since in this situation, by definition $\nu_Q(t)=\mc{L}(\tilde{X}^{\nu_Q}_t)$, we are thus interested in weak solutions to the limiting controlled McKean-Vlasov SDE:
\begin{align}
\label{eq:controlledMcKeanLimit}
  d\bar{X}_t&= \biggl[\bar{\beta}(\bar{X}_t,\mc{L}(\bar{X}_t))+\int_{\bb{T}^d\times\R^m}[\nabla_y \Phi(\bar{X}_t,y,\mc{L}(\bar{X}_t))+I]\sigma(\bar{X}_t,y,\mc{L}(\bar{X}_t))z \rho_t(dydz)\biggr]dt\\ 
  &+\bar{B}(\bar{X}_t,\mc{L}(\bar{X}_t))dW_t.\nonumber
\end{align}

Note that in the case that, decomposing $\rho_t$ as $\rho_t(dydz)=\gamma_t(dz;y)\beta_t(dy)$, if $\int_{\R^m}z\gamma_{t}(dz;y)=0$ almost-surely for almost every $t\in[0,1]$, this agrees with Equation \eqref{eq:McKeanLimit}.

The process triple $(\bar{X},\rho,W)$ can be given explicitly as the coordinate process on the probability space $(\mc{C},\mc{B}(\mc{C}),Q)$ endowed with the canonical filtration $\mc{G}_t\coloneqq \sigma\biggl((\bar{X}_s,\rho(s),W_s),0\leq s\leq t \biggr)$ (for predictability of a version of $\rho_t$ with respect to the canonical filtration, see e.g. \cite{LackerMarkovian} Lemma 3.2). Thus, for $\omega = (\phi,r,w)\in \mc{C}$,
\begin{align}
\label{eq:canonicalprocess}
    \bar{X}_t(\omega) = \phi(t),\hspace{1.5cm} \rho(t,\omega) = r|_{\mc{B}(\bb{T}^d\times\R^m\times [0,t])},\hspace{1.5cm}W_t(\omega)=w(t).
\end{align}

Thus, for $g:\bb{T}^d\times\R^m\tto \R$, when we write $\E^Q\biggl[\int_{\bb{T}^d\times\R^m\times [s,t]} g(y,z)\rho(t)(dydzd\tau) \biggr]$,
we mean
\begin{align*}
\E^Q\biggl[\int_{\bb{T}^d\times\R^m\times [s,t]} g(y,z)\rho(t)(dydzd\tau) \biggr]& =\int_{\mc{C}}\int_{\bb{T}^d\times\R^m\times [s,t]}g(y,z)\rho(t,\omega)(dydzd\tau)Q(d\omega) \\
&=  \int_{\mc{X}\times\mc{Y}\times\mc{W}}\int_{\bb{T}^d\times\R^m\times [s,t]}g(y,z)r|_{\mc{B}(\bb{T}^d\times\R^m\times [0,t])}(dydzd\tau)Q(d\phi dr dw)\\
& = \int_{\mc{X}\times\mc{Y}\times\mc{W}}\int_s^t\int_{\bb{T}^d\times\R^m}g(y,z)r_\tau(dydz)d\tau Q(d\phi dr dw).
\end{align*}

Throughout this paper we will only integrate $\rho(t,\omega)$ against time intervals of the form $[s,t]$, so we will simply write $\E^Q\biggl[\int_{\bb{T}^d\times\R^m\times [s,t]} g(y)\rho(dydzd\tau) \biggr]$ in the place of $\E^Q\biggl[\int_{\bb{T}^d\times\R^m\times [s,t]} g(y,z)\rho(t)(dydzd\tau) \biggr]$ and $r(dydzd\tau)$ in the place of $r|_{\mc{B}(\bb{T}^d\times\R^m\times [0,t])}(dydzd\tau)$.

\section{Statement of the Main Results}\label{S:MainResult}
The first result of this paper is a Law of Large Numbers for the multiscale empirical measures $\mu^{N}$:
\begin{theo}\label{thm:propofchaos}
Let $ev:\mc{X}\tto \R^d$ be the evaluation map at time $t$ and $\br{\mu^N}$ be as defined by Equation \eqref{eq:empiricalmesaure}. Under assumptions \ref{assumption:initialconditions}-\ref{assumption:limitinguniformellipticity}, $\mc{L}(\mu^{N})\tto \delta_{\mu^{*}}$ in $\mc{P}(\mc{P}(\mc{X}))$, where deterministic $\mu^{*}\in \mc{P}(\mc{X})$ satisfies $\mu^{*}\circ ev^{-1}(t)= \mc{L}(X_t),t\in[0,1]$ for $X$ solving the McKean-Vlasov SDE (\ref{eq:McKeanLimit}).
\end{theo}
\begin{proof}
This follows immediately from the proofs in Section \ref{sec:limitingbehavior} by taking $u^N\equiv 0$ for all $N\in\bb{N}$.
\end{proof}

In order to state the remaining main results of this paper, we need the following two definitions:
\begin{defi}\label{def:V}
{}We will say $\Theta\in\mc{P}(\mc{C})$ is in $\mc{V}$ if
\begin{enumerate}[label=(V\arabic*)]
    \item \label{V:V1}$\Theta$ corresponds to a weak solution $\bar{X}$ of (\ref{eq:controlledMcKeanLimit}).
    \item \label{V:V2}$\E^{\Theta}\biggl[\int_{\bb{T}^d\times\R^m\times [0,1]}|z|^2 \rho(dydzdt)  \biggr]<\infty$ .
    \item \label{V:V3}$\nu_\Theta(0)=\nu_0$ from Assumption \ref{assumption:initialconditions}.
    \item \label{V:V4}\begin{align*}\Theta\biggl(\biggl\lbrace(\phi,r,w)\in\mc{C}&: \exists [(s,y)\mapsto\gamma_s(\cdot;y)]\in \mc{M}([0,1]\times\bb{T}^d;\mc{P}(\R^m))\\&\text{ such that }r(dydzds)=\gamma_s(dz;y)\pi(dy|\phi(s),\nu_{\Theta}(s)) ds,\forall s\in[0,1]\biggr\rbrace\biggr)=1\end{align*}
\end{enumerate}

Where here we are using the notation for the coordinate process given in Equation \eqref{eq:canonicalprocess}.
\end{defi}

\begin{defi}
A function $I:\mc{P}(\mc{X})\tto[0,\infty]$ is called a \textit{(good) rate function} if for each $M<\infty$, the set $\br{\theta\in\mc{P}(\mc{X}):I(\theta)\leq M}$ is compact. We say that the Laplace Principle with speed $N$ holds for the family $\br{\mu^N}_{N\in\bb{N}}$ with rate function $I$ if for any bounded, continuous $F:\mc{P}(\mc{X})\tto\R$,
\begin{align}
\label{eq:laplaceprinciple}
\lim_{N\toinf} - \frac{1}{N}\log\E[\exp(-N F(\mu^N))] = \inf_{\theta\in\mc{P}(\mc{X})}\br{F(\theta)+I(\theta)}
\end{align}
\end{defi}

In order to prove the Laplace Principle for $\br{\mu^N}_{N\in\bb{N}}$, we make use of the following proposition:
\begin{proposition}\label{prop:varrep}
Under assumption \ref{assumption:LipschitzandBounded}, the prelimit expression in (\ref{eq:laplaceprinciple}) can be written as
\begin{align}
\label{eq:varrep}
-\frac{1}{N}\log\E[\exp(-N F(\mu^N))] &= \inf_{u^N\in\mc{U}_{N}}[\frac{1}{2}\E[\frac{1}{N}\int_0^1 |u^{N}(t)|^2dt]+\E[F(\bar{\mu}^N)]]\\
& = \inf_{u^N\in\mc{U}_{N}}[\frac{1}{2}\E[\frac{1}{N}\sum_{i=1}^{N}\int_0^1 |u_i^{N}(t)|^2dt]+\E[F(\bar{\mu}^N)]]\nonumber
\end{align}
for any $F\in\mc{C}_b(\mc{P}(\mc{X}))$ where $\bar{\mu}^N$ is given by (\ref{eq:barmu}) with $u^{N}=(u^{N}_1,...,u^{N}_{N})\in\mc{U}_{N}$ the control in Equation \eqref{eq:controlledprelimit}.
\end{proposition}
\begin{proof}
By Proposition \ref{prop:uniquestrongsol} and \cite{YW} there is Borel measurable $\psi^{i,N}$ such that
\begin{align*}
\psi^{i,N}((x^{1,N},...,x^{N,N}),(W^1,...,W^N)) = X^{i,N},
\end{align*} and by the characterization of $\mc{B}( \mc{P}(\mc{X}))$ given in  Lemma A.5.1 of
\cite{DE} $p:C([0,1];\R^d)^N\tto \mc{P}(C([0,1];\R^d))$ given by $p^N(\phi_1,...,\phi_N)=\frac{1}{N}\sum_{i=1}^N \delta_{\phi_i}$ is Borel measurable. So $$\bar{\mu}^N = p^N(\psi^{1,N}((x^{1,N},...,x^{N,N}),(W^1,...,W^N)),...,\psi^{N,N}((x^{1,N},...,x^{N,N}),(W^1,...,W^N)))$$ and is thus a Borel-measurable function of the driving Wiener processes for each $N$. Then Theorem 3.6 in \cite{BD} applies, giving us the desired result.
\end{proof}
Then, as is standard, we will prove the Laplace Principle via showing the Laplace Principle lower bound:
\begin{align}\label{eq:laplaceprinciplelowerbound}
\liminf_{N\toinf} - \frac{1}{N}\log\E[\exp(-N F(\mu^N))] \geq \inf_{\theta\in\mc{P}(\mc{X})}\br{F(\theta)+I(\theta)},
\end{align}
the Laplace Principle upper bound:
\begin{align}\label{eq:laplaceprincipleupperbound}
\limsup_{N\toinf} - \frac{1}{N}\log\E[\exp(-N F(\mu^N))] \leq \inf_{\theta\in\mc{P}(\mc{X})}\br{F(\theta)+I(\theta)},
\end{align}
and compactness of level sets of $I:\mc{P}(\mc{X})\tto [0,+\infty]$.

In proving the Laplace Principle upper bound \eqref{eq:laplaceprincipleupperbound}, we will need to make either of the following additional assumptions:
\begin{enumerate}[label=(B\arabic*)]
\item \label{assumption:samebrownianmotion} $f(x,y,\mu)\equiv 0,\sigma(x,y,\mu)=\sigma(x,\mu),$ and $d=m$
\item \label{assumption:xindependenceassumptionforupperbound} $f(x,y,\mu)=f(y,\mu)$ and $\sigma(x,y,\mu)=\sigma(y,\mu)$.
\end{enumerate}
\begin{remark}\label{remark:onadditionalassumption}
Note that these assumptions are not mutually inclusive. The first is essentially the regime where in the standard one-particle setting, strong ($L^2$) convergence of the multiscale slow process to the averaged slow process can be proved. The important feature for our proof in this setting is that, in the one particle regime, $(X^\epsilon,W)\tto (\bar{X},W)$, that is, in some sense the driving Brownian motion of the averaged system is the same as in the prelimit system. This is not the case in the full setting - see Remark \ref{remark:ontheoccmeasures}. The second assumption is a technical one which allows us to use Lipschitz arguments in order to approximate the a priori $L^2$ controls in Equation \eqref{eq:BDFlimitav} with bounded ones. The requirement for this approximation argument to go through is in fact that $\bar{B}(x,\mu)=\bar{B}(\mu)$, where $\bar{B}$ is the diffusion matrix of the averaged system as per Equations \eqref{eq:McKeanLimit} and \eqref{eq:limitingcoefficients}. This is clearly the case when $f$ and $\sigma$ do not depend on $x$, as all the terms which in the definition of $\bar{D}$ in Equation \eqref{eq:limitingcoefficients} depend only on $f$ and $\sigma$ through Equations \eqref{eq:pi} and \eqref{eq:cellproblem}. See also \cite{BCcurrents} Condition 2.3 i) and \cite{BCsmallnoise} Condition 2.3, where essentially the same assumption as \ref{assumption:xindependenceassumptionforupperbound}, for different but analogous reasons.
\end{remark}

While the natural form of the rate function which arises from the weak convergence approach taken in this paper is formulated in terms of the class of viable ``controls'' $\mc{V}$ from Definition \eqref{def:V} and the limiting controlled Equation \eqref{eq:controlledMcKeanLimit}, as we will see in Proposition \ref{prop:I=Iav}, there is an alternative form of the rate function which is formulated in terms of another controlled McKean-Vlasov Equation \eqref{eq:BDFlimitav} depending only on the limiting coefficients $\bar{\beta}$ and $\bar{B}$ from \eqref{eq:limitingcoefficients} and controls which do not have any $y$-dependence. Analogously to the situation in the joint small-noise and averaging limit for standard SDEs (see \cite{DS} Section 5), this alternative representation is useful in proving the Laplace Principal upper bound \eqref{eq:upperbound}. This is true under each of the Assumptions \ref{assumption:samebrownianmotion} and \ref{assumption:xindependenceassumptionforupperbound}, but in the case of assumption \ref{assumption:samebrownianmotion}, since we will be using the methods of \cites{BDF,BCsmallnoise}, we will need the analogous weak-sense uniqueness assumption to (A4) in \cite{BDF} for this alternative controlled McKean-Vlasov Equation. Hence we present the form of this alternative controlled Equation \eqref{eq:BDFlimitav} under Assumption \ref{assumption:samebrownianmotion} here:

Consider the space of relaxed controls $\mc{Z}\coloneqq R_1^1(\R^d)$ where
\begin{align}\label{eq:Zspace}
\mc{R}_1^\alpha(\R^d) \coloneqq \br{&r:r\text{ is a positive Borel measure on }\R^m\times [0,\alpha],r(\R^m\times[0,t])=t,\forall t\in [0,\alpha],\\
&\text{ and }\int_{\R^m\times [0,\alpha]}|z|r(dz dt)<\infty}.\nonumber
\end{align}
This is the space where the $\R^d\times[0,1]$-marginal of an element of $\mc{Y}$ takes values.

Consider also the controlled SDE with process triple $(\hat{X},\hat{\rho},\hat{W})\in \mc{X}\times\mc{Z}\times\mc{W}$ given by
\begin{align}\label{eq:limitingcontrolledmckeanvlasovsimp}
d\hat{X}_t&= \biggl[\int_{\bb{T}^d}b(\hat{X}_t,y,\mc{L}(\hat{X}_t))\pi(dy)+\sigma(\hat{X}_t,\mc{L}(\hat{X}_t))\int_{\R^d}z\hat{\rho}_t(dz)  \biggr] dt+\sigma(\hat{X}_t,\mc{L}(\hat{X}_t))d\hat{W}_t.
\end{align}
where $\hat{W}$ is a standard $d$-dimensional Brownian motion.

The sense in which we need uniqueness for Equation \eqref{eq:limitingcontrolledmckeanvlasovsimp} is as follows (compare with Definition 1 of \cite{BDF} and Lemma 3.3 in \cite{BCsmallnoise}):
\begin{defi}\label{def:weaksenseuniquness}
We will say weak-sense uniqueness holds for Equation \eqref{eq:limitingcontrolledmckeanvlasovsimp} if for $\Theta,\tilde{\Theta}\in \mc{P}(\mc{X}\times\mc{Z}\times\mc{W})$ such that:
\begin{enumerate}
\item $\Theta$ and $\tilde{\Theta}$ correspond to a weak solution $\hat{X}$ of \eqref{eq:limitingcontrolledmckeanvlasovsimp},
\item $\E^{\Theta}\biggl[\int_{\R^d\times [0,1]}|z|^2 \hat{\rho}(dzdt)\biggr],\E^{\tilde{\Theta}}\biggl[\int_{\R^d\times [0,1]}|z|^2 \hat{\rho}(dzdt)\biggr]<\infty$,
\item $\Theta\circ \vartheta^{-1}=\tilde{\Theta}\circ \vartheta^{-1}$, where $\vartheta:\mc{X}\times\mc{Z}\times\mc{W}\tto \R^d\times \mc{Z}\times\mc{W}$ is given by $\vartheta(\phi,r,w)=(\phi(0),r,w)$,
\end{enumerate}
we have $\Theta=\tilde{\Theta}.$
\end{defi}

Under assumption \ref{assumption:samebrownianmotion}, we will also assume:
\begin{enumerate}[label=(C\arabic*)]
\item \label{assumption:weaksenseuniqueness} Both:
\begin{enumerate}[label=\roman*)]
  \item Weak-sense uniqueness as defined in Definition \ref{def:weaksenseuniquness} for Equation (\ref{eq:limitingcontrolledmckeanvlasovsimp})\footnote{A statement analogous to Proposition C.1 in \cite{FischerFormofRateFunction} is needed in order to claim that Assumption \ref{assumption:weaksenseuniqueness}i) already holds under assumptions \ref{assumption:initialconditions}-\ref{assumption:uniformellipticity} and \ref{assumption:samebrownianmotion}. However, as an anonymous reviewer kindly and correctly pointed out, Proposition C.1 in \cite{FischerFormofRateFunction} is based on an erroneous localization argument. As far as we know, there is no proof currently available in the literature for Assumption \ref{assumption:weaksenseuniqueness}i) to hold under assumptions \ref{assumption:initialconditions}-\ref{assumption:uniformellipticity} and \ref{assumption:samebrownianmotion} alone. It is true, however, that if in addition one takes $\sigma(x,\mu)=\sigma(\mu)$, that a proof analogous to that of Proposition C.1 without the use of stopping times implies weak-sense uniqueness for Equation \eqref{eq:limitingcontrolledmckeanvlasovsimp} - see e.g. Lemma 3.4. in \cite{BCsmallnoise}. Finding precise and reasonable weaker assumptions under which Assumption \ref{assumption:weaksenseuniqueness}i), and hence the result of Theorem \ref{thm:LaplacePrinciple} under Assumption \ref{assumption:samebrownianmotion}, holds is, to the best of our knowledge, an open question.}
\item For any $\nu_0$-integrable $\phi:\R^d\tto \R$, $\frac{1}{N}\sum_{i=1}^N \phi(x^{i,N})\tto \langle \nu_0,\phi\rangle$, where $x^{i,N}$ are as in Assumption \ref{assumption:initialconditions}
\end{enumerate}
hold.
\end{enumerate}

Note that since we use the weaker notion of weak-sense uniqueness of \cites{BCsmallnoise,BCcurrents} where the joint distribution of the initial condition, control, and Brownian motion under $\Theta$ and $\tilde{\Theta}$ are assumed equal rather than the distribution of the initial condition and joint distribution of the control and Brownian motion separately as in \cite{BDF}, we also include the strengthened assumption on convergence of the initial values \ref{assumption:weaksenseuniqueness}ii). This is Condition 2.1 in \cite{BCsmallnoise} and Condition 2.3 (ii) in \cite{BCcurrents}, and is required for the same reason as in those papers. That is, in order to show the convergence stated in \eqref{eq:convergenceforweakuniqueness} when exploiting weak-sense uniqueness in the proof of the Laplace Principle Upper Bound \eqref{eq:upperbound}.

Our main result can now be summarized in the following theorem:
\begin{theo}
\label{thm:LaplacePrinciple}
Under assumptions \ref{assumption:initialconditions}-\ref{assumption:limitinguniformellipticity} the sequence of $\mc{P}(\mc{X})$-valued random variables $\br{\mu^N}_{N\in\bb{N}}$ as defined by Equation \eqref{eq:empiricalmesaure} satisfies the Laplace Principle Lower Bound \eqref{eq:laplaceprinciplelowerbound} with good rate function
\begin{align}
\label{eq:ratefunction}
I(\theta) = \inf_{\Theta\in\mc{V}:\Theta_{\mc{X}}=\theta} \E^{\Theta}\biggl[\frac{1}{2}\int_{\bb{T}^d\times\R^m\times [0,1]}|z|^2\rho(dydzdt)\biggr]
\end{align}
where $\inf(\emptyset):=+\infty.$

Further assuming either \ref{assumption:xindependenceassumptionforupperbound} or both \ref{assumption:samebrownianmotion} and \ref{assumption:weaksenseuniqueness},
$\br{\mu^N}_{N\in\bb{N}}$ satisfies the Laplace Principle Upper Bound \eqref{eq:laplaceprincipleupperbound} with rate function $I$ given as in \eqref{eq:ratefunction}, and hence satisfies the large deviations principle with speed $N$ and rate function $I$.
\end{theo}
\begin{proof}
We prove the Laplace Principle lower bound (\ref{eq:lowerbound}) in Section \ref{sec:lowerbound}. In Section \ref{sec:compactlevelsets} we prove that the level sets of $I$ are compact, so indeed $I$ is a good rate function. Under the further assumptions \ref{assumption:xindependenceassumptionforupperbound} or \ref{assumption:samebrownianmotion} and \ref{assumption:weaksenseuniqueness}, we prove the Laplace Principle upper bound (\ref{eq:upperbound}) in Section \ref{sec:upperbound}.

The main tool in these proofs is the Variational Representation Theorem for Functionals of Brownian Motion, given in Proposition \ref{prop:varrep}. Once we identify the law of large numbers result for the controlled process in Section \ref{sec:limitingbehavior}, the Laplace Principle lower bound \eqref{eq:laplaceprinciplelowerbound}
follows immediately from Fatou's lemma, as seen in Section \ref{sec:lowerbound}. Compactness of level sets follows from the methods of proving tightness of measures in the level sets and employing a version of Fatou's lemma.

To prove the Laplace Principle upper bound \eqref{eq:laplaceprincipleupperbound} in Section \ref{sec:upperbound}, we use the Equivalent formulation for the rate function provided by Proposition \ref{prop:I=Iav}. Under assumptions \ref{assumption:samebrownianmotion} and \ref{assumption:weaksenseuniqueness}, we are able then to use the methods of \cites{BDF,BCsmallnoise}, where weak-sense uniqueness (see Definition \ref{def:weaksenseuniquness}) and a construction of IID controls paired with IID Brownian motions allows one to construct a controlled empirical measure which converges to any near-optimal controlled process associated to Equation \ref{eq:BDFlimitav}.

 Under assumption \ref{assumption:xindependenceassumptionforupperbound}, we must take a more novel approach to proving the Laplace Principle upper bound \eqref{eq:laplaceprincipleupperbound}, as the construction from \cites{BDF,BCsmallnoise} doesn't directly apply for reasons outlined in the beginning of Subsection \ref{subsec:upperboundundernotsameBM}. Instead, using tools and methods from Mean Field Games and the optimal control of McKean-Vlasov Equations found in \cites{Lacker,DPTequivalence,CDL,DPTdpp}, we show one can approximate the expression on the right hand side of Equation \eqref{eq:laplaceprincipleupperbound} with the law of a controlled process with controls that are in semi-Markovian feedback form in terms of their driving Brownian motions and initial conditions. This is where we require the additional assumption \ref{assumption:xindependenceassumptionforupperbound}, as this provides Lipschitz properties for the controlled process that enable us to make such an approximation. We can then use this characterization of the nearly-optimal controls to construct a sequence of controls in feedback form under which the controlled empirical measure will converge to an element of $\mc{P}(\mc{X})$ which approximates any given controlled solution to Equation \eqref{eq:BDFlimitav}.

Once we show these two bounds, we get
\begin{align*}
\inf_{\theta\in\mc{P}(\mc{X})}\br{F(\theta)+I(\theta)}&\leq \liminf_{N\toinf} - \frac{1}{N}\log\E[\exp(-N F(\mu^N))]\\
&\leq \lim_{N\toinf} - \frac{1}{N}\log\E[\exp(-N F(\mu^N))]\\
&\leq \limsup_{N\toinf} - \frac{1}{N}\log\E[\exp(-N F(\mu^N))]\\
&\leq \inf_{\theta\in\mc{P}(\mc{X})}\br{F(\theta)+I(\theta)},
\end{align*}
so that Equation \eqref{eq:laplaceprinciple} is satisfied. It is well known that in our setting the Laplace Principle holds if and only if  $\br{\mu^N}_{N\in\bb{N}}$ satisfies a LDP with rate function $I$. See \cite{DE} Theorem 1.2.3.
\end{proof}

Lastly, we provide the alternative form of the rate function in the form of \cite{DG}. To obtain this form of the rate function, we must use the contraction principle to treat the empirical measures \eqref{eq:empiricalmesaure} as elements of $C([0,1];\mc{P}(\R^d))$ rather than $\mc{P}(\mc{X})$. To define the rate function, it will be useful to consider the generator of the limiting McKeav-Vlasov Equation \eqref{eq:McKeanLimit} as parameterized by $\mu\in\mc{P}(\R^d)$. That is, $\bar{L}_\mu$ which acts on $g\in C^2_b(\R^d)$ by:
\begin{align}\label{eq:limitinggenerator}
\bar{L}_\mu g(x) \coloneqq \bar{\beta}(x,\mu)\cdot \nabla g(x)+ \frac{1}{2}\bar{D}(x,\mu):\nabla\nabla g(x),
\end{align}
where $\bar{\beta}$ and $\bar{D}$ are as in Equation \eqref{eq:limitingcoefficients}.
Then we have the following:
\begin{theo}\label{theorem:DGform}
Consider $J^{DG}:C([0,1];\mc{P}(\R^d))\tto [0,+\infty]$ given by
\begin{align}\label{eq:DGform}
J^{DG}(\theta) = \frac{1}{2}\int_0^1 \sup_{\phi\in C^\infty_c(\R^d):\langle \theta(t),\norm{\nabla_{\theta(t)}\phi}^2_{\theta(t)}\rangle \neq 0}\frac{|\langle\dot{\theta}(t)-\bar{L}^*_{\theta(t)}\theta(t),\phi\rangle|^2}{\langle \theta(t),\norm{\nabla_{\theta(t)}\phi}^2_{\theta(t)}\rangle}dt
\end{align}
if $\phi\mapsto \langle \theta,\phi\rangle$ is absolutely continuous in the sense of distributions (see Definition \ref{def:absolutelycontinuous}) and $\theta(0)=\nu_0$, and $J^{DG}(\theta)=+\infty$ otherwise. In the above $\bar{L}^*_\mu$ is the formal adjoint of $\bar{L}_\mu$ as defined in Equation \eqref{eq:limitinggenerator} acting on $\mc{P}(\R^d)$, $\dot{\theta}$ is the time derivative in the distribution sense of the aforementioned absolutely continuous mapping, and $\norm{\nabla_{\theta(t)}\phi}^2_{\theta(t)}\coloneqq \nabla^\top \phi(\cdot)\bar{D}(\cdot,\theta(t))\nabla \phi(\cdot) $ (see Equation \eqref{eq:reimannianstructure}).

Under Assumptions \ref{assumption:initialconditions}-\ref{assumption:limitinguniformellipticity} and \ref{assumption:xindependenceassumptionforupperbound} or both \ref{assumption:samebrownianmotion} and \ref{assumption:weaksenseuniqueness}, $\br{t\mapsto\mu^N\circ ev^{-1}(t)}_{N\in\bb{N}}$ from Equation \eqref{eq:empiricalmesaure} satisfy the large deviations principle with speed $N$ and rate function $J^{DG}:C([0,1];\mc{P}(\R^d))\tto [0,+\infty]$ given by Equation \eqref{eq:DGform}, where $\text{ev}:\mc{X}\tto \R^d$ is the evaluation map at time $t$.
\end{theo}
\begin{proof}
See Subsection \ref{subsec:DGequivalentform}.
\end{proof}

\section{Example: A class of aggregation-diffusion equations  } \label{S:DawsonExample}
In this section we discuss a class of aggregation-diffusion equations which fall into the regime of \eqref{eq:multimeanfield} and how the law of large numbers and rate function for the empirical measure look. We also remark on extending the analysis of this paper to systems for which the drift coefficients are not necessarily bounded in $x$ and $\mu$, so that a ``rough'' potential version of the system with bi-stable confining potential and Currie-Weiss interactions considered in the classical work of \cite{Dawson} also satisfies the LDP presented in Theorems \ref{thm:LaplacePrinciple} and \ref{theorem:DGform}.

Consider the system of weakly interacting diffusions:
\begin{align}\label{eq:prelimituncontrolledaggregationdiffusionnomultiscale}
dX^{i,N}_t&= -\nabla V(X^{i,N}_t)dt - \frac{1}{N}\sum_{j=1}^N \nabla W(X^{i,N}_t-X^{j,N}_t)dt +\sigma dW^i_t,\quad X^{i,N}_0=x^{i,N},
\end{align}
where $V,W:\R^d\tto \R$ are sufficiently smooth, $\sigma>0$, and $W^i$ are IID $d-$dimensional Brownian motions.

Such interacting particle systems in the applications including biology, ecology, social sciences, economics, molecular dynamics, and in study of spatially homogeneous granular media (see \cites{MT,Garnier1,BCCP,KCBFL} and the references therein). In such a system $V$ is referred to as the confining potential, and $W$ is referred to as the interaction potential.

In some applications, the confining potential is known to be most accurately modeled by a so-called ``rough potential'' \cite{Zwanzig}. This means that the confining potential in Equation \eqref{eq:prelimituncontrolledaggregationdiffusionnomultiscale} takes the form $V^\epsilon(x)=V_1(x)+V_2(x/\epsilon)$, where $V_2$ is periodic and $\epsilon>0$ is a small parameter which represents the period of the overlayed roughness from $V_2$ over $V_1$. Making this replacement, Equation \eqref{eq:prelimituncontrolledaggregationdiffusionnomultiscale} becomes:
 \begin{align}\label{eq:prelimituncontrolledaggregationdiffusionmultiscale}
dX^{i,N}_t& = \biggl[-\nabla V_1(X^{i,N}_t)-\frac{1}{\epsilon}\nabla V_2(X^{i,N}_t/\epsilon) - \frac{1}{N}\sum_{j=1}^N \nabla W(X^{i,N}_t-X^{j,N}_t)\biggr]dt +\sigma dW^i_t,\quad X^{i,N}_0=x^{i,N}.
\end{align}
This corresponds to our Equation \eqref{eq:multimeanfield} with $f(x,y,\mu)=-\nabla V_2(y)$, $b(x,y,\mu)=-\nabla V_1(x)-\langle \mu, \nabla W(x-\cdot)\rangle$, $d=m$, $\sigma(x,y,\mu)=\sigma I$. Assume the initial conditions satisfy \ref{assumption:initialconditions}.

Note that Assumption \ref{assumption:xindependenceassumptionforupperbound} holds in this situation.

If $V_2\in C^2_b(\bb{T}^d)$ and $V_1,W\in C^2_b(\R^d)$, Assumptions \ref{assumption:LipschitzandBounded}-\ref{assumption:fsigmaregularity} hold (for Lipschitz continuity in $\mu$ of $g$ one can use \cite{CD} Section 5.2.2 Example 1 and Remark 5.27). Moreover, one can compute that explicitly that the invariant measure $\pi$ from Equation \eqref{eq:pi} admits a density $\tilde{\pi}$ given by:
\begin{align*}
\tilde{\pi}(y)=Z^{-1}\exp(-2V_2(y)/\sigma^2), \quad Z=\int_{\bb{T}^d}\exp(-2V_2(y)/\sigma^2)dy
\end{align*}
so that Assumption \ref{assumption:centeringcondition} holds. Finally, we note that via Remark \ref{remark:altdiffusionrep}, $\bar{D}$ is constant and given by
\begin{align*}
\bar{D}=\sigma^2 \int_{\bb{T}^d}[I+\nabla_y \Phi(y)][I+\nabla_y \Phi(y)]^\top \pi(dy),
\end{align*}
so that Assumption \ref{assumption:limitinguniformellipticity} can be readily verified.

Thus, Theorems \ref{thm:propofchaos}, \ref{thm:LaplacePrinciple}, and \ref{theorem:DGform} directly apply to the empirical measure as defined in Equation \eqref{eq:empiricalmesaure} for the system \eqref{eq:prelimituncontrolledaggregationdiffusionmultiscale}.

Further assuming that $V_2$ is separable, i.e. $V_2(y_1,...,y_d)=Q_1(y_1)+Q_2(y_2)+...+Q_d(y_d)$, we have via an explicit calculation that
\begin{align}\label{eq:Gamma}
&\Gamma\coloneqq \int_{\bb{T}^d}[I+\nabla_y \Phi(y)] \pi(dy)=\int_{\bb{T}^d}[I+\nabla_y \Phi(y)][I+\nabla_y \Phi(y)]^\top \pi(dy)= \text{diag}\biggl[Z_1^{-1}\hat{Z}_1^{-1},..., Z_d^{-1}\hat{Z}_d^{-1}\biggr]\\
Z_k&\coloneqq \int_0^1 \exp(-2Q_k(y)/\sigma^2)dy,\quad\hat{Z}_k\coloneqq \int_0^1 \exp(2Q_k(y)/\sigma^2)dy,k=1,...,d\nonumber.
\end{align}
Note in particular that the entries of $\Theta$ are less than $1$, so that the averaging effect is seen to decrease the effective diffusivity of the particle system in all directions (in addition to decreasing the magnitude of both the confining and interaction potential in all directions) - see the discussion at the end of Section 3 in \cite{BezemekSpiliopoulosAveraging2022}.

We have then the following Corollary (compare with the analogous Corollary 5.4 in \cite{DS}):
\begin{corollary}\label{cor:aggregationdiffusionequationresult}
Consider the system of interacting particles in a rough environment given by Equation \eqref{eq:prelimituncontrolledaggregationdiffusionmultiscale}. Assume $\sigma>0$, $V_1,W\in C^2_b(\R^d)$, and $V_2(y_1,...,y_d)=Q_1(y_1)+Q_2(y_2)+...+Q_d(y_d)$, $Q_k\in C^2_b(\bb{T})$. Then the empirical measures $\br{\mu^N}_{N\in\bb{N}}$ on the paths of these particles as defined in Equation \eqref{eq:empiricalmesaure} converge in distribution as $\mc{P}(\mc{X})$-valued random variables to $\mc{L}(X)$, where $X$ satisfies:
\begin{align*}
dX_t &= - \biggl[\Gamma \nabla V_1(X_t)+ \hat{\E}[\Gamma \nabla W(x-\hat{X}_t)]\biggl|_{x=X_t}\biggr]dt +\sigma \sqrt{\Gamma}d\hat W_t\\
X_0\sim \nu_0,
\end{align*}
where $\hat W$ is a $d$-dimensional standard Brownian motion on a (possibly different) probability space, $\hat{X}_t$ is an independent copy of $X_t$ on a copy of that space on which we denote the expectation by $\hat{\E}$, and $\Gamma$ is as in Equation \eqref{eq:Gamma}.

Moreover, $\br{t\mapsto \mu^N\circ ev^{-1}(t)}$ satisfies the large deviations principle on $C([0,1];\mc{P}(\R^d))$ with speed $N$ and good rate function given by:
\begin{align*}
J^{DG}(\theta) &= \frac{1}{2\sigma^2}\int_0^1 \sup_{\phi\in C^\infty_c(\R^d):\langle \theta(t),\nabla^\top\phi(\cdot) \Gamma \nabla\phi(\cdot)\rangle \neq 0}\frac{|\langle\dot{\theta}(t)-\bar{L}^*_{\theta(t)}\theta(t),\phi\rangle|^2}{\langle \theta(t),\nabla^\top\phi(\cdot) \Gamma \nabla\phi(\cdot)\rangle}dt\\
\bar{L}_\mu \phi (x) &\coloneqq -\Gamma \nabla V_1(x)\cdot \nabla \phi(x) -  \langle \mu, \Gamma\nabla W(x-\cdot)\rangle\cdot \nabla \phi (x) + \frac{\sigma^2}{2}\Gamma : \nabla \nabla \phi(x)
\end{align*}
if $\phi\mapsto \langle \theta,\phi\rangle$ is absolutely continuous in the sense of Definition \ref{def:absolutelycontinuous} and $\theta(0)=\nu_0$, and $J^{DG}(\theta)=+\infty$ otherwise.
\end{corollary}
\begin{proof}
This follows immediately from the above discussion and Theorems \ref{thm:propofchaos} and \ref{theorem:DGform}.
\end{proof}

\begin{remark}\label{remark:unboundedconfiningandinteraction}
One may be interested in applying the result of Corollary \ref{cor:aggregationdiffusionequationresult} to systems such as that found in the classical work of \cite{Dawson}. Such systems are used as a simple model of cooperative behavior, and exhibit interesting phase transitions in the mean field limit. The system of \cite{Dawson}, in the notation of Equation \eqref{eq:prelimituncontrolledaggregationdiffusionnomultiscale}, has $d=1$, $V(x)=\frac{1}{4}x^4-\frac{1}{2}x^2,$ and $W(x)=\frac{\kappa}{2}x^2$ for $\kappa>0$ a parameter controlling the strength of the interaction. Clearly here we do not have $V_1=V\in C^2_b(\R)$ or $W\in C^2_b(\R)$.

The additional considerations required to extend the proof of Theorems \ref{thm:propofchaos}, \ref{thm:LaplacePrinciple}, and \ref{theorem:DGform} to this setting are two-fold:

First, one must establish sufficient uniform integrability of the controlled empirical measure \eqref{eq:barmu} corresponding to this system in order to be able to pass to the limits in the proofs in Subsection \ref{subsection:identifylimit}, and to gain tightness of the $\mc{X}$- component of the occupation measures in Subsection \ref{subsubsection:TightnessQ^N_X} (and the analogous compactness of level sets for the rate function in Section \ref{sec:compactlevelsets}). However, using bounds on the explicit solution of the Cell Problem $\Phi(y)$ from Equation \eqref{eq:cellproblem} which are available in the 1-D setting (see \cite{GS} Proposition A.4), one can find via a long but straightforward calculation that if $\sup_{N\in\bb{N}}\E\biggl[\frac{1}{N}\sum_{i=1}^N \int_0^1 |u_i^N(t)|^2dt\biggr]\leq C_{con},C_{con}\in[0,\infty)$  then for each $t\in [0,1]$,
\begin{align*}
\sup_{N\in\bb{N}}\E\biggl[\sup_{t\in[0,1]}\frac{1}{N}\sum_{i=1}^N |\bar{X}^{i,N}_t|^4\biggr]+ \sup_{N\in\bb{N}}\E\biggl[\frac{1}{N}\sum_{i=1}^N \int_0^1 |\bar{X}^{i,N}_t|^6 dt\biggr]&\leq C(\kappa,\sigma,C_{con}),
\end{align*}
where $\bar{X}^{i,N}$ are as in Equation \eqref{eq:controlledprelimit} with this specific choice of coefficients. This provides enough uniform integrability for the proofs to go through.

Secondly, due to the lack of Lipschitz property of the limiting drift coefficient $$\bar{\beta}(x,\mu) = -\Gamma x^3 +\Gamma[1-\kappa] x +\kappa\Gamma \int_{\R}v\mu(dv)$$
in $x$, one must be careful when making the approximation argument of \cite{Lacker} in the proof of the Laplace Principle Upper Bound in Subsection \ref{subsec:upperboundundernotsameBM}. However, thanks to the one-sided Lipschitz property of the polynomial part of the drift, using similar arguments but applying It\^o's formula to estimate the square expectation rather than directly squaring the equations, the arguments go through. See Appendix A of \cite{Dawson} for guidance.
\end{remark}

\section{Connections to Rate Functions in the Existing Literature}\label{section:connectiontootheratefunctions}
 The goal of this section is to connect the rate function \eqref{eq:ratefunction} to existing rate functions for (non-multiscale) empirical measures of weakly interacting diffusions in the literature. We will assume \ref{assumption:initialconditions}-\ref{assumption:limitinguniformellipticity} throughout. Our first result will connect our rate function with the rate function of \cite{BDF} associated to the system where one makes the ansatz that the effective behavior of each particle will be of the same form as the effective averaged dynamics in the case of one-particle systems. This equivalent form of the rate function, found in Proposition \ref{prop:I=Iav}, is a key tool in proving the Laplace Principle Upper Bound in Section \ref{sec:upperbound}.

We then use the contraction principle, the mimicking theorem of \cite{BSmimicking}, and a Reisz Representation argument in order to further obtain a form of the rate function, posed on $C([0,1];\mc{P}(\R^d))$ rather than $\mc{P}(\mc{X})=\mc{P}(C([0,1];\R^d))$, for the empirical measure process which corresponds to the form found in the seminal paper \cite{DG} of Dawson-G\"artner. This is, to our knowledge, the first time that the large deviations rate function of \cite{BDF} has be shown to be equivalent to that of \cite{DG} (once the contraction principle is applied) in a rigorous way. This result is stated in Section \ref{S:MainResult} as Theorem \ref{theorem:DGform}. This form of the rate function further allows us to compare the relation between our joint averaging and propagation-of-chaos rate function \eqref{eq:ratefunction} and the propagation-of-chaos rate function without multiscale structure of \cite{BDF} with the relation of the joint averaging and small-noise rate function from \cite{DS} with the small-noise without multiscale structure rate function of \cite{FW} - See Remark \ref{remark:comparingwiththesmallnoisecase}.
\subsection{An Alternative Variational Form of the Rate Function}\label{subsec:altvariationalform}
Recall the space of relaxed controls $\mc{Z}\coloneqq R_1^1(\R^d)$ from Equation \eqref{eq:Zspace}.

Consider the controlled SDE with process triple $(\hat{X},\hat{\rho},\hat{W})\in \mc{X}\times\mc{Z}\times\mc{W}$ given by
\begin{align}\label{eq:BDFlimitav}
d\hat{X}_t &= [\bar{\beta}(\hat{X}_t,\mc{L}(\hat{X}_t))+\bar{B}(\hat{X}_t,\mc{L}(\hat{X}_t))\int_{\R^d}z\hat{\rho}_t(dz)]dt+\bar{B}(\hat{X}_t,\mc{L}(\hat{X}_t))d\hat{W}_t,
\end{align}
where $\bar{\beta}$ and $\bar{B}$ are as in Equations \eqref{eq:McKeanLimit} and \eqref{eq:limitingcoefficients} and $\hat{W}$ is a standard $d$-dimensional Brownian motion.

We define a class of measures in $\mc{P}(\mc{X}\times\mc{Z}\times\mc{W})$ by:
\begin{defi}\label{defi:Vav}
$\Theta\in\mc{P}(\mc{X}\times\mc{Z}\times\mc{W})$ is in $\mc{V}^{av}$ if
\begin{enumerate}[label=($V^{av}$\arabic*)]
\item \label{V:V1BDFav}$\Theta$ corresponds to a weak solution $\hat{X}$ of (\ref{eq:BDFlimitav}).
\item \label{V:V2BDFav}$\E^{\Theta}\biggl[\int_{\R^d\times [0,1]}|z|^2 \hat{\rho}(dzdt)\biggr]<\infty$ .
\item \label{V:V3BDFav}$\hat{\nu}_{\Theta}(0)=\nu_0$, where $\hat{\nu}_{\Theta}$ is as in Equation \eqref{eq:nuQ}, but parameterized by $\Theta\in\mc{P}(\mc{X}\times\mc{Z}\times\mc{W})$ rather than $\Theta\in\mc{P}(\mc{C})$.
\end{enumerate}
\end{defi}
and a function $I^{av}:\mc{P}(\mc{X})\tto [0,+\infty]$ by:
\begin{align}\label{eq:BDFratefunctionav}
I^{av}(\theta) = \inf_{\Theta\in\mc{V}^{av}:\Theta_{\mc{X}}=\theta} \E^{\Theta}\biggl[\frac{1}{2}\int_{\R^d\times [0,1]}|z|^2\hat{\rho}(dzdt)\biggr]
\end{align}
where $\inf(\emptyset):=+\infty$.

Here we are using the coordinate process notation from Equation \eqref{eq:canonicalprocess}, but where $\hat{\rho}\in\mc{Z}$ rather than $\mc{Y}$.

\begin{remark}\label{remark:tildeconstruction}
Note that $I^{av}$ from Equation \eqref{eq:BDFratefunctionav} is the rate function for the sequence of empirical measures $\br{\hat{\mu}^N}_{N\in\bb{N}}\subset\mc{P}(\mc{X})$ given by
\begin{align}\label{eq:averagedempiricalmeasure}
\hat{\mu}^N\coloneqq \frac{1}{N}\sum_{i=1}^N\delta_{\hat{X}^{i,N}},
\end{align}
where
\begin{align}\label{eq:averagingfirstlimit}
d\hat{X}^{i,N}_t &= \bar{\beta}(\hat{X}^{i,N}_t,\hat{\mu}^N_t)dt+\bar{B}(\hat{X}^{i,N}_t,\hat{\mu}^N_t)d\hat{W}^i_t,
\end{align}
and $\hat{W}^i$ are independent $d$-dimensional standard Brownian motions, as per Theorem 3.1 in \cite{BDF}. Equation \eqref{eq:averagingfirstlimit} is the equation which one arrives at by replacing the coefficients from Equation \eqref{eq:multimeanfield} with those obtained from sending $\epsilon\downarrow 0$ with $N$ fixed.
\end{remark}

We have the following Proposition:
\begin{proposition}\label{prop:I=Iav}
$I^{av}=I$, $I$ is as in Equation \eqref{eq:ratefunction} and $I^{av}$ is as in Equation \eqref{eq:BDFratefunctionav}.
\end{proposition}
\begin{proof}
Let $\theta\in \mc{P}(\mc{X})$ be such that $I^{av}(\theta)<\infty$. Let $\eta>0$ and $\hat{\Theta}\in\mc{V}^{av}$ be such that $\E^{\hat{\Theta}}\biggl[\frac{1}{2}\int_{\R^d\times [0,1]}|z|^2\hat{\rho}(dzdt)\biggr]\leq I^{av}(\theta)+\eta$.

Consider $\Theta\in \mc{P}(\mc{C})$ be given by $\Theta=\hat{\Theta}\circ G^{-1}$, where $G:\mc{X}\times\mc{Z}\times\mc{W}\tto \mc{C}$ is defined by $G(\phi,\hat{r},w)=(\phi,r_{\phi,\hat{r}},w)$, where for $A\in\mc{B}(\bb{T}^d),B\in\mc{B}(\R^m),\Gamma\in \mc{B}([0,1])$,
\begin{align*}
r_{\phi,\hat{r}}(A\times B\times \Gamma)=\int_{\Gamma}\int_{A}\delta_{\sigma^\top(\phi(t),y,\theta(t))[I+\nabla_y\Phi(\phi(t),y,\theta(t))]^\top (\bar{B}^\top)^{-1}(\phi(t),\theta(t))\int_{\R^d}z\hat{r}_t(dz)}(B)\pi(dy;\phi(t),\theta(t))dt.
\end{align*}
Here Assumption \ref{assumption:LipschitzandBounded}, Proposition \ref{prop:Phiexistenceregularity}, and Corollary \ref{cor:limitingcoefficientsregularity} ensure sufficient regularity of the coefficients for measurability of $G$.
By construction, $\Theta_{\mc{X}}=\hat{\Theta}_{\mc{X}}=\theta$. Thus, since $\hat{\Theta}$ satisfies \ref{V:V3BDFav}, $\Theta$ satisfies \ref{V:V3}. Moreover, since $\nu_{\Theta}(t)=\theta(t)$ for all $t$, by construction $\Theta$ satisfies \ref{V:V4}. $\Theta_{\mc{W}}=\hat{\Theta}_{\mc{W}}$, so the $\mc{W}$-marginal of $\Theta$ is the standard Wiener measure, and  by \ref{V:V1BDFav}, $\hat{\Theta}=\mc{L}(\hat{X},\hat{\rho},\hat{W})$ satisfying Equation \ref{eq:BDFlimitav}, so $\Theta=\mc{L}(G(\hat{X},\hat{\rho},\hat{W}))=\mc{L}(\hat{X},\rho,\hat{W})$ where $\rho=r_{\hat{X},\hat{\rho}}$ as defined above.

We verify that $(\hat{X},\rho,\hat{W})$ satisfies the desired Equation \eqref{eq:controlledMcKeanLimit}, since by definition, for any $s\in [0,1]:$
\begin{align*}
&\int_0^s\int_{\bb{T}^d\times\R^d}[\nabla_y \Phi(\hat{X}_t,y,\mc{L}(\hat{X}_t))+I]\sigma(\hat{X}_t,y,\mc{L}(\hat{X}_t))z\rho_t(dydz) dt\\
& = \int_0^s\int_{\bb{T}^d}[\nabla_y \Phi(\hat{X}_t,y,\mc{L}(\hat{X}_t))+I]\sigma(\hat{X}_t,y,\mc{L}(\hat{X}_t))\sigma^\top (\hat{X}_t,y,\mc{L}(\hat{X}_t))[\nabla_y \Phi(\hat{X}_t,y,\mc{L}(\hat{X}_t))+I]^\top \pi(dy;\hat{X}_t,\mc{L}(\hat{X}_t))\\
&\hspace{8cm}(\bar{B}^{\top})^{-1}(\hat{X}_t,\mc{L}(\hat{X}_t))\int_{\R^d}z\hat{\rho}_t(dz) dt\\
&= \int_0^s \bar{B}(\hat{X}_t,\mc{L}(\hat{X}_t))\bar{B}^{\top}(\hat{X}_t,\mc{L}(\hat{X}_t))(\bar{B}^{\top})^{-1}(\hat{X}_t,\mc{L}(\hat{X}_t))\int_{\R^d}z\hat{\rho}_t(dz) dt\\
& = \int_0^s \bar{B}(\hat{X}_t,\mc{L}(\hat{X}_t))\int_{\R^d}z\hat{\rho}_t(dz) dt,
\end{align*}
and since $(\hat{X},\hat{\rho},\hat{W})$ satisfies Equation \eqref{eq:BDFlimitav}, indeed $(\hat{X},\rho,\hat{W})$ satisfies Equation \eqref{eq:controlledMcKeanLimit}, and hence $\Theta$ satisfies \ref{V:V1}.

Lastly, by the change-of-variables formula:
\begin{align*}
&\E^{\Theta}\biggl[\frac{1}{2}\int_{\bb{T}^d\times \R^m\times [0,1]}|z|^2\rho(dydzdt)\biggr]\\
& = \E^{\hat{\Theta}}\biggl[\frac{1}{2}\int_0^1\int_{\bb{T}^d}\biggl|\sigma^\top (\hat{X}_t,y,\mc{L}(\hat{X}_t))[\nabla_y \Phi(\hat{X}_t,y,\mc{L}(\hat{X}_t))+I]^\top(\bar{B}^{\top})^{-1}(\hat{X}_t,\mc{L}(\hat{X}_t))\int_{\R^d}z\hat{\rho}_t(dz)\biggr|^2 \pi(dy;\hat{X}_t,\mc{L}(\hat{X}_t))dt\biggr]\\
& = \E^{\hat{\Theta}}\biggl[\frac{1}{2}\int_0^1\biggl[\int_{\R^d}z\hat{\rho}_t(dz)\biggr]^\top\bar{B}^{-1}(\hat{X}_t,\mc{L}(\hat{X}_t))\int_{\bb{T}^d}[\nabla_y \Phi(\hat{X}_t,y,\mc{L}(\hat{X}_t))+I]\sigma(\hat{X}_t,y,\mc{L}(\hat{X}_t))\\
&\hspace{2cm}\sigma^\top (\hat{X}_t,y,\mc{L}(\hat{X}_t))[\nabla_y \Phi(\hat{X}_t,y,\mc{L}(\hat{X}_t))+I]^\top \pi(dy;\hat{X}_t,\mc{L}(\hat{X}_t))(\bar{B}^{\top})^{-1}(\hat{X}_t,\mc{L}(\hat{X}_t))\int_{\R^d}z\hat{\rho}_t(dz)dt\biggr]\\
& = \E^{\hat{\Theta}}\biggl[\frac{1}{2}\int_0^1\biggl[\int_{\R^d}z\hat{\rho}_t(dz)\biggr]^\top\bar{B}^{-1}(\hat{X}_t,\mc{L}(\hat{X}_t))\bar{B}(\hat{X}_t,\mc{L}(\hat{X}_t))\bar{B}^\top(\hat{X}_t,\mc{L}(\hat{X}_t))(\bar{B}^{\top})^{-1}(\hat{X}_t,\mc{L}(\hat{X}_t))\int_{\R^d}z\hat{\rho}_t(dz)dt\biggr]\\
& = \E^{\hat{\Theta}}\biggl[\frac{1}{2}\int_0^1\biggl|\int_{\R^d}z\hat{\rho}_t(dz)\biggr|^2dt\biggr]\\
&\leq \E^{\hat{\Theta}}\biggl[\frac{1}{2}\int_0^1\int_{\R^d}|z|^2\hat{\rho}_t(dz)dt\biggr]\\
&\leq I^{av}(\theta)+\eta,
\end{align*}
so $\Theta$ satisfies \ref{V:V2}, and hence $\Theta\in \mc{V}$ and
\begin{align*}
I(\theta)\leq \E^{\Theta}\biggl[\frac{1}{2}\int_{\bb{T}^d\times \R^m\times [0,1]}|z|^2\rho(dydzdt)\biggr]\leq I^{av}(\theta)+\eta.
\end{align*}
Since $\eta$ and $\theta$ were arbitrary, $I\leq I^{av}$.

Now let $\theta\in \mc{P}(\mc{X})$ be such that $I(\theta)<\infty$. Let $\eta>0$ and $\Theta\in\mc{V}$ be such that $\E^{\Theta}\biggl[\frac{1}{2}\int_{\bb{T}^d\times \R^m\times [0,1]}|z|^2\rho(dydzdt)\biggr]\leq I(\theta)+\eta$.

Consider $\hat{\Theta}\in \mc{P}(\mc{X}\times\mc{Z}\times\mc{W})$ be given by $\hat{\Theta}=\Theta\circ \hat{G}^{-1}$, where $\hat{G}:\mc{C}\tto \mc{X}\times\mc{Z}\times\mc{W}$ is defined by $\hat{G}(\phi,r,w)=(\phi,\hat{r}_{\phi,r},w)$, where for $A\in\mc{B}(\R^m),\Gamma\in \mc{B}([0,1])$:
\begin{align*}
\hat{r}_{\phi,r}(A\times\Gamma)=\int_\Gamma \delta_{B^{-1}(\phi(t),\theta(t))\int_{\bb{T}^d\times\R^m}[I+\nabla_y\Phi(\phi(t),y,\theta(t))]\sigma(\phi(t),y,\theta(t))zr_t(dydz)}(A)dt.
\end{align*}
Again Assumption \ref{assumption:LipschitzandBounded}, Proposition \ref{prop:Phiexistenceregularity}, and Corollary \ref{cor:limitingcoefficientsregularity} ensure sufficient regularity of the coefficients for measurability of $\hat{G}$, and by construction, $\hat{\Theta}_{\mc{X}}=\Theta_{\mc{X}}=\theta$. Thus, since $\Theta$ satisfies \ref{V:V3}, $\hat\Theta$ satisfies \ref{V:V3BDFav}. In addition, $\hat{\Theta}_{\mc{W}}=\Theta_{\mc{W}}$, so the $\mc{W}$-marginal of $\hat\Theta$ is the standard Wiener measure, and by \ref{V:V1}, $\Theta=\mc{L}(\bar{X},\rho,W)$ satisfying Equation \ref{eq:controlledMcKeanLimit}, so $\hat\Theta=\mc{L}(\hat{G}(\bar{X},\rho,W))=\mc{L}(\bar{X},\hat{\rho},W)$ where $\hat{\rho}=\hat{r}_{X,\rho}$ as defined above.

We verify that $(\bar{X},\hat{\rho},W)$ satisfies the desired Equation \eqref{eq:BDFlimitav}, since by definition, for any $s\in [0,1]:$
\begin{align*}
&\int_0^s \bar{B}(\bar{X}_t,\mc{L}(\bar{X}_t))\int_{\R^d}z\hat{\rho}_t(dz)dt\\
&=\int_0^s \bar{B}(\bar{X}_t,\mc{L}(\bar{X}_t))\bar{B}^{-1}(\bar{X}_t,\mc{L}(\bar{X}_t))\int_{\bb{T}^d\times\R^m}[\nabla_y \Phi(\bar{X}_t,y,\mc{L}(\bar{X}_t))+I]\sigma(\bar{X}_t,y,\mc{L}(\bar{X}_t))z\rho_t(dydz)dt\\
& = \int_0^s \int_{\bb{T}^d\times\R^m}[\nabla_y \Phi(\bar{X}_t,y,\mc{L}(\bar{X}_t))+I]\sigma(\bar{X}_t,y,\mc{L}(\bar{X}_t))z\rho_t(dydz)dt
\end{align*}
and since $(\bar{X},\rho,W)$ satisfies Equation \eqref{eq:controlledMcKeanLimit}, indeed $(\bar{X},\hat{\rho},W)$ satisfies Equation \eqref{eq:BDFlimitav}, and hence $\Theta$ satisfies \ref{V:V1BDFav}.

Once again, by the change-of-variables formula:
\begin{align*}
&\E^{\hat{\Theta}}\biggl[\frac{1}{2}\int_{\R^d\times [0,1]}|z|^2\hat{\rho}(dzdt)\biggr]\\
& = \E^{\Theta}\biggl[\frac{1}{2}\biggl|\bar{B}^{-1}(\bar{X}_t,\mc{L}(\bar{X}_t))\int_{\bb{T}^d\times\R^m}[\nabla_y \Phi(\bar{X}_t,y,\mc{L}(\bar{X}_t))+I]\sigma(\bar{X}_t,y,\mc{L}(\bar{X}_t))z\rho_t(dydz)\biggr|^2dt\biggr]\\
& = \E^{\Theta}\biggl[\frac{1}{2}\int_{\bb{T}^d\times\R^m}z^\top\sigma^\top(\bar{X}_t,y,\mc{L}(\bar{X}_t))[\nabla_y \Phi(\bar{X}_t,y,\mc{L}(\bar{X}_t))+I]^\top\rho_t(dydz)(\bar{B}^\top)^{-1}(\bar{X}_t,\mc{L}(\bar{X}_t))\bar{B}^{-1}(\bar{X}_t,\mc{L}(\bar{X}_t))\\
&\hspace{2cm}\int_{\bb{T}^d\times\R^m}[\nabla_y \Phi(\bar{X}_t,y,\mc{L}(\bar{X}_t))+I]\sigma(\bar{X}_t,y,\mc{L}(\bar{X}_t))z\rho_t(dydz)dt\biggr]\\
& = \E^{\Theta}\biggl[\frac{1}{2}\int_{\bb{T}^d}\int_{\R^m}z^\top\gamma_t(dz;y)\sigma^\top(\bar{X}_t,y,\mc{L}(\bar{X}_t))[\nabla_y \Phi(\bar{X}_t,y,\mc{L}(\bar{X}_t))+I]^\top\pi(dy;\bar{X}_t,\mc{L}(\bar{X}_t))(\bar{B}^\top)^{-1}(\bar{X}_t,\mc{L}(\bar{X}_t))\\
&\bar{B}^{-1}(\bar{X}_t,\mc{L}(\bar{X}_t))\int_{\bb{T}^d}[\nabla_y \Phi(\bar{X}_t,y,\mc{L}(\bar{X}_t))+I]\sigma(\bar{X}_t,y,\mc{L}(\bar{X}_t))\int_{\R^m}z\gamma_t(dz;y)\pi(dy;\bar{X}_t,\mc{L}(\bar{X}_t))dt\biggr] \text{ by \ref{V:V4}}\\
&\leq \E^{\Theta}\biggl[\frac{1}{2}\int_{\bb{T}^d}\biggl|\int_{\R^m}z\gamma_t(dz;y)\biggr|^2\pi(dy;\bar{X}_t,\mc{L}(\bar{X}_t))dt\biggr] \text{ by Lemma 5.1 in \cite{DS}}\\
&\leq \E^{\Theta}\biggl[\frac{1}{2}\int_{\bb{T}^d}\int_{\R^m}|z|^2\gamma_t(dz;y)\pi(dy;\bar{X}_t,\mc{L}(\bar{X}_t))dt\biggr]\\
& = \E^{\Theta}\biggl[\frac{1}{2}\int_{\bb{T}^d\times \R^m\times [0,1]}|z|^2\rho(dydzdt)\biggr]\\
&\leq I(\theta)+\eta.
\end{align*}
so $\hat\Theta$ satisfies \ref{V:V2BDFav}, and hence $\hat\Theta\in \mc{V}^{av}$ and
\begin{align*}
I^{av}(\theta)\leq \E^{\hat{\Theta}}\biggl[\frac{1}{2}\int_{\R^d\times [0,1]}|z|^2\hat{\rho}(dzdt)\biggr]\leq I(\theta)+\eta.
\end{align*}
Since $\eta$ and $\theta$ were arbitrary, $I^{av}\leq I$.

%\red{I think the only thing left to comment on here is $G$ and $\hat{G}$ being measurable, and possibly the adaptedness of each constructed process. But this should be clear.}
\end{proof}
As stated in the introduction of this section, Proposition \ref{prop:I=Iav} will be used crucially in the proof of the Laplace Principle Upper Bound \eqref{eq:laplaceprinciple} in Section \ref{sec:upperbound}. Once Theorem \ref{thm:LaplacePrinciple} is proved, since the rate function associated to a sequence of random variables is unique (see e.g. Theorem 1.3.1 in \cite{DE}), we immediately get the following Corollary:
\begin{corollary}\label{corollary:LDPcommutes}
Under assumptions \ref{assumption:initialconditions}-\ref{assumption:limitinguniformellipticity} and \ref{assumption:xindependenceassumptionforupperbound} or both \ref{assumption:samebrownianmotion} and \ref{assumption:weaksenseuniqueness}, $\br{\mu^N}_{N\in\bb{N}}$ satisfies the large deviations principle with speed $N$ and rate function $I^{av}$.
\end{corollary}

\begin{remark}\label{remark:ontheequivalenceofIandIav}
Theorem \ref{thm:LaplacePrinciple} shows that the limits $\epsilon\downarrow 0$ and $N\toinf$ for the system \eqref{eq:multimeanfield} commute at the level of the law of large numbers. Corollary \ref{corollary:LDPcommutes} along with Theorem 3.1 of \cite{BDF} implies that, ever further, the empirical measures $\hat{\mu}^N$ from Equation \eqref{eq:averagedempiricalmeasure} obtained from first sending $\epsilon\downarrow 0$ and the multiscale empirical measures $\mu^N$ from Equation \eqref{eq:empiricalmesaure} satisfy the same large deviations principle. This parallels the situation in the small-noise diffusion setting - see Remark \ref{remark:comparingwiththesmallnoisecase}.
\end{remark}

We end this subsection with the remark that, as per Remark 3.2 in \cite{BDF}, in the definition of the rate function $I^{av}$ we may take the relaxed (meaning $\mc{Z}$-valued) controls $\hat{\rho}$ to in fact be standard open-loop controls. That is:
\begin{align}\label{eq:BDFratefunctionavstandard}
I^{av}(\theta) = \inf_{\Theta\in\mc{V}^{av}:\Theta_{\mc{X}}=\theta} \E^{\Theta}\biggl[\frac{1}{2}\int_0^1 |u(t)|^2dt\biggr]
\end{align}
where $\inf(\emptyset):=+\infty$, $u(t)\coloneqq \int_{\R^d}z\rho_t(dz)$, and we re-characterize \ref{V:V1BDFav} as $\Theta$ corresponding to a weak solution of:
\begin{align}\label{eq:BDFlimitavstandard}
d\hat{X}_t &= [\bar{\beta}(\hat{X}_t,\mc{L}(\hat{X}_t))+\bar{B}(\hat{X}_t,\mc{L}(\hat{X}_t))u(t)]dt+\bar{B}(\hat{X}_t,\mc{L}(\hat{X}_t))d\hat{W}_t.
\end{align}

\subsection{Connection to Dawson-G\"artner Form of the Rate Function}\label{subsec:DGequivalentform}
In this subsection we show how to connect the function $I^{av}$ given in Equation \eqref{eq:BDFratefunctionav} (equivalently in Equation \eqref{eq:BDFratefunctionavstandard}) to the ``negative Sobolev norm'' form of the rate function given in the seminal paper \cite{DG}. This is done via a series of Lemmas, with the main result being Theorem \ref{theorem:DGform}.

Since the rate function in \cite{DG} is for the flow of the empirical measures rather than the empirical measures on path space, that is treating them as elements of $C([0,1];\mc{P}(\R^d))$ rather than $\mc{P}(\mc{X})$. Thus our first step is to apply the contraction principle to $I^{av}$ to obtain a rate function for $\br{\mu^N_\cdot}$ (and $\br{\hat{\mu}^N_\cdot}$) on $C([0,1];\mc{P}(\R^d))$. Note that we use here the subscript $\mu^N_\cdot$ to distinguish the $C([0,1];\mc{P}(\R^d))$-valued random variable $[t\mapsto \mu^N\circ ev^{-1}(t)]$ from the $\mc{P}(\mc{X})$-valued random variable $\mu^N$. For a discussion of a similar matter, see Remark 6.12 in \cite{DLR}.

\begin{proposition}\label{prop:ratefunctionforflow}
Consider $J:C([0,1];\mc{P}(\R^d))\tto [0,+\infty]$ given by:
\begin{align}\label{eq:ordinaryratefunctionpaths}
J(\theta) = \inf_{\Theta\in\mc{V}^{av}:\hat{\nu}_{\Theta}(t)=\theta(t),\forall t\in[0,1]} \frac{1}{2}\E^{\Theta}\biggl[\int_0^1|u(t)|^2 dt\biggr].
{}\end{align}
Then, under Assumptions \ref{assumption:initialconditions}-\ref{assumption:limitinguniformellipticity} and \ref{assumption:xindependenceassumptionforupperbound} or both \ref{assumption:samebrownianmotion} and \ref{assumption:weaksenseuniqueness}, $\br{\mu^N_\cdot}_{N\in\bb{N}}$ from Equation \eqref{eq:empiricalmesaure} and $\br{\hat{\mu}^N_\cdot}$ from Equation \eqref{eq:averagedempiricalmeasure} satisfy the same large deviations principle with speed $N$ and rate function $J$.
\end{proposition}
\begin{proof}
We first claim that $\bar{\Psi}:\mc{P}(\mc{X})\tto C([0,1];\mc{P}(\R^d))$ sending $\mu$ to $t\mapsto \mu\circ [ev(t)]^{-1},t\in [0,1]$ is continuous. To see this, it is useful to use the bounded Lipschitz metric on both $\mc{P}(\R^d)$ and $\mc{P}(\mc{X})$, which agrees with the topology of weak convergence in both cases (see e.g. Proposition 11.3.2. in \cite{Dudley}). We have for $\mu_1,\mu_2\in \mc{P}(\mc{X})$, that
\begin{align*}
&d_{C([0,1];\mc{P}(\R^d))}(\bar{\Psi}(\mu_1),\bar{\Psi}(\mu_2)) \\
&= \sup_{t\in[0,1]}\sup_{g\in C_b(\R^d):\sup_{x\in\R^d}|g(x)|\leq 1,\sup_{x\neq y\in\R^d}\frac{|g(x)-g(y)|}{|x-y|}\leq 1}\biggl|\int_{\R^d}g(x)\mu_1\circ [ev(t)]^{-1}(t)(dx)-\int_{\R^d}g(x)\mu_2\circ [ev(t)]^{-1}(dx)\biggr|\\
& = \sup_{t\in[0,1]}\sup_{g\in C_b(\R^d):\sup_{x\in\R^d}|g(x)|\leq 1,\sup_{x\neq y\in\R^d}\frac{|g(x)-g(y)|}{|x-y|}\leq 1}\biggl|\int_{\mc{X}}g(\phi(t))\mu_1(d\phi)-\int_{\mc{X}}g(\phi(t))\mu_2(d\phi)\biggr|\\
&\leq \sup_{G\in C_b(\mc{X}):\sup_{\phi\in\mc{X}}|G(\phi)|\leq 1,\sup_{\phi\neq\psi\in\mc{X}}\frac{|G(\phi)-G(\psi)|}{\norm{\phi-\psi}_{\mc{X}}}\leq 1}\biggl|\int_{\mc{X}}G(\phi)\mu_1(d\phi)-\int_{\mc{X}}G(\phi)\mu_2(d\phi) \biggr|\\
& = d_{\mc{P}(\mc{X})}(\mu_1,\mu_2)
\end{align*}
Where in the inequality, we used for any $g\in C_b(\R^d)$ such that $\sup_{x\in\R^d}|g(x)|\leq 1$ and $\sup_{x\neq y\in\R^d}\frac{|g(x)-g(y)|}{|x-y|}\leq 1$, $G_t\in C_b(\mc{X})$ defined by $G_t(\phi) = g(\phi(t))$ satisfies $\sup_{\phi\in\mc{X}}|G_t(\phi)|\leq 1$ and $\sup_{\phi\neq\psi\in\mc{X}}\frac{|G_t(\phi)-G_t(\psi)|}{\norm{\phi-\psi}_{\mc{X}}}\leq 1$ for all $t\in[0,1]$.

Thus, in fact $\bar{\Psi}$ is Lipschitz continuous.

Then the contraction principle (see e.g. \cite{DE} Theorem 1.3.2) gives that, since $\br{\mu^N}\subset \mc{P}(\mc{X})$ satisfies and LDP with rate function $I^{av}:\mc{P}(\mc{X})\tto [0,+\infty]$, $\br{\mu^N_\cdot}=\br{\bar{\Psi}(\mu^N)}$ satisfies an LDP with rate function $J:C([0,1];\mc{P}(\R^d))\tto [0,+\infty]$ given by
\begin{align*}
J(\nu) &= \inf\br{I(\mu):\mu\in \bar{\Psi}^{-1}(\nu)}\\
& = \inf\br{I(\mu):\mu\in\mc{P}(\mc{X}) \text{ has the same one dimensional time marginals as }\nu}.
\end{align*}

Thus Corollary \ref{corollary:LDPcommutes} and Remark \ref{remark:tildeconstruction} along with the form of $I^{av}$ given in Equation \eqref{eq:BDFratefunctionavstandard} yield the desired result.
\end{proof}

Now that we have the rate function $J$ which acts on $C([0,1];\mc{P}(\R^d))$, we are ready to start to prove an equivalent form of $J$. The first step is to use the arguments along the lines of the proof of Theorem 3.7 in \cite{LackerMarkovian} (see also the discussion at the end of Subsection 6.2.5 of \cite{CD}), where we apply an extension of the well-known mimicking result of Gy\"ongy \cite{gyongymimicking} due to Brunick and Shreve \cite{BSmimicking} to obtain that the ``open-loop'' controls in $J$ can in fact be taken to be in Markovian feedback form.

To this end, we define:
\begin{defi}
Given $\theta\in C([0,1];\mc{P}(\R^d))$, define the class of functions $H(\theta)$ to be the set of measurable functions $h:[0,1]\times \R^d\tto \R^d$ such that:
\begin{enumerate}[label=($H$\arabic*)]
\item \label{H:H1} $\theta(t) = \mc{L}(X^h_t)$ for all $t\in [0,1]$, where $X^h$ satisfies Equation \eqref{eq:limitingcontrolledmckeanvlasovfeedback} on some filtered probability space supporting a $d$-dimensional Brownian motion $\hat{W}$.
\item \label{H:H2}$\int_0^1 \int_{\R^d}|h(t,x)|^2\theta(t)(dx)dt<\infty$ .
\end{enumerate}
Here
\begin{align}\label{eq:limitingcontrolledmckeanvlasovfeedback}
  dX^h_t&= [\bar{\beta}(X^h_t,\mc{L}(X^h_t))+\bar{B}(X^h_t,\mc{L}(X^h_t))h(t,X^h_t)]dt+\bar{B}(X^h_t,\mc{L}(X^h_t))d\hat{W}_t.
\end{align}
\end{defi}
We also define $J^m:C([0,1];\mc{P}(\R^d))\tto [0,+\infty]$ by
\begin{align}\label{eq:markovianratefunction}
J^m(\theta) = \inf_{h\in H(\theta)}\frac{1}{2}\int_0^1 \int_{\R^d}|h(t,x)|^2\theta(t)(dx)dt
\end{align}
if $\theta(0)=\nu_0$ and $J^m(\theta)=+\infty$ otherwise. As always, we take $\inf\br{\emptyset}=+\infty$.

\begin{proposition}\label{prop:markovianratefunction}
$J^m=J$, where $J$ is as in \eqref{eq:ordinaryratefunctionpaths} and $J^m$ is as in \eqref{eq:markovianratefunction}.
\end{proposition}
\begin{proof}
Let $\theta\in C([0,1];\mc{P}(\R^d))$ be such that $J(\theta)<\infty$. Let $\eta>0$, and take $\Theta\in\mc{V}^{av}$ such that $\hat{\nu}_{\Theta}(t)=\theta(t),\forall t\in [0,1]$ and $\frac{1}{2}\E^{\Theta}\biggl[\int_0^1|u(t)|^2 dt\biggr]\leq J(\theta)+\eta$.

Then, by \cite{BSmimicking} Corollary 3.7, there exists $\hat{\theta}\in \mc{P}(\mc{X})$, a filtered probability space, a Brownian motion $\hat{W}$, and a measurable $h:[0,1]\times \R^d\tto \R^d$ such that $\hat{\theta}(t) = \theta(t),\forall t\in [0,1]$, and $\hat{\theta}=\mc{L}(\hat{X})$ solving Equation (\ref{eq:limitingcontrolledmckeanvlasovfeedback}) with this choice of $h$. Thus $h$ satisfies \ref{H:H1} in the definition of $H(\theta) = H(\bar{\Psi}(\hat{\theta}))$ for $\bar{\Psi}$ as in the proof of Proposition \ref{prop:ratefunctionforflow} (note that $\hat{\theta}=\theta$ as elements of $C([0,1];\mc{P}(\R^d))$). Moreover, for $\hat{X}_t$ as in Equation \eqref{eq:BDFlimitavstandard}, for Lebesgue almost every $t\in[0,1]$,
\begin{align*}
h(t,\hat{X}_t) = \E[u(t)|\hat{X}_t].
\end{align*}

Thus:
\begin{align*}
\frac{1}{2}\int_0^1 \int_{\R^d}|h(t,x)|^2\hat{\theta}(t)(dx)dt& =\frac{1}{2}\int_0^1 \int_{\R^d}|h(t,x)|^2\theta(t)(dx)dt\\
& = \frac{1}{2}\E^{\Theta}\biggl[\int_0^1 |h(t,\hat{X}_t)|^2  dt\biggr]\\
& = \frac{1}{2}\E^{\Theta}\biggl[\int_0^1 |\E[u(t)|\hat{X}_t]|^2  dt\biggr]\\
&\leq \frac{1}{2}\E^{\Theta}\biggl[\int_0^1 |u(t)|^2  dt\biggr] \text{ by Jensen's inequality}\\
&\leq J(\theta)+\eta.
\end{align*}
Thus $h$ satisfies \ref{H:H2}, and hence $h\in H(\theta)$. So
\begin{align*}
J^m(\theta)&\leq \frac{1}{2}\int_0^1 \int_{\R^d}|h(t,x)|^2\hat{\theta}(t)(dx)dt\leq J(\theta)+\eta.
\end{align*}
Since $\theta$ and $\eta$ were arbitrary, $J^m\leq J$.

Now, let  $\theta\in C([0,1];\mc{P}(\R^d))$ be such that $J^m(\theta)<\infty$. Let $\eta>0$, and take $h\in H(\theta)$ such that $\frac{1}{2}\int_0^1 \int_{\R^d}|h(t,x)|^2\theta(t)(dx)dt\leq J^m(\theta)+\eta$. Let $(\hat{\W},\hat{\F},\hat{\Prob}),\br{\hat{\F}_t},(\hat{X},\hat{W})$ be a weak solution to Equation \eqref{eq:limitingcontrolledmckeanvlasovfeedback} with this choice of $h$ (which exists via \ref{H:H1}). Consider $F:\mc{X}\times\mc{W}\tto \mc{X}\times \mc{Z}\times \mc{W}$ given by $F(\phi,w) = (\phi,\hat{r}_{\phi},w)$ where for $A\in\mc{B}(\R^d),\Gamma\in\mc{B}([0,1])$:
\begin{align*}
\hat{r}(A\times\Gamma)=\int_\Gamma \delta_{h(t,\phi(t))}(A)dt.
\end{align*}
Define $\Theta\in \mc{P}(\mc{X}\times\mc{Z}\times\mc{W})$ by $\hat{\Prob}\circ (\hat{X},\hat{W})^{-1}\circ F^{-1}$. Then $\hat{\nu}_{\Theta}(t)=\theta(t),\forall t\in[0,1]$, and $\Theta$ satisfies \ref{V:V1BDFav} and \ref{V:V3BDFav} from the definition of $\mc{V}^{av}$. Moreover,
\begin{align*}
\frac{1}{2}\E^{\Theta}\biggl[\int_0^1 |z|^2\hat{\rho}_t(dz)  dt\biggr] & = \frac{1}{2}\E^{\Theta}\biggl[\int_0^1 |h(t,\hat{X}_t)|^2 dt\biggr] =\frac{1}{2}\int_0^1 \int_{\R^d}|h(t,x)|^2\theta(t)(dx)dt.
\end{align*}
So $\Theta$ satisfies \ref{V:V2BDFav} from the definition of $\mc{V}^{av}$, and hence $\Theta\in \mc{V}^{av}$. Then
\begin{align*}
J(\theta)&\leq \frac{1}{2}\E^{\Theta}\biggl[\int_0^1 |z|^2\hat{\rho}_t(dz)  dt\biggr] = \frac{1}{2}\int_0^1 \int_{\R^d}|h(t,x)|^2\theta(t)(dx)dt]\leq  J^m(\theta)+\eta.
\end{align*}
Once again, $\theta$ and $\eta$ were arbitrary, so we are done. \end{proof}

We now use a similar Reisz-representation argument to Lemma 4.8 of \cite{DG} to gain an equivalent ``negative Sobolev'' form of $J^m$ (equivalently of $J$) from Equation \eqref{eq:markovianratefunction}.

In order to do so, we first need to introduce some notation, as borrowed from p.270-271 of \cite{DG}.

For $t\in [0,T]$ and $\theta\in C([0,1];\mc{P}(\R^d))$, we define $\nabla_{\theta(t)},(\cdot,\cdot)_{\theta(t)},$ and $|\cdot|_{\theta(t)}$ be (formally) the Riemannian gradient, inner product, and Riemannian norm in the tangent space of the Riemannian structure on $\R^d$ induced by the diffusion matrix $t\mapsto \bar{D}(\cdot,\theta(t))$, where $\bar{D}$ is defined as in Equation \eqref{eq:McKeanLimit}. I.e:
\begin{align}\label{eq:reimannianstructure}
(\nabla_{\theta(t)} \phi)^i& \coloneqq \sum_{j=1}^d\bar{D}^{i,j}(\cdot,\theta(t))\frac{d\phi}{dx^j},i=1,...,d\\
(X,Y)_{\theta(t)} & \coloneqq \sum_{i,j=1}^d[\bar{D}^{-1}(\cdot,\theta(t))]^{i,j}X^iY^j\nonumber\\
|X|_{\theta(t)}&\coloneqq (X,X)^{1/2}_{\theta(t)}.\nonumber
\end{align}
Note in particular that
\begin{align*}
|\nabla_{\theta(t)}\phi|^2_{\theta(t)}= \sum_{i,j=1}^d[\bar{D}(\cdot,\theta(t))]^{i,j}\frac{d\phi}{dx^i}\frac{d\phi}{dx^j},
\end{align*}
and recall that $\bar{D}(x,\mu)$ uniformly positive definite in $x\in\R^d,\mu\in\mc{P}(\R^d)$ under Assumption \ref{assumption:limitinguniformellipticity}.

Also, we define for fixed $\theta\in C([0,1];\mc{P}(\R^d))$ the linear functional $F_\theta:C^\infty_c(U\times\R^d)\tto \R$ by
\begin{align}\label{eq:Ftheta}
F_\theta(\psi) = \langle \theta(1),\psi(1,\cdot)\rangle -\langle \theta(0),\psi(0,\cdot)\rangle-\int_0^1 \langle \theta(t),\dot{\psi}(t,\cdot)\rangle + \langle \theta(t),\bar{L}_{\theta(t)}\psi(t,\cdot)\rangle dt,
\end{align}
where here $U$ is any open interval in $\R$ containing $[0,1]$ and $\bar{L}_\mu$ is as in Equation \eqref{eq:limitinggenerator}.

We our now ready to define our final intermediate form of the rate function as it acts on $C([0,1];\mc{P}(\R^d))$, which is analogous to Equation (4.21) in \cite{DG}. We define $\bar{J}:C([0,1];\mc{P}(\R^d))\tto [0,+\infty]$ by:
\begin{align}\label{eq:Jbar}
\bar{J}(\theta) = \sup_{\psi\in C^\infty_c(U\times \R^d)}\left\{F_\theta(\psi)-\frac{1}{2}\int_0^1 \langle \theta(t),|\nabla_{\theta(t)}\psi(t,\cdot)|^2_{\theta(t)}\rangle dt\right\}
\end{align}
if $\theta(0)=\nu_0$ and $J^m(\theta)=+\infty$ otherwise.
\begin{lem}\label{lem:intermediateDGform}
$J^m=\bar{J}$, where $J^m$ is as in Equation \eqref{eq:markovianratefunction} and $\bar{J}$ is as in Equation \eqref{eq:Jbar}.
\end{lem}
\begin{proof}
Let $\theta\in C([0,1];\mc{P}(\R^d))$ be such that $J^m(\theta)<\infty$, and consider any $h\in H(\theta)$. Letting $\psi \in C^\infty_c(U\times\R^d)$ and applying It\^o's formula to $\psi(1,X^h_1)$ for $X^h$ a solution to \eqref{eq:limitingcontrolledmckeanvlasovfeedback} and taking expectations, we get
\begin{align*}
F_{\theta}(\psi) = \int_0^1 \langle \theta(t),\biggl(\bar{B}(\cdot,\theta(t)h(t,\cdot)\biggr)\cdot\nabla_x\psi(t,\cdot)\rangle dt.
\end{align*}
Thus we have for all $\psi \in C^\infty_c(U\times\R^d)$:
\begin{align}\label{eq:Fthetabounded}
|F_{\theta}(\psi)|&\leq \biggl(\int_0^1 \langle \theta(t),(\nabla_x\psi)^\top(t,\cdot)\bar{B}(\cdot,\theta(t))\bar{B}^\top(\cdot,\theta(t))\nabla_x\psi(t,\cdot)\rangle dt\biggr)^{1/2} \biggl(\int_0^1\langle \theta(t),|h(t,\cdot)|^2\rangle dt\biggr)^{1/2} \\
&= \biggl(\int_0^1 \langle \theta(t),(\nabla_x\psi)^\top(t,\cdot)\bar{D}(\cdot,\theta(t))\nabla_x\psi(t,\cdot)\rangle dt\biggr)^{1/2} \biggl(\int_0^1\langle \theta(t),|h(t,\cdot)|^2\rangle dt\biggr)^{1/2} \nonumber\\
&= \biggl(\int_0^1\langle \theta(t),|\nabla_{\theta(t)}\phi(\cdot)|^2_{\theta(t)}\rangle dt\biggr)^{1/2}\biggl(\int_0^1\langle \theta(t),|h(t,\cdot)|^2\rangle dt\biggr)^{1/2}.\nonumber
\end{align}
So if $\int_0^1\langle \theta(t),|\nabla_{\theta(t)}\phi(\cdot)|^2_{\theta(t)}\rangle dt=0$, then $F_\theta(\psi)=0$. Thus, using that $F_\theta$ is linear and that $\psi\in C^\infty_c(U\times \R^m)$ if and only if $c \psi \in C^\infty_c(U\times \R^m)$ for all $c\in\R\setminus\br{0}$:
\begin{align*}
\bar{J}(\theta) &= \sup_{\psi\in C^\infty_c(U\times \R^d):\int_0^1\langle \theta(t),|\nabla_{\theta(t)}\phi(\cdot)|^2_{\theta(t)}\rangle dt\neq 0}\br{F_\theta(\psi)-\frac{1}{2}\int_0^1 \langle \theta(t),|\nabla_{\theta(t)}\psi(t,\cdot)|^2_{\theta(t)}\rangle dt}\vee 0\\
& = \sup_{\psi\in C^\infty_c(U\times \R^d):\int_0^1\langle \theta(t),|\nabla_{\theta(t)}\phi(\cdot)|^2_{\theta(t)}\rangle dt\neq 0}\sup_{c\in\R}\br{cF_\theta(\psi)-\frac{c^2}{2}\int_0^1 \langle \theta(t),|\nabla_{\theta(t)}\psi(t,\cdot)|^2_{\theta(t)}\rangle dt}\vee 0\\
& = \sup_{\psi\in C^\infty_c(U\times \R^d):\int_0^1\langle \theta(t),|\nabla_{\theta(t)}\phi(\cdot)|^2_{\theta(t)}\rangle dt\neq 0}\frac{1}{2}\frac{|F_\theta(\psi)|^2}{\int_0^1 \langle \theta(t),|\nabla_{\theta(t)}\psi(t,\cdot)|^2_{\theta(t)}\rangle dt}\vee 0 \\
& = \sup_{\psi\in C^\infty_c(U\times \R^d):\int_0^1\langle \theta(t),|\nabla_{\theta(t)}\phi(\cdot)|^2_{\theta(t)}\rangle dt\neq 0}\frac{1}{2}\frac{|F_\theta(\psi)|^2}{\int_0^1 \langle \theta(t),|\nabla_{\theta(t)}\psi(t,\cdot)|^2_{\theta(t)}\rangle dt}.
\end{align*}
So squaring both sides of the inequality \eqref{eq:Fthetabounded}, we get
\begin{align*}
\frac{1}{2}\frac{|F_\theta(\psi)|^2}{\int_0^1 \langle \theta(t),|\nabla_{\theta(t)}\psi(t,\cdot)|^2_{\theta(t)}\rangle dt}\leq \frac{1}{2}\int_0^1\langle \theta(t),|h(t,\cdot)|^2\rangle dt
\end{align*}
for all $\psi\in C^\infty_c(U\times \R^d)$ such that $\int_0^1\langle \theta(t),|\nabla_{\theta(t)}\phi(\cdot)|^2_{\theta(t)}dt\neq 0$ and all $h\in H(\theta)$. Thus $\bar{J}\leq J^m$.

Now take $\theta\in C([0,1];\mc{P}(\R^d))$ such that $\bar{J}(\theta)<\infty$. Then, since
\begin{align*}
\bar{J}(\theta) &= \sup_{\psi\in C^\infty_c(U\times \R^d)}\sup_{c\in\R}\br{cF_\theta(\psi)-\frac{c^2}{2}\int_0^1 \langle \theta(t),|\nabla_{\theta(t)}\psi(t,\cdot)|^2_{\theta(t)}\rangle dt}\\
& = +\infty
\end{align*}
if there exists $\psi\in C^\infty_c(U\times \R^d)$ such that $\int_0^1 \langle \theta(t),|\nabla_{\theta(t)}\psi(t,\cdot)|^2_{\theta(t)}\rangle dt=0$ and $F_\theta(\psi)\neq 0$, we have once again that
\begin{align*}
\bar{J}(\theta) = \sup_{\psi\in C^\infty_c(U\times \R^d):\int_0^1\langle \theta(t),|\nabla_{\theta(t)}\phi(\cdot)|^2_{\theta(t)}\rangle dt\neq 0}\frac{1}{2}\frac{|F_\theta(\psi)|^2}{\int_0^1 \langle \theta(t),|\nabla_{\theta(t)}\psi(t,\cdot)|^2_{\theta(t)}\rangle dt}.
\end{align*}
Moreover, since $\bar{J}(\theta)$ is bounded by some $C>0$, we get
\begin{align}\label{eq:multiFZbounded}
|F_\theta(\psi)|\leq \sqrt{2 C} \biggl(\int_0^1 \langle \theta(t),|\nabla_{\theta(t)}\psi(t,\cdot)|^2_{\theta(t)}\rangle dt\biggr)^{1/2},\forall \psi \in C^\infty_c(U\times\R^d).
\end{align}

Now, as on p.279 in \cite{DG}, we define $L^2_\theta[0,1]$ to be the Hilbert space of measurable maps $g:[0,1]\times\R^d\tto \R^d$ with finite norm
\begin{align*}
\norm{g}_{L^2_\theta[0,1]} \coloneqq \biggl(\int_0^1 \langle \theta(t), |g(t,\cdot)|^2_{\theta(t)} \rangle dt\biggr)^{1/2}
\end{align*}
and inner product
\begin{align*}
[g_1,g_2]_{L^2_\theta[0,1]}& \coloneqq \int_0^1 \langle \theta(t),(g_1(t,\cdot),g_2(t,\cdot))_{\theta(t)}\rangle dt.
\end{align*}
Denote by $L^2_{\nabla,\theta}[0,1]$ the closure in $L^2_\theta[0,1]$ of the linear subset $L_{\nabla,\theta}$ consisting of all maps $(s,x)\mapsto \nabla_{\theta(s)}\psi(s,x)$, $\psi \in C^\infty_c(U\times\R^d)$. Then $F_\theta$ can be viewed as a linear functional on $L_{\nabla,\theta}$, and by the bound \eqref{eq:multiFZbounded}, is bounded. Then, by the Reisz Representation Theorem, there exists $\bar{h}\in L^2_{\nabla,\theta}[0,1]$ such that
\begin{align}\label{eq:FZreiszrep}
F_\theta(\psi) = \int_0^1 \langle \theta(s),(\bar{h}(s,\cdot),\nabla_{\theta(s)}\psi(s,\cdot))_{\theta(s)} \rangle ds = \int_0^1 \langle \theta(s),\nabla_x\psi(s,\cdot)\cdot \bar{h}(s,\cdot)\rangle ds,\forall \psi \in C^\infty_c(U\times\R^d).
\end{align}
Note that actually, $L_{\nabla,\theta}$ must be considered not as a class of functions, but as a set of equivalence classes of functions agreeing $\theta(t)(dx)dt$-almost surely. This is of no consequence, however, since the bound \eqref{eq:multiFZbounded} ensures that $F_\theta(\psi)=F_\theta(\tilde{\psi})$ if $\nabla_{\theta(\cdot)}\psi$ and $\nabla_{\theta(\cdot)}\tilde{\psi}$ are in the same equivalence class (see p.279 in \cite{DG} and Appendix D.5 in \cite{FK} for a more thorough treatment of the space $L^2_{\nabla,\theta}[0,1]$ and its dual).
Consider now $\tilde{h}:[0,1]\times \R^d\tto \R^d$ given by
\begin{align}\label{eq:tildeh}
\tilde{h}(t,x) &\coloneqq \bar{B}^{-1}(x,\theta(t))\bar{h}(t,x).
\end{align}
Then
\begin{align*}
&\int_0^1 \int_{\R^d}|\tilde{h}(t,x)|^2\theta(t)(dx)dt\\
& = \int_0^1 \int_{\R^d}\bar{h}^\top(t,x)(\bar{B}^{-1})^\top(x,\theta(t))\bar{B}^{-1}(x,\theta(t))\bar{h}(t,x)\theta(t)(dx)dt\\
& = \int_0^1 \int_{\R^d}\bar{h}^\top(t,x)\bar{D}^{-1}(x,\theta(t))\bar{h}(t,x)\theta(t)(dx)dt\\
& = \norm{\bar{h}}_{L^2_\theta[0,1]}^2<\infty.
\end{align*}
Moreover, for any $\psi\in C^\infty_c(U\times\R^d)$, we have
\begin{align*}
&\int_0^1 \langle\theta(t),\biggl[\bar{B}(\cdot,\theta(t))\tilde{h}(t,\cdot)\biggr]\cdot \nabla_x\psi(t,\cdot)\rangle dt = \int_0^1 \langle\theta(t),\bar{h}(t,\cdot)\cdot \nabla_x\psi(t,\cdot)\rangle dt = F_\theta(\psi)
\end{align*}
by Equations \eqref{eq:tildeh} and \eqref{eq:FZreiszrep}. From this, via an approximation argument taking $\psi(t,x)=\beta^k(t)\phi(x)$ for any $\phi\in C^\infty_c(\R^d)$ and $\beta^k:U\tto \R$ which approach $\1_{[0,t]}$, we see for all $\phi\in C^\infty_c(\R^d)$ and all $t\in [0,1]$,
\begin{align*}
\langle \theta(t),\phi\rangle = \langle \nu_0,\phi\rangle + \int_0^t \langle \theta(s),\bar{L}_{\theta(s)}\phi\rangle ds + \int_0^t \langle\theta(s),\biggl[\bar{B}(\cdot,\theta(s))\tilde{h}(s,\cdot)\biggr]\cdot \nabla_x\phi(\cdot)\rangle ds.
\end{align*}
By Section 2 in \cite{BR}, this shows that $\theta(t)=\mc{L}(X^{\tilde{h}}_t)$ for $X^{\tilde{h}}_t$ solving Equation \eqref{eq:limitingcontrolledmckeanvlasovfeedback} with $\tilde{h}$ in the place of $h$. Thus, $\tilde{h}\in H(\theta)$.

Now, since $\bar{h}\in L^2_{\nabla,\theta}[0,1]$, we can take a sequence $\br{\tilde{\psi}^n} \subset L_{\nabla,\theta}$ such that $\tilde{\psi}^n\tto \bar{h}$ in $L^2_\theta[0,T]$. By virtue of $\tilde{\psi}^n\in L_{\nabla,\theta}$, we have for each $n$, there is $\psi^n \in C^\infty_c(U\times\R^d)$ such that $\tilde{\psi}^n(s,x) = \nabla_{\theta(s)}\psi^n(s,x)$. Then $\norm{\nabla_{\theta(\cdot)}\psi^n}_{L^2_\theta[0,1]}^2 \tto \norm{\bar{h}}_{L^2_\theta[0,1]}^2$, and $[\bar{h},\nabla_{\theta(\cdot)}\psi^n]_{L^2_\theta[0,1]}\tto \norm{\bar{h}}_{L^2_\theta[0,1]}^2$. In particular, since
\begin{align}\label{eq:tildehvsbarhnorm}
\int_0^1 \int_{\R^d}|\tilde{h}(t,x)|^2\theta(t)(dx)dt=\norm{\bar{h}}_{L^2_\theta[0,1]}^2
\end{align}
and $\tilde{h}\in H(\theta)$, if $\norm{\bar{h}}_{L^2_\theta[0,1]}^2=0$ then $J^m(\theta)=0$, so we can without loss of generality assume $\norm{\bar{h}}_{L^2_\theta[0,1]}^2\neq 0$ and take a subsequence of $\br{\tilde{\psi}^n}$ such that $\norm{\nabla_{\theta(\cdot)}\psi^n}_{L^2_\theta[0,1]}\neq 0,\forall n\in\bb{N}$.

Then, for all $n\in\bb{N}$,
\begin{align*}
\bar{J}(\theta)&\geq \frac{1}{2}\frac{|F_\theta(\psi^n)|^2}{\int_0^1 \langle \theta(t),|\nabla_{\theta(t)}\psi^n(t,\cdot)|^2_{\theta(t)}\rangle dt}\\
& = \frac{1}{2}\frac{[\bar{h},\nabla_{\theta(\cdot)}\psi^n]^2_{L^2_\theta[0,1]}}{\norm{\nabla_{\theta(\cdot)}\psi^n}_{L^2_\theta[0,1]}^2} \text{ by Equation \eqref{eq:FZreiszrep} and the definition of the inner product}\\
&\tto \frac{1}{2}\norm{\bar{h}}_{L^2_\theta[0,1]}^2 \text{ as }n\toinf \\
& = \int_0^1 \int_{\R^d}|\tilde{h}(t,x)|^2\theta(t)(dx)dt\\
&\geq J^m(\theta).
\end{align*}
\end{proof}
Now we are ready to prove a final Proposition, which will yield the main result of this subsection (Theorem \ref{theorem:DGform}). In order to do so, we will need to recall first the following from \cite{DG}, which is needed for the definition of the rate function $J^{DG}$ from \eqref{eq:DGform}:
\begin{defi}\label{def:absolutelycontinuous}(Definition 4.1 in \cite{DG})
For a compact set $K\subset \R$, we will denote the subspace of $C^\infty_c(\R^d)$ which have compact support contained in $K$ by $\mc{S}_K$.
Let $I$ be an interval on the real line. A map $Z:I\tto \mc{S}'$ is called absolutely continuous if for each compact set $K\subset \R$, there exists a neighborhood of $0$ in $\mc{S}_K$ and an absolutely continuous function $H_K:I\tto \R$ such that
\begin{align*}
|\langle Z(u),\phi\rangle - \langle Z(v),\phi\rangle | \leq |H_K(u)-H_K(v)|
\end{align*}
for all $u,v\in I$ and $\phi\in U_K$.
\end{defi}

It is also useful to recall the following result:
\begin{lem}\label{lemma:DG4.2}(Lemma 4.2 in \cite{DG})
Assume that $Z:I\tto \mc{S}'$ is absolutely continuous. Then the real function $\langle Z,\phi\rangle$ is absolutely continuous for each $\phi\in C^\infty_c(\R^d)$ and the derivative in the distribution sense
\begin{align*}
\dot{Z}(t)\coloneqq \lim_{h\downarrow 0}h^{-1}[Z(t+h)-Z(t)]
\end{align*}
exists for Lebesgue almost-every $t\in I$.
\end{lem}
In the above $\mc{S}$ is the space of Schwartz test function on $\R^d$ and $\mc{S}'$ is its dual. Then, we have the following:
\begin{proposition}\label{prop:intermediatetofinalDG}
$\bar{J}=J^{DG}$, where $\bar{J}$ is as in Equation \eqref{eq:Jbar} and $J^{DG}$ is as in Equation \eqref{eq:DGform}.
\end{proposition}
\begin{proof}
Once again, we are using ideas from the proof of Lemma 4.8 in \cite{DG}.

Take $\theta\in C([0,1];\mc{P}(\R^d))$ such that $I^{DG}(\theta)<\infty$. Note that, similarly to as in the proof of Lemma \ref{lem:intermediateDGform}:
\begin{align*}
&\sup_{\phi\in C^\infty_c(\R^d):\norm{\nabla_{\theta(t)}\phi}_{\theta(t)}\neq 0}\biggl\lbrace\langle\dot{\theta}(t)-\bar{L}^*_{\theta(t)}\theta(t),\phi\rangle - \frac{1}{2}\langle \theta(t),\norm{\nabla_{\theta(t)}\phi}^2_{\theta(t)}\rangle \biggr\rbrace\\
& = \sup_{\phi\in C^\infty_c(\R^d):\norm{\nabla_{\theta(t)}\phi}_{\theta(t)}\neq 0}\sup_{c\in\R}\biggl\lbrace c\langle\dot{\theta}(t)-\bar{L}^*_{\theta(t)}\theta(t),\phi\rangle - \frac{c^2}{2}\langle \theta(t),\norm{\nabla_{\theta(t)}\phi}^2_{\theta(t)}\rangle \biggr\rbrace\\
& = \frac{1}{2}\sup_{\phi\in C^\infty_c(\R^d):\norm{\nabla_{\theta(t)}\phi}_{\theta(t)}\neq 0}\frac{|\langle\dot{\theta}(t)-\bar{L}^*_{\theta(t)}\theta(t),\phi\rangle|^2}{\langle \theta(t),\norm{\nabla_{\theta(t)}\phi}^2_{\theta(t)}\rangle}.
\end{align*}
So for any $\psi \in C^\infty_c(U\times \R^d)$:
\begin{align*}
J^{DG}(\theta) &= \frac{1}{2}\int_0^1 \sup_{\phi\in C^\infty_c(\R^d):\norm{\nabla_{\theta(t)}\phi}_{\theta(t)}\neq 0}\frac{|\langle\dot{\theta}(t)-\bar{L}^*_{\theta(t)}\theta(t),\phi\rangle|^2}{\langle \theta(t),\norm{\nabla_{\theta(t)}\phi}^2_{\theta(t)}\rangle}dt\\
& = \int_0^1 \sup_{\phi\in C^\infty_c(\R^d):\norm{\nabla_{\theta(t)}\phi}_{\theta(t)}\neq 0}\biggl\lbrace\langle\dot{\theta}(t)-\bar{L}^*_{\theta(t)}\theta(t),\phi\rangle - \frac{1}{2}\langle \theta(t),\norm{\nabla_{\theta(t)}\phi}^2_{\theta(t)}\rangle \biggr\rbrace dt\\
&\geq \int_0^1 \langle\dot{\theta}(t)-\bar{L}^*_{\theta(t)}\theta(t),\psi(t,\cdot)\rangle - \frac{1}{2}\langle \theta(t),\norm{\nabla_{\theta(t)}\psi(t,\cdot)}^2_{\theta(t)}\rangle dt\\
& = \langle \theta(1),\psi(1,\cdot)\rangle -\langle \theta(0),\psi(0,\cdot)\rangle-\int_0^1 \langle \theta(t),\dot{\psi}(t,\cdot)\rangle + \langle \theta(t),\bar{L}_{\theta(t)}\psi(t,\cdot)\rangle dt\\
&\hspace{6cm}- \int_0^1 \frac{1}{2}\langle \theta(t),\norm{\nabla_{\theta(t)}\psi(t,\cdot)}^2_{\theta(t)}\rangle dt,
\end{align*}
where in the last step we used Lemma 4.3 in \cite{DG}. Then taking the supremum over all $\psi\in C^\infty_c(U\times\R^d)$, we get $J^{DG}(\theta)\geq \bar{J}(\theta)$.

Now take $\theta\in C([0,1];\mc{P}(\R^d))$ such that $\bar{J}(\theta)<\infty$. Using that, per Lemma \ref{lem:intermediateDGform}, $\bar{J}=J^m$, and recalling that $\tilde{h}$ defined in Equation \eqref{eq:tildeh} is in $H(\theta)$, we have that $\theta(t)=\mc{L}(X^{\tilde{h}}_t)$ solving Equation \eqref{eq:limitingcontrolledmckeanvlasovfeedback} with $h=\tilde{h}$. Thus, taking any $0\leq s<t\leq 1$ and $\phi\in C^\infty_c(\R^d)$ and applying It\^o's formula to $\phi(X^{\tilde{h}}_t)$ and $\phi(X^{\tilde{h}}_s)$ and taking expectations, we have
\begin{align*}
\langle \theta(t),\phi\rangle - \langle \theta(s),\phi\rangle & = \int_s^t \langle \theta(u),\bar{L}_{\theta(u)}\phi\rangle du + \int_s^t \langle \theta(u),\bar{h}(u,\cdot)\cdot \nabla \phi\rangle du,
\end{align*}
where here we recall the definition of $\bar{h}$ from Equation \eqref{eq:FZreiszrep}.
We also have that, as per the proof of Lemma \ref{lem:intermediateDGform} and Equation \eqref{eq:tildehvsbarhnorm}:
\begin{align*}
\bar{J}(\theta) = J^m(\theta) = \frac{1}{2}\norm{\bar{h}}^2_{L^2_\theta[0,1]}<\infty
\end{align*}
by assumption, so by Definition \ref{def:absolutelycontinuous} and boundedness of the coefficients from Corollary \ref{cor:limitingcoefficientsregularity}, $\theta$ is an absolutely continuous map from $[0,1]$ to $\mc{S}'$.

Then, using Lemma \ref{lemma:DG4.2}, we have for each $\phi\in C^\infty_c(\R^d)$:
\begin{align}\label{eq:Zdotrep}
\langle \dot{\theta}_t,\phi\rangle &= \langle \theta(t),\bar{L}_{\theta(t)}\phi\rangle + \langle \theta(t),\bar{h}(t,\cdot)\cdot \nabla \phi\rangle.
\end{align}
Using a density argument, we can make sure this holds simultaneously for all $\phi \in C^\infty_c(\R)$ and Lebesgue almost every $t\in [0,T]$ (see p.280 of \cite{DG}).

This gives:
\begin{align*}
J^{DG}(\theta) = \frac{1}{2}\int_0^1 \sup_{\phi\in C^\infty_c(\R^d):\norm{\nabla_{\theta(t)}\phi}_{\theta(t)}\neq 0}\frac{|\langle \theta(t),\bar{h}(t,\cdot)\cdot \nabla \phi\rangle|^2}{\langle \theta(t),\norm{\nabla_{\theta(t)}\phi}^2_{\theta(t)}\rangle}dt.
\end{align*}
For any $\phi \in C^\infty_c(\R^d)$ and $t\in [0,1]$ such that $\norm{\nabla_{\theta(t)}\phi}_{\theta(t)}\neq 0$, we have
\begin{align*}
\frac{|\langle \theta(t),\bar{h}(t,\cdot)\cdot \nabla \phi\rangle|^2}{\norm{\nabla_{\theta(t)}\phi}^2_{\theta(t)}}& = \frac{|\langle \theta(t),(\bar{h}(t,\cdot),\nabla_{\theta(t)}\phi)_{\theta(t)}|^2}{\langle \theta(t),\norm{\nabla_{\theta(t)}\phi}^2_{\theta(t)}\rangle}\\
&\leq \frac{|\langle \theta(t),\norm{\bar{h}(t,\cdot)}_{\theta(t)}\norm{\nabla_{\theta(t)}\phi}_{\theta(t)}|^2}{\langle \theta(t),\norm{\nabla_{\theta(t)}\phi}^2_{\theta(t)}\rangle}\\
&\leq \frac{\langle \theta(t),\norm{\bar{h}(t,\cdot)}^2_{\theta(t)}\rangle \langle \theta(t),\norm{\nabla_{\theta(t)}\phi}^2_{\theta(t)}}{\langle \theta(t),\norm{\nabla_{\theta(t)}\phi}^2_{\theta(t)}\rangle}\\
& = \langle \theta(t),\norm{\bar{h}(t,\cdot)}^2_{\theta(t)}\rangle.
\end{align*}
So
\begin{align*}
J^{DG}(\theta) &\leq \frac{1}{2}\int_0^1 \langle \theta(t),\norm{\bar{h}(t,\cdot)}^2_{\theta(t)}\rangle dt = \frac{1}{2}\norm{\bar{h}}^2_{L^2_\theta[0,1]} = \bar{J}(\theta).
\end{align*}
\end{proof}
Now we prove Theorem \ref{theorem:DGform}:
\begin{proof}\textit{(Theorem \ref{theorem:DGform})}
Once again, we are using that the rate function for a sequence of random variables satisfying the large deviations principle is unique. Via Proposition \ref{prop:ratefunctionforflow}, we know both $\br{\mu^N_\cdot}$ from Equation \eqref{eq:empiricalmesaure} and $\br{\hat{\mu}^N_\cdot}$  from Equation \eqref{eq:averagedempiricalmeasure} satisfy the LDP with rate function $J$ given by Equation \eqref{eq:ordinaryratefunctionpaths} under the given assumptions. By Proposition \ref{prop:markovianratefunction}, we know $J=J^m$, where $J^m$ is given by Equation \eqref{eq:markovianratefunction}, by Lemma \ref{lem:intermediateDGform} we know $J^m=\bar{J}$ where $\bar{J}$ is given by Equation \eqref{eq:Jbar}, and by Proposition \ref{prop:intermediatetofinalDG} we know $\bar{J}=J^{DG}$. So in fact we have not only that $\br{\mu^N_\cdot}$ satisfies the result of Theorem \ref{theorem:DGform}, but so does $\br{\hat{\mu}^N_\cdot}$.
\end{proof}

\begin{remark}\label{remark:comparingwiththesmallnoisecase}
Note that the result of Theorem \ref{theorem:DGform} is completely analogous to the situation in the small noise case. Via the classical work of Freidlin and Wentzell \cite{FW}, the large deviations rate function on $C([0,1];\R^d)$ for small noise SDEs is given by the $L^2$ in time norm of $\dot{\phi}(t)-b(\phi(t))$, where the law of large numbers for the SDE is such that this term is $0$ for all time, and where the norm on $\R^d$ used at each time is that induced by the inverse of the diffusion matrix evaluated at $\phi(t)$. Similarly, the rate function on $C([0,1];\mc{P}(\R^d))$ for the empirical measure of weakly interacting diffusions, as per the classical work of Dawson and G\"artner \cite{DG} (and as extended to the case where the diffusion coefficient depends on the measure parameter in Theorem \eqref{theorem:DGform} above), is the $L^2$ in time norm of $\dot{\theta}(t)-\mc{L}^*_{\theta(t)}\theta(t)$, where the law of large numbers for the empirical measure is such that this term is $0$ for all time, and the norm on $\mc{S}'$ is that induced by the inverse of the diffusion matrix evaluated at $\theta(t)$.

To define the appropriate norm on $\mc{S}'$, we treat the time derivative of $\theta$ as an element of \newline$\overline{\br{\nabla \phi : \phi \in C^\infty_c(\R^d)}}^{L^2(\R^d,\theta(s))}$ via identification with its representative element in this space as a bounded linear functional on it, and with the norm on $\R^d$ once again being that induced by the inverse diffusion matrix. This space is known to be the tangent space to $\theta(s)$ in $\mc{P}_2(\R^d)$ (see Section 8.5 in \cite{GradientFlows}), and note that, in a parallel direction, the tangent space to $\phi(s)$ in the small noise case is simply $\R^d$.

Moreover, in light of our results, it seems this more general principle pervades with the addition of averaging in the joint limit. In Regime 1 of \cite{DS}, it is shown that the analogous large deviations principle to \cite{FW} with the addition of averaging results in a rate function which is exactly that of \cite{FW}, but where both the drift and diffusion are replaced by the effective drift and diffusion obtained from sending $\epsilon\downarrow 0$ before taking the small noise limit. Similarly here, we obtain that the analogous rate function to \cite{DG} with the addition of averaging results in a rate function which is exactly that of \cite{FW}, but where both the drift and diffusion and replaced with the effective drift and diffusion obtained from sending $\epsilon\downarrow 0$ before taking the large particle limit.

For a more concrete example where this effect can easily be observed, see Corollary \ref{cor:aggregationdiffusionequationresult} here and the analogous Corollary 5.4 in \cite{DS}.

\end{remark}

\section{Limiting Behavior of the Controlled Empirical Measure}\label{sec:limitingbehavior}
Throughout this Section we assume \ref{assumption:initialconditions}-\ref{assumption:limitinguniformellipticity}.

Our object of study in this section is the family of occupation measures $\br{Q^{N}}_{N\in\bb{N}} \in \mc{P}(\mc{C})$ defined by:
\begin{align}
\label{eq:occmeas}
    Q^N_\omega(A\times B\times C) = \frac{1}{N}\sum_{i=1}^N \delta_{\bar{X}^{i,N}(\omega)}(A) \delta_{\rho^{i,N}(\omega)}(B)\delta_{\bar{W}^{i,N}(\omega)}(C)
\end{align}
for $A\times B\times C\in \mc{B}(\mc{C}),\omega\in\W$, $\rho^{i,N}$ the relaxed controls corresponding to $u^{N}_i$ via Equation \eqref{eq:canonicalprocess}, $\bar{X}^{i,N}$ as in (\ref{eq:controlledprelimit}) controlled by the same controls used to construct $\rho^{i,N}$, and
\begin{align}\label{eq:barW}
\bar{W}^{i,N}_\cdot(\omega)&\coloneqq \int_0^\cdot \bar{B}^{-1}(\bar{X}^{i,N}_s(\omega),\bar{\mu}^N_s(\omega))[I+\nabla_y\Phi(\bar{X}^{i,N}_s(\omega),\bar{X}^{i,N}_s(\omega)/\epsilon,\bar{\mu}^N_s(\omega))]\\
&\hspace{6cm}\sigma(\bar{X}^{i,N}_s(\omega),\bar{X}^{i,N}_s(\omega)/\epsilon,\bar{\mu}^N_s(\omega))dW^i_s(\omega).\nonumber
\end{align}
Here we recall the definition of $\bar{B}(x,\mu)$ from Equations \eqref{eq:McKeanLimit} and \eqref{eq:equivalentaveragediffusion} and of $\bar{\mu}^N$ from Equation \eqref{eq:barmu}. Also, the Brownian Motions used to construct $\bar{W}^{i,N}$ are the same as those driving the controlled particles $\bar{X}^{i,N}$ as per Equation \eqref{eq:controlledprelimit}, and $\bar{B}(x,\mu)$ is invertible for all $x\in \R^d,\mu\in\mc{P}(\R^d)$ with uniformly bounded inverse by Corollary \ref{cor:limitingcoefficientsregularity}. We use the convention that for $s>1$, $u^N_i(s)=0,\forall i,N\in\bb{N}$.
\begin{remark}\label{remark:ontheoccmeasures}
The intuition for the construction of the occupation measures in Equation \eqref{eq:occmeas} is as follows:

$Q^N$ will be shown to converge in distribution to $Q$ which is almost surely in $\mc{V}$ from Definition \ref{def:V}. In particular, $Q$ will almost-surely satisfy \eqref{eq:controlledMcKeanLimit}. The $Q^N_{\mc{X}}$ is the controlled empirical measure $\bar{\mu}^N$ which enters the right-hand side of the prelimit Laplace Principle expression \eqref{eq:varrep}, and thus tracking this component is clearly necessary when taking the limits \eqref{eq:laplaceprinciplelowerbound} and \eqref{eq:laplaceprincipleupperbound}.

$Q^N_{\mc{Y}}$ is the empirical measure on the relaxed prelimit controls joined with the fast process, and will converge to $\rho_t(dydz)dt$ such that the $\bb{T}^d$-marginal of $\rho$ is almost surely the invariant measure $\pi$, as per \ref{V:V4}. It is extremely important to our proof, and in particular to the proof of the Laplace Principle Upper Bound in Section \ref{sec:upperbound}, that the invariant measure and controls are joined in this way. Namely, considering the term $u(t,y)\coloneqq \int_{\R^m}z\gamma_t(dz;y)$ which appears in Equation \eqref{eq:controlledMcKeanLimit} by decomposing $\rho_t(dydz)=\gamma_t(dz;y)\pi(dy;\bar{X}_t,\mc{L}(\bar{X}_t))$, it is precisely the degree of freedom gained by having $u$ depend on $y\in\bb{T}^d$ in addition to $t\in [0,1]$ that allows us to prove the equivalent variational formulation \eqref{eq:BDFratefunctionav} for the rate function used in the upper bound in Subsection \ref{subsec:altvariationalform}.

Lastly, $Q^N_{\mc{W}}$ will converge to the law of the driving Brownian motion in Equation \eqref{eq:controlledMcKeanLimit}. This component is vital for tracking the joint distribution of the prelimit controls and Brownian motions. The unusual construction of this marginal compared to the case without averaging compared to, e.g. \cite{BDF} Equation (5.2), where the IID Brownian motions from Equation \eqref{eq:multimeanfield} can be used instead, is due to the averaging effect. Namely, even in the one-dimensional setting, it is not $(X^\epsilon,W)$ which will jointly converge to the averaged limit $(\bar{X},W)$, but rather $(X^\epsilon,\bar{W}^\epsilon)$, where $\bar{W}^\epsilon$ is as in Equation \eqref{eq:barW} but without the empirical measure dependence. Indeed, if $m\neq d$, then we do not even have $W^i\in \mc{W}$. See \cite{Bensoussan} Remark 3.4.4 for an illustrative example regarding the effect of the change in driving Brownian motion with averaging, and \cite{Kushner} p.76 for a construction of a similar martingale to $\bar{W}^{i,N}$ in a simpler, one-dimensional setting in the context of singularly-perturbed control problems.
\end{remark}

Assume that there exists $C_{con}>0$ such that
\begin{align}
\label{eq:controlL2boundunspecific}
\sup_{N\in\bb{N}}\E\biggl[\frac{1}{N}\sum_{i=1}^N \int_0^1 |u^N_i(t)|^2dt\biggr]\leq C_{con}.
\end{align}

We will prove the following two propositions:
\begin{proposition}
Under assumption (\ref{eq:controlL2boundunspecific}), the sequence $\br{\mc{L}(Q^N)}_{N\in\bb{N}}$ is precompact in $\mc{P}(\mc{P}(\mc{C}))$.
\end{proposition}
\begin{proof}
See Subsection \ref{subsection:tightness}.
\end{proof}
\begin{proposition}
Under assumption (\ref{eq:controlL2boundunspecific}), for $Q$ such that $\mc{L}(Q^N)\tto \mc{L}(Q)$ in $\mc{P}(\mc{P}(\mc{C}))$ along any subsequence, $Q \in\mc{V}$ almost surely, where the class of measures $\mc{V}$ is described in Definition \ref{def:V}.
\end{proposition}
\begin{proof}
See Subsection \ref{subsection:identifylimit}.
\end{proof}

\subsection{Tightness of the Occupation Measures}
\label{subsection:tightness}

We prove tightness of the occupation measures $\br{Q^N}_{N\in\bb{N}}$ defined in Equation \eqref{eq:occmeas} as $\mc{P}(\mc{C})$-valued random variables by proving tightness of each of the marginals, $Q^N_\mc{Y}$, $Q^N_{\mc{W}}$, and $Q^N_{\mc{X}}=\bar{\mu}^N$ (as defined in Equation \eqref{eq:barmu}).

\subsubsection{Tightness of $Q^N_{\mc{Y}}$}\label{subsubsection:TightnessQ^N_Z}
This will follow analogously to p.88-89 of \cite{BDF}. We recall here the arguments for the readers convenience. Observe that
\begin{align*}
g(r):= \int_{\R^{m}\times [0,1]} |z|^2r(dydzdt)
\end{align*}
is a tightness function on $\mc{Y}$. Namely,  it is bounded from below and has relatively compact level sets. Indeed, boundedness from below is obvious and in order to confirm  the second property, for $c\in(0,\infty)$ let us set $R_c:=\br{r\in\mc{Y}:g(r)\leq c}$. Chebyshev's inequality for $M>0$ gives that
\begin{align}
\label{eq:chebybound}
\sup_{r\in R_c} r(\bb{T}^d\times\br{z\in\R^{m}:|z|>M}\times[0,1])\leq \frac{c}{M^2}.
\end{align}

Therefore, $R_c$ is tight and thus relatively compact as a subset of $\mc{Y}$. Let $\br{r_n}_{n\in\bb{N}}\subset R_c$ be such that $\br{r_n}_{n\in\bb{N}}$ converges weakly to $r_*\in \mc{Y}$. We need to show that $r_*$ has finite first moment and that first moments of $\br{r_n}_{n\in\bb{N}}$ converge to the first moment of $r_*$. By Jensen's inequality and  Fatou's lemma (Theorem A.3.12 in \cite{DE}),
\begin{align*}
\sqrt{c}\geq \liminf_{n\toinf} \sqrt{g(r_n)} \geq  \liminf_{n\toinf} \int_{\bb{T}^d\times\R^{m}\times[0,1]}|z|r_n(dydzdt)\geq \int_{\bb{T}^d\times\R^{m}\times [0,1]} |z|r_*(dydzdt) .
\end{align*}

Now letting $M>0$, by Equation \eqref{eq:chebybound} and H\"older's inequality, we have for all $r\in R_c$,
\begin{align*}
\int_{\bb{T}^d\times\R^{m}\times [0,1]} \1_{\br{z\in\R^{m}:|y|>M}} |z|r(dydzdt)&\leq \sqrt{\int_{\bb{T}^d\times\R^{m}\times [0,1]} |z|^2 r(dydzdt)\int_{\bb{T}^d\times\R^{m}\times [0,1]}\1_{\br{z\in\R^{m}:|z|>M}} r(dydz dt)}\\
&\leq \sqrt{c \frac{c}{M^2}} = \frac{c}{M}.
\end{align*}

So by reverse Fatou's Lemma we get
\begin{align*}
\limsup_{n\toinf}\int_{\bb{T}^d\times\R^{m}\times [0,1]} |z|r_n(dydzdt)&\leq \frac{c}{M}+\int_{\bb{T}^d\times\R^{m}\times [0,1]}\1_{\br{z\in\R^{m}:|z|\leq M}}|z|r_*(dydz dt)\\
&\leq \frac{c}{M}+\int_{\bb{T}^d\times\R^{m}\times [0,1]} |z|r_*(dydzdt).
\end{align*}

Given that $M$ may be taken to be arbitrarily large, we have
\begin{align*}
\lim_{n\toinf}\int_{\bb{T}^d\times\R^{m}\times [0,1]} |z|r_n(dydzdt) = \int_{\bb{T}^d\times\R^{m}\times [0,1]} |z|r_*(dydz dt).
\end{align*}

Thus we have $g$ is a tightness function on $\mc{R}_1$. Now define $G:\mc{P}(\mc{Y})\tto [0,\infty]$ by
\begin{align*}
G(\Theta):=\int_{\mc{Y}}g(r)\Theta(dr).
\end{align*}

Then $G$ is a tightness function on $\mc{P}(\mc{Y})$ (see Theorem A.3.17 in \cite{DE}). Thus in order to prove tightness of $\br{Q^N_{\mc{Y}}}_{N\in\bb{N}}$, it is enough to show that
\begin{align*}
\sup_{N\in\bb{N}} \E[G(Q^N)]<\infty.
\end{align*}

But this follows immediately from assumption (\ref{eq:controlL2boundunspecific}), since by definition of $G$ and $Q^N$,
\begin{align*}
\E[G(Q^N_{\mc{Y}})] &= \E[\int_{\mc{Y}} \int_{\bb{T}^d\times\R^{m}\times [0,1]} |z|^2 r(dydzdt)Q^N_{\mc{Y}}(dr)]\\
&= \E[\frac{1}{N}\sum_{i=1}^N\int_{\bb{T}^d\times\R^{m}\times [0,1]} |z|^2 \rho^{i,N}(dydzdt)]\\
&= \E[\frac{1}{N}\sum_{i=1}^N\int_0^1 |u_i^N(t)|^2 dt]\\
&<\infty.
\end{align*}

\subsubsection{Tightness of $Q^N_{\mc{W}}$}\label{subsubsection:TightnessQ^N_W}
Consider the function $G:\mc{P}(\mc{W})\tto [0,+\infty]$ given by
\begin{align}\label{eq:PXtightnessfunction}
G(\mu) = \int_{\mc{X}}\biggl\lbrace |\phi(0)| + \sup_{\eta\in(0,1]}\eta^{-1/4}\sup_{s,t\in [0,1]: |s-t|\leq \eta}|\phi(s)-\phi(t)|\bigg\rbrace\mu(d\phi).
\end{align}
Then, as per the discussion on p. 1798 of \cite{FischerFormofRateFunction}, $G$ is a tightness function on $\mc{P}(\mc{W})$ in the sense of \cite{DE} p.309. So, to show $\br{Q^N_{\mc{W}}}_{N\in\bb{N}}$ is tight as a sequence of $\mc{P}(\mc{W})$-valued random variables, it suffices to show
\begin{align*}
\sup_{N\in\bb{N}}\E[G(Q^N_{\mc{W}})]<\infty.
\end{align*}
We have for each $N\in\bb{N}$:
\begin{align*}
\E[G(Q^N_{\mc{W}})]& = \frac{1}{N}\sum_{i=1}^N\E\biggl[\sup_{\eta\in(0,1]}\eta^{-1/4}\sup_{s,t\in [0,1]: |s-t|\leq \eta}\biggl|\int_s^t\bar{B}^{-1}(\bar{X}^{i,N}_\tau,\bar{\mu}^N_\tau)[I+\nabla_y\Phi(\bar{X}^{i,N}_\tau,\bar{X}^{i,N}_\tau/\epsilon,\bar{\mu}^N_\tau)]\\
&\hspace{6cm}\sigma(\bar{X}^{i,N}_\tau,\bar{X}^{i,N}_\tau/\epsilon,\bar{\mu}^N_\tau)dW^i_\tau \biggr|\biggr]
\end{align*}
Since by assumption \ref{assumption:LipschitzandBounded} $\sigma$ is bounded, by Proposition \ref{prop:Phiexistenceregularity} $\nabla_y\Phi$ is bounded, and by Corollary \ref{cor:limitingcoefficientsregularity} $\bar{B}^{-1}$ is bounded, we can apply Lemma C.1 in \cite{FischerFormofRateFunction} to get there is $C>0$ such that for all $i\in\br{1,...,N}$:
\begin{align*}
&\E\biggl[\sup_{\eta\in(0,1]}\eta^{-1/4}\sup_{s,t\in [0,1]: |s-t|\leq \eta}\biggl|\int_s^t\bar{B}^{-1}(\bar{X}^{i,N}_\tau,\bar{\mu}^N_\tau)[I+\nabla_y\Phi(\bar{X}^{i,N}_\tau,\bar{X}^{i,N}_\tau/\epsilon,\bar{\mu}^N_\tau)]\sigma(\bar{X}^{i,N}_\tau,\bar{X}^{i,N}_\tau/\epsilon,\bar{\mu}^N_\tau)dW^i_\tau \biggr|\biggr]\\
&\leq C.
\end{align*}
Then
\begin{align*}
\E[G(Q^N_{\mc{W}})]\leq C,
\end{align*}
and since this bound is uniform in $N$, tightness of $\br{Q^N_{\mc{W}}}_{N\in\bb{N}}$ is proved.

\subsubsection{Tightness of $Q^N_{\mc{X}}$ (In $\mc{P}(\mc{X})$)}\label{subsubsection:TightnessQ^N_X}

\begin{proposition}
$\br{Q^N_{\mc{X}}}_{N\in\bb{N}}$ is tight as a sequence of $\mc{P}(\mc{X})$-valued random variables.
\end{proposition}
\begin{proof}
This follows along the lines of Proposition C.3. in \cite{FischerFormofRateFunction}, with appropriate modifications in order to account of the multiscale structure in the controlled interacting particle system \eqref{eq:controlledprelimit}.

We begin by showing an attempt to use the approach of Proposition C.3. in \cite{FischerFormofRateFunction}, after which we proceed to resolve a minor technical problem which arises in the course of the proof.

Consider the tightness function $G:\mc{P}(\mc{X})\tto [0,+\infty]$ given by Equation \eqref{eq:PXtightnessfunction}. Recall that $\mc{X}=\mc{W}=C([0,1];\R^d)$, so indeed we can define the tightness function for the $\mc{X}$-marginals in the same way as for the $\mc{W}$-marginals. So, to show $\br{Q^N_{\mc{X}}}$ is tight as a sequence of $\mc{P}(\mc{X})$-valued random variables, it suffices to show
\begin{align*}
\sup_{N\in\bb{N}}\E[G(Q^N_{\mc{X}})]=\sup_{N\in\bb{N}}\E[G(\bar{\mu}^N)]<\infty.
\end{align*}

Applying It\^o's formula (using Equations \eqref{eq:empfirstder} and \eqref{eq:empsecondder} in Proposition \ref{prop:empprojderivatives}, the regularity of $\Phi$ from Proposition \ref{prop:Phiexistenceregularity}, and Proposition \ref{prop:uniformXL2bound}) to $\Phi(\bar{X}^{i,N}_t ,\bar{X}^{i,N}_t/\epsilon,\bar{\mu}^N_t)$ and rearranging, we have for all $t\in [0,1]$:
\begin{align*}
\bar{X}^{i,N}_t = x^{i,N}+ \sum_{k=1}^8 C^{i,N}_k(t)
\end{align*}
where
\begin{align*}
C^{i,N}_{1}(t)& = \int_0^t \biggl[[I+\epsilon \nabla_x \Phi +\nabla_y \Phi ]b+\nabla_x \Phi f+ A:[\nabla_x\nabla_y\Phi +\frac{\epsilon}{2}\nabla_x\nabla_x \Phi]\biggr] d\tau\\
C^{i,N}_2(t)& = \int_0^t \biggl[[I + \epsilon\nabla_x \Phi +\nabla_y \Phi]\sigma u_i^N\biggr] d\tau\\
C^{i,N}_3(t)& = \int_0^t  [I +\epsilon\nabla_x\Phi +\nabla_y \Phi]\sigma dW^i_\tau \\
C^{i,N}_4(t)& =\int_0^t\biggl[\int_{\R^d} \partial_\mu \Phi(\bar{X}^{i,N}_\tau,\bar{X}^{i,N}_\tau/\epsilon,\bar{\mu}^N_\tau)(v)\biggl[f(v,v/\epsilon,\bar{\mu}^N_\tau)+ \epsilon b(v,v/\epsilon,\bar{\mu}^N_\tau) \biggr]+ \frac{1}{2}A(v,v/\epsilon,\bar{\mu}^N_\tau)\\
&\hspace{2cm}:\biggl[\epsilon \partial_v\partial_\mu \Phi(\bar{X}^{i,N}_\tau,\bar{X}^{i,N}_\tau/\epsilon,\bar{\mu}^N_\tau)(v) + \frac{\epsilon}{N}\partial^2_\mu\Phi(\bar{X}^{i,N}_\tau,\bar{X}^{i,N}_\tau/\epsilon,\bar{\mu}^N_\tau)(v,v)  \biggr] \bar{\mu}^N_\tau(dv)\biggr] d\tau\\
C^{i,N}_5(t)& = \int_0^t \biggl[\frac{\epsilon}{N}\sum_{j=1}^N \partial_\mu \Phi(\bar{X}^{i,N}_\tau,\bar{X}^{i,N}_\tau/\epsilon,\bar{\mu}^N_\tau)(\bar{X}^{j,N}_\tau)\sigma(\bar{X}^{j,N}_\tau,\bar{X}^{j,N}_\tau/\epsilon,\bar{\mu}^N_\tau) u_j^N(s)\biggr] d\tau\\
C^{i,N}_6(t)& =  \int_0^t \frac{\epsilon}{N} \sum_{j=1}^N \partial_\mu \Phi(\bar{X}^{i,N}_\tau,\bar{X}^{i,N}_\tau/\epsilon,\bar{\mu}^N_\tau)(\bar{X}^{j,N}_\tau)\sigma(\bar{X}^{j,N}_\tau,\bar{X}^{j,N}_\tau/\epsilon,\bar{\mu}^N_\tau)dW^j_\tau\\
C^{i,N}_7(t)& = \int_0^t\biggl[ A:[\frac{1}{N}\nabla_{y} \partial_\mu \Phi(\bar{X}^{i,N}_\tau,\bar{X}^{i,N}_\tau/\epsilon,\bar{\mu}^N_\tau)(\bar{X}^{i,N}_\tau) + \frac{\epsilon}{N}\nabla_{x} \partial_\mu \Phi(\bar{X}^{i,N}_\tau,\bar{X}^{i,N}_\tau/\epsilon,\bar{\mu}^N_\tau)(\bar{X}^{i,N}_\tau)  ] \biggr] d\tau \\
C^{i,N}_8(t)& = \epsilon \biggl[\Phi(x^{i,N},x^{i,N}/\epsilon,\bar{\mu}^N_0) -  \Phi (\bar{X}^{i,N}_t,\bar{X}^{i,N}_t/\epsilon,\bar{\mu}^N_t)\biggr],
\end{align*}
where all arguments where omitted are $(\bar{X}^{i,N}_\tau,\bar{X}^{i,N}_\tau/\epsilon,\bar{\mu}^N_\tau)$.
Then, by construction, we have for all $N\in\bb{N}$
\begin{align}\label{eq:GbarmuN}
&\E[G(\bar{\mu}^N)]= \frac{1}{N}\sum_{i=1}^N |x^{i,N}| + \frac{1}{N}\sum_{i=1}^N\E\biggl[\sup_{\eta\in(0,1]}\eta^{-1/4}\sup_{s,t\in [0,1]: |s-t|\leq \eta}\biggl|\sum_{k=1}^8 \br{C^{i,N}_k(t)-C^{i,N}_k(s)}\biggr|\biggr]\\
&\leq \frac{1}{N}\sum_{i=1}^N |x^{i,N}| + \frac{1}{N}\sum_{i=1}^N\E\biggl[\sup_{\eta\in(0,1]}\eta^{-1/4}\sum_{k=1}^8 \sup_{s,t\in [0,1]: |s-t|\leq \eta}|C^{i,N}_k(t)-C^{i,N}_k(s)|\biggr]\nonumber
\end{align}
Immediately we can see there is a possible issue with the term
\begin{align*}
C^{i,N}_8(t)-C^{i,N}_8(s)=\epsilon \biggl[\Phi(\bar{X}^{i,N}_s,\bar{X}^{i,N}_s/\epsilon,\bar{\mu}^N_s) -  \Phi (\bar{X}^{i,N}_t,\bar{X}^{i,N}_t/\epsilon,\bar{\mu}^N_t)\biggr],
\end{align*}
since it is not clear that $\sup_{s,t\in [0,1]: |s-t|\leq \eta}|C^{i,N}_8(t)-C^{i,N}_8(s)|\leq C \eta^\beta$ for $\beta\geq 1/4$.

In the standard case of averaging for single particle systems, where one only needs to establish tightness of $X^\epsilon$ in $\mc{X}$, this is overcome by the fact that one can take $\limsup_{N\toinf}$ before taking $|s-t|\tto 0$ in the characterization of compactness for $\mc{X}$-valued random variables - see e.g. Theorem 7.3 on p. 8.2 of \cite{billingsley}. As $\epsilon\downarrow 0$, the ``undesirable'' term captured in $C^{i,N}_8$ vanishes, so it becomes a non-issue (see Equation (3.2) in \cite{DS} and the argument thereafter).

Since it is this same characterization of compact sets in $\mc{P}(\mc{X})$ from which we construct our tightness function $G$, we expect the same to be true here. However, it is not immediately obvious how to account for the fact that $C^{i,N}_8$ vanishes as $N\toinf$ when constructing $G$.

We avoid this problem by constructing an auxiliary sequence of random measures which is tight, and from which we will be able to conclude the tightness of $\br{Q^N_{\mc{X}}}_{N\in\bb{N}}= \br{\bar{\mu}^N}_{N\in\bb{N}}$ via the continuous mapping theorem.

Define $\nu^N_\omega \coloneqq \frac{1}{N}\sum_{i=1}^N \delta_{\tilde{X}^{i,N}(\omega)}\delta_{\bar{X}^{i,N}(\omega)-\tilde{X}^{i,N}(\omega)}$ in $\mc{P}(\mc{X}\times\mc{X})$
\begin{align*}
\nu^N_\omega(A\times B) \coloneqq \frac{1}{N}\sum_{i=1}^N \delta_{\tilde{X}^{i,N}(\omega)}(A)\delta_{\bar{X}^{i,N}(\omega)-\tilde{X}^{i,N}(\omega)}(B)
\end{align*}
for $A,B\in\mc{B}(\mc{X})$.
Here
\begin{align*}
\tilde{X}^{i,N}_t \coloneqq \bar{X}^{i,N}_t-C^{i,N}_8(t)=\bar{X}^{i,N}_t-\epsilon \biggl[\Phi(x^{i,N},x^{i,N}/\epsilon,\bar{\mu}^N_0) -  \Phi (\bar{X}^{i,N}_t,\bar{X}^{i,N}_t/\epsilon,\bar{\mu}^N_t)\biggr],\forall t\in [0,1].
\end{align*}
Note by $\nu^N_1$ the first marginal of $\nu^N$ and by $\nu^N_2$ the second marginal of $\nu^N$. As always, to prove tightness of $\br{\nu^N}_{N\in\bb{N}}$ as a sequence of $\mc{P}(\mc{X}\times\mc{X})$-valued random variables, it suffices to prove tightness of $\br{\nu^N_1}_{N\in\bb{N}}$ and $\br{\nu^N_2}_{N\in\bb{N}}$ as $\mc{P}(\mc{X})$-valued random variables (this follows from Prokhorov's Theorem - see p.232-233 of \cite{DE}).

For $\nu^N_1$, we make use of the tightness function $G$ from Equation \eqref{eq:PXtightnessfunction}.
By definition of $\nu^N_1$ and comparing with Equation \eqref{eq:GbarmuN}, we get for all $N\in\bb{N}$:
\begin{align*}
&\E[G(\nu_1^N)]\leq \frac{1}{N}\sum_{i=1}^N |x^{i,N}| + \frac{1}{N}\sum_{i=1}^N\E\biggl[\sup_{\eta\in(0,1]}\eta^{-1/4}\sum_{k=1}^7 \sup_{s,t\in [0,1]: |s-t|\leq \eta}|C^{i,N}_k(t)-C^{i,N}_k(s)|\biggr]
\end{align*}
For the first term, we have
\begin{align*}
\sup_{N\in\bb{N}}\frac{1}{N}\sum_{i=1}^N |x^{i,N}|\leq \sup_{N\in\bb{N}}\biggl(\frac{1}{N}\sum_{i=1}^N |x^{i,N}|^2\biggr)^{1/2}<\infty
\end{align*}
by Jensen's inequality and Assumption \ref{assumption:initialconditions}. So we only need to focus on the second term.

Since $\epsilon\downarrow 0$ as $N\toinf$, we may assume without loss of generality that $0\leq \epsilon\leq 1$.

Then for deterministic constants $C$ independent of $N$ and $\eta$, by Assumption \ref{assumption:LipschitzandBounded} and Proposition \ref{prop:Phiexistenceregularity}, we have by repeated applications of H\"older's and Young's inequalities:
\begin{align*}
\sup_{s,t\in [0,1]: |s-t|\leq \eta}|C^{i,N}_k(t)-C^{i,N}_k(s)|&\leq C\eta,\quad k=1,4,7\\
\sup_{s,t\in [0,1]: |s-t|\leq \eta}|C^{i,N}_2(t)-C^{i,N}_2(s)|&\leq C\eta^{1/2}\biggl(\int_0^1 |u^N_i(\tau)|^2d\tau\biggr)^{1/2}\leq C\eta^{1/2}[1+\int_0^1 |u^N_i(\tau)|^2d\tau]\\
 \sup_{s,t\in [0,1]: |s-t|\leq \eta}|C^{i,N}_5(t)-C^{i,N}_5(s)|&\leq C\eta^{1/2}\biggl( \frac{1}{N}\sum_{j=1}^N\int_0^1 |u^N_j(s)|^2 ds   \biggr)^{1/2}\leq C\eta^{1/2}\biggl[1+\frac{1}{N}\sum_{j=1}^N\int_0^1 |u^N_j(s)|^2 ds \biggr].
\end{align*}
See the proof of the bound \eqref{eq:vanishingmartingaleprobterms} in Section \ref{subsubsection:V1} for more details on a very similar computation. In addition, by Assumption \ref{assumption:LipschitzandBounded} and Proposition \ref{prop:Phiexistenceregularity}, the martingale terms $C^{i,N}_3(t)$ and $C^{i,N}_6(t)$ satisfy the conditions of Lemma C.1 in \cite{FischerFormofRateFunction}.
Combining the above bounds and the result of the aforementioned Lemma, we get
\begin{align*}
\sup_{N\in\bb{N}}\E[G(\nu_1^N)]&\leq C\biggl\lbrace 1+\sup_{N\in\bb{N}}\frac{1}{N}\sum_{i=1}^N\E\biggl[\int_0^1 |u^N_i(\tau)|^2d\tau\biggr]\biggr\rbrace\\
&\leq C[1+C_{con}]
\end{align*}
by the assumed bound \eqref{eq:controlL2boundunspecific}. So $\br{\nu^N_1}_{N\in\bb{N}}$ is tight as a sequence of $\mc{P}(\mc{X})$-valued random variables.

For $\br{\nu^N_2}_{N\in\bb{N}}$, we let $\bb{W}_{1,\mc{X}}$ be the 1-Wasserstein metric on $\mc{P}_1(\mc{X})$, as defined in Definition \ref{def:lionderivative}.

Then by construction:
\begin{align*}
\E[\bb{W}_{1,\mc{X}}(\nu^N_2,\delta_0)]&\leq \E[\int_{\mc{X}} \sup_{t\in [0,1]} |\phi(t)| \nu_2^{N}(d\phi)] \\
&= \frac{1}{N}\sum_{i=1}^N \epsilon \E[\sup_{t\in [0,1]} |\Phi(\bar{X}^{i,N}_0,\bar{X}^{i,N}_0/\epsilon,\bar{\mu}^N_0) -  \Phi (\bar{X}^{i,N}_t,\bar{X}^{i,N}_t/\epsilon,\bar{\mu}^N_t)|]\\
&\leq 2\epsilon \norm{\Phi}_\infty \tto 0 \text{ as }N\toinf,
\end{align*}
so by Markov's inequality, $\nu_2^{N}$ converges in probability to $\delta_0$ as $N\toinf$ as a $\mc{P}_1(\mc{X})$-valued random variable, and hence as a $\mc{P}(\mc{X})$-valued random variable (see, e.g., Lemma 8.2.1. on p. 175 of \cite{Rachev}). Here $\delta_0\in \mc{P}_1(\mc{X})$ is denoting the unit mass on $\phi\in\mc{X}$ which is uniformly $0$ for all $t\in [0,1]$, and we used the boundedness of $\Phi$ from Proposition \ref{prop:Phiexistenceregularity}. Thus $\nu^N_2\tto \delta_0$ in distribution, and by Prokhorov's theorem, $\br{\nu_2^N}_{N\in\bb{N}}$ is tight as a sequence of $\mc{P}(\mc{X})$-valued random variables.

Now we know that $\br{\nu^N}_{N\in\bb{N}}$ is tight as a sequence of $\mc{P}(\mc{X}\times\mc{X})$-valued random variables. We claim there is a continuous mapping $\bar{F}:\mc{P}(\mc{X}\times\mc{X})\tto \mc{P}(\mc{X})$ such that $\bar{F}(\nu^N) = \bar{\mu}^N,\forall N\in \bb{N}$. Once we show this, tightness of $\br{\bar{\mu}^N}_{N\in\bb{N}}$ as a sequence of $\mc{P}(\mc{X})$-valued random variables will follow immediately from Prokhorov's theorem and the continuous mapping theorem (\cite{billingsley} Theorem 2.7).

We define $\bar{F}:\mc{P}(\mc{X}\times\mc{X})\tto \mc{P}(\mc{X})$ as $\bar{F}(\mu) = \mu \circ \tilde{F}^{-1}$, where $\tilde{F}:\mc{X}\times \mc{X}\tto \mc{X}$ is given by$\tilde{F}(\phi,\psi) = \phi+\psi$.

First we show $\bar{F}(\nu^{N}) = \bar{\mu}^N$. To see this, take and $g\in C_b(\mc{X})$. Then
\begin{align*}
\int_{\mc{X}} g(\phi) \bar{F}(\nu^N)(d\phi) &= \int_{\mc{X}\times\mc{X}} g(\tilde{F}(\phi,\psi)) \nu^N(d\phi,d\psi)\\
&= \int_{\mc{X}\times\mc{X}} g(\phi+\psi) \nu^N(d\phi,d\psi)\\
&=\frac{1}{N}\sum_{i=1}^N g(\tilde{X}^{i,N}+\bar{X}^{i,N}-\tilde{X}^{i,N})\\
& = \frac{1}{N}\sum_{i=1}^N g(\bar{X}^{i,N})\\
& = \int_{\mc{X}} g(\phi) \bar{\mu}^N(d\phi).
\end{align*}
Since $C_b(\mc{X})$ is separating (Chapter 3 Section 4 in \cite{EK}), we indeed have $\bar{F}(\nu^{N}) = \bar{\mu}^N$.

Now we show $\bar{F}$ is continuous. Take $\br{\pi^N}$ such that $\pi^N\tto \nu$ in $\mc{P}(\mc{X}\times\mc{X})$. Then for any $g\in C_{b,L}(\mc{X})$:
\begin{align*}
\lim_{N\toinf}\int_{\mc{X}}g(\phi)\bar{F}(\pi^N)(d\phi) & =\lim_{N\toinf}\int_{\mc{X}\times\mc{X}} g(\tilde{F}(\phi,\psi)) \pi^N(d\phi,d\psi)\\
& = \int_{\mc{X}\times\mc{X}} g(\tilde{F}(\phi,\psi)) \nu(d\phi,d\psi)\\
&= \int_{\mc{X}}g(\phi)\bar{F}(\nu)(d\phi),
\end{align*}
where in the second equality we used $g(\tilde{F}(\phi,\psi))\in C_b(\mc{X}\times\mc{X})$, since
\begin{align*}
\norm{\tilde{F}(\phi_1,\psi_1)-\tilde{F}(\phi_2,\psi_2)}_{\mc{X}} &= \norm{\phi_1-\phi_2+\psi_1-\psi_2}_{\mc{X}}\leq \norm{\phi_1-\phi_2}_{\mc{X}}+\norm{\psi_1-\psi_2}_{\mc{X}}= \norm{(\phi_1,\psi_1)-(\phi_2,\psi_2)}_{\mc{X}\times\mc{X}}
\end{align*}
and hence $\tilde{F}\in C_{b,L}(\mc{X}\times\mc{X};\mc{X})$. So by, e.g. Theorem 11.3.3 in \cite{Dudley}, $\bar{F}(\pi^N)\tto \bar{F}(\nu)$ in $\mc{P}(\mc{X})$, and $\bar{F}$ is continuous.

Thus $\br{Q^N_{\mc{X}}}_{N\in\bb{N}} = \br{\bar{\mu}^N}_{N\in\bb{N}}$ is tight as a sequence of $\mc{P}(\mc{X})$-valued random variables, and we are done.

\end{proof}

\subsection{Identification of the Limit}
\label{subsection:identifylimit}
Again, throughout this Section we assume \ref{assumption:initialconditions}-\ref{assumption:limitinguniformellipticity}.

Extract a convergent subsequence from $\br{Q^N}_{N\in\bb{N}}$ and relabel with new indexes so that $\br{Q^N}_{N\in\bb{N}}$ converges to some $Q$ weakly as a $\mc{P}(\mc{C})$-valued random variable. Let $(\tilde\W,\tilde\F,\tilde \Prob)$ be the probability space on which $Q$ lies.

We wish to identify the limit $Q$ as a member of $\mc{V}$ for $\tilde\Prob$ almost every $\tilde\omega\in\tilde\W$. Our main tool here is the associated martingale problem to weak solutions of (\ref{eq:paramMcKeanLimit}). An important element to the proof which is special to the joint limit as $N\toinf,\epsilon\downarrow 0$ is that we must first show that \ref{V:V4} holds before identifying the SDE associated to $Q$ to prove \ref{V:V1}. This is because, as proven in (\ref{eq:D5vanishes}), in the prelimit there are terms which are a priori $O(1)$ in $N$, but that are fact $0$ in the limit due to the centering condition \ref{assumption:centeringcondition}. As the centering condition is a statement involving the invariant measure $\pi$ (see Equation \eqref{eq:pi}), it is necessary to the proof that we have already identified the $\mc{Y}$ component of the limiting Coordinate Process \eqref{eq:canonicalprocess} as being concentrated on elements with $\bb{T}^d$-marginal $\pi$.

\subsubsection{Proof of \ref{V:V4}}
As stated, we offer the proof of \ref{V:V4} first. We begin with a Lemma:
\begin{lem}\label{lemma:willimplyV4}
For almost every $\tilde\omega\in\tilde\W$ and $\forall t\in [0,1],g\in C^2_b(\bb{T}^d)$,
\begin{align*}
\E^{Q_{\tilde{\omega}}}\biggl[\biggl|\int_{\bb{T}^d\times\R^m\times [0,t]} \mc{L}^1_{\bar{X}_s,\nu_{Q_{\tilde{\omega}}(s)}}g(y)\rho(dydzds)\biggr|\biggr]=0,
\end{align*}
where $\mc{L}^1$ is given in Equation \eqref{eq:L1}.
\end{lem}
\begin{proof}
Let $g_l:\bb{T}^d\tto \R,l\in\bb{N}$ be smooth and bounded with bounded derivatives and dense in $C^2_b(\bb{T}^d)$. This set exists by using Stone–Weierstrass and taking rational coefficients. Let $\bar{Y}^{i,N}=\bar{X}^{i,N}/\epsilon$. Considering the operator which acts on $g\in C^2_b(\bb{T}^d)$ by
\begin{align*}
\mc{A}_{x,z,\mu}[g](y) \coloneqq \biggl[\frac{1}{\epsilon^2} f(x,y,\mu)+\frac{1}{\epsilon}(b(x,y,\mu)+\sigma(x,y,\mu)z)\biggr]\cdot \nabla g(y)+ \frac{1}{2\epsilon^2} A(x,y,\mu): \nabla\nabla g(y).
\end{align*}

Note that by \ref{assumption:LipschitzandBounded}, for $t\in [0,1]$ and fixed $N\in\bb{N}$,
\begin{align*}
M^{i,N}_t&\coloneqq g_l(\bar{Y}^{i,N}_t)- g_l(x^{i,N}_0/\epsilon)-\int_0^t \mc{A}_{\bar{X}^{i,N}_s,u_i^N(s),\bar{\mu}^N_s} [g_l](\bar{Y}^{i,N}_s)ds \\
&= \frac{1}{\epsilon} \int_0^t \nabla_y g_l(\bar{Y}^{i,N}_s)\cdot (\sigma(\bar{X}^{i,N}_s,\bar{Y}^{i,N}_s,\bar{\mu}^N_s)dW^i_t)
\end{align*}
is an $\F_t$-martingale. By definition, for $t\in[0,1]$,
\begin{align*}
\int_0^t \mc{A}_{\bar{X}^{i,N}_s,u_i^N(s),\bar{\mu}^N_s} g_l(\bar{Y}^{i,N}_s)ds &= \frac{1}{\epsilon^2} \int_0^t \mc{L}^1_{\bar{X}^{i,N}_s,\bar{\mu}^N_s} g_l(\bar{Y}^{i,N}_s)ds\\
& + \frac{1}{\epsilon}  \int_0^t  \biggl[b(\bar{X}^{i,N}_s,\bar{Y}^{i,N}_s,\bar{\mu}^N_s)+\sigma(\bar{X}^{i,N}_s,\bar{Y}^{i,N}_s,\bar{\mu}^N_s)u_i^N(s)\biggr]\cdot \nabla_y g_l(\bar{Y}^{i,N}_s)ds
\end{align*}

Consider now the operator which acts on $g\in C^2_b(\bb{T}^d)$ by
\begin{align*}
\mc{B}_{x,z,\mu}[g](y)\coloneqq \biggl[b(x,y,\mu)+\sigma(x,y,\mu)z\biggr]\cdot\nabla g(y).
\end{align*}

Then
\begin{align}\label{eq:V4eqn}
&\frac{1}{N}\sum_{i=1}^N\biggl|\epsilon^2 \biggl(-M_t^{i,N}+g_l(\bar{Y}^{i,N}_t)-g_l(x_0^{i,N}/\epsilon) \biggr) - \epsilon \int_0^t  \mc{B}_{\bar{X}^{i,N}_s,u^N_i(s),\bar{\mu}^N_s}[g_l](\bar{Y}^{i,N}_s)ds \biggr|\\
& =\frac{1}{N}\sum_{i=1}^N\biggl|\int_0^t  \mc{L}^1_{\bar{X}^{i,N}_s,\bar{\mu}^N_s} g_l(\bar{Y}^{i,N}_s)ds\biggr|.\nonumber
\end{align}
We will show the right hand side of Equation \eqref{eq:V4eqn} converges in distribution to \newline$\E^Q\biggl[\biggl|\int_{\bb{T}^d\times\R^m\times[0,t]} \mc{L}^1_{\bar{X}_s,\nu_{Q}(s)}g_l(y)\rho(dydzds)\biggr|\biggr]$ and the left hand side converges in distribution to $0$, so by a density argument the result holds.

The proof that the right hand side of Equation \eqref{eq:V4eqn} $\tto \E^Q\biggl[\biggl|\int_{\bb{T}^d\times\R^m\times[0,t]} \mc{L}^1_{\bar{X}_s,\nu_{Q}(s)}g_l(y)\rho(dydzds)\biggr|\biggr]$ in distribution follows from the observation that
\begin{align*}
\frac{1}{N}\sum_{i=1}^N\biggl|\int_0^t  \mc{L}^1_{\bar{X}^{i,N}_s,\bar{\mu}^N_s} g_l(\bar{Y}^{i,N}_s)ds\biggr|& = \int_{\mc{C}} \biggl| \int_{\bb{T}^d\times\R^m\times[0,t]} \mc{L}^1_{\phi(s),\nu_{Q^N}(s)}g_l(y)r(dydzds)\biggr| Q^N(d\phi dr dw).
\end{align*}

We invoke Skorokhod's representation theorem (Theorem 3.1.8 in \cite{EK}) to assume the convergence of $Q^N\tto Q$ holds with probability $1$. Without making a distinction between the original probability space in the new one, we will prove
\begin{align}
\label{eq:RHSconv}
\E\biggl[\biggl|\frac{1}{N}\sum_{i=1}^N\biggl|\int_0^t  \mc{L}^1_{\bar{X}^{i,N}_s,\bar{\mu}^N_s} g_l(\bar{Y}^{i,N}_s)ds\biggr| - \E^Q\biggl[\biggl|\int_{\bb{T}^d\times\R^m\times[0,t]} \mc{L}^1_{\bar{X}_{s},\nu_{Q}(s)}g_l(y)\rho(dydzds)\biggr|\biggr]\biggr|\biggr]\tto 0,
\end{align}
so that by Chebyshev's inequality the convergence holds in probability and hence in distribution.

First we note that the left hand side of \eqref{eq:RHSconv} can be written as
\begin{align*}
&\E\biggl[ \biggl|\int_{\mc{C}} \biggl| \int_{\bb{T}^d\times[0,t]} \mc{L}^1_{\phi(s),\nu_{Q^N}(s)}g_l(y)r(dydzds)\biggr| Q^N(d\phi dr dw)- \int_{\mc{C}} \biggl| \int_{\bb{T}^d\times[0,t]} \mc{L}^1_{\phi(s),\nu_{Q}(s)}g_l(y)r(dydzds)\biggr| Q(d\phi  dr dw) \biggr| \biggr].
\end{align*}
By assumption \ref{assumption:LipschitzandBounded}, the term inside the expectation is bounded, so
\begin{align*}
&\lim_{N\toinf}\E\biggl[ \biggl|\int_{\mc{C}} \biggl| \int_{\bb{T}^d\times\R^m\times[0,t]} \mc{L}^1_{\phi(s),\nu_{Q^N}(s)}g_l(y)r(dydzds)\biggr| Q^N(d\phi dr dw)\\
& \hspace{2cm}- \int_{\mc{C}} \biggl| \int_{\bb{T}^d\times\R^m\times[0,t]} \mc{L}^1_{\phi(s),\nu_{Q}(s)}g_l(y)r(dydzds)\biggr| Q(d\phi dr dw) \biggr| \biggr]\\
& = \E\biggl[\lim_{N\toinf} \biggl|\int_{\mc{C}} \biggl| \int_{\bb{T}^d\times\R^m\times[0,t]} \mc{L}^1_{\phi(s),\nu_{Q^N}(s)}g_l(y)r(dydzds)\biggr| Q^N(d\phi dr dw)\\
& \hspace{2cm}- \int_{\mc{C}} \biggl| \int_{\bb{T}^d\times\R^m\times[0,t]} \mc{L}^1_{\phi(s),\nu_{Q}(s)}g_l(y)r(dydzds)\biggr| Q(d\phi dr dw) \biggr| \biggr]\\
\end{align*}
Now, observing that
\begin{align*}
(\phi,r,\Theta)\mapsto \biggl| \int_{\bb{T}^d\times\R^m\times[0,t]} \mc{L}^1_{\phi(s),\nu_{\Theta}(s)}g_l(y)r(dydzds)\biggr|
\end{align*}
is bounded and continuous (using Proposition \ref{prop:nucontmeasure} and Assumption \ref{assumption:LipschitzandBounded}),  the results follows via Theorem A.3.18 in \cite{DE} and almost sure convergence of $Q^N\tto Q$.

To prove the left hand side of (\ref{eq:V4eqn}) converges to zero in distribution, we will show that
\begin{align*}
\E\biggl[\frac{1}{N}\sum_{i=1}^N\biggl|\epsilon^2 \biggl(-M_t^{i,N}+g_l(\bar{Y}^{i,N}_t)-g_l(x_0^{i,N}/\epsilon) \biggr) - \epsilon \int_0^t  \mc{B}_{\bar{X}^{i,N}_s,u^N_i(s),\bar{\mu}^N_s}[g_l](\bar{Y}^{i,N}_s)ds \biggr|\biggr]\tto 0
\end{align*}
as $N\toinf$, so the result will follow by Chebyshev's inequality.

We first note that
\begin{align*}
\epsilon^2\frac{1}{N}\sum_{i=1}^N\E\biggl[ \biggl| M_t^{i,N}\biggr|\biggr]&\leq \epsilon^2\frac{1}{N}\sum_{i=1}^N\E\biggl[ \biggl( M_t^{i,N}\biggr)^2\biggr]^{1/2}\leq \epsilon \norm{\nabla g_l}_\infty C \text{ by Assumption \ref{assumption:LipschitzandBounded} and It\^o Isometry.}
\end{align*}

Also,
\begin{align*}
\epsilon^2\frac{1}{N}\sum_{i=1}^N\E\biggl[ \biggl| g_l(\bar{Y}^{i,N}_t)-g_l(x^{i,N}/\epsilon)\biggr|\biggr]\leq 2\epsilon^2\norm{g_l}_\infty.
\end{align*}

Lastly,
\begin{align*}
&\epsilon\frac{1}{N}\sum_{i=1}^N\E\biggl[\biggl|\int_0^t  \mc{B}_{\bar{X}^{i,N}_s,u^N_i(s),\bar{\mu}^N_s}[g_l](\bar{Y}^{i,N}_s)ds \biggr|\biggr]\leq \epsilon \norm{\nabla g_l}_\infty C(C_{con})\\
& \text{ by H\"older's inequality, Assumption (\ref{eq:controlL2boundunspecific}), and Assumption \ref{assumption:LipschitzandBounded} .}
\end{align*}

Now we have that for each $g_l$ and $t\in[0,1]$, there exists a set $N_{g_l,t}$ such that $\tilde{P}(N_{g_l,t})=0$ and
\begin{align*}
\E^{Q_{\tilde{\omega}}}\biggl[\biggl|\int_{\bb{T}^d\times\R^m\times [0,t]} \mc{L}^1_{\bar{X}_{s},\nu_{Q_{\tilde{\omega}}(s)}}g_l(y)\rho(dydzds)\biggr|\biggr]=0, \forall \tilde{\omega}\in\tilde\W\setminus N_{g_l,t}.
\end{align*}
Taking a countable dense set $D\subset [0,1]$ and letting $N = \cup_{l\in\bb{N}}\cup_{t\in D}N_{g_l,t}$, we have $\tilde\Prob(N)=0$ and
\begin{align*}
\E^{Q_{\tilde{\omega}}}\biggl[\biggl|\int_{\bb{T}^d\times\R^m\times [0,t]} \mc{L}^1_{\bar{X}_{s},\nu_{Q_{\tilde{\omega}}(s)}}g(y)\rho(dydzds)\biggr|\biggr]=0, \forall \tilde{\omega}\in\tilde\W\setminus N,\forall g\in C^2_b(\bb{T}^d).
\end{align*}
\end{proof}
Now we can prove that the limit point $Q$ satisfies \ref{V:V4}.

\begin{proposition}\label{remark:V4explaination}
$Q$ satisfies \ref{V:V4} $\tilde\Prob$ almost-surely.
\end{proposition}
\begin{proof}
Note that we can write $\Theta|_{\mc{B}(\mc{X}\times\mc{Y})}(d\phi dr) = \lambda(dr|\phi)\Theta_{\mc{X}}(d\phi)$. Then the result of Lemma \ref{lemma:willimplyV4} can be written as:
\begin{align*}
&\forall t\in [0,1],g\in C^2_b(\bb{T}^d),\int_{\mc{X}}\int_{\mc{Y}}\biggl|\int_{\bb{T}^d\times\R^d\times [0,t]} \mc{L}^1_{\phi(s),\nu_\Theta(s)}g(y)r(dydzds)\biggr| \lambda(dr|\phi)\Theta_{\mc{X}}(d\phi)=0,\tilde\Prob-\text{a.s.}\\
\end{align*}

So $\tilde\Prob-\text{a.s.}$:
\begin{align}\label{eq:V4implication}
&\int_{\mc{Y}}\biggl|\int_{\bb{T}^d\times\R^m\times[0,t]} \mc{L}^1_{\phi(s),\nu_\Theta(s)} g(y)r(dydzds)\biggr| \lambda(dr|\phi) =0, \forall t\in [0,1],\forall g\in C^2_b(\bb{T}^d),\Theta_{\mc{X}}-\text{a.e. }\phi\in\mc{X}\nonumber\\
&\Rightarrow \text{for }\Theta_{\mc{X}}\text{-a.e.}\phi\in\mc{X},\hspace{.2cm}\lambda(\br{r\in\mc{Y}:r(dydzds)=\gamma_s(dz;y)\eta_s(dy)ds\\
&\hspace{6cm}\text{ such that}(\mc{L}^1_{\phi(s),\nu_\Theta(s)})^*\eta_s=0,\forall s\in[0,1]}|\phi)=1.\nonumber
\end{align}

But by Proposition \ref{prop:invtmeasure}, the invariant measure associated with $\mc{L}^1_{\phi(s),\nu_{\Theta}(s)}$ is unique for each $\phi,\Theta,$ and $s$, so $\lambda(\cdot|\phi)$ is concentrated on measures of the form $\gamma_s(dz;y)\pi(dy|\phi(s),\nu_{\Theta}(s))ds$ for $\Theta_{\mc{X}}$-a.e. $\phi$. Thus we have $\tilde\Prob-\text{a.s.}$:
\begin{align*}
&\Theta\biggl(\biggl\lbrace(\phi,r,w)\in\mc{C}: \exists \gamma\in \mc{M}([0,1]\times\bb{T}^d;\mc{P}(\R^m)) \text{ such that }r_s(dydz) = \gamma_s(dz;y)\pi(dy|\phi(s),\nu_{\Theta}(s)),\forall s\in[0,1]\biggr\rbrace\biggr)\\
&= \Theta|_{\mc{B}(\mc{X}\times\mc{Y})}\biggl(\biggl\lbrace(\phi,r)\in\mc{X}\times\mc{Y}: \exists \gamma\in \mc{M}([0,1]\times\bb{T}^d;\mc{P}(\R^m))&\\
&\hspace{6cm} \text{ such that }r_s(dydz) = \gamma_s(dz;y)\pi(dy|\phi(s),\nu_{\Theta}(s)),\forall s\in[0,1]\biggr\rbrace\biggr)\\
&\text{ since the set we are measuring doesn't depend on $w$}\\
& = \int_{\mc{X}} \lambda\biggl(\biggl\lbrace r\in\mc{Y}: \exists \gamma\in \mc{M}([0,1]\times\bb{T}^d;\mc{P}(\R^m))\\
&\hspace{6cm} \text{ such that }r_s(dydz) = \gamma_s(dz;y)\pi(dy|\phi(s),\nu_{\Theta}(s)),\forall s\in[0,1]\biggr\rbrace|\phi\biggr)\Theta_{\mc{X}}(d\phi)\\
& = \int_{A} \lambda\biggl(\biggl\lbrace r\in\mc{Y}: \exists \gamma\in \mc{M}([0,1]\times\bb{T}^d;\mc{P}(\R^m)) \\
&\hspace{6cm}\text{ such that }r_s(dydz) = \gamma_s(dz;y)\pi(dy|\phi(s),\nu_{\Theta}(s)),\forall s\in[0,1]\biggr\rbrace|\phi\biggr)\Theta_{\mc{X}}(d\phi)\\
& \text{ for }A\in\mc{B}(\mc{X}) \text{ such that }\eqref{eq:V4implication} \text{ holds }\forall \phi\in A,\Theta_{\mc{X}}(A)=1\\
& =1.
\end{align*}
\end{proof}

\subsubsection{Proof of \ref{V:V1}}\label{subsubsection:V1}
We wish to prove that $Q_{\tilde\omega}$ corresponds to $\bar{X}$, a weak solution of Equation \eqref{eq:controlledMcKeanLimit}, for $\tilde{\Prob}$-a.e. $\tilde{\omega}\in\tilde{\W}$.

Given $g\in C^2_b(\R^d\times\R^d)$ and $\Theta\in\mc{P}(\mc{C})$, define a real-valued process $\br{M^\Theta_g(t)}_{t\in[0,1]}$ on $(\mc{C},\mc{B}(\mc{C}),\Theta)$ given by
\begin{align}
\label{eq:M}
    &M^\Theta_g(t,(\phi,r,w)) = g(\phi(t),w(t))-g(\phi(0),w(0))-\int_0^t\int_{\bb{T}^d\times\R^m} \mc{A}[g](\phi(s),y,z,\nu_\Theta(s),w(s))r_s(dydz)ds\\
    &-\int_0^t \sqrt{\int_{\bb{T}^d\times\R^m}\tilde{D}(\phi(s),y,\nu_{\Theta}(s))  r_s(dydz)}:\nabla_p\nabla_x g(\phi(s),w(s)) ds \nonumber
\end{align}
where
\begin{align} \label{eq:decomposeA}
    \mc{A}[g](x,y,z,\nu_\Theta(s),p) &\coloneqq\biggl[\beta(x,y,\nu_\Theta(s))+ [\nabla_y\Phi(x,y,\nu_\Theta(s))+I]\sigma(x,y,\nu_\Theta(s))z\biggr]\cdot \nabla_x g(x,p)\\
    &\quad+ \frac{1}{2}D(x,y,\nu_\Theta(s)):\nabla_x\nabla_x g(x,p)+\frac{1}{2}I:\nabla_p\nabla_pg(x,p)\nonumber,
\end{align}
and $\beta$, $D$ are as in Equation \eqref{eq:limitingcoefficients} and $\tilde{D}$ as in Equation \eqref{eq:tildeD}.
Here by the square root of a symmetric positive semi-definite matrix $M\in \R^{d\times d}$, we mean the unique positive definite $\sqrt{M}\in \R^{d\times d}$ is such that $\sqrt{M}\sqrt{M}=\sqrt{M}\sqrt{M}^{\top} = M$.%

We will say $\Theta\in \mc{P}(\mc{C})$ solves the martingale problem associated to $\tilde{X}^{\nu_\Theta}$ with initial distribution $\nu_0$ if for all $0\leq s\leq t\leq 1$ and $g\in C^2_b(\R^d\times\R^d)$,
\begin{align}
    \label{eq:martingaleproblem}
    \E^\Theta\biggl[M^\Theta_g(t)|\mc{G}_s \biggr]=M^\Theta_g(s),\hspace{2cm} \nu_\Theta(0)=\nu_0
\end{align}
for $\br{\mc{G}_t}_{t\in[0,1]}$ the canonical filtration on the coordinate process as defined in Equation (\ref{eq:canonicalprocess}). To identify the limit $Q$ as a weak solution $\tilde{X}^{\nu_Q}$ to \eqref{eq:paramMcKeanLimit}, by the density argument offered at the end of this subsection, it suffices to show that for fixed $g\in C^\infty_c(\R^d\times\R^d)$, $0\leq s\leq t\leq 1$, and $\mc{G}_s$- measurable $\Psi\in C_b(\mc{C})$ that
\begin{thm}
\label{thm:V1}
\begin{align*}
    \E^{Q_{\tilde\omega}}\biggl[\Psi(M_g^{Q_{\tilde\omega}}(t)-M_g^{Q_{\tilde\omega}}(s))\biggr]=0
\end{align*}
for almost every $\tilde{\omega}\in \tilde{\W}$.
\end{thm}

We note that it is enough to prove Theorem \ref{thm:V1} for $g\in C^\infty_c(\R^d\times\R^d)$, due to the fact that $C^\infty_c(\R^d\times\R^d)$ is separating in the sense of Chapter 3 Section 4 in \cite{EK} (see Chapter 4 Section 8 in \cite{EK}). Moreover, we are using implicitly here that the $\bb{T}^d$-marginal of the $\mc{Y}$-coordinate process under the limit point $Q$ is almost-surely $\pi(dy;\phi(s),\nu_Q(s))$ via Proposition \ref{remark:V4explaination}, so that via Remark \ref{remark:altdiffusionrep} the diffusion coefficient in Equation \ref{eq:controlledMcKeanLimit} can be written as either the $\tilde{D}$ integrated against $\rho$ or $D$ integrated against $\rho$, and indeed proving Theorem \ref{thm:V1} shows $Q$ almost surely satisfies \ref{V:V1}.

In order to prove Theorem \ref{thm:V1} we will prove Lemma \ref{lem:marprob1} and Lemma \ref{lem:marprob2}.
\begin{lem}
\label{lem:marprob1}
\begin{align}\label{eq:marprob1}
    \E^{Q^N}\biggl[\Psi(M_g^{Q^N}(t)-M_g^{Q^N}(s))\biggr]\tto \E^{Q}\biggl[\Psi(M_g^{Q}(t)-M_g^{Q}(s))\biggr] \text{ in distribution},
\end{align}
\end{lem}

\begin{lem}
\label{lem:marprob2}
\begin{align}
    \E^{Q^N}\biggl[\Psi(M_g^{Q^N}(t)-M_g^{Q^N}(s))\biggr]\tto 0 \text{ in distribution}.
\end{align}
\end{lem}
Then the conclusion follows.

We proceed with the proof of Lemma \ref{lem:marprob1}.
\begin{proof}[Proof of  Lemma \ref{lem:marprob1}]
Unpacking the notation in Equation \eqref{eq:marprob1}, we see what we are trying to show is that
\begin{align*}
&\int_{\mc{C}} \Psi(\phi,r,w)\biggl[ g(\phi(t),w(t)) -g(\phi(s),w(s))  - \int_s^t \int_{\bb{T}^d\times\R^m}  \mc{A}[g](\phi(\tau),y,z,\nu_{Q^N}(\tau),w(\tau))r_\tau (dydz)d\tau \\
&\hspace{4cm}-\int_s^t \sqrt{\int_{\bb{T}^d\times\R^m}\tilde{D}(\phi(\tau),y,\nu_{Q^N}(\tau))  r_\tau(dydz)}:\nabla_p\nabla_x g(\phi(\tau),w(\tau)) d\tau\biggr] Q^N(d\phi dr dw)\\
&\tto \int_{\mc{C}} \Psi(\phi,r,w)\biggl[g(\phi(t),w(t)) -g(\phi(s),w(s))  - \int_s^t \int_{\bb{T}^d\times\R^m} \mc{A}[g](\phi(\tau),y,z,\nu_{Q}(\tau),w(\tau))r_\tau (dydz)d\tau\biggr]\\
&\hspace{4cm}-\int_s^t \sqrt{\int_{\bb{T}^d\times\R^m}\tilde{D}(\phi(\tau),y,\nu_{Q}(\tau))  r_\tau(dydz)}:\nabla_p\nabla_x g(\phi(\tau),w(\tau)) d\tau\biggr] Q(d\phi dr dw)
\end{align*}
in distribution. We invoke Skorokhod's representation theorem to assume the convergence of $Q^N\tto Q$ occurs with probability 1, without making a distinction in the notation between the new probability space and the original one. Then we see Lemma \ref{lem:marprob1} essentially follows from the definition of convergence of measures in the space $\mc{P}(\mc{C})$. The only caveat is that the integrand is not a priori in $C_b(\mc{C})$, since it grows linearly in the control.

We will show
\begin{align}
\label{eq:conv2}
\E\biggl[\bigg|\E^{Q^N}\biggl[\Psi(M_g^{Q^{N}}(t)-M_g^{Q^{N}}(s))\biggr]-\E^{Q}\biggl[\Psi(M_g^{Q}(t)-M_g^{Q}(s))\biggr]\biggr|\biggr] \tto 0.
\end{align}
Once this limit is established, by Chebyshev's inequality Lemma \ref{lem:marprob1} will be proved.
\begin{align*}
&\E\biggl[\bigg|\E^{Q^N}\biggl[\Psi(M_g^{Q}(t)-M_g^{Q}(s))\biggr]-\E^{Q}\biggl[\Psi(M_g^{Q}(t)-M_g^{Q}(s))\biggr]\biggr|\biggr]\\
&=\E\biggl[\bigg|\int_{\mc{C}} \Psi(\phi,r,w)\biggl[ g(\phi(t),w(t)) -g(\phi(s),w(s))  - \int_s^t \int_{\bb{T}^d\times\R^m}  \mc{A}[g](\phi(\tau),y,z,\nu_{Q^N}(\tau),w(\tau))r_\tau (dydz)d\tau \\
&\hspace{4cm}-\int_s^t \sqrt{\int_{\bb{T}^d\times\R^m}\tilde{D}(\phi(\tau),y,\nu_{Q^N}(\tau))  r_\tau(dydz)}:\nabla_p\nabla_x g(\phi(\tau),w(\tau)) d\tau\biggr] Q^N(d\phi dr dw)\\
&- \int_{\mc{C}} \Psi(\phi,r,w)\biggl[g(\phi(t),w(t)) -g(\phi(s),w(s))  - \int_s^t \int_{\bb{T}^d\times\R^m} \mc{A}[g](\phi(\tau),y,z,\nu_{Q}(\tau),w(\tau))r_\tau (dydz)d\tau\biggr]\\
&\hspace{4cm}-\int_s^t \sqrt{\int_{\bb{T}^d\times\R^m}\tilde{D}(\phi(\tau),y,\nu_{Q}(\tau))  r_\tau(dydz)}:\nabla_p\nabla_x g(\phi(\tau),w(\tau)) d\tau\biggr] Q(d\phi dr dw)\biggr|\biggr].
\end{align*}

Noting that $\Psi(\phi,r,w)[g(\phi(t),w(t))-g(\phi(s),w(s))]$ is bounded and writing $\mc{A}[g]$ as in Equation \eqref{eq:decomposeA}, we see that only $\int_{\bb{T}^d\times\R^m} [\nabla_y\Phi(x,y,\nu_\Theta(s))+I]\sigma(x,y,\nu_\Theta(s))zr_s(dydz)$ exhibits growth in the control $r$. Since the desired convergence occurs immediately by boundedness of the integrand, almost-sure convergence of $Q^N\tto Q$, Proposition \ref{prop:nucontmeasure}, and Theorem A.3.18 on \cite{DE}, and boundedness and continuity of the matrix square root, we only show work to show the convergence of this term. Let
\begin{align*}
B^1(M)&\coloneqq \br{r\in\mc{Y}:\int_0^1\int_{\R^m} |z|r(dzdt)\leq M}\\
B^2(M)&\coloneqq \br{r\in\mc{Y}:\int_0^1\int_{\bb{T}^d\times\R^m} |z|r_t(dydz)dt> M},\\
\psi_M^1(r,x,\mu,t)&\coloneqq\1_{B^1(M)}(r)\int_{\bb{T}^d\times \R^m} [I+\nabla_y\Phi(x,y,\mu)]\sigma(x,y,\mu)z r_t(dydz)\\
&+\1_{B^2(M)}(r)\frac{M}{\int_0^1\int_{\bb{T}^d\times\R^m} |z|r_t(dydz)dt}\int_{\bb{T}^d\times \R^m} [I+\nabla_y\Phi(x,y,\mu)]\sigma(x,y,\mu)z r_t(dydz)\\
\psi_M^2(r,x,\mu,t)&\coloneqq \1_{B^2(M)}(r)\int_{\bb{T}^d\times \R^m} [I+\nabla_y\Phi(x,y,\mu)]\sigma(x,y,\mu)z r_t(dydz)  \biggl(1-\frac{M}{\int_0^1\int_{\bb{T}^d\times\R^m} |z|r_t(dydz)dt}\biggr)
\end{align*}
Then for all $r\in\mc{Y},M>0,t\in[0,1]$,  $\psi^1_M(r,x,\mu,t)+\psi^2_M(r,x,\mu,t)=\int_{\bb{T}^d\times \R^m} [I+\nabla_y\Phi(x,y,\mu)]\sigma(x,y,\mu)z r_t(dydz),$ $$\sup_{x\in\R^d,\mu\in\mc{P}(\R^d)}\int_0^1 |\psi^1_M(r,x,\mu,t)|dt\leq \norm{[I+\nabla_y\Phi]\sigma}_\infty M,$$ and $$\sup_{x\in\R^d,\mu\in\mc{P}(\R^d)}\int_0^1 |\psi^2_M(r,x,\mu,t)|dt< \1_{B^2(M)}(r)\norm{[I+\nabla_y\Phi]\sigma}_\infty\int_0^1\int_{\bb{T}^d\times\R^m} |z|r_t(dydz)dt.$$

So
\begin{align*}
&\E\biggl[\bigg|\int_{\mc{C}}\Psi(\phi,r,w)\int_s^t \biggl(\int_{\bb{T}^d\times\R^m} [\nabla_y\Phi+I]\sigma(\phi(\tau),y,\nu_{Q^{N}}(\tau))zr_{\tau}(dydz)\biggr)\cdot \nabla g (\phi(\tau),w(\tau))d\tau Q^N(d\phi dr dw)\biggr]\\
&\hspace{1cm}-\int_{\mc{C}}\Psi(\phi,r,w)\int_s^t \biggl(\int_{\bb{T}^d\times\R^m}[\nabla_y\Phi+I]\sigma(\phi(\tau),y,\nu_{Q}(\tau))zr_{\tau}(dydz)\biggr)\cdot \nabla g (\phi(\tau),w(\tau))d\tau Q(d\phi dr dw)\biggr|\biggr]\\
& = \E\biggl[\bigg|\int_{\mc{C}}\Psi(\phi,r,w)\int_s^t \psi^1_M(r,\phi(\tau),\nu_{Q^N}(\tau),\tau)\cdot \nabla g (\phi(\tau),w(\tau))d\tau Q^N(d\phi drd\phi)\\
&+\int_{\mc{C}}\Psi(\phi,r,w)\int_s^t \psi^2_M(r,\phi(\tau),\nu_{Q^N}(\tau),\tau)\cdot \nabla g (\phi(\tau),w(\tau))d\tau Q^N(d\phi dr dw)\\
&-\int_{\mc{C}}\Psi(\phi,r,w)\int_s^t \psi^1_M(r,\phi(\tau),\nu_{Q}(\tau),\tau)\cdot \nabla g (\phi(\tau),w(\tau))d\tau Q(d\phi dr dw)\\
&-\int_{\mc{C}}\Psi(\phi,r,w)\int_s^t \psi^2_M(r,\phi(\tau),\nu_{Q}(\tau),\tau)\cdot \nabla g (\phi(\tau),w(\tau))d\tau Q(d\phi dr dw)\biggr|\biggr]\\
&\leq \E\biggl[\bigg|\int_{\mc{C}}\Psi(\phi,r,w)\int_s^t \psi^1_M(r,\phi(\tau),\nu_{Q^N}(\tau),\tau)\cdot \nabla g (\phi(\tau),w(\tau))d\tau Q^N(d\phi dr dw)\\
&-\int_{\mc{C}}\Psi(\phi,r,w)\int_s^t \psi^1_M(r,\phi(\tau),\nu_{Q}(\tau),\tau)\cdot \nabla g (\phi(\tau),w(\tau))d\tau Q(d\phi dr dw)\biggr|\biggr]\\
& + 2\norm{\nabla g}_\infty\norm{\Psi}_\infty \sup_{N\in \bb{N}}\sup_{\mu\in\mc{P}(\R^d)}\E\biggl[\int_{\mc{C}}\int_s^t\biggl|\psi^2_M(r,\phi(\tau),\mu,\tau) \biggr|d\tau Q^N(d\phi dr dw)\biggr]\text{ by Theorem A.3.12 in \cite{DE}}.
\end{align*}
By Bounded Convergence Theorem and almost-sure convergence of $Q^N\tto Q$, the first term vanishes as $N\toinf$ in the same manner as discussed for the other bounded terms (for continuity of the time integral of $\psi^1_M$ see Lemma 5.3.4/3.3.1 in \cite{DE}). To handle the second term, we have:
\begin{align*}
&\sup_{N\in \bb{N}}\sup_{\mu\in\mc{P}(\R^d)}\E\biggl[\int_{\mc{C}}\int_s^t\biggl|\psi^2_M(r,\phi(\tau),\mu,\tau) \biggr|d\tau Q^N(d\phi dr dw)\biggr]\\
&\leq C \sup_{N\in \bb{N}}\E\biggl[\int_{\mc{C}}\1_{B^2(M)}(r)\int_0^1\int_{\bb{T}^d\times\R^m} |z|r_t(dydz)dt Q^N(d\phi dr dw)\biggr]\\
&  = C \sup_{N\in \bb{N}}\E\biggl[\int_{\mc{C}}\1_{B^2(M)}(r) \frac{\biggl(\int_0^1\int_{\bb{T}^d\times\R^m} |z|r_t(dydz)dt\biggr)^2}{\int_0^1\int_{\bb{T}^d\times\R^m} |z|r_t(dydz)dt}Q^N(d\phi dr dw)\biggr]\\
&\leq \frac{C}{M} \sup_{N\in \bb{N}}\E\biggl[\int_{\mc{C}}\biggl(\int_0^1\int_{\bb{T}^d\times\R^m} |z|r_t(dydz)dt\biggr)^2 Q^N(d\phi dr dw)\biggr] \text{ by the definition of }B^2(M)\\
&\leq \frac{C}{M} \sup_{N\in \bb{N}}\E\biggl[\int_{\mc{C}}\int_0^1\int_{\bb{T}^d\times\R^m} |z|^2r_t(dydz)dtQ^N(d\phi dr dw)\biggr] \text{ by Jensen's inequality}\\
& = \frac{C}{M} \sup_{N\in \bb{N}}\E\biggl[\frac{1}{N}\sum_{i=1}^N\int_0^1 |u_i^N(\tau)|^2 d\tau\biggr] \text{ by definition of }Q^N\\
&\leq \frac{CC_{con}}{M} \text{ by Assumption (\ref{eq:controlL2boundunspecific})}.
\end{align*}
Taking $N\toinf$ then $M\toinf$ the result follows.

\end{proof}
Now we prove Lemma \ref{lem:marprob2}.
\begin{proof}[Proof of Lemma \ref{lem:marprob2}]
Again, we invoke Skorokhod's representation theorem to assume the convergence of $Q^N\tto Q$ occurs with probability 1.

We will show $\E\biggl[\biggl|\E^{Q^N}\biggl[\Psi(M_g^{Q^N}(t)-M_g^{Q^N}(s))\biggr]\biggr|\biggr]\tto 0 \text{ as } N\toinf$, and so the conclusion will follow via Chebyshev's inequality.

Applying It\^o's formula to $g(\bar{X}^{i,N}_t,\bar{W}^{i,N}_t)$, we get (suppressing the arguments for notational convenience)
\begin{align*}
g(\bar{X}^{i,N}_t,\bar{W}^{i,N}_t)& = g(x^{i,N},0)+\int_0^t [\frac{1}{\epsilon} f+b +\sigma u_i^N ]\cdot \nabla_x g +\frac{1}{2} A : \nabla_x\nabla_x g +\biggl[A [I+\nabla_y\Phi]^\top(\bar{B}^\top)^{-1}\biggr]:\nabla_p\nabla_x g\\
&\hspace{4cm}+\frac{1}{2}\biggl[\bar{B}^{-1}[I+\nabla_y\Phi]A[I+\nabla_y\Phi]^\top (\bar{B}^\top)^{-1}\biggr]:\nabla_p\nabla_p gds\\
&+\int_0^t \nabla_x g \cdot (\sigma dW^i_s)+\int_0^t \nabla_p g\cdot (\bar{B}^{-1}[I+\nabla_y\Phi]\sigma dW^i_s).
\end{align*}

In order to control the term that blows up as $\epsilon\to 0$, we define $\psi_l(x,p,y,\mu)\coloneqq \Phi_l(x,y,\mu) g_{x_l}(x,p),l=1,...,d$, for $\Phi$ as in Equation \eqref{eq:cellproblem}. Then $\psi_l$ solves
\begin{align*}
\mc{L}^1_{x,\mu} \psi_l(x,p,y,\mu) &= -f_l(x,y,\mu)g_{x_l}(x,p).
\end{align*}

Now applying It\^o's formula to $\psi_l$  (using Equations \eqref{eq:empfirstder} and \eqref{eq:empsecondder} in Proposition \ref{prop:empprojderivatives}, the regularity of $\Phi$ from Proposition \ref{prop:Phiexistenceregularity}, and Proposition \ref{prop:uniformXL2bound}), we get%

\begin{align*}
&\psi_l(\bar{X}^{i,N}_t,\bar{W}^{i,N}_t,\bar{X}^{i,N}_t/\epsilon,\bar{\mu}^N_t)\\
& = \psi_l(x^{i,N},0,x^{i,N}/\epsilon,\bar{\mu}^N_0)+\int_0^t [\frac{1}{\epsilon}f+b+\sigma u_i^N]\cdot \nabla_x \psi_l+\frac{1}{2} A:\nabla_x\nabla_x \psi_l +\frac{1}{\epsilon} [b+\sigma u_i^N]\cdot \nabla_y \psi_l\\
&+ \frac{1}{\epsilon} A : \nabla_x\nabla_y \psi_l +\biggl[A[I+\nabla_y\Phi]^\top (\bar{B}^\top)^{-1}\biggr]:\nabla_p\nabla_x\psi_l+\frac{1}{\epsilon}A[I+\nabla_y\Phi]^\top (\bar{B}^\top)^{-1}:\nabla_p\nabla_y\psi_l\\
&\hspace{6cm}+\frac{1}{2}\biggl[\bar{B}^{-1}[I+\nabla_y\Phi]A[I+\nabla_y\Phi]^\top (\bar{B}^\top)^{-1}\biggr]:\nabla_p\nabla_p \psi_lds\\
&+\int_0^t \nabla_x \psi_l\cdot(\sigma dW^i_s) +\frac{1}{\epsilon} \int_0^t \nabla_y \psi_l \cdot (\sigma dW^i_s)+\int_0^t \nabla_p\psi_l \cdot (\bar{B}^{-1}[I+\nabla_y\Phi]\sigma dW^i_s)  -\frac{1}{\epsilon^2}\int_0^t f_l g_{x_l}ds\\
&+ \frac{1}{N}\sum_{j=1}^N \biggl[\int_0^t \partial_\mu \psi_l(\bar{X}^{i,N}_s,\bar{W}^{i,N}_s,\bar{X}^{i,N}_s/\epsilon,\bar{\mu}^N_s)(\bar{X}^{j,N}_s)\\
&\hspace{2cm}\cdot \biggl(\frac{1}{\epsilon}f(\bar{X}^{j,N}_s,\bar{X}^{j,N}_s/\epsilon,\bar{\mu}^N_s) +b(\bar{X}^{j,N}_s,\bar{X}^{j,N}_s/\epsilon,\bar{\mu}^N_s)+\sigma(\bar{X}^{j,N}_s,\bar{X}^{j,N}_s/\epsilon,\bar{\mu}^N_s) u_j^N(s)\biggr)\\
& \hspace{5cm}+\frac{1}{2}A(\bar{X}^{j,N}_s,\bar{X}^{j,N}_s/\epsilon,\bar{\mu}^N_s):\partial_v\partial_\mu \psi_l(\bar{X}^{i,N}_s,\bar{W}^{i,N}_s,\bar{X}^{i,N}_s/\epsilon,\bar{\mu}^N_s)(\bar{X}^{j,N}_s)  \\
& \hspace{5cm}+\frac{1}{2N}A(\bar{X}^{j,N}_s,\bar{X}^{j,N}_s/\epsilon,\bar{\mu}^N_s):\partial^2_\mu \psi_l(\bar{X}^{i,N}_s,\bar{W}^{i,N}_s,\bar{X}^{i,N}_s/\epsilon,\bar{\mu}^N_s)(\bar{X}^{j,N}_s,\bar{X}^{j,N}_s) ds \\
&\hspace{5cm}+ \int_0^t \partial_\mu \psi_l(\bar{X}^{i,N}_s,\bar{W}^{i,N}_s,\bar{X}^{i,N}_s/\epsilon,\bar{\mu}^N_s)(\bar{X}^{j,N}_s)\cdot ( \sigma(\bar{X}^{j,N}_s,\bar{X}^{j,N}_s/\epsilon,\bar{\mu}^N_s) dW_s^j)\biggr]\\
&+\frac{1}{N}\int_0^t A:\nabla_{x} \partial_\mu \psi_l(\bar{X}^{i,N}_s,\bar{W}^{i,N}_s,\bar{X}^{i,N}_s/\epsilon,\bar{\mu}^N_s)(\bar{X}^{i,N}_s) ds + \frac{1}{N \epsilon} \int_0^t A:\nabla_{y} \partial_\mu \psi_l(\bar{X}^{i,N}_s,\bar{W}^{i,N}_s,\bar{X}^{i,N}_s/\epsilon,\bar{\mu}^N_s)(\bar{X}^{i,N}_s)  ds\\
&+\frac{1}{N}\int_0^t \biggl[A[I+\nabla_y\Phi]^\top (\bar{B}^\top)^{-1}\biggr]:\nabla_{p} \partial_\mu \psi_l(\bar{X}^{i,N}_s,\bar{W}^{i,N}_s,\bar{X}^{i,N}_s/\epsilon,\bar{\mu}^N_s)(\bar{X}^{i,N}_s) ds,
\end{align*}
where in all coefficients where the argument is suppressed, the argument is $(\bar{X}^{i,N}_s,\bar{X}^{i,N}_s/\epsilon,\bar{\mu}^N_s)$, and the argument of $\psi_l$ where suppressed is $(\bar{X}^{i,N}_s,\bar{W}^{i,N}_s,\bar{X}^{i,N}_s/\epsilon,\bar{\mu}^N_s)$. Solving for $\frac{1}{\epsilon}\int_0^t f_l g_{x_l}ds$ and plugging into our representation for $h$, we get
\begin{align}
\label{eq:hxin}
g(\bar{X}_t^{i,N},\bar{W}^{i,N}_t)& = g(x^{i,N},0)+\sum_{k=1}^8 B^{i,N}_k(t)
\end{align}
where
\begin{align*}
B^{i,N}_1(t)& = \int_0^t b\cdot \nabla_{x} g +\frac{1}{2} A :\nabla_{x}\nabla_{x} g+\int_0^t A [I+\nabla_y\Phi]^\top(\bar{B}^\top)^{-1}:\nabla_p\nabla_x g +\sum_{l=1}^d \biggl\lbrace [f+\epsilon b]\cdot [\nabla_x \Phi_l g_{x_l}+ \Phi_l \nabla_x g_{x_l}] \\
&+\frac{\epsilon}{2} A:[g_{x_l} \nabla_x\nabla_x\Phi_l  + 2\nabla_x\Phi_l \otimes  \nabla_x g_{x_l} + \nabla_x\nabla_x g_{x_l} \Phi_l  ]+b \cdot [g_{x_l}\nabla_y\Phi_l ]\\
& +A : [g_{x_l}\nabla_x\nabla_y \Phi_l +\nabla_x g_{x_l}\otimes\nabla_y\Phi_l  ]\\
&+\biggl[A[I+\nabla_y\Phi]^\top (\bar{B}^\top)^{-1}\biggr]:[\epsilon \nabla_x\Phi_l\otimes \nabla_pg_{x_l}+\epsilon \nabla_p\nabla_xg_{x_l}\Phi_l+\nabla_y\Phi_l\otimes \nabla_pg_{x_l}]\\
&+\frac{1}{2}\biggl[\bar{B}^{-1}[I+\nabla_y\Phi]A[I+\nabla_y\Phi]^\top (\bar{B}^\top)^{-1}\biggr]:[\epsilon\nabla_p\nabla_pg_{x_l} \Phi_l+\nabla_p\nabla_p g]\biggl\rbrace ds\\
& = \int_0^t \biggl[[I+\epsilon \nabla_x \Phi +\nabla_y \Phi ]b+\nabla_x \Phi f+ A:[\nabla_x\nabla_y\Phi +\frac{\epsilon}{2}\nabla_x\nabla_x \Phi]\biggr]\cdot \nabla_{x} g \\
&+\biggl[[\frac{1}{2}+\epsilon\nabla_x \Phi+\nabla_y\Phi] A+[f+\epsilon b]\otimes \Phi \biggr]:\nabla_{x}\nabla_{x} g +\biggl[[I+ \epsilon \nabla_x\Phi+\nabla_y\Phi] A[I+\nabla_y\Phi]^\top(\bar{B}^\top)^{-1}\biggr]:\nabla_p\nabla_x g\\
&+\frac{1}{2}\biggl[\bar{B}^{-1}[I+\nabla_y\Phi]A[I+\nabla_y\Phi]^\top (\bar{B}^\top)^{-1}\biggr]:\nabla_p\nabla_p g\\
&\hspace{1cm}+\sum_{l=1}^d\biggl\lbrace\frac{\epsilon}{2} [A :\nabla_x\nabla_x g_{x_l}]\Phi_l +\epsilon \biggl[A[I+\nabla_y\Phi]^\top (\bar{B}^\top)^{-1}\biggr]: \nabla_p\nabla_xg_{x_l}\Phi_l\\
&\hspace{4cm}+\frac{\epsilon}{2}\biggl[\bar{B}^{-1}[I+\nabla_y\Phi]A[I+\nabla_y\Phi]^\top (\bar{B}^\top)^{-1}\biggr]:\nabla_p\nabla_pg_{x_l} \Phi_l\biggr\rbrace ds\\
B^{i,N}_2(t)& = \int_0^t [\sigma u_i^N]\cdot \nabla_{x} g +\sum_{l=1}^d \biggl\lbrace \epsilon[\sigma u_i^N]\cdot [\nabla_x \Phi_l g_{x_l}+ \Phi_l \nabla_x g_{x_l}] +[\sigma u_i^N]\cdot [g_{x_l}\nabla_y \Phi_l] \biggr\rbrace ds\\
& = \int_0^t \biggl[[I + \epsilon\nabla_x \Phi +\nabla_y \Phi]\sigma u_i^N\biggr]\cdot \nabla_{x} g + \epsilon[\sigma u^N_i]\otimes \Phi : \nabla_{x}\nabla_{x} g ds\\
B^{i,N}_3(t)& = \int_0^t \nabla_{x} g\cdot (\sigma dW^i_s) +\int_0^t \nabla_p g\cdot (\bar{B}^{-1}[I+\nabla_y\Phi]\sigma dW^i_s)\\
&+\sum_{l=1}^d\biggl\lbrace \int_0^t  [\epsilon [\nabla_x \Phi_l g_{x_l}+ \Phi_l \nabla_x g_{x_l}]+[g_{x_l}\nabla_y \Phi_l]]\cdot (\sigma dW^i_s)+\epsilon\int_0^t \Phi_l\nabla_p g_{x_l} \cdot (\bar{B}^{-1}[I+\nabla_y\Phi]\sigma dW^i_s)\biggr\rbrace \\
& = \int_0^t \nabla_{x} g\cdot ([I +\epsilon\nabla_x\Phi +\nabla_y \Phi]\sigma dW^i_s )+\int_0^t \nabla_p g\cdot (\bar{B}^{-1}[I+\nabla_y\Phi]\sigma dW^i_s) \\
&+\int_0^t \epsilon \nabla_{x}\nabla_{x} g: \Phi\otimes (\sigma dW^i_s)+\int_0^t \epsilon \nabla_{x}\nabla_p g: \Phi\otimes (\bar{B}^{-1}[I+\nabla_y\Phi]\sigma dW^i_s)\\
B^{i,N}_4(t)& = \epsilon \sum_{l=1}^d \biggl\lbrace g_{x_l}(x^{i,N},0)\Phi_l(x^{i,N},x^{i,N}/\epsilon,\bar{\mu}^N_0)-g_{x_l}(\bar{X}^{i,N}_t,\bar{W}^{i,N}_t)\Phi_l (\bar{X}^{i,N}_t,\bar{X}^{i,N}_t/\epsilon,\bar{\mu}^N_t)\biggr\rbrace\\
& = \epsilon \biggl[\nabla_{x} g(x^{i,N}) \cdot\Phi(x^{i,N},x^{i,N}/\epsilon,\bar{\mu}^N_0) - \nabla_{x} g(\bar{X}^{i,N}_t)\cdot \Phi (\bar{X}^{i,N}_t,\bar{X}^{i,N}_t/\epsilon,\bar{\mu}^N_t)\biggr]\\
B^{i,N}_5(t)& =\frac{\epsilon}{N}\sum_{l=1}^d  \biggl\lbrace \sum_{j=1}^N \biggl[\int_0^t [g_{x_l}\partial_\mu \Phi_l(\bar{X}^{i,N}_s,\bar{X}^{i,N}_s/\epsilon,\bar{\mu}^N_s)(\bar{X}^{j,N}_s)]\\
&\hspace{4cm}\cdot \biggl(\frac{1}{\epsilon}f(\bar{X}^{j,N}_s,\bar{X}^{j,N}_s/\epsilon,\bar{\mu}^N_s) +b(\bar{X}^{j,N}_s,\bar{X}^{j,N}_s/\epsilon,\bar{\mu}^N_s)\biggr)\\
&\hspace{1cm}+\frac{1}{2}A(\bar{X}^{j,N}_s,\bar{X}^{j,N}_s/\epsilon,\bar{\mu}^N_s):[g_{x_l}\partial_v\partial_\mu \Phi_l(\bar{X}^{i,N}_s,\bar{X}^{i,N}_s/\epsilon,\bar{\mu}^N_s)(\bar{X}^{j,N}_s)]\\
&\hspace{1cm} +\frac{1}{2N}A(\bar{X}^{j,N}_s,\bar{X}^{j,N}_s/\epsilon,\bar{\mu}^N_s):[g_{x_l}\partial^2_\mu \Phi_l(\bar{X}^{i,N}_s,\bar{X}^{i,N}_s/\epsilon,\bar{\mu}^N_s)(\bar{X}^{j,N}_s,\bar{X}^{j,N}_s)] ds\biggr]\biggr\rbrace\\
& = \int_0^t\biggl[\int_{\R^d} \partial_\mu \Phi(\bar{X}^{i,N}_s,\bar{X}^{i,N}_s/\epsilon,\bar{\mu}^N_s)(v)\biggl[f(v,v/\epsilon,\bar{\mu}^N_s)+ \epsilon b(v,v/\epsilon,\bar{\mu}^N_s) \biggr]+ \frac{1}{2}A(v,v/\epsilon,\bar{\mu}^N_s)\\
&\hspace{2cm}:\biggl[\epsilon \partial_v\partial_\mu \Phi(\bar{X}^{i,N}_s,\bar{X}^{i,N}_s/\epsilon,\bar{\mu}^N_s)(v) + \frac{\epsilon}{N}\partial^2_\mu\Phi(\bar{X}^{i,N}_s,\bar{X}^{i,N}_s/\epsilon,\bar{\mu}^N_s)(v,v)  \biggr] \bar{\mu}^N_s(dv)\biggr]\\
&\hspace{4cm}\cdot \nabla_{x} g ds\\
B^{i,N}_6(t)& = \frac{\epsilon}{N}\sum_{l=1}^d  \biggl\lbrace \sum_{j=1}^N \biggl[\int_0^t [g_{x_l}\partial_\mu \Phi_l(\bar{X}^{i,N}_s,\bar{X}^{i,N}_s/\epsilon,\bar{\mu}^N_s)(\bar{X}^{j,N}_s)]\cdot\biggl( \sigma(\bar{X}^{j,N}_s,\bar{X}^{j,N}_s/\epsilon,\bar{\mu}^N_s) u_j^N(s)\biggr) ds\biggr]\biggr\rbrace \\
& = \int_0^t \biggl[\frac{\epsilon}{N}\sum_{j=1}^N \partial_\mu \Phi(\bar{X}^{i,N}_s,\bar{X}^{i,N}_s/\epsilon,\bar{\mu}^N_s)(\bar{X}^{j,N}_s)\sigma(\bar{X}^{j,N}_s,\bar{X}^{j,N}_s/\epsilon,\bar{\mu}^N_s) u_j^N(s)\biggr] \cdot \nabla_{x} g ds\\
B^{i,N}_7(t)& = \frac{\epsilon}{N}\sum_{l=1}^d  \biggl\lbrace \sum_{j=1}^N \biggl[\int_0^t [g_{x_l}\partial_\mu \Phi_l(\bar{X}^{i,N}_s,\bar{X}^{i,N}_s/\epsilon,\bar{\mu}^N_s)(\bar{X}^{j,N}_s)]\cdot ( \sigma(\bar{X}^{j,N}_s,\bar{X}^{j,N}_s/\epsilon,\bar{\mu}^N_s) dW_s^j)\biggr]\biggr\rbrace \\
& = \int_0^t \nabla_{x} g \cdot \biggl[ \frac{\epsilon}{N} \sum_{j=1}^N \partial_\mu \Phi(\bar{X}^{i,N}_s,\bar{X}^{i,N}_s/\epsilon,\bar{\mu}^N_s)(\bar{X}^{j,N}_s)\sigma(\bar{X}^{j,N}_s,\bar{X}^{j,N}_s/\epsilon,\bar{\mu}^N_s)dW^j_s\biggr]\\
B^{i,N}_8(t)& =  \frac{1}{N}\sum_{l=1}^d \biggl\lbrace\epsilon \int_0^t A:[\nabla_{x}g_{x_l}\otimes \partial_\mu \Phi_l(\bar{X}^{i,N}_s,\bar{X}^{i,N}_s/\epsilon,\bar{\mu}^N_s)(\bar{X}^{i,N}_s)\\
&\hspace{1cm}+g_{x_l}\nabla_x\partial_\mu \Phi_l(\bar{X}^{i,N}_s,\bar{X}^{i,N}_s/\epsilon,\bar{\mu}^N_s)(\bar{X}^{i,N}_s)]\\
&\hspace{1cm}+\biggl[A[I+\nabla_y\Phi]^\top (\bar{B}^\top)^{-1}\biggr]:\nabla_{p}g_{x_l} \otimes\partial_\mu \Phi_l(\bar{X}^{i,N}_s,\bar{X}^{i,N}_s/\epsilon,\bar{\mu}^N_s)(\bar{X}^{i,N}_s) ds\\
&\hspace{1cm} +  \int_0^t A:[g_{x_l}\nabla_{y} \partial_\mu \Phi_l(\bar{X}^{i,N}_s,\bar{X}^{i,N}_s/\epsilon,\bar{\mu}^N_s)(\bar{X}^{i,N}_s)]  ds \biggr\rbrace\\
& = \int_0^t\biggl[ A:[\frac{1}{N}\nabla_{y} \partial_\mu \Phi(\bar{X}^{i,N}_s,\bar{X}^{i,N}_s/\epsilon,\bar{\mu}^N_s)(\bar{X}^{i,N}_s) + \frac{\epsilon}{N}\nabla_{x} \partial_\mu \Phi(\bar{X}^{i,N}_s,\bar{X}^{i,N}_s/\epsilon,\bar{\mu}^N_s)(\bar{X}^{i,N}_s)  ] \biggr]\cdot \nabla_{x} g \\
&\hspace{1cm}+ \biggl[\frac{\epsilon}{N}\partial_\mu \Phi(\bar{X}^{i,N}_s,\bar{X}^{i,N}_s/\epsilon,\bar{\mu}^N_s)(\bar{X}^{i,N}_s) A \biggr]:\nabla_{x}\nabla_{x} g\\
&+\frac{\epsilon}{N}\biggl[\partial_\mu \Phi(\bar{X}^{i,N}_s,\bar{X}^{i,N}_s/\epsilon,\bar{\mu}^N_s)(\bar{X}^{i,N}_s)A[I+\nabla_y\Phi]^\top (\bar{B}^\top)^{-1}\biggr]:\nabla_p\nabla_x g ds .
\end{align*}

Rearranging the above, and using that by symmetry
\begin{align*}
\frac{1}{2}D(x,y,\mu):\nabla_x\nabla_x g(x,p)=[\nabla_y\Phi(x,y,\mu)+\frac{1}{2}]A(x,y,\mu):\nabla_x\nabla_x g(x,p)+f\otimes \Phi(x,y,\mu):\nabla_x\nabla_x g(x,p),
\end{align*}
for all $x,p\in\R^d,y\in\bb{T}^d,\mu\in\mc{P}(\R^d),$ we get that
\begin{align*}
g(\bar{X}^{i,N}_t,\bar{W}^{i,N}_t) -g(\bar{X}^{i,N}_s,\bar{W}^{i,N}_s)& = \int_s^t \int_{\bb{T}^d\times\R^m}\mc{A}[g]( \bar{X}^{i,N}_\tau,y,z,\bar{\mu}^N_\tau,\bar{W}^{i,N}_\tau)\rho^{i,N}_\tau(dydz)d\tau\\
&+\int_s^t \sqrt{\int_{\bb{T}^d\times\R^m}\tilde{D}(\bar{X}^{i,N}_\tau,y,\bar{\mu}^N_\tau)  \rho^{i,N}_\tau(dydz)}:\nabla_p\nabla_x g(\bar{X}^{i,N}_\tau,\bar{W}^{i,N}_\tau) d\tau\\
&+\sum_{k=1}^{7} D^{i,N}_k,
\end{align*}
where
\begin{align*}
D^{i,N}_1& = \int_s^t \biggl[\frac{1}{N}A:[\nabla_{y} \partial_\mu \Phi(\bar{X}^{i,N}_\tau,\bar{X}^{i,N}_\tau/\epsilon,\bar{\mu}^N_\tau)(\bar{X}^{i,N}_\tau)]\\
&+\epsilon\biggl( \nabla_x \Phi b+ \frac{1}{2}A:\nabla_x\nabla_x \Phi +\frac{1}{N}A:\nabla_{x} \partial_\mu \Phi(\bar{X}^{i,N}_\tau,\bar{X}^{i,N}_\tau/\epsilon,\bar{\mu}^N_\tau)(\bar{X}^{i,N}_\tau)\\
&\hspace{1cm} +\biggl\lbrace\int_{\R^d} \partial_\mu \Phi(\bar{X}^{i,N}_\tau,\bar{X}^{i,N}_\tau/\epsilon,\bar{\mu}^N_\tau)(v) b(v,v/\epsilon,\bar{\mu}^N_\tau) +\frac{1}{2}A(v,v/\epsilon,\bar{\mu}^N_\tau)\\
&\hspace{2cm}:\biggl[\partial_v\partial_\mu \Phi(\bar{X}^{i,N}_\tau,\bar{X}^{i,N}_\tau/\epsilon,\bar{\mu}^N_\tau)(v) + \frac{1}{N}\partial^2_\mu\Phi(\bar{X}^{i,N}_\tau,\bar{X}^{i,N}_\tau/\epsilon,\bar{\mu}^N_\tau)(v,v)  \biggr] \bar{\mu}^N_\tau(dv)\biggr\rbrace\biggr)\biggr]\\
&\hspace{11cm}\cdot \nabla_{x} g(\bar{X}^{i,N}_\tau,\bar{W}^{i,N}_\tau)d\tau \\
&+\epsilon\int_s^t\biggl[\nabla_x \Phi A+b\otimes \Phi +\frac{1}{N}\partial_\mu \Phi(\bar{X}^{i,N}_\tau,\bar{X}^{i,N}_\tau/\epsilon,\bar{\mu}^N_\tau)(\bar{X}^{i,N}_\tau) A \biggr]:\nabla_{x}\nabla_{x} g(\bar{X}^{i,N}_\tau,\bar{W}^{i,N}_\tau)d\tau \\
&+\epsilon\int_s^t \biggl[\nabla_x\Phi A[I+\nabla_y\Phi]^\top(\bar{B}^\top)^{-1}+\frac{1}{N}\partial_\mu \Phi(\bar{X}^{i,N}_s,\bar{X}^{i,N}_s/\epsilon,\bar{\mu}^N_s)(\bar{X}^{i,N}_s)A[I+\nabla_y\Phi]^\top (\bar{B}^\top)^{-1}\biggr]\\
&\hspace{11cm}:\nabla_p\nabla_xg(\bar{X}^{i,N}_\tau,\bar{W}^{i,N}_\tau)d\tau\\
&+\epsilon\int_s^t\sum_{l=1}^d\biggl\lbrace\frac{1}{2} [A :\nabla_x\nabla_x g_{x_l}(\bar{X}^{i,N}_\tau,\bar{W}^{i,N}_\tau)]\Phi_l+\biggl[A[I+\nabla_y\Phi]^\top (\bar{B}^\top)^{-1}\biggr]: \nabla_p\nabla_xg_{x_l}\Phi_l\\
&\hspace{4cm}+\frac{1}{2}\biggl[\bar{B}^{-1}[I+\nabla_y\Phi]A[I+\nabla_y\Phi]^\top (\bar{B}^\top)^{-1}\biggr]:\nabla_p\nabla_pg_{x_l}(\bar{X}^{i,N}_\tau,\bar{W}^{i,N}_\tau) \Phi_l\biggr\rbrace  d\tau\\
D^{i,N}_2& = \epsilon\int_s^t \biggl[\nabla_x \Phi \sigma u_i^N(\tau) +\biggl\lbrace\frac{1}{N}\sum_{j=1}^N \partial_\mu \Phi(\bar{X}^{i,N}_\tau,\bar{X}^{i,N}_\tau/\epsilon,\bar{\mu}^N_\tau)(\bar{X}^{j,N}_\tau)\sigma(\bar{X}^{j,N}_\tau,\bar{X}^{j,N}_\tau/\epsilon,\bar{\mu}^N_\tau) u_j^N(\tau)\biggr\rbrace\biggr]\\
&\hspace{11cm}\cdot \nabla_{x} g(\bar{X}^{i,N}_\tau,\bar{W}^{i,N}_\tau)\\
& + [\sigma u^N_i(\tau)]\otimes \Phi : \nabla_{x}\nabla_{x} g(\bar{X}^{i,N}_\tau,\bar{W}^{i,N}_\tau) d\tau\\
D^{i,N}_3& = \int_s^t \nabla_{x} g(\bar{X}^{i,N}_\tau,\bar{W}^{i,N}_\tau)\cdot ([I +\epsilon\nabla_x\Phi +\nabla_y \Phi]\sigma dW^i_\tau ) +\int_s^t \epsilon \nabla_{x}\nabla_{x} g(\bar{X}^{i,N}_\tau,\bar{W}^{i,N}_\tau): \Phi\otimes (\sigma dW^i_\tau)\\
& +\int_s^t \nabla_{x} g(\bar{X}^{i,N}_\tau,\bar{W}^{i,N}_\tau)\cdot \biggl[ \frac{\epsilon}{N} \sum_{j=1}^N \partial_\mu \Phi(\bar{X}^{i,N}_\tau,\bar{X}^{i,N}_\tau/\epsilon,\bar{\mu}^N_\tau)(\bar{X}^{j,N}_\tau)\sigma(\bar{X}^{j,N}_\tau,\bar{X}^{j,N}_\tau/\epsilon,\bar{\mu}^N_\tau)dW^j_\tau\biggr]\\
&+\int_s^t \nabla_p g(\bar{X}^{i,N}_\tau,\bar{W}^{i,N}_\tau)\cdot (\bar{B}^{-1}[I+\nabla_y\Phi]\sigma dW^i_s)+\int_s^t \epsilon \nabla_{x}\nabla_{p} g(\bar{X}^{i,N}_\tau,\bar{W}^{i,N}_\tau): \Phi\otimes (\bar{B}^{-1}[I+\nabla_y\Phi]\sigma dW^i_s)\\
D^{i,N}_4& = \epsilon \biggl[\nabla_{x} g(\bar{X}^{i,N}_\tau,\bar{W}^{i,N}_\tau) \cdot\Phi(\bar{X}^{i,N}_s,\bar{X}^{i,N}_s/\epsilon,\bar{\mu}^N_s) - \nabla_{x} g(\bar{X}^{i,N}_\tau,\bar{W}^{i,N}_\tau)\cdot \Phi (\bar{X}^{i,N}_t,\bar{X}^{i,N}_t/\epsilon,\bar{\mu}^N_t)\biggr]\\
D^{i,N}_5& = \int_s^t \biggl[\frac{1}{N}\sum_{j=1}^N \partial_\mu \Phi(\bar{X}^{i,N}_\tau,\bar{X}^{i,N}_\tau/\epsilon,\bar{\mu}^N_\tau)(\bar{X}^{j,N}_\tau)f(\bar{X}^{j,N}_\tau,\bar{X}^{j,N}_\tau/\epsilon,\bar{\mu}^N_\tau)\biggr]\cdot \nabla_{x} g(\bar{X}^{i,N}_\tau,\bar{W}^{i,N}_\tau)d\tau\\
D^{i,N}_6& =\int_s^t\biggl[\frac{1}{2}\bar{B}^{-1}(\bar{X}^{i,N}_\tau,\bar{\mu}^N_\tau)\tilde{D}(\bar{X}^{i,N}_\tau,\bar{X}^{i,N}_\tau/\epsilon,\bar{\mu}^N_\tau) (\bar{B}^\top)^{-1}(\bar{X}^{i,N}_\tau,\bar{\mu}^N_\tau)-\frac{1}{2}I\biggr]:\nabla_p\nabla_pg(\bar{X}^{i,N}_\tau,\bar{W}^{i,N}_\tau)d\tau\\
D^{i,N}_7&=\int_s^t\biggl[\tilde{D}(\bar{X}^{i,N}_\tau,\bar{X}^{i,N}_\tau/\epsilon,\bar{\mu}^N_\tau) (\bar{B}^\top)^{-1}(\bar{X}^{i,N}_\tau,\bar{\mu}^N_\tau)-\sqrt{\int_{\bb{T}^d\times\R^m}\tilde{D}(\bar{X}^{i,N}_\tau,y,\bar{\mu}^N_\tau)  \rho^{i,N}_\tau(dydz)}\biggr]\\
&\hspace{11cm}:\nabla_p\nabla_xg(\bar{X}^{i,N}_\tau,\bar{W}^{i,N}_\tau)d\tau,
\end{align*}

$\rho^{i,N}$ are defined as in Equation \eqref{eq:occmeas} and the arguments which are omitted are taken to be $(\bar{X}^{i,N}_\tau,\bar{X}^{i,N}_\tau/\epsilon,\bar{\mu}^N_\tau)$, or in the case of $\bar{B}$, $(\bar{X}^{i,N}_\tau,\bar{\mu}^N_\tau)$. Thus
\begin{align*}
\E^{Q^N}\biggl[\Psi(M_g^{Q^N}(t)-M_g^{Q^N}(s))\biggr] =  \frac{1}{N}\sum_{i=1}^N \Psi(\bar{X}^{i,N},\rho^{i,N},\bar{W}^{i,N})\sum_{k=1}^{7} D^{i,N}_k.
\end{align*}

Using Assumption \ref{assumption:LipschitzandBounded} and Proposition \ref{prop:Phiexistenceregularity}, we first show that for large enough $N$:
\begin{align}\label{eq:vanishingmartingaleprobterms}
\E\biggl[\biggr|\frac{1}{N}\sum_{i=1}^N  \Psi(\bar{X}^{i,N},\rho^{i,N},\bar{W}^{i,N})\sum_{k=1}^4 D^{i,N}_k\biggr|\biggr]\leq \max\br{\epsilon,\frac{1}{N^{1/2}}}C
\end{align}
where $C$ depends only on the sup norms of $\Psi$ and $h$ and its first 3 derivatives. This vanishes as $N\toinf$, so once we prove the following (\ref{eq:D5vanishes}), Lemma \ref{lem:marprob2} will be proved:
\begin{align}
\lim_{N\toinf}\E\biggl[| \frac{1}{N}\sum_{i=1}^N \Psi(\bar{X}^{i,N},\rho^{i,N},\bar{W}^{i,N}) D_k^{i,N}|\biggr]=0,k=5,6,7.\label{eq:D5vanishes}
\end{align}
%\end{lem}

Let's first show \eqref{eq:vanishingmartingaleprobterms}.
Firstly, we observe that
\begin{align*}
\E\biggl[\biggr|\frac{1}{N}\sum_{i=1}^N  \Psi(\bar{X}^{i,N},\rho^{i,N},\bar{W}^{i,N})D^{i,N}_4\biggr|\biggr]&\leq 2\epsilon \norm{\Psi}_\infty\norm{\nabla_x g}_\infty  \frac{1}{N}\sum_{i=1}^N \E[\sup_{t\in [0,1]}  |\Phi(\bar{X}^{i,N}_t,\bar{X}^{i,N}_t/\epsilon,\bar{\mu}^N_t)|] \\
&\leq \epsilon C
\end{align*}
by Proposition \ref{prop:Phiexistenceregularity}.

Next, we observe that
\begin{align*}
&\E\biggl[\biggr|\frac{1}{N}\sum_{i=1}^N  \Psi(\bar{X}^{i,N},\rho^{i,N},\bar{W}^{i,N})D^{i,N}_2\biggr|\biggr]\\
&\leq \epsilon C(m,d)\norm{\Psi}_\infty\norm{g}_{C^2_b(\R^d\times\R^d)} \frac{1}{N}\sum_{i=1}^N\E\biggl[ \biggl( \int_{s}^{t} | \nabla_x \Phi \sigma|^2 d\tau  \int_0^1 |u^N_i(\tau)|^2 d\tau   \biggr)^{1/2}\\
&+ \biggl(\int_{s}^{t} |\Phi|^2 |\sigma|^2 d\tau \int_0^1 |u^N_i(\tau)|^2 d\tau   \biggr)^{1/2}\\
&+\frac{1}{N} \sum_{j=1}^N \biggl(\int_s^t|\partial_\mu \Phi(\bar{X}^{i,N}_\tau,\bar{X}^{i,N}_\tau/\epsilon,\bar{\mu}^N_\tau)(\bar{X}^{j,N}_\tau)\sigma(\bar{X}^{j,N}_\tau,\bar{X}^{j,N}_\tau/\epsilon,\bar{\mu}^N_\tau)|^2 ds\int_0^1 |u_j^N(\tau)|^2 d\tau\biggr)^{1/2}\biggr]\\
&\leq \epsilon C(m,d)\norm{\Psi}_\infty\norm{g}_{C^2_b(\R^d\times\R^d)}\biggl\lbrace|t-s|^{1/2} \biggl(\frac{1}{N}\sum_{i=1}^N\E\biggl[\int_0^1 |u^N_i(\tau)|^2 d\tau   \biggr]\biggr)^{1/2}\\
&+\E\biggl[\biggl(\int_s^t \norm{\partial_\mu \Phi(\bar{X}^{i,N}_\tau,\bar{X}^{i,N}_\tau/\epsilon,\bar{\mu}^N_\tau)(\cdot)}_{L^2(\R^d,\bar{\mu}^N_\tau)}^2 d\tau\biggr)^{1/2} \biggl(\frac{1}{N} \sum_{j=1}^N \int_0^1 |u_j^N(\tau)|^2 d\tau\biggr)^{1/2}\biggr]\biggr\rbrace\\
&\leq  \epsilon C(m,d)\norm{\Psi}_\infty\norm{g}_{C^2_b(\R^d\times\R^d)}|t-s|^{1/2} \biggl(\frac{1}{N}\sum_{i=1}^N\E\biggl[\int_0^1 |u^N_i(\tau)|^2 d\tau   \biggr]\biggr)^{1/2}\\
&\leq \epsilon C(m,d)C_{con}^{1/2}\norm{\Psi}_\infty\norm{g}_{C^2_b(\R^d\times\R^d)}
\end{align*}
by H\"older's inequality, monotonicity of the time integrals, Assumption \ref{assumption:LipschitzandBounded}, Proposition \ref{prop:Phiexistenceregularity}, Jensen's inequality, and the control bound \eqref{eq:controlL2boundunspecific}.

In addition,
\begin{align*}
&\E\biggl[\biggr|\frac{1}{N}\sum_{i=1}^N  \Psi(\bar{X}^{i,N},\rho^{i,N},\bar{W}^{i,N})D^{i,N}_3\biggr|\biggr]\\
&\E\biggl[\biggr|\frac{1}{N}\sum_{i=1}^N  \Psi(\bar{X}^{i,N},\rho^{i,N},\bar{W}^{i,N})D^{i,N}_3\biggr|^2\biggr]^{1/2} \\
&\leq C(m,d)\norm{\Psi}_\infty\norm{g}_{C^2_b(\R^d\times\R^d)}\E\biggl[\frac{1}{N^2}\sum_{i=1}^N\int_s^t 1+\epsilon^2|\nabla_x\Phi|^2+|\nabla_y\Phi|^2+|\Phi|^2 \\
&\hspace{8cm}+ \epsilon^2\norm{\partial_\mu \Phi(\bar{X}^{i,N}_\tau,\bar{X}^{i,N}_\tau/\epsilon,\bar{\mu}^N_\tau)(\cdot)}^2_{L^2(\R^d,\bar{\mu}^N_\tau)}d\tau \biggr]^{1/2}\\
&\leq  C(m,d)\norm{\Psi}_\infty\norm{g}_{C^2_b(\R^d\times\R^d)}\frac{1}{N^{1/2}}|t-s|^{1/2}(1+\epsilon)\\
&\leq  C(m,d)\norm{\Psi}_\infty\norm{g}_{C^2_b(\R^d\times\R^d)}\frac{1}{N^{1/2}}.
\end{align*}
by Jensen's Inequality, It\^o Isometry (using orthogonality of the martingales and $\mc{G}_s$-measurability of $\Psi$ - see e.g. \cite{BDF} p.93), Assumption \ref{assumption:LipschitzandBounded}, Proposition \ref{prop:Phiexistenceregularity}, and Corollary \ref{cor:limitingcoefficientsregularity}.

Lastly, we observe that
\begin{align*}
&\E\biggl[\biggr|\frac{1}{N}\sum_{i=1}^N  \Psi(\bar{X}^{i,N},\rho^{i,N},\bar{W}^{i,N})D^{i,N}_1\biggr|\biggr]\\
&\leq C(m,d)\norm{\Psi}_\infty\norm{g}_{C^3_b(\R^d\times\R^d}\frac{1}{N}\sum_{i=1}^N\E\biggl[\int_s^{t}\frac{1}{N}\biggl\lbrace|\nabla_{y} \partial_\mu \Phi(\bar{X}^{i,N}_\tau,\bar{X}^{i,N}_\tau/\epsilon,\bar{\mu}^N_\tau)(\bar{X}^{i,N}_\tau)|\\
&+\epsilon |\nabla_{x} \partial_\mu \Phi(\bar{X}^{i,N}_\tau,\bar{X}^{i,N}_\tau/\epsilon,\bar{\mu}^N_\tau)(\bar{X}^{i,N}_\tau)|+\epsilon|\partial_\mu \Phi(\bar{X}^{i,N}_\tau,\bar{X}^{i,N}_\tau/\epsilon,\bar{\mu}^N_\tau)(\bar{X}^{i,N}_\tau)|(1+|\nabla_y\Phi|)\ \biggr\rbrace \\
&+\epsilon\biggl\lbrace\int_{\R^d} |\partial_\mu \Phi(\bar{X}^{i,N}_\tau,\bar{X}^{i,N}_\tau/\epsilon,\bar{\mu}^N_\tau)(v)|+|\partial_v\partial_\mu \Phi(\bar{X}^{i,N}_\tau,\bar{X}^{i,N}_\tau/\epsilon,\bar{\mu}^N_\tau)(v)| + \frac{1}{N}|\partial^2_\mu\Phi(\bar{X}^{i,N}_\tau,\bar{X}^{i,N}_\tau/\epsilon,\bar{\mu}^N_\tau)(v,v)| \bar{\mu}^N_\tau(dv)\\
&+|\nabla_x\Phi| (1+|\nabla_y\Phi|)+|\nabla_x\nabla_x \Phi|+(1+|\nabla_y\Phi|)^2 |\Phi| \biggr\rbrace   d\tau \biggr]\\
&\leq C(m,d)\norm{\Psi}_\infty\norm{g}_{C^3_b(\R^d\times\R^d}\frac{1}{N}\sum_{i=1}^N\E\biggl[\int_s^{t}\frac{1}{N}\biggl\lbrace|\nabla_{y} \partial_\mu \Phi(\bar{X}^{i,N}_\tau,\bar{X}^{i,N}_\tau/\epsilon,\bar{\mu}^N_\tau)(\bar{X}^{i,N}_\tau)|\\
&+\epsilon |\nabla_{x} \partial_\mu \Phi(\bar{X}^{i,N}_\tau,\bar{X}^{i,N}_\tau/\epsilon,\bar{\mu}^N_\tau)(\bar{X}^{i,N}_\tau)|+\epsilon|\partial_\mu \Phi(\bar{X}^{i,N}_\tau,\bar{X}^{i,N}_\tau/\epsilon,\bar{\mu}^N_\tau)(\bar{X}^{i,N}_\tau)|\ \biggr\rbrace \\
&+\epsilon\biggl\lbrace\norm{\partial_\mu \Phi(\bar{X}^{i,N}_\tau,\bar{X}^{i,N}_\tau/\epsilon,\bar{\mu}^N_\tau)(\cdot)}_{L^2(\R^d,\bar{\mu}^N_\tau)}+\norm{\partial_v\partial_\mu \Phi(\bar{X}^{i,N}_\tau,\bar{X}^{i,N}_\tau/\epsilon,\bar{\mu}^N_\tau)(\cdot)}_{L^2(\R^d,\bar{\mu}^N_\tau)} \\
&+ \norm{\partial^2_\mu\Phi(\bar{X}^{i,N}_\tau,\bar{X}^{i,N}_\tau/\epsilon,\bar{\mu}^N_\tau)(\cdot,\circ)}_{L^2(\R^d,\bar{\mu}^N_\tau)\otimes L^2(\R^d;\bar{\mu}^N_\tau)}+1 \biggr\rbrace   d\tau \biggr]\\
&\leq C(m,d)\norm{\Psi}_\infty\norm{g}_{C^3_b(\R^d\times\R^d}\biggl\lbrace\E\biggl[\int_s^{t}\frac{1}{N^2}\sum_{i=1}^N\sum_{j=1}^N\biggl\lbrace|\nabla_{y} \partial_\mu \Phi(\bar{X}^{i,N}_\tau,\bar{X}^{i,N}_\tau/\epsilon,\bar{\mu}^N_\tau)(\bar{X}^{j,N}_\tau)|\\
&+\epsilon |\nabla_{x} \partial_\mu \Phi(\bar{X}^{i,N}_\tau,\bar{X}^{i,N}_\tau/\epsilon,\bar{\mu}^N_\tau)(\bar{X}^{j,N}_\tau)|+\epsilon|\partial_\mu \Phi(\bar{X}^{i,N}_\tau,\bar{X}^{i,N}_\tau/\epsilon,\bar{\mu}^N_\tau)(\bar{X}^{j,N}_\tau)|\biggr\rbrace\1_{i=j} d\tau\biggr] +|t-s|\epsilon\biggr\rbrace \\
&\leq C(m,d)\norm{\Psi}_\infty\norm{g}_{C^3_b(\R^d\times\R^d}\biggl\lbrace\E\biggl[\int_s^{t}\frac{1}{N^2}\sum_{i=1}^N\biggl(\sum_{j=1}^N\biggl\lbrace|\nabla_{y} \partial_\mu \Phi(\bar{X}^{i,N}_\tau,\bar{X}^{i,N}_\tau/\epsilon,\bar{\mu}^N_\tau)(\bar{X}^{j,N}_\tau)|\\
&+\epsilon |\nabla_{x} \partial_\mu \Phi(\bar{X}^{i,N}_\tau,\bar{X}^{i,N}_\tau/\epsilon,\bar{\mu}^N_\tau)(\bar{X}^{j,N}_\tau)|+\epsilon|\partial_\mu \Phi(\bar{X}^{i,N}_\tau,\bar{X}^{i,N}_\tau/\epsilon,\bar{\mu}^N_\tau)(\bar{X}^{j,N}_\tau)|\biggr\rbrace^2\sum_{j=1}^N(\1_{i=j})^2\biggr)^{1/2} d\tau\biggr] +|t-s|\epsilon\biggr\rbrace \\
&\leq C(m,d)\norm{\Psi}_\infty\norm{g}_{C^3_b(\R^d\times\R^d}\biggl\lbrace\E\biggl[\int_s^{t}\frac{1}{N^2}\sum_{i=1}^N\biggl(\sum_{j=1}^N|\nabla_{y} \partial_\mu \Phi(\bar{X}^{i,N}_\tau,\bar{X}^{i,N}_\tau/\epsilon,\bar{\mu}^N_\tau)(\bar{X}^{j,N}_\tau)|^2\\
&+\epsilon^2 |\nabla_{x} \partial_\mu \Phi(\bar{X}^{i,N}_\tau,\bar{X}^{i,N}_\tau/\epsilon,\bar{\mu}^N_\tau)(\bar{X}^{j,N}_\tau)|^2+\epsilon^2|\partial_\mu \Phi(\bar{X}^{i,N}_\tau,\bar{X}^{i,N}_\tau/\epsilon,\bar{\mu}^N_\tau)(\bar{X}^{j,N}_\tau)|^2\biggr)^{1/2} d\tau\biggr] +|t-s|\epsilon\biggr\rbrace \\
&\leq C(m,d)\norm{\Psi}_\infty\norm{g}_{C^3_b(\R^d\times\R^d}\biggl\lbrace\E\biggl[\frac{1}{N^{3/2}}\sum_{i=1}^N\int_s^{t}\biggl(\norm{|\nabla_{y} \partial_\mu \Phi(\bar{X}^{i,N}_\tau,\bar{X}^{i,N}_\tau/\epsilon,\bar{\mu}^N_\tau)(\cdot)}_{L^2(\R^d,\tilde{\mu}^N_\tau)}\\
&+\epsilon \norm{\nabla_{x} \partial_\mu \Phi(\bar{X}^{i,N}_\tau,\bar{X}^{i,N}_\tau/\epsilon,\bar{\mu}^N_\tau)(\cdot)}_{L^2(\R^d,\tilde{\mu}^N_\tau)}+\epsilon\norm{\partial_\mu \Phi(\bar{X}^{i,N}_\tau,\bar{X}^{i,N}_\tau/\epsilon,\bar{\mu}^N_\tau)(\cdot)}_{L^2(\R^d,\tilde{\mu}^N_\tau)}\biggr) d\tau\biggr] +|t-s|\epsilon\biggr\rbrace \\
&\leq C(m,d)\norm{\Psi}_\infty\norm{g}_{C^3_b(\R^d\times\R^d}|t-s|\biggl\lbrace\frac{1}{N^{1/2}}+\frac{\epsilon}{N^{1/2}}+\epsilon\biggr\rbrace \\
&\leq C(m,d)\norm{\Psi}_\infty\norm{g}_{C^3_b(\R^d\times\R^d}\max\br{\frac{1}{N^{1/2}},\epsilon}
\end{align*}
by Assumption \ref{assumption:LipschitzandBounded}, Proposition \ref{prop:Phiexistenceregularity}, and Corollary \ref{cor:limitingcoefficientsregularity}. So indeed \eqref{eq:vanishingmartingaleprobterms} holds.

Now we will show \eqref{eq:D5vanishes}. Unlike those appearing in \eqref{eq:vanishingmartingaleprobterms}, the terms we wish to vanish in (\ref{eq:D5vanishes}) are a priori $\mc{O}(1)$ in $N$. However, as we will see, the fact that $Q$ almost surely satisfies \ref{V:V4} via Proposition \ref{remark:V4explaination} and the centering condition from Assumption \ref{assumption:centeringcondition} will result in these terms vanishing when we pass to the limit. We first observe that
\begin{align*}
\E\biggl[\biggl|\frac{1}{N}\sum_{i=1}^N \Psi(\bar{X}^{i,N},\rho^{i,N},\bar{W}^{i,N})D_{k}^{i,N}\biggr|\biggr]&\leq \norm{\Psi}_\infty \norm{g}_{C^2_b(\R^d\times\R^d)} \E\biggl[D^N_{k} \biggr],k=5,6,7
\end{align*}
where
\begin{align*}
D^N_5&\coloneqq \int_s^t \frac{1}{N}\sum_{i=1}^N\biggl|\frac{1}{N}\sum_{j=1}^N \partial_\mu \Phi(\bar{X}^{i,N}_\tau,\bar{X}^{i,N}_\tau/\epsilon,\bar{\mu}^N_\tau)(\bar{X}^{j,N}_\tau)f(\bar{X}^{j,N}_\tau,\bar{X}^{j,N}_\tau/\epsilon,\bar{\mu}^N_\tau)\biggr|d\tau\\
D^N_6& =\int_s^t\frac{1}{N}\sum_{i=1}^N\biggl|\frac{1}{2}\bar{B}^{-1}(\bar{X}^{i,N}_\tau,\bar{\mu}^N_\tau)\tilde{D}(\bar{X}^{i,N}_\tau,\bar{X}^{i,N}_\tau/\epsilon,\bar{\mu}^N_\tau) (\bar{B}^\top)^{-1}(\bar{X}^{i,N}_\tau,\bar{\mu}^N_\tau)-\frac{1}{2}I\biggr|d\tau\\
D^N_7&=\int_s^t\frac{1}{N}\sum_{i=1}^N\biggl|\tilde{D}(\bar{X}^{i,N}_\tau,\bar{X}^{i,N}_\tau/\epsilon,\bar{\mu}^N_\tau) (\bar{B}^\top)^{-1}(\bar{X}^{i,N}_\tau,\bar{\mu}^N_\tau)-\sqrt{\int_{\bb{T}^d}\tilde{D}(\bar{X}^{i,N}_\tau,y,\bar{\mu}^N_\tau)  m^{i,N}_\tau(dy)}\biggr|d\tau,
\end{align*}

We can rewrite $D^N_{k},k=5,6,7$ in terms of the occupation measures defined in Equation \eqref{eq:occmeas} as:
\begin{align*}
D^N_{5}= \int_s^t \int_{\mc{C}}\int_{\bb{T}^d\times\R^m} \biggl|\int_{\mc{C}}\int_{\bb{T}^d\times\R^m} \partial_\mu \Phi(\phi(\tau),y,\bar{\mu}^N_\tau)(\psi(\tau))f(\psi(\tau),\hat{y},\bar{\mu}^N_\tau) r_\tau(d\hat{y}d\hat{z})&Q^N(d\psi dr dw)  \biggr|\rho_\tau(dydz) \\
&Q^N(d\phi d\rho dv)d\tau.
\end{align*}
and
\begin{align*}
D^N_6& =\int_{\mc{C}}\int_s^t\biggl|\frac{1}{2}\bar{B}^{-1}(\phi(\tau),\bar{\mu}^N_\tau)\int_{\bb{T}^d\times\R^m} \tilde{D}(\phi(\tau),y,\bar{\mu}^N_\tau)r_\tau(dydz) (\bar{B}^\top)^{-1}(\phi(\tau),\bar{\mu}^N_\tau)-\frac{1}{2}I\biggr|d\tau Q^N(d\phi dr dw)\\
D^N_7&=\int_{\mc{C}}\int_s^t\biggl|\int_{\bb{T}^d\times\R^m} \tilde{D}(\phi(\tau),y,\bar{\mu}^N_\tau)r_\tau(dydz) (\bar{B}^\top)^{-1}(\phi(\tau),\bar{\mu}^N_\tau)\\
&\hspace{7cm}-\sqrt{\int_{\bb{T}^d\times\R^m} \tilde{D}(\phi(\tau),y,\bar{\mu}^N_\tau)  r_\tau(dydz)}\biggr|d\tau Q^N(d\phi dr dw).
\end{align*}
We will show each of these vanishes in expectation, at which point Lemma \ref{lem:marprob2} will be proved. Since this is simpler for $D^N_6$ and $D^N_7$, discuss these two terms first.
We first observe that, by boundedness and continuity of $\bar{B}^{-1}$ and $\tilde{D}$ from Corollary \ref{cor:limitingcoefficientsregularity}, that
\begin{align*}
&\lim_{N\toinf}\E[D^N_6]\\
& = \E\biggl[ \lim_{N\toinf}\int_{\mc{C}}\int_s^t\biggl|\frac{1}{2}\bar{B}^{-1}(\phi(\tau),\nu_{Q^N}(\tau))\int_{\bb{T}^d\times\R^m} \tilde{D}(\phi(\tau),y,\nu_{Q^N}(\tau))r_\tau(dydz) (\bar{B}^\top)^{-1}(\phi(\tau),\nu_{Q^N}(\tau))\\
&\hspace{12cm}-\frac{1}{2}I\biggr|d\tau Q^N(d\phi dr dw)\biggr]\\
& = \E\biggl[ \int_{\mc{C}}\int_s^t\biggl|\frac{1}{2}\bar{B}^{-1}(\phi(\tau),\nu_{Q}(\tau))\int_{\bb{T}^d\times\R^m} \tilde{D}(\phi(\tau),y,\nu_{Q}(\tau))r_\tau(dydz) (\bar{B}^\top)^{-1}(\phi(\tau),\nu_{Q}(\tau))\\
&\hspace{12cm}-\frac{1}{2}I\biggr|d\tau Q(d\phi dn dr dw)\biggr],
\end{align*}
where in the second step we used continuity of $\nu_\cdot(t)$ from Proposition \ref{prop:nucontmeasure} and Theorem A.3.18 in \cite{DE}. Then using that $Q$ almost surely satisfies \ref{V:V4} via Proposition \ref{remark:V4explaination},
\begin{align*}
&\E\biggl[ \int_{\mc{X}}\int_{\mc{Y}}\int_s^t\biggl|\frac{1}{2}\bar{B}^{-1}(\phi(\tau),\nu_{Q}(\tau))\int_{\bb{T}^d\times\R^m} \tilde{D}(\phi(\tau),y,\nu_{Q}(\tau))r_\tau(dydz) (\bar{B}^\top)^{-1}(\phi(\tau),\nu_{Q}(\tau))\\
&\hspace{12cm}-\frac{1}{2}I\biggr|d\tau \lambda(dr|\phi)Q_{\mc{X}}(d\phi)\biggr]\\
&=\E\biggl[ \int_{\mc{X}}\int_s^t\biggl|\frac{1}{2}\bar{B}^{-1}(\phi(\tau),\nu_{Q}(\tau))\int_{\bb{T}^d\times\R^m} \tilde{D}(\phi(\tau),y,\nu_{Q}(\tau))\pi(dy;\phi(\tau),\nu_Q(\tau)) (\bar{B}^\top)^{-1}(\phi(\tau),\nu_{Q}(\tau))\\
&\hspace{12cm}-\frac{1}{2}I\biggr|d\tau Q_{\mc{X}}(d\phi)\biggr]\\
& = \E\biggl[ \int_{\mc{X}}\int_s^t\biggl|\frac{1}{2}\bar{B}^{-1}(\phi(\tau),\nu_{Q}(\tau))\bar{D}(\phi(\tau),\nu_{Q}(\tau)) (\bar{B}^\top)^{-1}(\phi(\tau),\nu_{Q}(\tau))-\frac{1}{2}I\biggr|d\tau Q_{\mc{X}}(d\phi)\biggr]\text{ by Remark \ref{remark:altdiffusionrep}}\\
& = \E\biggl[ \int_{\mc{X}}\int_s^t\biggl|\frac{1}{2}\bar{B}^{-1}(\phi(\tau),\nu_{Q}(\tau))\bar{B}(\phi(\tau),\nu_{Q}(\tau))\bar{B}^\top(\phi(\tau),\nu_{Q}(\tau)) (\bar{B}^\top)^{-1}(\phi(\tau),\nu_{Q}(\tau))-\frac{1}{2}I\biggr|d\tau Q_{\mc{X}}(d\phi)\biggr]\\
& = \E\biggl[ \int_s^t\biggl|\frac{1}{2}I-\frac{1}{2}I\biggr|d\tau )\biggr]\\
&=0.
\end{align*}
Similarly, for $D^N_7$, using in addition the continuity of the matrix square root:
\begin{align*}
&\lim_{N\toinf}\E[D^N_7]\\
& = \E\biggl[\lim_{N\toinf}\int_{\mc{C}}\int_s^t\biggl|\int_{\bb{T}^d\times\R^m} \tilde{D}(\phi(\tau),y,\nu_{Q^N}(\tau))r_\tau(dydz) (\bar{B}^\top)^{-1}(\phi(\tau),\nu_{Q^N}(\tau))\\
&\hspace{6cm}-\sqrt{\int_{\bb{T}^d\times\R^m} \tilde{D}(\phi(\tau),y,\nu_{Q^N}(\tau))  r_\tau(dydz)}\biggr|d\tau Q^N(d\phi dn dr dw\biggr]\\
& = \E\biggl[\int_{\mc{C}}\int_s^t\biggl|\int_{\bb{T}^d\times\R^m} \tilde{D}(\phi(\tau),y,\nu_{Q}(\tau))r_\tau(dydz) (\bar{B}^\top)^{-1}(\phi(\tau),\nu_{Q}(\tau))\\
&\hspace{6cm}-\sqrt{\int_{\bb{T}^d\times\R^m} \tilde{D}(\phi(\tau),y,\nu_{Q}(\tau))  r_\tau(dydz)}\biggr|d\tau Q(d\phi dn dr dw\biggr]\\
&= \E\biggl[\int_{\mc{X}}\int_{\mc{Y}}\int_s^t\biggl|\int_{\bb{T}^d\times\R^m} \tilde{D}(\phi(\tau),y,\nu_{Q}(\tau))r_\tau(dydz) (\bar{B}^\top)^{-1}(\phi(\tau),\nu_{Q}(\tau))\\
&\hspace{6cm}-\sqrt{\int_{\bb{T}^d\times\R^m} \tilde{D}(\phi(\tau),y,\nu_{Q}(\tau))  r_\tau(dydz)}\biggr|d\tau\lambda(dr|\phi)Q_{\mc{X}}(d\phi)\biggr]\\
& = \E\biggl[\int_{\mc{X}}\int_s^t\biggl|\int_{\bb{T}^d\times\R^m} \tilde{D}(\phi(\tau),y,\nu_{Q}(\tau))\pi(dy;\phi(\tau),\nu_Q(\tau)) (\bar{B}^\top)^{-1}(\phi(\tau),\nu_{Q}(\tau))\\
&\hspace{6cm}-\sqrt{\int_{\bb{T}^d\times\R^m} \tilde{D}(\phi(\tau),y,\nu_{Q}(\tau))  \pi(dy;\phi(\tau),\nu_Q(\tau))}\biggr|d\tau Q_{\mc{X}}(d\phi)\biggr]\\
&=\E\biggl[\int_{\mc{X}}\int_s^t\biggl|\bar{D}(\phi(\tau),\nu_{Q}(\tau)) (\bar{B}^\top)^{-1}(\phi(\tau),\nu_{Q}(\tau))-\sqrt{\bar{D}(\phi(\tau),\nu_{Q}(\tau))}\biggr|d\tau Q_{\mc{X}}(d\phi)\biggr]\\
& = \E\biggl[\int_{\mc{X}}\int_s^t\biggl|\bar{B}(\phi(\tau),\nu_{Q}(\tau))\bar{B}^\top(\phi(\tau),\nu_{Q}(\tau)) (\bar{B}^\top)^{-1}(\phi(\tau),\nu_{Q}(\tau))-\bar{B}(\phi(\tau),\nu_{Q}(\tau))\biggr|d\tau Q_{\mc{X}}(d\phi)\biggr]\\
& = \E\biggl[\int_{\mc{X}}\int_s^t\biggl|\bar{B}(\phi(\tau),\nu_{Q}(\tau))-\bar{B}(\phi(\tau),\nu_{Q}(\tau))\biggr|d\tau Q_{\mc{X}}(d\phi)\biggr]\\
&=0.
\end{align*}

Now turning to $D^N_5$, we have by Proposition \ref{prop:Phiexistenceregularity}, Assumption \ref{assumption:LipschitzandBounded}, and Bounded Convergence Theorem:
\begin{align*}
\lim_{N\toinf}\E[D^N_{5}]&\leq \E\biggl[\lim_{N\toinf} \int_{\mc{C}}\int_{\mc{C}}\int_s^t\int_{\bb{T}^d\times\R^m} \biggl|\int_{\bb{T}^d\times\R^m} \partial_\mu \Phi(\phi(\tau),y,\nu_{Q^N}(\tau))(\psi(\tau))f(\psi(\tau),\hat{y},\nu_{Q^N}(\tau)) r_\tau(d\hat{y}d\hat{z}) \biggr|\\
&\hspace{8cm}\rho_\tau(dydz)d\tau Q^N(d\psi dr dw) Q^N(d\phi d\rho dv)\biggr].
\end{align*}

Since $Q^N\tto Q$ in $\mc{P}(\mc{C})$ almost surely and the integrand is bounded and continuity, Proposition \ref{prop:nucontmeasure}, Proposition 4.6 on p.115 of \cite{EK}, and Theorem A.3.18 in \cite{DE} imply:
\begin{align*}
&\E\biggl[\lim_{N\toinf} \int_{\mc{C}}\int_{\mc{C}}\int_s^t\int_{\bb{T}^d\times\R^m} \biggl|\int_{\bb{T}^d\times\R^m} \partial_\mu \Phi(\phi(\tau),y,\nu_{Q^N}(\tau))(\psi(\tau))f(\psi(\tau),\hat{y},\nu_{Q^N}(\tau)) r_\tau(d\hat{y}d\hat{z}) \biggr|\rho_\tau(dydz) d\tau\\
&\hspace{12cm} Q^N(d\psi dr dw) Q^N(d\phi d\rho dv)\biggr]\\
& = \E\biggl[\int_{\mc{C}}\int_{\mc{C}}\int_s^t\int_{\bb{T}^d\times\R^m} \biggl|\int_{\bb{T}^d\times\R^m} \partial_\mu \Phi(\phi(\tau),y,\nu_{Q}(\tau))(\psi(\tau))f(\psi(\tau),\hat{y},\nu_{Q}(\tau)) r_\tau(d\hat{y}d\hat{z}) \biggr|\rho_\tau(dydz) d\tau \\
&\hspace{12cm}Q(d\psi dr dw) Q(d\phi d\rho dv )\biggr].
\end{align*}

Now by H\"older's inequality and Tonelli's Theorem,
\begin{align*}
&\E\biggl[\int_{\mc{C}}\int_{\mc{C}}\int_s^t\int_{\bb{T}^d\times\R^m} \biggl| \partial_\mu \Phi(\phi(\tau),y,\nu_{Q}(\tau))(\psi(\tau))\biggl\lbrace \int_{\bb{T}^d\times\R^m} f(\psi(\tau),\hat{y},\nu_{Q}(\tau)) r_\tau(d\hat{y}d\hat{z})\biggr\rbrace \biggr|\rho_\tau(dydz) d\tau\\
&\hspace{12cm} Q(d\psi dr dw) Q(d\phi d\rho dv )\biggr]\\
&\leq \E\biggl[\biggl(\int_s^t\biggl(\int_{\mc{C}}\biggl(\int_{\mc{C}}\int_{\bb{T}^d\times\R^m} |\partial_\mu \Phi(\phi(\tau),y,\nu_{Q}(\tau))(\psi(\tau))|^2\rho_\tau(dydz) Q(d\psi dr dw)\biggr)^{1/2}Q(d\phi d\rho dv )\biggr)^{2}d\tau\biggr)^{1/2}\\
&\hspace{4cm}\times\biggl(\int_s^t\int_{\mc{C}} \biggl|\int_{\bb{T}^d\times\R^m} f(\psi(\tau),\hat{y},\nu_{Q}(\tau)) r_\tau(d\hat{y}d\hat{z})\biggr|^2 Q(d\psi dr dw)d\tau\biggr)^{1/2} \biggr]\\
& = \E\biggl[\biggl(\int_s^t\biggl(\int_{\mc{C}}\biggl(\int_{\bb{T}^d\times\R^m}  \norm{\partial_\mu \Phi(\phi(\tau),y,\nu_{Q}(\tau))(\cdot)}^2_{L^2(\R^d,\nu_{Q}(\tau))}\rho_\tau(dydz)\nu_{Q}(\tau)(dx) \biggr)^{1/2}Q(d\phi d\rho dv )\biggr)^2d\tau\biggr)^{1/2}\\
&\hspace{4cm}\times\biggl(\int_s^t \int_{\mc{C}} \biggl|\int_{\bb{T}^d\times\R^m} f(\psi(\tau),\hat{y},\nu_{Q}(\tau)) r_\tau(d\hat{y}d\hat{z})\biggr|^2  Q(d\psi dr dw)d\tau \biggr)^{1/2}\biggr]\\
&\leq C|t-s|^{1/2}\E\biggl[\biggl(\int_s^t \int_{\mc{C}} \biggl|\int_{\bb{T}^d\times\R^m} f(\psi(\tau),\hat{y},\nu_{Q}(\tau)) r_\tau(d\hat{y}d\hat{z})\biggr|^2  Q(d\psi dr dw)d\tau \biggr)^{1/2}\biggr],
\end{align*}
where we used uniform boundedness of the $L^2(\R^d,\nu)$ norm of $\partial_\mu \Phi(x,y,\nu)(\cdot)$ from Proposition \ref{prop:Phiexistenceregularity} in the last step.

Using that $Q$ almost surely satisfies \ref{V:V4} via Proposition \ref{remark:V4explaination}:
\begin{align*}
&\E\biggl[\biggl(\int_s^t\int_{\mc{C}} \biggl|\int_{\bb{T}^d\times\R^m} f(\psi(\tau),\hat{y},\nu_{Q}(\tau)) r_\tau(d\hat{y}d\hat{z})\biggr|^2  Q(d\psi dr dw)d\tau \biggr)^{1/2}\biggr] \\
& = \E\biggl[\biggl(\int_s^t\int_{\mc{C}}\int_{\mc{X}}\int_{\mc{Y}} \biggl|\int_{\bb{T}^d\times\R^m} f(\psi(\tau),\hat{y},\nu_{Q}(\tau)) r_\tau(d\hat{y}d\hat{z})\biggr|^2  \lambda(dr|\psi)Q_{\mc{X}}(d\psi)d\tau \biggr)^{1/2}\biggr]\\
& = \E\biggl[\biggl(\int_s^t\int_{\mc{C}}\int_{\mc{X}}\biggl|\int_{\bb{T}^d\times\R^m} f(\psi(\tau),\hat{y},\nu_{Q}(\tau)) \pi(d\hat{y}|\psi(\tau),\nu_Q(\tau))\biggr|^2  Q_{\mc{X}}(d\psi)d\tau \biggr)^{1/2}\biggr]\\
& = 0 \text{ by Assumption \ref{assumption:centeringcondition}}.
\end{align*}

Thus (\ref{eq:D5vanishes}) holds, and the proof of Lemma \ref{lem:marprob2} is complete.
\end{proof}

We have then that for each $(s,t,\Psi,g)\in [0,1]\times[0,1]\times C_b(\mc{C}) \times C^\infty_c(\R^d\times\R^d)$ there is a set $Z_{(s,t,\Psi,g)}\in \tilde{\mc{F}}$ such that $\tilde\Prob(Z_{(s,t,\Psi,g)})=0$ and
\begin{align*}
\E^{Q_{\tilde\omega}}\biggl[\Psi(M_g^{Q_{\tilde\omega}}(t)-M_g^{Q_{\tilde\omega}}(s))\biggr]=0,\forall \tilde\omega\in\tilde\W\setminus Z_{(s,t,\Psi,g)}.
\end{align*}

Since there is a a countable collection of $g\in C^\infty_c(\R^d\times\R^d)$ which is dense in $C^\infty_c(\R^d\times\R^d)$, a countable collection $(s,t)\in [0,1]^2$ which is dense in $[0,1]^2$, and countably many $\Phi \in C_b(\mc{C})$ generating each of the countably many sigma algebras $\mc{G}_{s_l}$ (see \cite{Lacker} Lemma A.1), letting $Z$ be the union over all these countable collections of $Z_{(s,t,\Psi,g)}$, we have $Z\in\tilde\F$, $\tilde\Prob(Z)=0$, and
\begin{align*}
\E^{Q_{\tilde\omega}}\biggl[\Psi(M_g^{Q_{\tilde\omega}}(t)-M_g^{Q_{\tilde\omega}}(s))\biggr]=0,\forall \tilde\omega\in\tilde\W\setminus Z.
\end{align*}
So Theorem \ref{thm:V1} is proved.

\subsubsection{Proof of \ref{V:V2}}
By Skorokhod's representation theorem, we can invoke another probability space on which the convergence of $Q^N\tto Q$ occurs with probability 1. Without making a distinction in the notation between that probability space and our original one, we note that by Fatou's lemma
\begin{align*}
\E\biggl[\E^{Q}\biggl[\int_{\bb{T}^d\times\R^m\times [0,1]}|z|^2 \rho(dydzdt) \biggr]\biggr]&\leq \liminf_{N\toinf} \E\biggl[\int_{\mc{C}}\biggl\lbrace\int_{\bb{T}^d\times\R^m\times [0,1]}|z|^2 r(dydzdt)  \biggr\rbrace Q^N(d\phi dr dw)\biggr]\\
& = \liminf_{N\toinf} \E\biggl[\frac{1}{N}\sum_{i=1}^N\int_0^1 |u^N_i(s)|^2ds\biggr]\\
&<\infty \text{ by Assumption (\ref{eq:controlL2boundunspecific})}\red{.}
\end{align*}

\subsubsection{Proof of \ref{V:V3}}
This follows immediately from weak convergence, since by Proposition \ref{prop:nucontmeasure}, for $g\in C_b(\R^d)$, $\Theta\mapsto \int_{\R^d} g(x)\nu_\Theta(0)(dx)$ is a continuous bounded map from $\mc{P}(\mc{C})$ to $\R$. Thus, again invoking Skorokhod's representation theorem:
\begin{align*}
\int_{\R^d}g(x)\nu_0(dx)& = \lim_{N\toinf} \int_{\R^d} g(x) \biggl(\frac{1}{N}\sum_{i=1}^N \delta_{x^{i,N}}\biggr)(dx)\\
& = \lim_{N\toinf} \int_{\R^d} g(x) \nu_{Q^N}(0)(dx)\\
& = \int_{\R^d} g(x)\nu_Q(0)(dx) \\
& = \int_{\R^d} g(x)\nu_{Q}(0)(dx),
\end{align*}
almost surely for each $g$. By a density argument we can ensure there is a null set on which the equality fails that is independent of the choice of $g$. Thus we get that $Q$ $\tilde\Prob$-.a.s. satisfies \ref{V:V3}.

\section{The Laplace Principle Lower Bound}
\label{sec:lowerbound}
We now proceed with proving the Laplace Principle Lower Bound:
\begin{align}
\label{eq:lowerbound}
\liminf_{N\toinf} -\frac{1}{N}\log\E[\exp(-N F(\mu^N))] &\geq \inf_{\theta\in\mc{P}(\mc{X})}\br{F(\theta)+I(\theta)}.
\end{align}

It suffices to prove this bound along any subsequence such that the left hand side converges. Such a sequence exists since  $-\frac{1}{N}\log\E[\exp(-NF(\mu^N))]\leq \norm{F}_\infty$. Fix $\eta>0$. By Proposition \ref{prop:varrep}, for each $N\in \bb{N}$, there exists $v_N\in \mc{U}_N$ such that
\begin{align*}
-\frac{1}{N}\log\E[\exp(-NF(\mu^N))]\geq \frac{1}{2}\E[\frac{1}{N}\sum_{i=1}^{N}\int_0^1 |v_i^{N}(t)|^2dt]+\E[F(\bar{\mu}^N)]-\eta.
\end{align*}

Note also that for this choice of controls, we have for all $N\in\bb{N}$,
\begin{align}
\label{eq:controlL2bound}
\E[\frac{1}{N}\sum_{i=1}^{N}\int_0^1|v_i^{N}(t)|^2dt]\leq 4\norm{F}_\infty + 2\eta.
\end{align}

Thus the bound \eqref{eq:controlL2boundunspecific} is satisfied, so the results of Section \ref{sec:limitingbehavior} apply with $\br{v^N}_{N\in\bb{N}}$ as our choice of controls, and for $\br{Q^N}_{N\in\bb{N}}$ as in Equation \eqref{eq:occmeas} with $\mc{Y}$-marginal determined by $\br{v^N}_{N\in\bb{N}}$, $Q^N\tto Q$ as $\mc{P}(\mc{C})$-valued random variables such that $Q\in\mc{V}$ almost-surely. So
\begin{align*}
    \liminf_{N\toinf}-\frac{1}{N}\log\E[\exp(-NF(\mu^N))] &\geq \liminf_{N\toinf}\left[\frac{1}{2}\E[\frac{1}{N}\sum_{i=1}^N |v^N_i(t)|^2dt]+\E[F(\bar{\mu}^N]\right]-\eta\\
    & = \liminf_{N\toinf}\left[\E[\frac{1}{2}\int_{\mc{Y}}\int_{\bb{T}^d\times\R^m\times[0,1]} |z|^2 r(dydzdt)Q^N_{\mc{Y}}(dr)] +\E[F(Q^N_{\mc{X}})]\right]-\eta\\
    &\geq \frac{1}{2}\tilde{\E}\biggl[\int_{\mc{Y}}\int_{\bb{T}^d\times\R^m\times[0,1]} |z|^2 r(dydzdt)Q_{\mc{Y}}(dr) +F(Q_{\mc{X}})\biggr]-\eta\\
    & \text{ by Fatou's Lemma}\\
    &\geq \inf_{\theta\in \mc{P(\mc{X})}}\left\{\inf_{\Theta\in\mc{V}:\Theta_{\mc{X}}=\theta}\E^{\Theta}\biggl[\frac{1}{2}\int_{\bb{T}^d\times\R^m\times[0,1]} |z|^2 \rho(dydzdt)\biggr] +F(\theta)\right\}-\eta\\
    & = \inf_{\theta\in \mc{P(\mc{X}^*)}}\br{ I(\theta)+F(\theta)}-\eta.
\end{align*}
Since $\eta$ is arbitrary the lower bound (\ref{eq:lowerbound}) is proved.

\section{Compactness of Level Sets}\label{sec:compactlevelsets}

Consider $I$ as defined in Equation \eqref{eq:ratefunction}. We want to prove that, assuming \ref{assumption:uniformellipticity}-\ref{assumption:limitinguniformellipticity}, for each $s\in[0,\infty)$, the set
\begin{align}
I_s\coloneqq \br{\theta\in \mc{P}(\mc{X}):I(\theta)\leq s}\label{eq:LevelSet}
\end{align}
is a compact subset of $\mc{P}(\mc{X})$. This will imply that indeed $I$ is a good rate function.

Since in this section we are dealing with sequences of measures all of which coincide with weak solutions of Equation \eqref{eq:controlledMcKeanLimit}, but with possibly different controls, we introduce a new notation for the coordinate process which allows us to keep track of which measure the $\mc{X}$-component of the coordinate process corresponds to. For this we use the parameterized version of the limiting Equation \eqref{eq:paramMcKeanLimit}.

For $Q$ corresponding to a weak solution of Equation \eqref{eq:controlledMcKeanLimit}, $Q$ also corresponds to a solution of Equation \eqref{eq:paramMcKeanLimit} with $\nu_{Q}$ as defined in Equation \eqref{eq:nuQ} in the place of $\nu$. Thus we consider the process triple $(\tilde{X}^{\nu_{Q}},\rho,W)$, which can be given explicitly as the coordinate process on the probability space $(\mc{C},\mc{B}(\mc{C}),Q)$ endowed with the canonical filtration $\mc{G}_t\coloneqq \sigma\biggl((\tilde{X}^{\nu_{Q}}_s,\rho(s),W_s),0\leq s\leq t \biggr)$. Thus, for $\omega = (\phi,r,w)\in \mc{C}$,
\begin{align}
\label{eq:paramcanonicalprocess}
    \tilde{X}^{\nu_{Q}}_t(\omega) = \phi(t),\hspace{1.5cm} \rho(t,\omega) = r|_{\mc{B}(\R^m\times [0,t])},\hspace{1.5cm}W_t(\omega)=w(t).
\end{align}

\begin{lem}
\label{lem:precompactlevelsets}
Fix $K<\infty$ and consider a sequence $\br{Q^N}_{N\in\bb{N}}\subset \mc{P}(\mc{C})$ such that for every $N\in\bb{N}$, $Q^N$ is in $\mc{V}$ from Definition \ref{def:V} and
\begin{align*}
\E^{Q^N}\biggl[\int_{\bb{T}^d\times\R^m\times [0,1]}|z|^2 \rho(dydzdt) \biggr]<K .
\end{align*}
Then $\br{Q^N}_{N\in\bb{N}}$ is tight.
\end{lem}
\begin{proof}
As in Subsection \ref{subsection:tightness}, it suffices to show tightness of each of the marginals. It is worth noting that where before we were proving tightness of $\mc{L}(Q^N)$ in $\mc{P}(\mc{P}(\mc{C}))$, here we have that $Q^N$ are deterministic measures and we are proving tightness of the measures themselves in $\mc{P}(Q^N)$.

Tightness of the $\mc{W}$-marginals follows immediately since all are the standard Wiener measure by definition.

Tightness of the $\mc{Y}$-marginals is very similar to Subsection \ref{subsubsection:TightnessQ^N_Z}.
\begin{align*}
g(r):= \int_{\bb{T}^d\times\R^{m}\times [0,1]} |z|^2r(dydz dt)
\end{align*}
is a tightness function on $\mc{Y}$, so since
\begin{align*}
\E^{Q^N}\biggl[\int_{\R^m\times [0,1]}|z|^2 \rho(dzdt)\biggr]<\infty,
\end{align*}
$\br{Q^N_{\mc{Y}}}_{N\in\bb{N}}$ is tight.

For the tightness of the $\mc{X}$-marginals, we use that each $Q^N$ satisfies \ref{V:V1}; that is, $Q^N_{\mc{X}}=\mc{L}(\tilde{X}^{\nu_{Q^N}})$. Via Theorem 2.4.10 in \cite{KS}, it suffices to show that for every $\eta>0$,
\begin{align*}
\lim_{\rho\downarrow 0}\sup_{N\in\bb{N}}Q^N_{\mc{X}}\biggl(\sup_{|t_1-t_2|<\rho,0\leq t_1<t_2\leq 1} |\tilde{X}^{\nu_{Q^N}}_{t_1}-\tilde{X}^{\nu_{Q^N}}_{t_2}|\geq \eta\biggr)=0,
\end{align*}
where here we are using the notation from Equation \eqref{eq:paramcanonicalprocess}. We have that by Chebyshev's inequality,
\begin{align*}
&\lim_{\rho\downarrow 0}\sup_{N\in\bb{N}}Q^N_{\mc{X}}\biggl(\sup_{|t_1-t_2|<\rho,0\leq t_1<t_2\leq 1} |\tilde{X}^{\nu_{Q^N}}_{t_1}-\tilde{X}^{\nu_{Q^N}}_{t_2}|\geq \eta\biggr)\\
&\leq \lim_{\rho\downarrow 0}\frac{1}{\eta}\sup_{N\in\bb{N}}\E^{Q^N}\biggl[\sup_{|t_1-t_2|<\rho,0\leq t_1<t_2\leq 1} |\tilde{X}^{\nu_{Q^N}}_{t_1}-\tilde{X}^{\nu_{Q^N}}_{t_2}| \biggr].
\end{align*}
Since
\begin{align*}
&|\tilde{X}^{\nu_{Q^N}}_{t_1}-\tilde{X}^{\nu_{Q^N}}_{t_2}|\\
&=\biggl|\int_{t_1}^{t_2} \bar{\beta}(\tilde{X}_t^{\nu_{Q^N}},\nu_{Q^N}(t))dt+\int_{t_1}^{t_2}\int_{\bb{T}^d\times\R^m}[\nabla_y \Phi(\tilde{X}_t^{\nu_{Q^N}},y,\nu_{Q^N}(t))+I]\sigma(\tilde{X}_t^{\nu_{Q^N}},y,\nu_{Q^N}(t)) z \rho_t(dydz) dt\\
    &+\int_{t_1}^{t_2}\bar{B}(\tilde{X}_t^{\nu},\nu_{Q^N}(t))dW_t\biggr|,\nonumber
\end{align*}
we get via H\"older's inequality, It\^o isometry, Assumption \ref{assumption:LipschitzandBounded}, Proposition \ref{prop:Phiexistenceregularity} and Corollary \ref{cor:limitingcoefficientsregularity} that
\begin{align*}
|\tilde{X}^{\nu_{Q^N}}_{t_1}-\tilde{X}^{\nu_{Q^N}}_{t_2}|&\leq C\biggl((t_2-t_1) + \sqrt{t_2-t_1}\biggl(\sqrt{ \int_0^1\int_{\bb{T}^d\times\R^m} |z|^2 \rho_t(dydz)dt} + 1\biggr)\biggr).
\end{align*}
Then we have by  Young's inequality that
\begin{align*}
\sup_{|t_1-t_2|<\rho,0\leq t_1<t_2\leq 1} |\tilde{X}^{\nu_{Q^N}}_{t_1}-\tilde{X}^{\nu_{Q^N}}_{t_2}| &\leq C\biggl(\frac{5}{2} + \frac{1}{2}\int_0^1\int_{\bb{T}^d\times\R^m} |z|^2\rho_t(dydz)dt\biggr).
\end{align*}
Then
\begin{align*}
&\sup_{N\in\bb{N}}\E^{Q^N}\biggl[C\biggl(\frac{5}{2} + \frac{1}{2}\int_0^1\int_{\R^m} |z|^2\rho_t(dz)dt\biggr)\biggr]\leq C\biggl(\frac{5}{2} + \frac{1}{2}K\biggr) \text{ by assumption.}
\end{align*}
So by dominated convergence theorem, we have
\begin{align*}
&\lim_{\rho\downarrow 0}\sup_{N\in\bb{N}}Q^N_{\mc{X}}\biggl(\sup_{|t_1-t_2|<\rho,0\leq t_1<t_2\leq 1} |\tilde{X}^{\nu_{Q^N}}_{t_1}-\tilde{X}^{\nu_{Q^N}}_{t_2}|\geq \eta\biggr)\\
&\leq \lim_{\rho\downarrow 0}\frac{1}{\eta}\sup_{N\in\bb{N}}\E^{Q^N}\biggl[\sup_{|t_1-t_2|<\rho,0\leq t_1<t_2\leq 1} |\tilde{X}^{\nu_{Q^N}}_{t_1}-\tilde{X}^{\nu_{Q^N}}_{t_2}| \biggr]\\
& \leq \frac{C}{\eta}\sup_{N\in\bb{N}} \E^{Q^N}\biggl[\lim_{\rho\downarrow 0}\sup_{|t_1-t_2|<\rho,0\leq t_1<t_2\leq 1} (t_2-t_1) + \sqrt{t_2-t_1}\biggl(\sqrt{ \int_0^1\int_{\R^m} |z|^2 \rho_t(dz)dt} + 1\biggr)\biggr]\\
& = 0. %\text{ by continuity of }t\mapsto \tilde{X}_t^{\nu_{Q^N}}(\omega) \text{for a.e. }\omega\in \mc{C}\text{ and continuity of }t\mapsto\nu_{Q^N}(t).
\end{align*}
\end{proof}
\begin{lem}
\label{lem:limviablelevelsets}
Fix $K<\infty$ and consider a convergent sequence $\br{Q^N}_{N\in\bb{N}}\subset \mc{P}(\mc{C})$ such that for every $N\in\bb{N}$, $Q^N$ is in $\mc{V}$ from Definition \ref{def:V} and
\begin{align*}
\E^{Q^N}\biggl[\int_{\bb{T}^d\times\R^m\times [0,1]}|z|^2 \rho(dydzdt) \biggr]<K .
\end{align*}
Then for $Q$ such that $Q^N\tto Q$, $Q$ is in $\mc{V}$.
\end{lem}
\begin{proof}
The fact that $Q$ satisfies \ref{V:V2} follows immediately from Fatou's lemma. Since by Proposition \ref{prop:nucontmeasure} $\nu_0=\lim_{N\toinf}\nu_{Q^N}(0) = \nu_{Q}(0)$, \ref{V:V3} is satisfied.

We now prove $Q$ satisfies \ref{V:V1}. As before, our tool here is the martingale problem associated to Equation \eqref{eq:paramMcKeanLimit}. It suffices to show that for fixed $h\in C^\infty_c(\R^d\times\R^d),0\leq s\leq t\leq 1,$ and $\mc{G}_s$-measurable $\Psi\in C_b(\mc{C})$ that
\begin{align}\label{eq:martingaleproblemcompactnesslevelsets}
\E^Q\biggl[\Psi(M_h^Q(t)-M_h^Q(s))\biggr]=0
\end{align}
where $M_h^Q$ is given in Equation \eqref{eq:M}. Note that since we know $Q^N$ satisfies \ref{V:V4} for all $N$, we can in fact simplify the form of the process $M_h^Q$ to:
\begin{align}
\label{eq:M2}
    &M^Q_h(t,(\phi,r,w)) = h(\phi(t),w(t))-h(\phi(0),w(0))-\int_0^t\int_{\bb{T}^d\times\R^m} \mc{A}_2[h](\phi(s),y,z,\nu_Q(s),w(s))r_s(dydz)ds
\end{align}
where
\begin{align} \label{eq:decomposeA2}
    \mc{A}_2[g](x,y,z,\nu_Q(s),p) &\coloneqq\biggl[\bar{\beta}(x,\nu_Q(s))+ [\nabla_y\Phi(x,y,\nu_Q(s))+I]\sigma(x,y,\nu_Q(s))z\biggr]\cdot \nabla_x g(x,p)\\
    &\quad+ \frac{1}{2}\bar{D}(x,\nu_Q(s)):\nabla_x\nabla_x g(x,p)+\frac{1}{2}I:\nabla_p\nabla_pg(x,p)\nonumber\\
    &\quad+\bar{B}(x,\nu_Q(s)):\nabla_p\nabla_x g(x,p)\nonumber.
\end{align}
It suffices to show that
\begin{align*}
\E^{Q^N}\biggl[\Psi(M_h^{Q^N}(t)-M_h^{Q^N}(s))\biggr]\tto \E^{Q}\biggl[\Psi(M_h^{Q}(t)-M_h^{Q}(s))\biggr]
\end{align*}
since by \ref{V:V1}
\begin{align*}
\E^{Q^N}\biggl[\Psi(M_h^{Q^N}(t)-M_h^{Q^N}(s))\biggr]=0.
\end{align*}

Unlike in the previous proof of \ref{V:V1}, here the convergence is as a sequence of real numbers and not in distribution, since $Q^N$ are deterministic.

So that we can keep track of which measure $\rho$ and Brownian Motion $W$ correspond to in the Coordinate Process \eqref{eq:paramcanonicalprocess} on $\mc{C}$ under $Q^N$, we relabel it  $( \tilde{X}^{\nu_{Q^N}},\rho^N,W^N)$. Under $Q$, we keep the notation $(\tilde{X}^{\nu_{Q}},\rho,W)$.  Invoking Skorokhod's representation theorem to find another probability space on which the convergence of the random variables $ ( \tilde{X}^{\nu_{Q^N}},\rho^N,W^N)\tto ( \tilde{X}^{\nu_{Q}},\rho,W)$ occurs for almost every $\omega\in \W$, we have
\begin{align*}
\E\biggl[\Psi\left|M_h^{Q^N}(t)-M_h^{Q}(t)+M_h^{Q}(s) - M_h^{Q^N}(s)\right|\biggr]&\leq C(\norm{\Psi}_\infty)\biggl(\E\biggl[\biggl| M_h^{Q^N}(t)-M_h^{Q}(t)\biggr|\biggr]\\
&+\E\biggl[\biggl| M_h^{Q^N}(s)-M_h^{Q}(s)\biggr|\biggr]\biggr)
\end{align*}
and
\begin{align*}
&\E\biggl[\biggl| M_h^{Q^N}(t)-M_h^{Q}(t)\biggr|\biggr] \\
& = \E\biggl[\biggl|h(\tilde{X}^{\nu_{Q^N}}_t,W^N_t))-h(\tilde{X}^{\nu_{Q}}_t,W_t)+\int_0^t\int_{\bb{T}^d\times\R^m} \mc{A}_2[h](\tilde{X}^{\nu_{Q}}_s,y,z,\nu_Q(s),W_s)\rho_s(dydz)ds\\
&\hspace{6cm} -\int_0^t\int_{\bb{T}^d\times\R^m} \mc{A}_2[h](\tilde{X}^{\nu_{Q^N}}_s,y,z,\nu_{Q^N}(s),W^N_s)\rho^N_s(dydz)ds\biggr|\biggr].
\end{align*}
By continuity and boundedness of $g$ and convergence of $(\tilde{X}^{\nu_{Q^N}},W^N)\tto (\tilde{X}^{\nu_{Q}},W)$ along with bounded convergence theorem,
\begin{align*}
\E\biggl[\biggl|g(\tilde{X}^{\nu_{Q^N}}_t,W^N_t))-g(\tilde{X}^{\nu_{Q}}_t,W_t)\biggr|\biggr]\tto 0 \text{ as }N\toinf.
\end{align*}
By Assumption \ref{assumption:LipschitzandBounded} and Corollary \ref{cor:limitingcoefficientsregularity} along with dominated convergence theorem (using $L^2$ boundedness of the controls):
\begin{align*}
&\lim_{N\toinf}\E\biggl[\biggl|\int_0^t\int_{\bb{T}^d\times\R^m} \mc{A}_2[h](\tilde{X}^{\nu_{Q}}_s,y,z,\nu_Q(s),W_s)\rho_s(dydz)ds \\
&\hspace{6cm}-\int_0^t\int_{\bb{T}^d\times\R^m} \mc{A}_2[h](\tilde{X}^{\nu_{Q^N}}_s,y,z,\nu_{Q^N}(s),W^N_s)\rho^N_s(dydz)ds\biggr|\biggr]\\
& \leq \lim_{N\toinf}\E\biggl[\int_0^t\biggl|\int_{\bb{T}^d\times\R^m} \mc{A}_2[h](\tilde{X}^{\nu_{Q}}_s,y,z,\nu_Q(s),W_s)\rho_s(dydz)ds\\
&\hspace{6cm} -\int_{\bb{T}^d\times\R^m} \mc{A}_2[h](\tilde{X}^{\nu_{Q^N}}_s,y,z,\nu_{Q^N}(s),W^N_s)\rho^N_s(dydz)\biggr|ds\biggr]\\
& = \E\biggl[\int_0^t\lim_{N\toinf}\biggl|\int_{\bb{T}^d\times\R^m} \mc{A}_2[h](\tilde{X}^{\nu_{Q}}_s,y,z,\nu_Q(s),W_s)\rho_s(dydz)ds\\
&\hspace{6cm} -\int_{\bb{T}^d\times\R^m} \mc{A}_2[h](\tilde{X}^{\nu_{Q^N}}_s,y,z,\nu_{Q^N}(s),W^N_s)\rho^N_s(dydz)\biggr|ds\biggr].
\end{align*}

By continuity of the coefficients in $x$ and $\mu$ from Assumption \ref{assumption:LipschitzandBounded} and Corollary \ref{cor:limitingcoefficientsregularity}, along with the assumed uniform $L^2$ bound on the control and with the fact that the growth in the control is linear, if we can show that $\nu_{Q^N}(t)\tto \nu_{Q}(t)$ in $\mc{P}(\R^d)$ for each $t\in [0,1]$ then this term will vanish by essentially the same argument given in the proof of Lemma \ref{lem:marprob1}. But this follows immediately by the assumption that $Q^N\tto Q$ almost surely and Proposition \ref{prop:nucontmeasure}, and by Chebyshev's inequality and the same density argument as at the end of Subsection \ref{subsubsection:V1} we have that $Q$ satisfies \ref{V:V1}.

Finally we prove that $Q$ satisfies \ref{V:V4}. Again invoking Skorokhod's representation theorem to find another probability space on which the convergence of the random variables $ ( \tilde{X}^{\nu_{Q^N}},\rho^N,W^N)\tto ( \tilde{X}^{\nu_{Q}},\rho,W)$ occurs for almost every $\omega\in \W$, taking any $g\in C^2_b(\R^d)$ and $t\in[0,1]$:

\begin{align*}
&\E\biggl[\biggl|\int_{\bb{T}^d\times\R^m\times [0,t]} \mc{L}^1_{\tilde{X}^{\nu_{Q}}_s,\nu_Q(s)}g(y)\rho(dydzds)\biggr|\biggr]\\
&\leq \liminf_{N\toinf}\E\biggl[\biggl|\int_{\bb{T}^d\times\R^m\times [0,t]} \mc{L}^1_{\tilde{X}^{\nu_{Q^N}}_s,\nu_{Q^N}(s)}g(y)\rho^N(dydzds) \biggr|\biggr] \\
&=\liminf_{N\toinf}\E\biggl[\biggl|\int_0^t\int_{\bb{T}^d} \mc{L}^1_{\tilde{X}^{\nu_{Q^N}}_s,\nu_{Q^N}(s)}g(y) \pi(dy;\tilde{X}^{\nu_{Q^N}}_s,\nu_{Q^N}(s))\biggr|\biggr]\text{ by }\ref{V:V4},\\
&=0 \text{ since $\pi$ satisfies \eqref{eq:pi}},
\end{align*}
where to get to the second line we used Fatou's Lemma, continuity of $\mc{L}^1_{x,\mu}f(y)$ via Assumption \ref{assumption:LipschitzandBounded}, Proposition \ref{prop:nucontmeasure},  and Theorem A.3.18 in \cite{DE}. So the result of Lemma \ref{lemma:willimplyV4} holds for $Q$, and hence via the proof of Proposition \ref{remark:V4explaination}, $Q$ satisfies \ref{V:V4}.
\end{proof}

Lemma \ref{lem:precompactlevelsets} establishes precompactness of $I_s$ defined in (\ref{eq:LevelSet}). Now we will use both Lemmas \ref{lem:precompactlevelsets} and \ref{lem:limviablelevelsets} to prove the level sets $I_s$ are closed via showing lower-semicontinuity of $I$.
\begin{lem} \label{lem:lowersemicont}
The functional $I$ given in Equation \eqref{eq:ratefunction} is lower semi-continuous.
\end{lem}
\begin{proof}
Consider a sequence $\br{\theta^N}\subset \mc{P}(\mc{X})$ with limit $\theta$. We wish to show
\begin{align*}
\liminf_{N\toinf}I(\theta^N)\geq I(\theta).
\end{align*}
It suffices to consider the case there the left hand side is finite, so there is $M\in [0,\infty)$ such that $\liminf_{N\toinf}I(\theta^N)\leq M$. Then, recalling that
\begin{align*}
I(\theta^N) = \inf_{\Theta^N\in\mc{V}:\Theta^N_{\mc{X}}=\theta^N} \E^{\Theta^N}\biggl[\frac{1}{2}\int_{\bb{T}^d\times\R^m\times [0,1]}|z|^2\rho(dydzdt)\biggr],
\end{align*}
by taking a subsequence of $\br{\theta^N}$ if necessary, we can find measures $\Theta^N$ such that $\Theta^N_{\mc{X}}=\theta^N$,
 \begin{align}\label{eq:levelsetsbound}
 \sup_{N\in\bb{N}}\E^{\Theta^N}\biggl[\frac{1}{2}\int_{\bb{T}^d\times\R^m\times [0,1]}|z|^2\rho(dydzdt)\biggr]<M+1,
 \end{align}
 and
 \begin{align*}
 I(\theta^N) \geq \E^{\Theta^N}\biggl[\frac{1}{2}\int_{\bb{T}^d\times\R^m\times [0,1]}|z|^2\rho(dydzdt)\biggr] - \frac{1}{N}.
 \end{align*}
Then by Lemma \ref{lem:precompactlevelsets} we can consider a subsequence along which $\br{\Theta^N}$ converges to some $\Theta$. By Lemma \ref{lem:limviablelevelsets} $\Theta$ is viable. Hence by Fatou's lemma,
\begin{align*}
\liminf_{N\toinf} I(\theta^N) &\geq \liminf_{N\toinf}\E^{\Theta^N}\biggl[\frac{1}{2}\int_{\bb{T}^d\times\R^m\times [0,1]}|z|^2\rho(dydzdt)\biggr] - \frac{1}{N}\\
&\geq \E^{\Theta}\biggl[\frac{1}{2}\int_{\bb{T}^d\times\R^m\times [0,1]}|z|^2\rho(dydzdt)\biggr]\\
&\geq \inf_{\Theta\in\mc{V}:\Theta_{\mc{X}}=\theta}\E^{\Theta}\biggl[\frac{1}{2}\int_{\bb{T}^d\times\R^m\times [0,1]}|z|^2\rho(dydzdt)\biggr]\\
& = I(\theta),
\end{align*}
so lower semi-continuity of $I$ is proved.
\end{proof}
\section{The Laplace Principle Upper Bound}\label{sec:upperbound}
In order to close the proof of Theorem \ref{thm:LaplacePrinciple}, we now need to show for any $F\in C_b(\mc{P}(\mc{X}))$ that the Laplace Principle Upper Bound \eqref{eq:laplaceprincipleupperbound} holds. Thanks to Proposition \ref{prop:I=Iav}, we know that $I=I^{av}$, so we can equivalently show:
\begin{align}
\label{eq:upperbound}
    \limsup_{N\toinf}-\frac{1}{N}\log\E[\exp(-NF(\mu^N))]\leq \inf_{\theta\in \mc{P}(\mc{X})}\br{I^{av}(\theta)+F(\theta)}.
\end{align}

We will prove this bound via two different methods, under the additional assumptions \ref{assumption:samebrownianmotion}/\ref{assumption:weaksenseuniqueness} and \ref{assumption:xindependenceassumptionforupperbound} respectively.

In the setting of the additional assumptions \ref{assumption:samebrownianmotion}/\ref{assumption:weaksenseuniqueness}, as we will see, the result follows along the same lines of Section 6 of \cite{BDF} (see also Sections 3.2 and 4.2 of \cite{BCsmallnoise} and Section 3.4 in \cite{BCcurrents}). We are able to do so thanks to the fact that in this regime, $\bar{W}^{i,N}$ in Equation \eqref{eq:barW} being used to construct the occupation measures \eqref{eq:occmeas} are simply the original driving Brownian Motions for the particles \eqref{eq:multimeanfield}- that is $\bar{W}^{i,N}=W^i$ for all $i,N\in\bb{N}$.

In the setting of the additional assumption \ref{assumption:xindependenceassumptionforupperbound}, we do not have such a luxury, and thus tracking the joint distribution of the driving controls and Brownian motions in Equation \eqref{eq:controlledprelimit} is more complicated. We thus make an approximation argument via semi-Markovian controls which are continuous, bounded functions of $\bar{W}^{i,N}$, owing largely to ideas found in the proof of Theorem 2.4 in \cite{Lacker} and results from \cite{CDL}.

In both subsections, we make the standing assumptions \ref{assumption:initialconditions}-\ref{assumption:limitinguniformellipticity}.

\subsection{Proof under the additional assumptions \ref{assumption:samebrownianmotion} and \ref{assumption:weaksenseuniqueness}}\label{subsection:sameBMupperboundproof}\label{subsec:upperboundundersameBM}
Under this additional assumption, $\bar{B}(x,\mu)=\sigma(x,\mu)$ and $\Phi\equiv 0$, so Equation \eqref{eq:BDFlimitav} from the definition of $\mc{V}^{av}$ in Definition \ref{defi:Vav} is given by Equation \eqref{eq:limitingcontrolledmckeanvlasovsimp}.
\begin{remark}\label{eq:onhavingacoupledslowfastsystem}
Note that, due to the fact that we consider the case where the fast motion of Equation \eqref{eq:multimeanfield} is the same as the slow motion but on a time scale $\frac{1}{\epsilon}$ faster, we have in fact that in this simplified regime that $\pi(dy;x,\mu)=dy$; that is, the invariant measure associated to the fast dynamics is Lebesgue measure on the torus. However, this has no bearing on the proof of the equivalence of $I$ and $I^{av}$ in Proposition \ref{prop:I=Iav}, and hence we could just as well consider the empirical measure of Equation \eqref{eq:multimeanfield} with $X^{i,N}_t/\epsilon$ replaced by some $Y^{i,N}$ satisfying
\begin{align*}
dY^{i,N}_t = \frac{1}{\epsilon}\biggl[\frac{1}{\epsilon}g(X^{i,N}_t,Y^{i,N}_t,\mu^N_t)+c(X^{i,N}_t,Y^{i,N}_t,\mu^N_t)\biggr]dt + \frac{1}{\epsilon}\tau(X^{i,N}_t,Y^{i,N}_t,\mu^N_t)dW^i_t
\end{align*}
and under sufficient coercive conditions on the fast process the LDP would hold (though with a different form of the limiting coefficients $\bar{\beta}$ and $\bar{D}$ from Equation \eqref{eq:limitingcoefficients}) -see \cites{MS,Rockner}. We thus present forthcoming proof without using the independence of $\pi$ from $x$ and $\mu$. %\orange{Maybe move this elsewhere?}
\end{remark}
Since, as previously stated, our method of proof under these assumptions mimics that of \cites{BDF,BCsmallnoise}, we will also be making heavy use of their notion of weak sense uniqueness, which one should now recall from Definition \ref{def:weaksenseuniquness}. This is the very property that the additional Assumption \ref{assumption:weaksenseuniqueness}i) is supposing.

We are now ready to prove the bound \eqref{eq:laplaceprincipleupperbound} (equivalently the bound \eqref{eq:upperbound}) in this regime.

\begin{proof}
Let $F\in C_b(\mc{P}(\mc{X}))$ and $\eta>0$. Take $\theta \in \mc{P}(\mc{X})$ such that
\begin{align*}
I^{av}(\theta)+F(\theta) \leq \inf_{\theta \in \mc{P}(\mc{X})} \br{ I^{av}(\theta)+F(\theta)} +\frac{\eta}{2}.
\end{align*}

Since the bound given in Equation \eqref{eq:upperbound} is trivial if the right hand side is $+\infty$, we may assume it is finite. %The ensuing proof is analogous to the one in Section 6 of \cite{BDF}.

By the definition of $I^{av}$, there exists $\Theta\in\mc{V}^{av}$ such that $\Theta_{\mc{X}}=\theta.$ By merit of \ref{V:V1BDFav}, we get that letting $(\hat{X},\hat{\rho},\hat{W})$ be the canonical process on $\mc{X}\times\mc{Z}\times\mc{W}$ as defined analogously to Equation \eqref{eq:canonicalprocess}, $$((\mc{X}\times\mc{Z}\times\mc{W},\mc{B}(\mc{X}\times\mc{Z}\times\mc{W}),\Theta),\br{\hat{\mc{G}}^\Theta_{t+}},(\hat{X},\hat{\rho},\hat{W}))$$ is a weak solution of (\ref{eq:limitingcontrolledmckeanvlasovsimp}). Note that we take the $\Theta$-augmentation and right limit of $\hat{\mc{G}}_t\coloneqq \sigma((\hat{X}_s,\hat{\rho}_s,\hat{W}_s),0\leq s\leq t)$ so that we have a filtration that satisfies the usual conditions. By \ref{V:V1BDFav} we know that the martingale problem \eqref{eq:martingaleproblemcompactnesslevelsets} (but with $r\in\mc{Y}$ replaced by $\hat{r}\in\mc{Z}$ and setting $f=\Phi\equiv 0$, $\sigma(x,y,\mu)=\sigma(x,\mu)=\bar{B}(x,\mu)$) is satisfied by the coordinate process, and by Exercise 5.4.13 in \cite{KS} it is also satisfied with $\hat{\mc{G}}^\Theta_{t+}$ in the place of $\hat{\mc{G}}_t$. See also Remark 4.2 in \cite{BDF} for a further discussion of this.

As per equations \eqref{eq:BDFratefunctionavstandard} and \eqref{eq:BDFlimitavstandard} and the discussion at the end of Subsection \ref{subsec:altvariationalform}, we can further always assume that $\hat{\rho}_t(dz)= \delta_{u(t)}(dz)$ for a square-integrable $\R^d$-valued process $u(t)$.

We thus can find $\tilde{\Theta}\in\mc{V}^{av}$ such that $\tilde{\Theta}_{\mc{X}}=\theta$, and $$((\mc{X}\times\mc{Z}\times\mc{W},\mc{B}(\mc{X}\times\mc{Z}\times\mc{W}),\tilde{\Theta}),\br{\hat{\mc{G}}^{\tilde{\Theta}}_{t+}},(\hat{X},\tilde{\rho},\hat{W}))$$ is a weak solution to (\ref{eq:limitingcontrolledmckeanvlasovsimp}) with $\tilde{\rho}_t(\omega)(D) = \delta_{\tilde{u}(t,\omega)}(D)$ for $D\in\mc{B}(\R^d)$, and $\tilde{u}$ an $\R^d$-valued process such that \begin{align*}
   \E^{\tilde\Theta}\biggl[\frac{1}{2}\int_0^1|\tilde{u}(t)|^2dt\biggr]= \E^{\tilde\Theta}\biggl[\frac{1}{2}\int_{\R^m\times [0,1]}|z|^2\tilde{\rho}(dzdt)\biggr]\leq I(\theta)+\frac{\eta}{2}.
\end{align*}

Recall the mapping $\vartheta$ from (3) in Definition \ref{def:weaksenseuniquness}. Decompose $\tilde{\Theta}\circ \vartheta^{-1}\in \mc{P}(\R^d\times \mc{Z}\times\mc{W})$ as $\tilde{\Theta}\circ \vartheta^{-1}(dx,dr,dw)=\Lambda(dr;x,w)\nu_0(dx)\tilde{\Theta}_{\mc{W}}(dw)$. Now let us define for each $N$ a probability space $(\W_\infty,\F_\infty,\Prob_N)$ by setting $\W_\infty = \otimes_{i=1}^\infty \mc{Z}\times\mc{W}$,$\F_\infty = \mc{B}(\W_\infty)$, and $$P_N(d(r_1,r_2,...),d(w_1,w_2,...))=\otimes_{i=1}^N\Lambda(dr_i;w_i,x^{i,N}) \tilde{\Theta}_{\mc{W}}(dw_i)\otimes_{i=N+1}^\infty \tilde{\Theta}|_{\mc{B}(\mc{Z}\times\mc{W})}(dr_idw_i),$$ where $x^{i,N}$ are the deterministic initial conditions from Assumption \ref{assumption:initialconditions}.

For $\omega=(\omega_1,\omega_2,...)=((r_1,w_1),(r_2,w_2),...)\in\W_\infty$, define (see also Section 3.2 of \cite{BCsmallnoise})
\begin{align*}
u^\infty_i(t,\omega)=\tilde{u}(t,\omega_i)=\int_{\R^d}zr_{i,t}(dz),\quad W^{i,\infty}_t(\omega) = \hat{W}_t(\omega_i)=w_i,\quad i\in\bb{N},t\in[0,1].
\end{align*}
Here we decompose $r_i(dzdt)=r_{i,t}(dz)dt$. By construction and \ref{V:V1BDFav}, $W^{i,\infty},i=1,...,N$ are standard, mutually Brownian motions under $\Prob_N$, since $\tilde{\Theta}_{\mc{W}}$ is the classical Wiener measure by definition.

Denote by $\E^N$ the expectation under $\Prob_N$. Then

\begin{align}
\label{eq:L2controlupperbound}
&\limsup_{N\toinf}\E^N[\frac{1}{N}\sum_{i=1}^N \int_0^1|u^\infty_i(t)|^2dt] = \limsup_{N\toinf}\frac{1}{N}\sum_{i=1}^N\int_{\mc{W}}\int_{\mc{Z}} \int_0^1|\int_{\R^d}zr_{t}(dz)|^2dt\Lambda(dr;x^{i,N},w)\tilde{\Theta}(dw) \\
&\leq \int_{\R^d}\int_{\mc{W}}\int_{\mc{Z}} \int_0^1|\int_{\R^d}zr_{t}(dz)|^2dt\Lambda(dr;x,w)\tilde{\Theta}(dw)\nu_0(dx)\nonumber\\
&=\E^{\tilde{\Theta}}[\int_0^1|\int_{\R^d}z\tilde{\rho}_t(dz)|^2dt]= \E^{\tilde{\Theta}}[\int_{\R^d\times[0,1]}|z|^2\tilde{\rho}(dzdt)]<\infty\nonumber
\end{align}
by Assumption \ref{assumption:weaksenseuniqueness}ii) and \ref{V:V2BDFav}.

Let $\br{\tilde{X}^{i,N}}_{i\in\br{1,...,N}}$ be the unique solution to the system of SDEs on $(\W_\infty,\F_\infty,\Prob_N)$
\begin{align*}
    d\tilde{X}^{i,N}_t &= \biggl[b(\tilde{X}^{i,N}_t,\tilde{X}^{i,N}_t/\epsilon,\tilde{\mu}^N_t)+\sigma(\tilde{X}^{i,N}_t,\tilde{\mu}^N_t) u^{i,\infty}(t)\biggr]dt + \sigma(\tilde{X}^{i,N}_t,\tilde{\mu}^N_t) dW^{i,\infty}_t\\
    \tilde{X}^{i,N}_0&=x^{i,N}
\end{align*}
for $N\in\bb{N}$ and $\tilde{\mu}^N_t$ the empirical measure of $\tilde{X}^{1,N},...,\tilde{X}^{N,N}$ at time $t$ (the existence of such solutions is a consequence of Proposition \ref{prop:uniquestrongsol} via the discussion on p.81 of \cite{BDF}).

Define $\hat{\mc{Y}}=\mc{R}^1(\bb{T}^d)$, where
\begin{align*}
\mc{R}^\alpha(\bb{T}^d) \coloneqq \br{n:n\text{ is a positive Borel measure on }\bb{T}^d\times [0,\alpha] \text{ and }n(\bb{T}^d\times[0,t])=t,\forall t\in [0,\alpha]}.
\end{align*}

Note that while $\mc{Z}$ from Equation \eqref{eq:Zspace} is the space where the $\R^d$-marginal of an element of $\mc{Y}$ takes values, $\hat{\mc{Y}}$ is the space where the $\bb{T}^d$-marginal of an element of $\mc{Y}$ takes values.

Now define a sequence of random variables on $\mc{P}(\mc{X}\times\hat{\mc{Y}}\times\mc{Z}\times{\mc{W}})$ by for $A\in \mc{B}(\mc{X}),B\in\mc{B}(\hat{\mc{Y}}),C\in\mc{B}(\mc{Z}),D\in\mc{B}(\mc{W})$:
\begin{align}\label{eq:occmeasurewhenbarW=W}
    \tilde{Q}^N_\omega(A\times B\times C\times D) = \frac{1}{N}\sum_{i=1}^N \delta_{\tilde{X}^{i,N}(\cdot,\omega)}(A)\delta_{\tilde{m}^{i,N}(\omega)} (B) \delta_{\tilde{\rho}^{i,\infty}(\omega)}(C)\delta_{W^{i,\infty}}(D).
\end{align}
Here for $E\in\mc{B}(\R^d)$, $F\in \mc{B}(\bb{T}^d)$, and $I\in\mc{B}([0,1])$:
\begin{align}\label{eq:rhoconstruction}
\tilde{\rho}^{i,\infty}(\omega)(I\times E) \coloneqq \tilde{\rho}(\omega_i)(I\times E) = \int_I \delta_{\tilde{u}(t,\omega_i)}(E)dt =\int_I \delta_{u^{i,\infty}(t,\omega)}(E)dt
\end{align}
and
\begin{align}\label{eq:minconstruction}
\tilde{m}^{i,N}(\omega)(I\times F)\coloneqq \int_{I}\delta_{(\tilde{X}^{i,N}_t(\omega)/\epsilon) \text{mod}1}(F)dt.
\end{align}

Note that these occupation measures are defined similarly to those in Equation \eqref{eq:occmeas}, but that, crucially, in this restricted setting $\bar{W}^{i,N}=W^{i,\infty}$. Moreover, since we are using the equivalent formulation of the rate function $I^{av}$ from Equation \eqref{eq:BDFratefunctionav} rather than $I$ from Equation \eqref{eq:ratefunction}, we replace the second marginal, which was an element of $\mc{P}(\mc{Y})$ given the empirical measure on the $\rho^{i,N}$'s as constructed via the relation \eqref{eq:inducedrelaxedcontrol} by the occupation measures on the decoupled $\bb{T}^d$ and $\R^d$ marginals of $\br{\rho^{i,N}}$, $\br{\tilde{m}^{i,N}}\subset \mc{Y}$ and $\br{\tilde{\rho}^{i,\infty}}\subset \mc{Z}$ respectively. The role of the second marginal here is merely to track what becomes the invariant measure $\pi$ in the limit in a means that allows us to easily refer to the proofs in Section \ref{sec:limitingbehavior}.

We want to see that $\tilde{Q}^N|_{\mc{B}(\mc{X}\times\mc{Z}\times\mc{W})}$ converges weakly to $\tilde{Q}|_{\mc{B}(\mc{X}\times\mc{Z}\times\mc{W})}\in\mc{V}^{av}$ as a $\mc{P}(\mc{X}\times\mc{Z}\times\mc{W})$- valued random variable, and that $\tilde{Q}_{\mc{X}}=\theta$.

We first show tightness of $\br{\tilde{Q}^N}_{N\in\bb{N}}$ as a sequence of $\mc{P}(\mc{X}\times\hat{\mc{Y}}\times\mc{Z}\times\mc{W})$-valued random variables. Since Equation \eqref{eq:L2controlupperbound} holds, the tightness of the $\mc{X},\mc{Z}$, and $\mc{W}$ marginals follows exactly as in Subsection \ref{subsection:tightness}. For tightness of the $\hat{\mc{Y}}$ marginal, we have that $\bb{T}^d\times [0,1]$ is compact, so $\mc{M}^1(\bb{T}^d\times [0,1])$, where $\mc{M}^1(E)$ denotes the set of sub-probability measures on $E$, is compact by Corollary A.3.16 in \cite{DE} (this also works for $\mc{M}^\alpha(E)$, positive Borel measures $\mu$ on $E$ with $\mu(E)\leq\alpha$, for any $\alpha>0$). Then by the proof of Lemma 3.3.1 in \cite{DE}, $\mc{R}^1(\bb{T}^d)\subset \mc{M}^1(\bb{T}^d\times [0,1])$ is closed in the topology of weak convergence (if a weakly converging sequence of measures on $\bb{T}^d\times[0,1]$ has the property that for each member of the sequence, its second marginal is Lebesgue measure, then this will also be true of the limiting measure), and hence $\hat{\mc{Y}}= \mc{R}^1(\bb{T}^d)$ is compact. Then $\mc{P}(\hat{\mc{Y}})$ is compact, and hence $\mc{P}(\mc{P}(\hat{\mc{Y}}))$ is compact. Since $\br{\mc{L}(\tilde{Q}^N)}_{N\in\bb{N}}\subset  \mc{P}(\mc{P}(\hat{\mc{Y}}))$, and on a metrizable space compactness implies sequential compactness, we immediately get $\br{\tilde{Q}^N_{\mc{Y}}}_{N\in\bb{N}}$ is tight as a sequence of $\mc{P}(\mc{Y})$-random variables.

Thus we can extract a weakly convergent subsequence of $\br{\tilde{Q}^N}_{N\in\bb{N}}$, which we will not relabel in the notation, to some $\tilde{Q} \in \mc{P}(\mc{X}\times\mc{Y}\times\mc{Z}\times\mc{W})$.

 Then, via the same proofs as in Section 6.2.1. but with $\tilde{m}^{i,N}$ in the place of $\rho^{i,N}$ (noting that the integrals involved only depend on the $y$-marginal of $\hat{\rho}^{i,N}$, which is exactly $\tilde{m}^{i,N}$), we find that
\begin{align}\label{eq:secondmargisinvariantmeasuresimp}
\tilde{Q}\biggl(\biggl\lbrace(\phi,n,r,w)\in\mc{X}\times\mc{Y}\times\mc{Z}\times\mc{W}: n_s(dy) = \pi(dy|\phi(s),\nu_{\tilde{Q}}(s)),\forall s\in[0,1]\biggr\rbrace\biggr)=1.
\end{align}

Reformulating the martingale problem from Theorem \ref{thm:V1} by taking instead $\Psi\in C_b(\mc{X}\times\mc{Y}\times\mc{Z}\times \mc{W})$ which is measurable with respect to the filtration generated by the coordinate process on $\mc{X}\times\mc{Y}\times\mc{Z}\times\mc{W}$ and modifying $\mc{M}^\Theta_g$ to
\begin{align*}
\tilde{\mc{M}}^\Theta_g(t,(\phi,n,r,w))&=g(\phi(t),w(t))-g(\phi(0),0)\\
&-\int_0^1 \biggl[\int_{\bb{T}^d}b(\phi(s),y,\hat{\nu}_{\Theta}(s))n_s(dy)+\sigma(\phi(s),\hat{\nu}_{\Theta}(s))\int_{\R^d}zr_s(dz) \biggr]\cdot \nabla_x g(\phi(s),w(s))\\
&-\frac{1}{2}A(\phi(s),\hat{\nu}_{\Theta}(s)):\nabla_x\nabla_x g(\phi(s),w(s)) - \sigma(\phi(s),\hat{\nu}_{\Theta}(s)) :\nabla_p\nabla_x g(\phi(s),w(s))\\
&-\frac{1}{2}I:\nabla_p\nabla_p g(\phi(s),w(s))ds
\end{align*}
we get, using that in this simplified regime all terms are only integrated against the $y$ or $z$ marginal of what was $\rho^{i,N}$, the exact same proof as before shows that $\tilde{Q}$ is almost surely a weak solution to Equation \eqref{eq:limitingcontrolledmckeanvlasovsimp} with $m_t$ in the place of $\pi$, where here we mean the coordinate process $(\hat{X},m,\hat{\rho},\hat{W})$ on $\mc{X}\times\mc{Y}\times\mc{Z}\times\mc{W}$ satisfies the given Equation on some filtered probability space. But using \eqref{eq:secondmargisinvariantmeasuresimp}, we get that writing $\tilde{Q}(d\phi dn dr dw) = \lambda(dn|\phi, r, w)\tilde{Q}|_{\mc{B}(\mc{X}\times\mc{Z}\times\mc{W})}(d\phi dr dw)$ that $\lambda(dn|\phi, r, w) = \delta_{\bar{m}_\phi}(dn)$, where for $A\in\mc{B}(\bb{T}^d),I\in\mc{B}([0,1])$,
\begin{align*}
\bar{m}_\phi(A\times I)=\int_I \pi(A;\phi(t),\nu_{\tilde{Q}}(t))dt.
\end{align*}
This means that indeed $\tilde{Q}|_{\mc{B}(\mc{X}\times\mc{Z}\times\mc{W})}$ corresponds to a weak solution of Equation \eqref{eq:limitingcontrolledmckeanvlasovsimp}.

Lastly, we observe that the proofs in Sections 6.2.3 and 6.2.4 go through in the exact same manner, so we have $\tilde{Q}|_{\mc{B}(\mc{X}\times\mc{Z}\times\mc{W})}$ is almost surely in $\mc{V}^{av}$, as desired.

We wish now to conclude that, almost surely, $\tilde{Q}_{\mc{X}}=\tilde{\Theta}_{\mc{X}}$. Note that under assumption \ref{assumption:samebrownianmotion}, $\tilde{Q}|_{\mc{B}(\mc{X}\times\mc{Z}\times\mc{W})},\tilde{\Theta}\in \mc{V}^{av}$ imply conditions (1) and (2) in Definition \ref{def:weaksenseuniquness}. Thus, by Assumption \ref{assumption:weaksenseuniqueness}i), it suffices to prove that $\tilde{Q}|_{\mc{B}(\mc{X}\times\mc{Z}\times\mc{W})}\circ\vartheta^{-1}=\tilde{\Theta}\circ\vartheta^{-1}$, where $\vartheta$ is as in (3) in Definition \ref{def:weaksenseuniquness}.  By the mapping theorem (Theorem 2.7 in \cite{billingsley}) and continuity of $\vartheta$, we can simply show
\begin{align}\label{eq:convergenceforweakuniqueness}
\Prob_N\circ[\tilde{Q}^N|_{\mc{B}(\mc{X}\times\mc{Z}\times\mc{W})}\circ\vartheta^{-1}]^{-1} \tto \delta_{\tilde{\Theta}\circ\vartheta^{-1}}.
\end{align}
Since $(\tilde{\rho}^{i,\infty},W^{i,\infty})_{i=1}^N$ are independent and identically-distributed, if we ignore the initial conditions this would be a consequence of Varadarajan's theorem (\cite{Dudley} p.399). To account for the initial conditions is precisely the reason for the construction of $\Prob_N$ and the Assumption \ref{assumption:weaksenseuniqueness}ii), and we can verify directly via Chebyshev's inequality that for any $F\in C_b(\R^d\times\mc{Z}\times\mc{W})$ and $\eta>0$:
\begin{align*}
&\Prob_N\biggl(\biggl|\int_{\R^d\times\mc{Z}\times\mc{W}}F(x,r,w)\tilde{Q}^N|_{\mc{B}(\mc{X}\times\mc{Z}\times\mc{W})}\circ \vartheta^{-1}(dxdrdw)-\int_{\R^d\times\mc{Z}\times\mc{W}}F(x,r,w)\tilde{\Theta}\circ\vartheta^{-1}(dxdrdw)\biggr|>\eta\biggr)\\
&\leq \frac{8\norm{F}^2_\infty}{\eta^2 N}+\frac{1}{\eta^2}\biggl(\frac{1}{N}\sum_{i=1}^N\int_{\mc{W}}\int_{\mc{Z}}F(x^{i,N},r,w)\Lambda(dr;x^{i,N},w)\tilde{\Theta}_{\mc{W}}(dw)\\
&-\int_{\R^d}\int_{\mc{W}}\int_{\mc{Z}}F(x,r,w)\Lambda(dr;x,w)\tilde{\Theta}_{\mc{W}}(dw)\nu_0(dx) \biggr)^2\\
&\tto 0 \text{ as }N\toinf\text{ by Assumption \ref{assumption:weaksenseuniqueness}ii).}
\end{align*}
Then the fact that \eqref{eq:convergenceforweakuniqueness} holds follows via a standard density argument similar to \cite{Dudley} p.399.

Therefore $\tilde{Q}_{\mc{X}}=\tilde{\Theta}_{\mc{X}}$, and we have using Proposition \ref{prop:varrep}:
\begin{align*}
    \limsup_{N\toinf}-\frac{1}{N}\log\E[\exp(-NF(\mu^N))] &= \limsup_{N\toinf}\inf_{u^N\in\mc{U}^N}\left\{\frac{1}{2}\E[\frac{1}{N}\sum_{i=1}^N \int_0^1|u^N_i(t)|^2dt] + \E[F(\bar{\mu}^N)]\right\}\\
    & \leq \limsup_{N\toinf}\left\{\frac{1}{2}\E^N[\frac{1}{N}\sum_{i=1}^N \int_0^1|u^\infty_i(t)|^2dt]+\E^N[F(\bar{\mu}^N)]\right\}\\
    & \leq \E^{\tilde\Theta}\biggl[\frac{1}{2}\int_{\R^m\times [0,1]}|z|^2\tilde{\rho}(dzdt)\biggr]+\limsup_{N\toinf}\E^N[F(\tilde{Q}^N_{\mc{X}})]\\
    &\leq \E^{\tilde\Theta}\biggl[\frac{1}{2}\int_{\R^m\times [0,1]}|z|^2\tilde{\rho}(dzdt)\biggr]+F(\tilde{\Theta}_{\mc{X}})\\
    &\leq I(\theta)+F(\theta)+\frac{\eta}{2}\\
    &\leq \inf_{\theta\in\mc{P}(\mc{X})}\br{I^{av}(\theta)+F(\theta)} +\eta ,
\end{align*}
where the infimum in the first line is taken to be over all stochastic bases (see the discussion on p.84 of \cite{BDF} and Remark B.1. in \cite{Fischer}).
Since $\eta$ is arbitrary, Equation \eqref{eq:upperbound} is proved.
\end{proof}

\subsection{Proof under the additional assumption \ref{assumption:xindependenceassumptionforupperbound}}\label{subsec:upperboundundernotsameBM}
In this regime, we do not have $\bar{W}^{i,N}=W^i$ with $\bar{W}^{i,N}$ is in Equation \eqref{eq:barW} and $W^i$ the driving Brownian motions in Equation \eqref{eq:multimeanfield}. Thus, as can be seen in the proofs found in Subsection \ref{subsubsection:V1}, constructing occupation measures as in Equation \eqref{eq:occmeasurewhenbarW=W} will not yield a solution to Equation \eqref{eq:BDFlimitav}, since the $\mc{W}$-marginal will not converge to the driving Brownian motion $\hat{W}$ (see the discussion in Remark \eqref{remark:ontheoccmeasures}). So indeed, $W^{i,\infty}$ must be replaced by $\bar{W}^{i,N}$. However, in doing so, $\tilde{Q}^N_{\mc{Z}\times\mc{W}}$ is no longer a sequence of IID random variables, and thus we can no longer apply (essentially) Varadarajan's theorem to conclude that $\tilde{Q}|_{\mc{B}(\mc{Z}\times\mc{W})}=\tilde{\Theta}|_{\mc{B}(\mc{Z}\times\mc{W})}$. Therefore, tracking the joint distribution of some prelimit control and $\bar{W}^{i,N}$, which is what will converge to the driving Brownian motion, is much more subtle in this situation. We thus use what is essentially the idea proposed in the proof of Theorem 4.2.1 in \cite{Kushner} in the one-particle setting, and show that in this regime we can approximate a nearly-optimal control for $I^{av}$ as defined in Equation \eqref{eq:BDFratefunctionav} by a bounded, continuous control in semi-Markovian feedback form (that is, and control which is a function on $[0,1]\times\R^d\times \mc{W}$ which at time $t$ depends stochastically solely on this history of the driving Brownian motion $\hat{W}$ up to time $t$). This then allows us to construct an empirical measure of particles controlled by sequence of feedback controls that nearly approximate the desired limit, since at this point we will not have any additional randomness coming from the control process, and are not tasked with the difficulty of trying to track the joint law of the controls and Brownian motions.

The situation at hand is more complicated than that of \cite{Kushner} for many reasons. Beyond just the fact that we are considering the convergence of random variables in $\mc{P}(\mc{X})$ rather than just $\mc{X}$, there is also the fact that we do not know a priori that the control can be assumed to take values in a compact subset of $\R^d$, or the coordinate process corresponding $\Theta\in\mc{V}^{av}$ can be taken to be adapted to the filtration generated by the initial condition and driving Brownian Motion. These are exactly the issues addressed by Theorem 2.4 in \cite{Lacker}. Indeed, referencing Proposition 2.5 therein, the quantity
\begin{align*}
\inf_{\theta\in\mc{P}(\mc{X})}\br{I^{av}(\theta)+F(\theta)}
\end{align*}
in Equation \eqref{eq:laplaceprinciple} can be viewed in the language of \cite{Lacker} as a ``relaxed formulation'' for the McKean-Vlasov control problem associated to $\Gamma$ as defined in Equation (2.5) with $f(t,x,\mu,a)=-a^2$, and rather than $g:\R^d\times \mc{P}(\R^d)\tto \R$ giving a terminal constraint, we have $-F:\mc{P}(\mc{X})\tto \R$ providing a constraint on the entire path.

As we will see, due to a difference in assumptions, Theorem 2.4 in \cite{Lacker} cannot be applied verbatim, and we need to use some structure inherent to our McKean-Vlasov control problem in order to ensure the first step in its proof goes through. In particular, this is where the assumption \ref{assumption:xindependenceassumptionforupperbound} comes in. Under this additional assumption, the coefficients appearing in $\mc{L}^1_{x,\mu}$ in Equation \eqref{eq:L1} do not depend on $x$, so that $\pi$ from Equation \eqref{eq:pi} and $\Phi$ from Equation \eqref{eq:cellproblem} do not depend on $x$. Thus the controlled limiting McKean-Vlasov Equation \eqref{eq:BDFlimitav} used in the definition of $\mc{V}^{av}$ in Definition \ref{defi:Vav} and hence $I^{av}$ in Equation \eqref{eq:BDFratefunctionav} reduces to:
\begin{align}\label{eq:reducedBDFcontrolledsystemB2}
d\hat{X}_t=[\bar{\beta}(\hat{X}_t,\mc{L}(\hat{X}_t))+\bar{B}(\mc{L}(\hat{X}_t))\int_{\R^d}z\hat{\rho}_t(dz)]dt+\bar{B}(\mc{L}(\hat{X}_t))d\hat{W}_t
\end{align}
where $\bar{B}$ is as in Equation \eqref{eq:equivalentaveragediffusion}, that is
\begin{align*}
\bar{B}(\mu)\bar{B}(\mu)^\top &= \int_{\bb{T}^d} \tilde{D}(y,\mu) \pi(dy|\mu),\mu\in \mc{P}_2(\R^d)\\
\tilde{D}(y,\mu)&=[I+\nabla_y\Phi(y,\mu)]\sigma\sigma^\top(y,\mu)[I+\nabla_y\Phi(y,\mu)]^\top,y\in\bb{T}^d,\mu\in \mc{P}_2(\R^d).
\end{align*}

Put simply, in this regime the effective diffusion $\bar{B}(x,\mu)=\bar{B}(\mu)$. It is precisely for this reason that we make the assumption \ref{assumption:xindependenceassumptionforupperbound}, since, as we will see, this is what allows for the proof method of Theorem 2.4 in \cite{Lacker} to go through.

It is also worth noting at this point that while the forthcoming proof method does not use weak-sense uniqueness directly, it can be shown via a standard coupling method (see e.g. Lemma 3.4. in \cite{BCsmallnoise} and \cite{FischerFormofRateFunction} Proposition C.2) that the weak-sense uniqueness holds for \eqref{eq:reducedBDFcontrolledsystemB2} (replacing \eqref{eq:limitingcontrolledmckeanvlasovsimp} by \eqref{eq:reducedBDFcontrolledsystemB2} in Definition \ref{def:weaksenseuniquness}). Thus, while our proof method allows for accounting for the change in driving Brownian motion which occurs in our multiscale setting, its application in the standard setting of \cite{BDF} would not circumnavigate the need for weak-sense uniqueness of the limiting controlled McKean-Vlasov Equation.

With this discussion in mind, we now provide the proof of the Laplace Principle Upper Bound \eqref{eq:laplaceprincipleupperbound} (equivalently \eqref{eq:upperbound}) in this regime: %\red{Z: In the following I changed references to Equation \eqref{eq:BDFlimitav} to \eqref{eq:reducedBDFcontrolledsystemB2}, deleted dependence of $\sigma,\Phi,B$ on $x$ in two places, and changed the reference to Theorem A.3 in [28] to [53] p.1667}
\begin{proof}
Given $\eta>0$, take $\theta \in \mc{P}(\mc{X})$ such that
\begin{align*}
I^{av}(\theta)+F(\theta) \leq \inf_{\theta \in \mc{P}(\mc{X})} \br{ I^{av}(\theta)+F(\theta)} +\frac{\eta}{3}.
\end{align*}

Since the bound given in Equation \eqref{eq:upperbound} is trivial if the right hand side is $+\infty$, we may assume it is finite.

Consider $\Theta\in \mc{V}^{av}$ such that $\Theta_{\mc{X}}=\theta$ and
\begin{align*}
\E^{\Theta}\biggl[\frac{1}{2}\int_{\R^d\times [0,1]}|z|^2\rho(dzdt)\biggr]\leq I^{av}(\theta)+\frac{\eta}{3}.
\end{align*}

We will now show that we can then find, denoting by $\E$ the expectation with respect to $\Prob_{\nu_0}\in \mc{P}(\R^d\times\mc{W})$ defined as $\nu_0\otimes \mc{P}_W$ where $\mc{P}_W$ is the classical Wiener measure and, by $\hat{W}$ its coordinate process, $\phi:[0,1]\times\R^d\times\mc{W}\tto \R^d$ continuous and bounded and $\tilde{\theta}\in\mc{P}(\mc{X})$ such that
\begin{align}\label{eq:approximationresult}
\E\biggl[\int_0^1 |\phi(t,\xi,\hat{W}_{t\wedge\cdot})|^2 dt\biggr]+F(\tilde{\theta})\leq \E^{\Theta}\biggl[\frac{1}{2}\int_{\R^d\times [0,1]}|z|^2\hat{\rho}(dzdt)\biggr]+F(\theta)+\frac{\eta}{3},
\end{align}
and $\tilde{\theta}=\mc{L}(\hat{X})$ satisfying
\begin{align}\label{eq:feedbackapproximationlimit}
d\hat{X}_t&=\biggl[\bar{\beta}(\hat{X}_t,\mc{L}(\hat{X}_t))+\bar{B}(\mc{L}(\hat{X}_t))\phi(t,\xi,\hat{W}_{t\wedge \cdot})\biggr]dt+\bar{B}(\mc{L}(\hat{X}_t))d\hat{W}_t\\
\hat{X}_0&= \xi\sim \nu_0\nonumber.
\end{align}

As discussed, this follows essentially via the proof of Theorem 2.4 in \cite{Lacker}. The situation is in fact much simpler here though, since in the controlled dynamics, as given by Equation \eqref{eq:reducedBDFcontrolledsystemB2}, we do not have a control term appearing in the diffusion, and thus there is no need for the construction of martingale measures in Lemma 7.1. (It is still worth noting that, as pointed out in Remark 4.12 of \cite{DPTequivalence}, Lemma 7.1 in \cite{Lacker} is based on an unpublished, erroneous result. The proof is resolved in Section 4.1.2 in \cite{DPTequivalence}). Moreover, the cost on the right hand side of Equation \eqref{eq:approximationresult} does not involve any moments of the process $\hat{X}_t$, so we do not need to make the same uniform integrability considerations as \cite{Lacker}.

A caveat, however, is that while our diffusion term in Equation \eqref{eq:reducedBDFcontrolledsystemB2} does satisfy Assumption B of \cite{Lacker} with $p=2$ via Corollary \ref{cor:limitingcoefficientsregularity}, due to the linear dependence on the control, for the drift we have rather
\begin{align*}
|\bar{\beta}(x,\mu)-\bar{B}(\mu)a-\bar{\beta}(x',\mu')-\bar{B}(\mu')a|\leq C(|x-x'|+\bb{W}_2(\mu,\mu')(1+|a|)).
\end{align*}
That is, the Lipschitz assumption does not hold uniformly in $a$.

Thus, we show how to carry out the first step (approximating $\hat{\rho}$ by bounded controls) in the proof of Theorem 2.4 explicitly.

By virtue of \ref{V:V1BDFav} and \ref{V:V3BDFav}, there exists a filtered probability space and adapted processes $(\hat{X},\hat{\rho},\hat{W})$ which satisfy Equation \eqref{eq:reducedBDFcontrolledsystemB2} such that $\Theta=\mc{L}(\hat{X},\hat{\rho},\hat{W})$ and $\hat{\nu}_{\Theta}(0)=\nu_0$.

Let $\iota_l:\R^d\tto \R^d$ be a measurable function such that $|\iota_l(z)|\leq l$ and $\iota_l(z)=z$ for $|z|\leq l$. Define $\hat{\rho}^l_t=\hat{\rho}_t\circ \iota^{-1}_l$, so that $\hat{\rho}^l_t(dz)dt\tto \hat{\rho}_t(dz)dt$ almost surely as $\mc{Z}$-valued random variables. Consider $\hat{X}^l$ the unique solution to the McKean-Vlasov Equation on (possibly an enlargement of) the same probability space on which Equation \eqref{eq:reducedBDFcontrolledsystemB2} is posed:
\begin{align*}
d\hat{X}^l_t&= \biggl[\bar{\beta}(\hat{X}^l_t,\mc{L}(\hat{X}^l_t))+\bar{B}(\mc{L}(\hat{X}^l_t))\int_{\R^d}z\hat{\rho}^l_t(dz)  \biggr] dt+\bar{B}(\mc{L}(\hat{X}^l_t))d\hat{W}_t.\\
\hat{X}^l_0&=\xi\sim \nu_0.
\end{align*}

Note that the Lipschitz property of the coefficients from Corollary \ref{cor:limitingcoefficientsregularity} ensures the well-posedness of this equation for each $l$, as in \cite{Lacker} p.1667 (restricting to bounded controls, Assumption $B$ therein holds).

 Denoting by $\hat{\E}$ the expectation on the probability space on with both equations are posed, we have by Burkholder-Davis-Gundy inequality and the Lipschitz property of the coefficients from Corollary \ref{cor:limitingcoefficientsregularity}:
\begin{align*}
\hat{\E}\biggl[\sup_{s\in [0,t]}\biggl|\hat{X}_s-\hat{X}^l_s \biggr|^2 \biggr]&\leq C\biggl\lbrace\int_0^t\hat{\E}\biggl[ \biggl|\bar{\beta}(\hat{X}_s,\mc{L}(\hat{X}_s))-\bar{\beta}(\hat{X}^l_s,\mc{L}(\hat{X}^l_s))\biggr|^2 \biggr]ds \\
&+\hat{\E}\biggl[\sup_{s\in[0,t]}\biggl|\int_0^s\biggl(\bar{B}(\mc{L}(\hat{X}_\tau))-\bar{B}(\mc{L}(\hat{X}^l_\tau))\biggr)\int_{\R^d}z\hat{\rho}^l_\tau(dz)d\tau\biggr|^2]\\
&+ \hat{\E}\biggl[\sup_{s\in[0,t]}\biggl|\int_0^s\bar{B}(\mc{L}(\hat{X}_\tau))\biggl(\int_{\R^d}z\hat{\rho}_\tau(dz)-\int_{\R^d}z\hat{\rho}^l_\tau(dz)\biggr)d\tau\biggr|^2]\\
&+\int_0^t\hat{\E}\biggl[ \biggl|\bar{B}(\mc{L}(\hat{X}_s))-\bar{B}(\mc{L}(\hat{X}^l_s))\biggr|^2 \biggr]ds  \biggr\rbrace\\
&\leq C\biggl\lbrace\int_0^t\hat{\E}\biggl[ \sup_{\tau\in[0,s]}|\hat{X}_\tau-\hat{X}^l_\tau|^2 \biggr]ds\\
&+\hat{\E}\biggl[\int_0^t\biggl|\bar{B}(\mc{L}(\hat{X}_s))-\bar{B}(\mc{L}(\hat{X}^l_s))\biggr|^2ds\int_0^t\biggl|\int_{\R^d}z\hat{\rho}^l_s(dz)\biggr|^2ds\biggr]\\
&+ \norm{\bar{B}}_\infty\hat{\E}\biggl[\int_0^t\biggl|\int_{\R^d}z\hat{\rho}_s(dz)-\int_{\R^d}z\hat{\rho}^l_s(dz)\biggr|^2ds\biggr]\biggr\rbrace.\\
\end{align*}

Assumption \ref{assumption:xindependenceassumptionforupperbound} is made precisely to handle the second term above. We have:
\begin{align*}
&\hat{\E}\biggl[\int_0^t\biggl|\bar{B}(\mc{L}(\hat{X}_s))-\bar{B}(\mc{L}(\hat{X}^l_s))\biggr|^2ds\int_0^t\biggl|\int_{\R^d}z\hat{\rho}^l_s(dz)\biggr|^2ds\biggr]\\
&\leq C\hat{\E}\biggl[\int_0^t\biggl|\bb{W}_2(\mc{L}(\hat{X}_s),\mc{L}(\hat{X}^l_s))\biggr|^2ds\int_0^t\biggl|\int_{\R^d}z\hat{\rho}^l_s(dz)\biggr|^2ds\biggr]\\
&\leq C\int_0^t\hat{\E}\biggl[ \sup_{\tau\in[0,s]}|\hat{X}_\tau-\hat{X}^l_\tau|^2 \biggr]ds \hat{\E}\biggl[\int_0^t\biggl|\int_{\R^d}z\hat{\rho}^l_s(dz)\biggr|^2ds \biggr]\\
&\leq C\int_0^t\hat{\E}\biggl[ \sup_{\tau\in[0,s]}|\hat{X}_\tau-\hat{X}^l_\tau|^2 \biggr]ds \hat{\E}\biggl[\int_0^1\int_{\R^d}|z|^2\hat{\rho}^l_s(dz)ds \biggr]\\
&\leq C\int_0^t\hat{\E}\biggl[ \sup_{\tau\in[0,s]}|\hat{X}_\tau-\hat{X}^l_\tau|^2 \biggr]ds \hat{\E}\biggl[\int_0^1\int_{\R^d}|z|^2\hat{\rho}_s(dz)ds \biggr]\\
&\leq C\int_0^t\hat{\E}\biggl[ \sup_{\tau\in[0,s]}|\hat{X}_\tau-\hat{X}^l_\tau|^2 \biggr]ds \text{ by \ref{V:V2BDFav}.}
\end{align*}
Then Gronwall's inequality yields
\begin{align*}
\hat{\E}\biggl[\sup_{s\in [0,t]}\biggl|\hat{X}_s-\hat{X}^l_s \biggr|^2 \biggr]&\leq \norm{\bar{B}}_\infty\hat{\E}\biggl[\int_0^t\biggl|\int_{\R^d}z\hat{\rho}_s(dz)-\int_{\R^d}z\hat{\rho}^l_s(dz)\biggr|^2ds\biggr]\exp(Ct)\\
&\leq 2C\norm{\bar{B}}_\infty\hat{\E}\biggl[\int_0^t\int_{\R^d}\1_{|z|>l}|z|^2\hat{\rho}_s(dz)ds\biggr]\exp(Ct)\\
&\tto 0 \text{ as }l\toinf.
\end{align*}

Thus, letting $\theta_l=\mc{L}(\hat{X}^l)$, $\theta_l\tto \theta$ in $\mc{P}_2(\mc{X})$ (see Definition \ref{def:lionderivative}) and hence in $\mc{P}(\mc{X})$, so $F(\theta_l)\tto F(\theta)$.

Moreover, $\frac{1}{2}\int_{\R^d\times[0,1]}|z|^2\hat{\rho}^l_t(dz)dt\leq \frac{1}{2}\int_{\R^d\times[0,1]}|z|^2\hat{\rho}_t(dz)dt,\forall l$, so by dominated convergence theorem:
\begin{align*}
\lim_{l\toinf}\frac{1}{2}\hat{\E}\biggl[\int_{\R^d\times[0,1]}|z|^2\hat{\rho}^l_t(dz)dt \biggr]& =\frac{1}{2}\hat{\E}\biggl[\int_{\R^d\times[0,1]}|z|^2\hat{\rho}_t(dz)dt \biggr],
\end{align*}
so letting $\Theta^l=\mc{L}(\hat{X}^l,\hat{\rho}^l,\hat{W})$, $$\frac{1}{2}\int_{\mc{X}\times\mc{Z}\times\mc{W}}\int_{\R^d\times[0,1]}|z|^2r(dzdt)\Theta^l(d\phi dr dw)\tto \frac{1}{2}\int_{\mc{X}\times\mc{Z}\times\mc{W}}\int_{\R^d\times[0,1]}|z|^2r(dzdt)\Theta(d\phi dr dw).$$

Now we can assume without loss of of generality that the ordinary control associated to $\hat{\rho}$ on the right hand side of the desired bound \eqref{eq:approximationresult} takes values in $B_l$ for some $l>0$, and that $\theta$ is the law of $\hat{X}$ solving Equation \eqref{eq:reducedBDFcontrolledsystemB2} with this bounded control. At this point assumption B in \cite{Lacker} indeed holds for the drift as well, since $a$ is restricted to $B_l$. From here steps 2 and 3 of the proof of Theorem 2.4 in \cite{Lacker} follow verbatim, This yields that there exists a filtered probability space and an adapted process $(\hat{X},u,\hat{W})$ such that $\br{u(t)}_{t\in [0,1]}$ is a $\F^{W}_t\coloneqq \sigma((X_0,W_s),0\leq s\leq t)$-progressively measurable and $B_l$ valued, $(\hat{X},u,\hat{W})$ satisfies Equation \eqref{eq:reducedBDFcontrolledsystemB2} with $u(t)$ in the place of $\int_{\R^d}z\hat{\rho}_t(dz)$, and $\bar{\Theta}=\mc{L}(\hat{X},\delta_{u(t)}(dz)dt,\hat{W})$ is in $\mc{V}^{av}$ with $\bar{\Theta}_{\mc{X}}=\bar{\theta}$ and
\begin{align*}
\E^{\bar{\Theta}}\biggl[\int_0^1 |u(t)|^2 dt\biggr]+F(\bar{\theta})\leq \E^{\Theta}\biggl[\frac{1}{2}\int_{\R^d\times [0,1]}|z|^2\hat{\rho}(dzdt)\biggr]+F(\theta)+\frac{\eta}{3}.
\end{align*}

We just want now to conclude that we can in fact take $u(t)=\phi(t,\xi,\hat{W}_{\cdot\wedge t})$ for $\phi:[0,1]\times\R^d\times\mc{W}\tto \R^d$ bounded and continuous. This is almost implicit in the proof of Theorem 2.4 in \cite{Lacker}, since in the last step they appeal to Lemma 3.11 in \cite{CDL}. In the proof of that lemma, the authors construct a relaxed control via adapted, continuous on $\R^d\times\mc{W}$, and piecewise constant in time functions $\psi:[0,1]\times\R^d\times\mc{W}\tto \mc{P}(\R^d)$. There are many ways to see what we can further take $\psi$ to be of the form $\psi(t,x,w) = \delta_{\phi(t,x,w)}(dz)dt$, but we can arrive at this conclusion after the fact by first applying Proposition 10 in \cite{CTT} to see that $u(t)=\phi(t,\xi,W_{t\wedge \cdot})$ for $\phi:[0,1]\times\R^d\times\mc{W}\tto \R^d$ bounded and predictable, then applying the second part of Proposition C.1 in \cite{CDL} to approximate $\phi$ point-wise $\Prob_{\nu_0}\otimes dt$- almost surely by $\phi^k:[0,1]\times\R^d\times\mc{W}\tto \R^d$ which are bounded, continuous, and predictable. Then
\begin{align*}
\E^{\bar{\Theta}}\biggl[\int_0^1 |u(t)|^2 dt\biggr] =\E\biggl[\int_0^1 |\phi(t,\xi,\hat{W}_{t\wedge \cdot})|^2dt \biggr]=\lim_{k\toinf} \E\biggl[\int_0^1 |\phi^k(t,\xi,\hat{W}_{t\wedge \cdot})|^2dt \biggr],
\end{align*}
and, using the same method as in the previous approximation argument, we can see for $\theta^k=\mc{L}(\hat{X}^k)$ solving Equation \eqref{eq:BDFlimitav} with $\phi^k(t,\xi,\hat{W}_{t\wedge \cdot})$ in the place of $\int_{\R^d}z\hat{\rho}_t(dz)$, we have $\theta^k\tto \bar{\theta}$ in $\mc{P}(\mc{X})$, and hence $F(\theta^k)\tto F(\bar{\theta})$ (see also part 2 of the proof of Proposition 4.15 in \cite{DPTequivalence}).

Thus, indeed \eqref{eq:approximationresult} holds.

We now note that, since the coefficients of Equation \eqref{eq:feedbackapproximationlimit} are continuous (in particular are Lipschitz continuous in $\bb{W}_2$) and bounded, weak uniqueness holds in the classical sense: if $\xi_1,\xi_2\sim \nu_0$ and $W^1,W^2$ are $d$-dimensional standard Wiener processes on possibly different filtered probability spaces, and $(\xi_1,X_1,W^1)$ and $(\xi_2,X_2,W^2)$ both satisfy Equation \eqref{eq:feedbackapproximationlimit}, then $\mc{L}(X_1)=\mc{L}(X_2)=\tilde{\theta}$. This can be see from, e.g. Corollary A.4 in \cite{DPTdpp}, although the situation is much simpler in our setting.

We are now ready to construct a sequence of viable controls and controlled controlled empirical measures such that, inserting this choice into the right-hand side of Equation \eqref{eq:varrep}, we have convergence to the left hand side of Equation \eqref{eq:feedbackapproximationlimit}.

Consider on a possibly different probability space with expectation denoted $\E^\infty$ the collection of particles satisfying the exchangeable system of SDEs
\begin{align*}
d\bar{X}^{i,N}_t &= \biggl[\frac{1}{\epsilon}f(\bar{X}_t^{i,N}/\epsilon,\bar\mu^N_t)+b(\bar{X}_t^{i,N},\bar{X}_t^{i,N}/\epsilon,\bar \mu^N_t)\\
&+\sigma(\bar{X}_t^{i,N}/\epsilon,\bar \mu^N_t)\sigma^\top(\bar{X}_t^{i,N}/\epsilon,\bar \mu^N_t)[I+\nabla_y\Phi(\bar{X}_t^{i,N}/\epsilon,\bar \mu^N_t)]^\top (\bar{B}^\top)^{-1}(\bar \mu^N_t)\phi(t,\bar{X}^{i,N}_0,\bar{W}^{i,N}_{t\wedge\cdot})\biggr]dt\\
&+\sigma(\bar{X}_t^{i,N},\bar{X}_t^{i,N}/\epsilon,\bar\mu^N_t)dW^i_t\nonumber\\
\bar{X}^{i,N}_0&=x^{i,N},
\end{align*}
$W^i$ are independent $m$-dimensional Brownian Motions,
\begin{align*}
\bar{\mu}^{N}_t(\omega):=\frac{1}{N}\sum_{i=1}^N\delta_{\bar{X}^{i,N}_t(\omega)},\hspace{2cm}\bar{\mu}^{N}(\omega):=\frac{1}{N}\sum_{i=1}^N\delta_{\bar{X}^{i,N}(\omega)},
\end{align*}
and $\bar{W}^{i,N}$ are constructed from $\bar{X}^{i,N},\bar{\mu}^N,W^i$ as in Equation \eqref{eq:barW}.

That is, we consider the controlled system from Equation \eqref{eq:controlledprelimit} with semi-Markovian feedback controls
\begin{align}\label{eq:semimarkovfeedbackcontrols}
u^N_i(t)\coloneqq \sigma^\top(\bar{X}_t^{i,N}/\epsilon,\bar \mu^N_t)[I+\nabla_y\Phi(\bar{X}_t^{i,N}/\epsilon,\bar \mu^N_t)]^\top (\bar{B}^\top)^{-1}(\bar \mu^N_t)\phi(t,\bar{X}^{i,N}_0,\bar{W}^{i,N}_{t\wedge\cdot})
\end{align}

Note that it is not the driving Brownian motions themselves which enter the $\mc{W}$-component of $\phi$ in these constructions, but rather the martingales $\bar{W}^{i,N}$. Indeed, the dimension of the $W^{i}_t$'s may not even be equal to $d$, so that they do not even necessarily belong to $\mc{W}$.

Note that via Assumption \ref{assumption:LipschitzandBounded}, Proposition \ref{prop:Phiexistenceregularity}, and Corollary \ref{cor:limitingcoefficientsregularity}, along with the construction of $\phi$, the controls are bounded. Do to the the boundedness and Lipschitz properties of the coefficients, we further have the solutions of the interacting particle system are strong, and hence due to their feedback form, the controls are adapted to the filtration generated by the initial conditions and the driving Brownian motions. Thus indeed we have
\begin{align*}
&\limsup_{N\toinf}-\frac{1}{N}\log\E[\exp(-NF(\mu^N))] \\
&= \limsup_{N\toinf}\inf_{u^N\in\mc{U}^N}\left\{\frac{1}{2}\E[\frac{1}{N}\sum_{i=1}^N \int_0^1|u^N_i(t)|^2dt] + \E[F(\bar{\mu}^N)]\right\}\\
& \leq \limsup_{N\toinf}\left\{\frac{1}{2}\E^\infty[\frac{1}{N}\sum_{i=1}^N \int_0^1|u^N_i(t)|^2dt]+\E^\infty[F(\bar{\mu}^N)]\right\}\\
& = \limsup_{N\toinf}\biggl\lbrace\frac{1}{2}\E^\infty[\frac{1}{N}\sum_{i=1}^N \int_0^1\biggl|\sigma^\top(\bar{X}_t^{i,N}/\epsilon,\bar \mu^N_t)[I+\nabla_y\Phi(\bar{X}_t^{i,N}/\epsilon,\bar \mu^N_t)]^\top (\bar{B}^\top)^{-1}(\bar \mu^N_t)\phi(t,\bar{X}^{i,N}_0,\bar{W}^{i,N}_{t\wedge\cdot})\biggr|^2dt]\\
&\hspace{12cm}+\E^\infty[F(\bar{\mu}^N)]\biggr\rbrace\\
& = \limsup_{N\toinf}\biggl\lbrace\frac{1}{2}\E^\infty[ \int_0^1\phi^\top(t,\bar{X}^{i,N}_0,\bar{W}^{i,N}_{t\wedge\cdot})\bar{B}^{-1}(\bar \mu^N_t)[I+\nabla_y\Phi(\bar{X}_t^{i,N}/\epsilon,\bar \mu^N_t)]\sigma\sigma^\top(\bar{X}_t^{i,N}/\epsilon,\bar \mu^N_t)\\
&\hspace{4cm}[I+\nabla_y\Phi(\bar{X}_t^{i,N}/\epsilon,\bar \mu^N_t)]^\top (\bar{B}^\top)^{-1}(\bar \mu^N_t)\phi(t,\bar{X}^{i,N}_0,\bar{W}^{i,N}_{t\wedge\cdot})dt]+\E^\infty[F(\bar{\mu}^N)]\biggr\rbrace\\
& = \limsup_{N\toinf}\biggl\lbrace\frac{1}{2}\E^\infty[\int_{\mc{X}\times\mc{Y}\times\mc{W}}\int_0^1\phi^\top(t,\psi(0),w(t\wedge\cdot))\bar{B}^{-1}(\hat{\nu}_{Q^N}(t))\int_{\bb{T}^d}[I+\nabla_y\Phi(y,\hat{\nu}_{Q^N}(t))]\sigma\sigma^\top(y,\hat{\nu}_{Q^N}(t))\\
&\hspace{2cm}[I+\nabla_y\Phi(y,\hat{\nu}_{Q^N}(t))]^\top n(dy)(\bar{B}^\top)^{-1}(\hat{\nu}_{Q^N}(t))\phi(t,\psi(0),w(t\wedge\cdot))dt\bar{Q}^N(d\psi dn dw)]\\
&\hspace{12cm}+\E^\infty[F(\bar{Q}^N_{\mc{X}}]\biggr\rbrace,
\end{align*}
where here we defined the sequence of random measures $\br{\bar{Q}^N}_{N\in\bb{N}}\subset\mc{P}(\mc{X}\times\hat{\mc{Y}}\times\mc{W})$ by
\begin{align*}
  \bar{Q}^N_\omega(A\times B\times C) = \frac{1}{N}\sum_{i=1}^N \delta_{\bar{X}^{i,N}(\cdot,\omega)}(A)\delta_{\bar{m}^{i,N}(\omega)} (B) \delta_{\bar{W}^{i,N}}(C)
\end{align*}
Here $\bar{m}^{i,N}$ are as in Equation \eqref{eq:minconstruction} but with this choice of $\bar{X}^{i,N}$.

As with the construction of the occupation measures \eqref{eq:occmeasurewhenbarW=W} in Subsection \ref{subsection:sameBMupperboundproof}, the role of the $\hat{\mc{Y}}$-marginal is just to converge to what becomes the invariant measure $\pi$ in the limit in a means that allows us to easily refer to the proofs in Section \ref{sec:limitingbehavior}. We want to show that $\bar{Q}^N_{\mc{X}\times\mc{W}}$ converges in distribution as a sequence of $\mc{P}(\mc{X}\times\mc{W})$- valued random variables to the deterministic limit $\bar{Q}=\mc{L}(\hat{X},\hat{W})$ solving Equation \eqref{eq:approximationresult}, so that $\bar{Q}_{\mc{X}}=\tilde{\theta}$.

Tightness of $\bar{Q}^N$ follows exactly as in Subsection \ref{subsection:tightness}, with tightness of the $\hat{\mc{Y}}$-marginals holding trivially as in Subsection \eqref{subsection:sameBMupperboundproof}. In addition, in the same way as Subsection \eqref{subsection:sameBMupperboundproof}, taking a subsequence of $\br{\bar{Q}^N}$ such that $\bar{Q}^N\tto \bar{Q}$ in distribution,
\begin{align*}
\bar{Q}\biggl(\biggl\lbrace(\psi,n,w)\in\mc{X}\times\hat{\mc{Y}}\times\mc{W}: n_s=\pi(dy|\hat{\nu}_{\hat\Theta}(s)) ds,\forall s\in[0,1]\biggr\rbrace\biggr)=1.
\end{align*}
Moreover, due to the boundedness and continuity of the controls $u_i^N$ from Equation \eqref{eq:semimarkovfeedbackcontrols} as functions on $[0,1]\times\R^d\times\bb{T}^d\times\mc{P}(\R^d)\times\mc{W}$, the same proof of the form of the limiting equation as in Subsection \ref{subsubsection:V1} holds, but rather than having an external control, we treat the control as a standard part of the drift - that is, it plays essentially the same role as $b$ in the computation. From this, we get (noting that the effect of the averaging on the term $b$ as contained in $\bar{\beta}$ in Equations \eqref{eq:McKeanLimit} and \eqref{eq:limitingcoefficients} is to replace it by $[I+\nabla_y\Phi]b$ and integrate it against the invariant measure), that any subsequence of $\bar{Q}^{N}_{\mc{X}\times\mc{W}}$ converges in distribution to $\mc{L}(\hat{X},\hat{W})\in\mc{P}(\mc{X}\times\mc{W})$ satisfying
\begin{align*}
d\hat{X}_t &= \biggl[\bar{\beta}(\hat{X}_t,\mc{L}(\hat{X}_t))+\int_{\bb{T}^d}[I+\nabla_y\Phi(y,\mc{L}(\hat{X}_t))]\sigma(y,\mc{L}(\hat{X}_t))\sigma^\top(y,\mc{L}(\hat{X}_t))\\
&\hspace{2cm}[I+\nabla_y\Phi(y,\mc{L}(\hat{X}_t))]^\top\pi(dy;\mc{L}(\hat{X}_t))\bar{B}^{-1}(\mc{L}(\hat{X}_t))\phi(t,\xi,\hat{W}_{t\wedge\cdot})\biggr]dt\\
&+\bar{B}(\mc{L}(\hat{X}_t))d\hat{W}_t\\
& = \biggl[\bar{\beta}(\hat{X}_t,\mc{L}(\hat{X}_t))+\bar{B}(\mc{L}(\hat{X}_t))\bar{B}^\top(\mc{L}(\hat{X}_t))(\bar{B}^\top)^{-1}(\mc{L}(\hat{X}_t))\phi(t,\xi,\hat{W}_{t\wedge\cdot})\biggr]dt+\bar{B}(\mc{L}(\hat{X}_t))d\hat{W}_t\\
& = \biggl[\bar{\beta}(\hat{X}_t,\mc{L}(\hat{X}_t))+\bar{B}(\mc{L}(\hat{X}_t))\phi(t,\xi,\hat{W}_{t\wedge\cdot})\biggr]dt+\bar{B}(\mc{L}(\hat{X}_t))d\hat{W}_t.
\end{align*}
with $\hat{X}_0\sim \nu_0$.

Thus, by uniqueness of solutions to Equation \eqref{eq:feedbackapproximationlimit}, we have in fact $Q^N$ converges in distribution along its entire sequence to the deterministic limit $Q=\mc{L}(\hat{X},\pi(dy;\tilde{\theta}(t))dt,\hat{W})$, and in particular $Q_{\mc{X}}=\tilde{\theta}$. Then we can return to the inequality
\begin{align*}
&\limsup_{N\toinf}-\frac{1}{N}\log\E[\exp(-NF(\mu^N))] \\
& = \limsup_{N\toinf}\biggl\lbrace\frac{1}{2}\E^\infty[\int_{\mc{X}\times\hat{\mc{Y}}\times\mc{W}}\int_0^1\phi^\top(t,\psi(0),w(t\wedge\cdot))\bar{B}^{-1}(\hat{\nu}_{Q^N}(t))\int_{\bb{T}^d}[I+\nabla_y\Phi(y,\hat{\nu}_{Q^N}(t))]\sigma\sigma^\top(y,\hat{\nu}_{Q^N}(t))\\
&\hspace{1cm}[I+\nabla_y\Phi(y,\hat{\nu}_{Q^N}(t))]^\top n(dy)(\bar{B}^\top)^{-1}(\hat{\nu}_{Q^N}(t))\phi(t,\psi(0),w(t\wedge\cdot))dt\bar{Q}^N(d\psi dn dw)]+\E^\infty[F(\bar{Q}^N_{\mc{X}}]\biggr\rbrace,
\end{align*}
and use continuity of $F$ and boundedness and continuity of the integrand of the first equation along with Theorem A.3.18 in \cite{DE} to continue as:
\begin{align*}
&\limsup_{N\toinf}-\frac{1}{N}\log\E[\exp(-NF(\mu^N))]\\
&\leq \frac{1}{2}\int_{\mc{X}\times\hat{\mc{Y}}\times\mc{W}}\int_0^1\phi^\top(t,\psi(0),w(t\wedge\cdot))\bar{B}^{-1}(\hat{\nu}_{Q}(t))\int_{\bb{T}^d}[I+\nabla_y\Phi(y,\hat{\nu}_{Q}(t))]\sigma\sigma^\top(y,\hat{\nu}_{Q}(t))\\
&\hspace{2cm}[I+\nabla_y\Phi(y,\hat{\nu}_{Q}(t))]^\top n(dy)(\bar{B}^\top)^{-1}(\hat{\nu}_{Q}(t))\phi(t,\psi(0),w(t\wedge\cdot))dt\bar{Q}(d\psi dn dw)+F(\bar{Q}_{\mc{X}})\\
& = \frac{1}{2}\int_{\mc{X}\times\mc{W}}\int_0^1\phi^\top(t,\psi(0),w(t\wedge\cdot))\bar{B}^{-1}(\tilde{\theta}(t))\int_{\bb{T}^d}[I+\nabla_y\Phi(y,\tilde{\theta}(t))]\sigma\sigma^\top(y,\tilde{\theta}(t))\\
&\hspace{2cm}[I+\nabla_y\Phi(y,\tilde{\theta}(t))]^\top \pi(dy;\phi(t),\tilde{\theta}(t))(\bar{B}^\top)^{-1}(\tilde{\theta}(t))\phi(t,\psi(0),w(t\wedge\cdot))dt\bar{Q}_{\mc{X}\times\mc{W}}(d\psi dw)+F(\tilde{\theta})\\
& = \frac{1}{2}\int_{\mc{X}\times\mc{W}}\int_0^1\phi^\top(t,\psi(0),w(t\wedge\cdot))\bar{B}^{-1}(\tilde{\theta}(t))\bar{B}(\tilde{\theta}(t))\bar{B}^\top(\tilde{\theta}(t))(\bar{B}^\top)^{-1}(\tilde{\theta}(t))\phi(t,\psi(0),w(t\wedge\cdot))dt\bar{Q}_{\mc{X}\times\mc{W}}(d\psi dw)\\
&\hspace{12cm}+F(\tilde{\theta})\\
& = \frac{1}{2}\int_{\mc{X}\times\mc{W}}\int_0^1|\phi(t,\psi(0),w(t\wedge\cdot))|^2dt\bar{Q}_{\mc{X}\times\mc{W}}(d\psi dw)+F(\tilde{\theta})\\
& = \E\biggl[\int_0^1 |\phi(t,\xi,\hat{W}_{t\wedge\cdot})|^2 dt\biggr]+F(\tilde{\theta})\\
&\leq \E^{\Theta}\biggl[\frac{1}{2}\int_{\R^d\times [0,1]}|z|^2\rho(dzdt)\biggr]+F(\theta)+\frac{\eta}{3}\text{ by Equation \eqref{eq:approximationresult}}\\
&\leq I^{av}(\theta)+F(\theta)+\frac{2\eta}{3}\\
&\leq \inf_{\theta \in \mc{P}(\mc{X})} \br{ I^{av}(\theta)+F(\theta)}+\eta.
\end{align*}
Since $F$ and $\eta$ were arbitrary, the Laplace Principle Upper Bound \eqref{eq:laplaceprincipleupperbound} is proved.
\end{proof}

\section{Conclusions and Future Work}\label{S:Conclusions}
We have derived a large deviations principle and law of large numbers for the empirical measure of a system of weakly interacting particles in a two-scale environment in the joint many-particle and averaging limit. We use weak convergence methods, and obtain a variational form of the rate function. We saw that for the system (\ref{eq:multimeanfield}), the two limiting procedures commute.

The results of this paper bring to light many interesting problems to be explored in future work.  An interesting extension of this work would be to proving a large deviations principle for systems whose coefficients depend on the ``fast empirical measure'' $\mu^{N,\epsilon}_t \coloneqq \frac{1}{N}\sum_{i=1}^N \delta_{X^{i,N}_t/\epsilon}$ as well. Such a result would then capture the system explored in \cite{delgadino2020}, and would perhaps serve to give insight into the nature of bifurcations in the number of steady states for certain classes McKean-Vlasov systems as originally investigated in \cite{Dawson}. In addition, the connection of the variational form of the rate function from Theorem \ref{thm:LaplacePrinciple} to the Dawson-G\"artner form of the rate function from Theorem \ref{theorem:DGform} as proved in Subsection \ref{subsec:DGequivalentform} opens the doors to studying the dynamical effects of multiscale structure on phase transitions and exit times from basins of attraction for the empirical measures of weakly interacting diffusions, and may allow for the design provably optimal importance sampling schemes for functionals of the empirical measure in the multiscale and non-multiscale settings \cites{Spiliopoulos2013a,DSX,AS}. In particular, understanding this connection further may allow for leveraging many recently-developed tools coming from gradient flows on metric spaces \cite{GradientFlows} and optimal control of McKean-Vlasov equations \cites{Lacker,LackerMarkovian,Fischer,DPTequivalence,DPTdpp} in addition to Large Deviations Theory \cite{FW} to make progress in such directions.

\appendix
\section{Preliminary Results on the Prelimit System \eqref{eq:controlledprelimit} and the Operator $\nu_Q(t)$} \label{Appendix:Prelimit}
\begin{proposition}\label{prop:uniquestrongsol}
Under assumption \ref{assumption:LipschitzandBounded}, the system of mean-field SDEs \eqref{eq:multimeanfield} admits a unique strong solution for each $N\in\bb{N}$.
\end{proposition}
\begin{proof}
We observe that Equation \eqref{eq:multimeanfield} can be written as a standard $2dN$-dimensional SDE via
\begin{align*}
d\hat{X}^N_t = \biggl[\frac{1}{\epsilon} \hat{f}(\hat{X}^N_t)+\hat{b}(\hat{X}^N_t)\biggr]+\hat{\sigma}(\hat{X}^N_t)d\hat{W}_t^N
\end{align*}
where, letting $Y^{i,N}_t=X^{i,N}_t/\epsilon,\forall i\in\br{1,...,N}$, and $\hat{x}=(\hat{x}_1,...,\hat{x}_{2N})^\top,\hat{x}_i\in \R^d,i\in\br{1,...,2N}$, we have
$\hat{X}^N_t=(X^{1,N}_t,\cdots, X^{N,N}_t,Y^{1,N}_t,\cdots, Y^{N,N}_t)$ and where $\hat{W}^N_t=(W^1_t,...,W^N_t)^\top\in\R^{mN}$.

Let $g$ play the role of $f$ or $b$. For $g:\R^d\times \bb{T}^d\times \mc{P}_2(\R^d)\tto \R^d$, here we denote by $\hat{g}:\R^{2dN}\tto \R^{2dN}$ the function with $i$'th d-dimensional coordiante functions
\begin{align*}
\hat{g}_i(x_1,...,x_N,x_{N+1},...,x_{2N})=g(x_i,x_{N+i},\frac{1}{N}\sum_{j=1}^N\delta_{x_j}),\quad i=1,...,N
\end{align*}
and $\hat{g}_{N+i}=\frac{1}{\epsilon}\hat{g}_i,i=1,...,N.$ Similarly, $\hat{\sigma}:\R^{2dN}\tto \R^{2dN\times mN}$ is a matrix with entries given by $\hat{\sigma}_{i,j}:\R^{2dN}\tto \R^{d\times m}$ which are all $0$ except
\begin{align*}
\hat{\sigma}_{i,i}(x_1,...,x_N,x_{N+1},...,x_{2N})=\sigma(x_i,x_{N+i},\frac{1}{N}\sum_{j=1}^N\delta_{x_j}),\quad i=1,...,N
\end{align*}
and $\hat{\sigma}_{i+N,i}=\frac{1}{\epsilon}\hat{\sigma}_{i,i},i=1,...,N.$ One can then verify that \ref{assumption:LipschitzandBounded} implies that for each $N\in\bb{N},\exists C(N)$ such that for all $\hat{x}_1,\hat{x}_2\in\R^{2dN},$ $|\frac{1}{\epsilon}\hat{f}(\hat{x}_1)+\hat{b}(\hat{x}_1)-\frac{1}{\epsilon}\hat{f}(\hat{x}_2)-\hat{b}(\hat{x}_2)| + |\sigma(\hat{x}_1)-\sigma(\hat{x}_2)|\leq C(N)|\hat{x}_1-\hat{x}_2|$, so that by standard existence and uniqueness results for SDE's with globally Lipschitz coefficients, the proposition holds. See, for example, Theorem 5.2.1 in \cite{Oksendal}.
\end{proof}

\begin{proposition}\label{prop:uniformXL2bound}
For $\bar{X}^{i,N}$ as in Equation \eqref{eq:controlledprelimit} controlled by any $u^N\in\mc{U}_N$ satisfying almost surely for some $C_{con}>0$ the bound $\sup_{N\in \bb{N}} \E\biggl[\frac{1}{N}\sum_{i=1}^N \int_0^1 |u^N_i(t)|^2dt\biggr]\leq C_{con}$ and under assumptions \ref{assumption:initialconditions}-\ref{assumption:centeringcondition}:
\begin{align*}
\sup_{N\in\bb{N}}\E\biggl[\frac{1}{N}\sum_{i=1}^N \sup_{0\leq t\leq 1} |\bar{X}^{i,N}_t|^2\biggr] < \infty.%,\forall q>0.
\end{align*}
This ensures that for all $N$, there exists a modification of $\bar{\mu}^N\in C([0,1];\mc{P}_2(\R^d))$ so that $\bar{\mu}^N_t$ is in $\mc{P}_2(\R^d)$ for all time. %\orange{Maybe remove the proof to save pages?}
\end{proposition}
\begin{proof}
This follows via standard methods after using the same computations as in the proof of the bound \eqref{eq:vanishingmartingaleprobterms} in Subsection \ref{subsubsection:V1}, taking $g(x,p)=|x|^2$.

\end{proof}

We end this section with a proposition regarding the mapping defined in Equation \eqref{eq:nuQ}.
\begin{proposition}\label{prop:nucontmeasure}
For fixed $t\in[0,1]$, $Q\mapsto \nu_{Q}(t)$ is continuous, and for fixed $Q$, $t\mapsto \nu_Q(t)$ is continuous.
\end{proposition}
\begin{proof}
Take $\br{Q^n}\subset \mc{P}(\mc{C})$ such that $Q^n\tto Q$ and $f\in C_b(\R^d)$. Then, since $(\phi,r,w)\mapsto f(\phi(t))\in C_b(\mc{C})$ we get
\begin{align*}
\lim_{n\toinf}\int_{\R^d}f(x) \nu_{Q^n}(t)(dx) & =\lim_{n\toinf} \int_{\mc{C}} f(\phi(t)) Q^n(d\phi dr dw)
 = \int_{\mc{C}}f(\phi(t)) Q(d\phi dr dw)
 = \int_{\R^d} f(x)\nu_{Q}(t)(dx).
\end{align*}
Continuity in time follows as in Section 4 of \cite{BDF}.
\end{proof}

\section{On Lions Differentiation}\label{Appendix:LionsDifferentiation}
We will need the following two definitions from \cite{CD}:

\begin{defi}
\label{def:lionderivative}
Given a function $u:\mc{P}_2(\R^d)\tto \R$, we may define a lifting of $u$ to $\tilde{u}:L^2(\tilde\W,\tilde\F,\tilde\Prob;\R^d)\tto \R$ via $\tilde u (X) = u(\mc{L}(X))$ for $X\in L^2(\tilde\W,\tilde\F,\tilde\Prob;\R^d)$. Here we assume $\tilde\W$ is a Polish space, $\tilde\F$ its Borel $\sigma$-field, and $\tilde\Prob$ is an atomless probability measure (since $\tilde\W$ is Polish, this is equivalent to every singleton having zero measure).

Here, letting $\mc{S}$ be a Polish space with metric $\rho$, and denoting by $\mu(|\cdot|^r)\coloneqq \int_{\mc{S}}\rho(x,x_0)^r \mu(dx)$ for $r\geq 1$ and $x_0$ a fixed element of $\mc{S}$,
\begin{align*}
\mc{P}_{r}(\mc{S}) \coloneqq \br{ \mu\in \mc{P}(\mc{S}):\mu(|\cdot|^{r}])<\infty}.
\end{align*}
$\mc{P}_{r}(\mc{S})$ is a polish space under the $L^{r}$-Wasserstein distance
\begin{align*}
\bb{W}_{r,\mc{S}} (\mu_1,\mu_2)\coloneqq \inf_{\pi \in\mc{C}_{\mu_1,\mu_2}} \biggl[\int_{\mc{S}\times \mc{S}} \rho(x,y)^{r} \pi(dx,dy)\biggr]^{1/r},
\end{align*}
where $\mc{C}_{\mu_1,\mu_2}$ denotes the set of all couplings of $\mu_1,\mu_2$ (see p.360 of \cite{CD}). When $\mc{S}=\R^d$, we will simply write $\bb{W}_r$ rather than $\bb{W}_{r,\R^d}$.

We say $u$ is L-differentiable or Lions-differentiable at $\mu_0\in\mc{P}_2(\R^d)$ if there exists a random variable $X_0$ on some $(\tilde\W,\tilde\F,\tilde\Prob)$ satisfying the above assumptions such that $\mc{L}(X_0)=\mu_0$ and $\tilde u$ is Fr\'echet differentiable at $X_0$.

The Fr\'echet derivative of $\tilde u$ can be viewed as an element of $L^2(\tilde\W,\tilde\F,\tilde\Prob;\R^d)$ by identifying $L^2(\tilde\W,\tilde\F,\tilde\Prob;\R^d)$ and its dual. From this, one can find that if $u$ is L-differentiable at $\mu_0\in\mc{P}_2(\R^d)$, there is a deterministic measurable function $\xi: \R^d\tto \R^d$ such that $D\tilde{u}(X_0)=\xi(X_0)$, and that $\xi$ is uniquely defined $\mu_0$-almost everywhere on $\R^d$. We denote this equivalence class of $\xi\in L^2(\R^d,\mu_0;\R^d)$ by $\partial_\mu u(\mu_0)$ and call $\partial_\mu u(\mu_0)(\cdot):\R^d\tto \R^d$ the Lions derivative of $u$ at $\mu_0$. Note that this definition is independent of the choice of $X_0$ and $(\tilde\W,\tilde\F,\tilde\Prob)$. See \cite{CD} Section 5.2.

To avoid confusion when $u$ depends on more variables than just $\mu$, if $\partial_\mu u(\mu_0)$ is differentiable at $v_0\in\R^d$, we denote its derivative at $v_0$ by $\partial_v\partial_\mu u(\mu_0)(v_0)$.
\end{defi}
\begin{defi}
\label{def:fullyC2}
(\cite{CD} Definition 5.83) We say $u:\mc{P}_2(\R^d)\tto \R$ is Fully $\mathbf{C^2}$ if the following conditions are satisfied:
\begin{enumerate}
\item $u$ is $C^1$ in the sense of L-differentiation, and its first derivative has a jointly continuous version $\mc{P}_2(\R^d)\times \R^d\ni (\mu,v)\mapsto \partial_\mu u(\mu)(v)\in\R^d$.
\item For each fixed $\mu\in\mc{P}_2(\R^d)$, the version of $\R^d\ni v\mapsto \partial_\mu u(\mu)(v)\in\R^d$ from the first condition is differentiable on $\R^d$ in the classical sense and its derivative is given by a jointly continuous function $\mc{P}_2(\R^d)\times \R^d\ni (\mu,v)\mapsto \partial_v\partial_\mu u(\mu)(v)\in\R^{d\times d}$.
\item For each fixed $v\in \R^d$, the version of $\mc{P}_2(\R^d)\ni \mu\mapsto \partial_\mu u(\mu)(v)\in \R^d$ in the first condition is continuously L-differentiable component-by-component, with a derivative given by a function $\mc{P}_2(\R^d)\times \R^d\times \R^d\ni(\mu,v,v')\mapsto \partial^2_\mu u(\mu)(v)(v')\in\R^{d\times d}$ such that for any $\mu\in\mc{P}_2(\R^d)$ and $X\in L^2(\tilde\W,\tilde\F,\tilde\Prob;\R^d)$ with $\mc{L}(X)=\mu$, $\partial^2_\mu u(\mu)(v)(X)$ gives the Fr\'echet derivative at $X$ of $L^2(\tilde\W,\tilde\F,\tilde\Prob;\R^d)\ni X'\mapsto \partial_\mu u(\mc{L}(X'))(v)$ for every $v\in\R^d$. Denoting $\partial^2_\mu u(\mu)(v)(v')$ by $\partial^2_\mu u(\mu)(v,v')$, the map $\mc{P}_2(\R^d)\times \R^d\times \R^d\ni(\mu,v,v')\mapsto \partial^2_\mu u(\mu)(v,v')$ is also assumed to be continuous in the product topology.
\end{enumerate}
\end{defi}

We recall now a useful connection between the Lions derivative as defined in \ref{def:lionderivative} and the empirical measure.
\begin{proposition}\label{prop:empprojderivatives}
For $g:\mc{P}_2(\R^d)\tto \R^d$ which is Fully $C^2$ in the sense of definition \ref{def:fullyC2}, we can define the empirical projection of $g$, as $g^N: (\R^d)^N\tto \R^d$ given by
\begin{align*}
g^N(\beta_1,...,\beta_N)\coloneqq g(\frac{1}{N}\sum_{i=1}^N \delta_{\beta_i}).
\end{align*}

Then $g^N$ is twice differentiable on $(\R^d)^N$, and for each $\beta_1,..,\beta_N\in\R^d$, $(i,j)\in \br{1,...,N}^2$, $l\in\br{1,...,d}$
\begin{align}
\label{eq:empfirstder}
\nabla_{\beta_i} g^N_l(\beta_1,...,\beta_N)= \frac{1}{N} \partial_\mu  g_l(\frac{1}{N}\sum_{i=1}^N \delta_{\beta_i}) (\beta_i)
\end{align}
and
\begin{align}
\label{eq:empsecondder}
\nabla_{\beta_i} \nabla_{\beta_j} g^N_l(\beta_1,...,\beta_N)= \frac{1}{N} \partial_v \partial_\mu  g_l(\frac{1}{N}\sum_{i=1}^N \delta_{\beta_i}) (\beta_i) \1_{i=j} + \frac{1}{N^2} \partial^2_\mu g_l(\frac{1}{N}\sum_{i=1}^N \delta_{\beta_i})(\beta_i,\beta_j).
\end{align}

In particular, under assumptions \ref{assumption:LipschitzandBounded}-\ref{assumption:centeringcondition}, this holds for $\Phi(x,y,\cdot)$ for fixed $x\in\R^d$ and $y\in\bb{T}^d$.
\end{proposition}
\begin{proof}
This follows from Propositions 5.35 and 5.91 of \cite{CD}. Since by Proposition \ref{prop:Phiexistenceregularity} $\Phi$ is Fully $C^2$, it applies to $\Phi(x,y,\cdot)$.
\end{proof}

\section{On the Operator $\mc{L}^1$ and Related PDEs}\label{Appendix:L1}
\begin{proposition}\label{prop:invtmeasure}
Under assumptions \ref{assumption:LipschitzandBounded}-\ref{assumption:fsigmaregularity}, the invariant measure $\pi$ defined by Equation \eqref{eq:pi} is uniquely determined for each $x\in\R^d$ and $\mu\in\mc{P}_2(\R^d)$ and has a continuous, bounded density $\tilde{\pi}$.
\end{proposition}
\begin{proof}
This follows immediately from Theorem 4.4 in \cite{PS}/Theorem 4.3.4 in \cite{Bensoussan} and a standard embedding argument via Morrey's inequality. For continuity in the parameters, see e.g. Chapter 3, Section 6 of \cite{Bensoussan}.
\end{proof}

\begin{proposition}\label{prop:Phiexistenceregularity}
Under assumptions \ref{assumption:LipschitzandBounded}-\ref{assumption:centeringcondition}, there is a unique strong solution $\Phi$ to equation \eqref{eq:cellproblem}. Moreover, $\Phi$, all first and second order partial derivatives of $\Phi$ in $x$ and $y$ and $\nabla_y\nabla_x\nabla_x\Phi$ are bounded, $\Phi$ is Fully $C^2$ in the sense of Definition \ref{def:fullyC2}, and $\partial_\mu\Phi(x,y,\mu)(\cdot)$, $\partial_v\partial_\mu\Phi(x,y,\mu)(\cdot),$ $\nabla_x\partial_\mu\Phi(x,y,\mu)(\cdot)$, $\nabla_y\partial_\mu\Phi(x,y,\mu)(\cdot)$, $\nabla_y\nabla_x\partial_\mu\Phi(x,y,\mu)(\cdot)$, $\partial^2_\mu\Phi(x,y,\mu)(\cdot,\cdot)$ exist, are continuous with respect to all variables $x,v,v'\in\R^d,y\in\bb{T}^d,\mu\in\mc{P}(\R^d)$, and are uniformly bounded in $L^2(\R^d,\mu)(\otimes L^2(\R^d,\mu))$  with respect to $x$ and $y$.
\end{proposition}
\begin{proof}

Existence and uniqueness follows directly from Theorems 6.16 and 7.9 in \cite{PS}.

Consider the frozen process on $\bb{T}^d$ for fixed $x\in\R^d$, $\mu\in\mc{P}_2(\R^d)$, $y\in\bb{T}^d$ given by
\begin{align}
\label{eq:frozeneqn}
dY^{x,y,\mu}_t&=f(x,Y^{x,y,\mu}_t,\mu)dt+\sigma(x,Y^{x,y,\mu}_t,\mu)d\tilde{W}_t\\
Y^{x,y,\mu}_0&=y \nonumber
\end{align}
where $\tilde{W}_t$ is a $m$-dimensional, $\tilde{F}_t$-adapted Brownian motion on some probability space $(\tilde\W,\tilde\F,\tilde\Prob)$ satisfying the usual conditions.

As per Proposition 4.1 in \cite{Rockner} and Section 11.6 in \cite{PS}, $\Phi$ is given by
\begin{align}
\label{eq:stochrepphi}
\Phi(x,y,\mu)=\int_0^\infty\tilde{\E}[f(x,Y^{x,y,\mu}_s,\mu)]ds.
\end{align}

Then the fact that $\Phi$ is fully $C^2$ and smooth in $x$ and $y$ and boundedness of $\Phi$, along with regularity of $\Phi$ of the same type given in \ref{assumption:fsigmaregularity} (with an additional $y$ derivative) follows from the unique representation of the cell problem given by Equation \eqref{eq:stochrepphi}. This has been studied in many situations in the existing literature (see, for example \cite{PV1}, \cite{PV2} for general results on Euclidean space with no measure dependence, \cite{Bensoussan} Chapter 3 Section 6 for the case where the fast component is on the torus with no measure dependence, as well as \cite{Rockner} for when $\Phi$ depends on a measure). The particular regularity assumption imposed here as \ref{assumption:fsigmaregularity} mirror those of \cite{RocknerFullyCoupled}, where regularity of the Poisson Equation (found in Theorem 2.1) is derived via derivative transfer formulas stated therein as Lemma 3.2. These transfer formulas are extended to the Lions Derivative in the analogous setting in Lemma A.2 of \cite{BezemekSpiliopoulosAveraging2022}, with the regularity result for the Lions derivatives and mixed spacial and Lions derivatives of the Poisson Equation appearing there as Lemma A.5. Note that in both \cite{RocknerFullyCoupled} and \cite{BezemekSpiliopoulosAveraging2022}, the $y$ components of the coefficients are not assumed to be periodic, and hence both allow for polynomial growth of the coefficients in $y$. This is inconsequential here due to the fact that we are confining the fast motion to the compact space $\bb{T}^d$, and the arguments still go through. In particular, the needed exponential ergodicity for Equation \eqref{eq:frozeneqn}, still holds even without a recurrence assumption on the drift (see, e.g. Theorem 6.16 in \cite{PS} and Chapter 3 Section 3 in \cite{Bensoussan}).
\end{proof}

\begin{corollary}\label{cor:limitingcoefficientsregularity}
Under assumptions \ref{assumption:LipschitzandBounded}-\ref{assumption:centeringcondition}, $\bar{\beta}$ and $\bar{D}$ as defined in Equation \eqref{eq:limitingcoefficients} and bounded and Lipschitz continuous in $(x,\bb{W}_2)$. Under the additional assumption \ref{assumption:limitinguniformellipticity}, $\bar{B}$ as defined in Equation \eqref{eq:McKeanLimit} is bounded and Lipschitz continuous in $(x,\bb{W}_2)$, and $\bar{B}^{-1}$ exists and is bounded and continuous.
\end{corollary}
\begin{proof}
Boundedness of $\bar{\beta}$ and $\bar{D}$ is immediate, since by assumption \ref{assumption:LipschitzandBounded} and Proposition \ref{prop:Phiexistenceregularity}, the functions which comprise $\beta$ and $D$ are bounded.

From Proposition \ref{prop:Phiexistenceregularity}, we have for $g=\Phi,\nabla_y\Phi,\nabla_x\Phi,$ or $\nabla_x\nabla_y\Phi$ that $\partial_\mu g(x,y,\mu)(\cdot)$ is bounded in $L^2(\R^d,\mu)$  for all $x\in\R^d,y\in\bb{T}^d,\mu\in\mc{P}_2(\R^d)$. This implies that $g(x,y,\cdot)$ is Lipschitz continuous in $\bb{W}_2$ for each $x\in\R^d$ and $y\in\bb{T}^d$ by Remark 5.27 in \cite{CD}. Similarly, $\nabla_x g(x,y,\mu)$ is uniformly bounded in in $x$, $y$, and $\mu$, so $g(\cdot,y,\mu)$ is Lipschitz continuous in $x$ for each $y\in\bb{T}^d,$ and $\mu\in\mc{P}_2(\R^d)$. Then, by Assumption \ref{assumption:LipschitzandBounded} and the aforementioned boundedness of all terms appearing in $\beta$ and $D$, $\beta(\cdot,y,\cdot)$ and $D(\cdot,y,\cdot)$ are jointly Lipschitz in $(x,\bb{W}_2)$ for each $y$. The Lipschitz continuity of $\bar{\beta}$ and $\bar{D}$ now follows as in Lemma A.6 of \cite{BezemekSpiliopoulosAveraging2022} (see also Lemma 3.2 ii) in \cite{RocknerFullyCoupled}).

Under the additional assumption \ref{assumption:limitinguniformellipticity}, we note that the mapping which takes a positive-definite matrix to its unique positive-definite square root is Fr\'echet differentiable up to arbitrary order, with all derivatives being bounded on sets of uniformly bounded, uniformly positive definite matrices (see Equation (6) in \cite{MatrixRoot}). Thus, as the composition of the bounded, Lipschitz continuous mappings $M\mapsto \sqrt{M}$ and $(x,\mu)\mapsto \bar{D}(x,\mu)$, $\bar{B}$ is itself bounded and Lipschitz continuous. Moreover, the matrix inverse is bounded and continuous on sets of uniformly bounded, uniformly positive definite matrices, so similarly $\bar{B}^{-1}$ is bounded and continuous.
\end{proof}

%\end{thebibliography}
\end{document}